\documentclass[12pt,twoside]{article}
\usepackage{authblk}

\usepackage[includeheadfoot,margin=2.54cm]{geometry}
\usepackage[dvipsnames]{xcolor}
\usepackage{latexsym,amssymb}
\usepackage{amsmath,amsthm,amsxtra,amsbsy,accents}

\usepackage{mathtools}
\mathtoolsset{centercolon} 
\usepackage[nointegrals]{wasysym}

\newcommand{\speck}{\boldsymbol{k}}
\renewcommand{\speck}{\mathsf{k}}
\usepackage{dutchcal}
  \let\dutchcalmathcal\mathcal 
\renewcommand{\speck}{\dutchcalmathcal k}

\usepackage{amsfonts}
  \let\amsfontsmathcal\mathcal
\usepackage{eucal}
\usepackage{mathrsfs} 
\usepackage{tikz-cd}
\usepackage[font=small,margin=\parindent]{caption}
\usepackage{longtable}

\usepackage{enumitem} 
\usepackage{pdflscape}
\usepackage{afterpage}

\usepackage{booktabs,tabularx}

\usepackage{lmodern,microtype} 

\usepackage{upgreek}
\newcommand\moc{\upomega}

\usepackage{bm} 

\usepackage{bbm}

\usepackage{tikz-cd}

\newmuskip\halfthinmuskip
\newcommand{\ip}{\mathord{\mskip\halfthinmuskip\cdot\mskip\halfthinmuskip}}

\DeclareMathAlphabet{\mathup}{OT1}{\familydefault}{m}{n}
\newcommand{\dx}[1]{\mathop{}\!\mathup{d} #1}
\newcommand{\dxx}[2]{\mathop{}\!\mathup{d} #1\mathup{d} #2}
\newcommand{\dxxx}[3]{\mathop{}\!\mathup{d} #1\mathup{d} #2\mathup{d} #3}

\newcommand{\pderiv}[3][]{\frac{\mathop{}\!\mathup{d}^{#1} #2}{\mathop{}\!\mathup{d} #3^{#1}}}

\DeclarePairedDelimiter{\abs}{\lvert}{\rvert}
\DeclarePairedDelimiter{\norm}{\lVert}{\rVert}
\DeclarePairedDelimiter{\bra}{(}{)}
\DeclarePairedDelimiter{\pra}{[}{]}
\DeclarePairedDelimiter{\set}{\{}{\}}
\DeclarePairedDelimiter{\skp}{\langle}{\rangle}

\makeatletter
\newcommand{\customlabel}[2]{%
   \protected@write \@auxout {}{\string \newlabel {#1}{{#2}{\thepage}{#2}{#1}{}} }%
   \hypertarget{#1}{}
}
\makeatother

\numberwithin{figure}{section}

\def\calD{{\mathcal D}} \def\calE{{\mathcal E}} \def\calF{{\mathcal F}}
\def\calG{{\mathcal G}} \def\calH{{\mathcal H}} \def\calI{{\mathcal I}}
  
\def\calM{{\mathcal M}} \def\calN{{\mathcal N}} 
\def\calP{{\mathcal P}}  \def\calR{{\mathcal R}}

\def\calY{{\mathcal Y}} \def\calZ{{\mathcal Z}}
  
\def\rmd{{\mathrm d}} \def\rme{{\mathrm e}} 
  
\def\rmj{{\mathrm j}}

\def\rmD{{\mathrm D}}

\def\rmY{{\mathrm Y}} 

\def\sfa{{\mathsf a}} \def\sfb{{\mathsf b}} \def\sfc{{\mathsf c}}

\def\sfm{{\mathsf m}}  
  \def\sfr{{\mathsf r}}
  
\def\sfv{{\mathsf v}} \def\sfw{{\mathsf w}} \def\sfx{{\mathsf x}}
\def\sfy{{\mathsf y}} \def\sfz{{\mathsf z}}

  \def\sfC{{\mathsf C}}
 \def\sfE{{\mathsf E}} \def\sfF{{\mathsf F}}
 \def\sfH{{\mathsf H}} 
  \def\sfL{{\mathsf L}}
\def\sfM{{\mathsf M}} \def\sfN{{\mathsf N}} 
  \def\sfR{{\mathsf R}}
\def\sfS{{\mathsf S}}  
\def\sfV{{\mathsf V}}  \def\sfX{{\mathsf X}}
 \def\sfZ{{\mathsf Z}}

\def\scrD{{\mathscr  D}} \def\scrE{{\mathscr  E}} 
 \def\scrH{{\mathscr  H}} \def\scrI{{\mathscr  I}}
  
 \def\scrN{{\mathscr  N}} 
  \def\scrR{{\mathscr  R}}
  
\def\scrV{{\mathscr  V}}

 \def\bbE{{\mathbb E}}

\def\bbP{{\mathbb P}}  
\def\bbS{{\mathbb S}}

\newcommand{\Expectation}{\bbE}

\newcommand\wt{\widetilde}

\def\dd{\mathrm{d}}
\let\e\varepsilon
\let\eps\varepsilon
\newcommand{\CB}{\mathrm{C}_{\mathrm{b}}}

\newcommand\hj{\wh \jmath}
\newcommand\hu{\wh u}

\newcommand\hnodes{\wh \nodes}
\newcommand\hnodesT{\wh \nodes_{T}}
\newcommand\hpi{\wh \pi}
\newcommand\htheta{\wh \theta}

\newcommand\wh{\widehat}
\newcommand\hF{\wh F}
\newcommand\hp{\wh p}

\newcommand{\hUpsilone}{\wh \Upsilon_\e}
\newcommand{\hUpsilon}{\wh \Upsilon}
\newcommand{\ol}[1]{\overline{#1}}
\newcommand{\R}{\mathbb{R}}
\newcommand{\N}{\mathbb{N}}

\renewcommand\div{\mathop{\mathrm{div}}\nolimits}

\newcommand\dual[2]{\langle #1,#2 \rangle}
\newcommand\Dual[2]{\left\langle #1,#2 \right\rangle}
\newcommand{\CE}{\mathrm{CE}}
\newcommand{\hrho}{\wh \rho}

\newcommand\cev[1]{\overleftarrow{#1}}
\let\ds\displaystyle
\def\longrightharpoonup{\relbar\joinrel\rightharpoonup}

\newcommand{\ee}{\rme}
\newcommand{\fast}{\mathup{fast}}
\newcommand{\term}{\mathup{term}}
\newcommand{\Chain}{\mathup{Chain}}
\newcommand{\Series}{\mathup{Series}}

\newcommand{\weight}{\gamma}
\newcommand{\vol}{\mathup{vol}}

\newcommand{\Pot}{\mathup{Pot}}
\newcommand{\Upwind}{\mathup{Upwind}}
\newcommand{\SG}{\mathup{SG}}

\newcommand{\PartSum}{\mathup{Z}}

\newcommand{\ona}{\dnabla}
\newcommand{\odiv}[1]{\mathop{\overline{\mathrm{div}}}\nolimits#1}
\newcommand{\dnabla}{\overline\nabla}
\newcommand\nodes{\sfV}
\newcommand\edges{\sfE}
\newcommand\adiv{\mathop{\mathsf{div}}\nolimits}
\makeatletter
\newcommand{\anabla}{{\mathpalette\a@nabla\relax}}
\newcommand\a@nabla[2]{%
	\setbox\z@=\hbox{$\m@th#1\bigtriangledown$}%
	\ht\z@.7\ht\z@
	\raise\dp\z@\box\z@
}
\makeatother

\newcommand{\gdiv}{\odiv}
\newcommand{\gnabla}{\ona}

\newcommand{\weakto}{\rightharpoonup}
\newcommand{\weaksto}{\stackrel{*}{\rightharpoonup}} 
\newcommand{\longweaksto}{\stackrel{*}{\longrightharpoonup}} 
\DeclareMathOperator{\law}{law}
\DeclareMathOperator*{\argmin}{arg\,min}

\DeclareMathOperator{\arsinh}{arsinh}
\DeclareMathOperator{\ProbMeas}{\calP}

\DeclareMathOperator\Prob{Prob}

\DeclareMathOperator{\diag}{diag}

\DeclareMathOperator{\sign}{sign}

\DeclareMathOperator{\supp}{supp}
\newcommand\KR{{\mathsf K\kern -1.3 pt\mathsf R}}
\DeclareMathOperator{\Int}{Int}
\DeclareMathOperator{\capacity}{cap}
\DeclareMathOperator{\dom}{dom}
\DeclareMathOperator{\RelEnt}{\mathcal H}
\newcommand{\GenEnt}{\sfH}

\DeclareMathOperator{\CCs}{\mathcal{C}}
\DeclareMathOperator{\Contact}{\mathscr{C}}

\DeclareMathOperator{\Generator}{\mathscr{A}}
\DeclareMathOperator{\RateFunc}{\mathscr{L}}
\newcommand{\Redge}{\sfR}
\DeclareMathOperator{\redge}{\sfr}

\newcounter{step}

\usepackage{ifthen}
\newlength{\leftstackrelawd}
\newlength{\leftstackrelbwd}
\def\leftstackrel#1#2{\settowidth{\leftstackrelawd}%
{${{}^{#1}}$}\settowidth{\leftstackrelbwd}{$#2$}%
\addtolength{\leftstackrelawd}{-\leftstackrelbwd}%
\leavevmode\ifthenelse{\lengthtest{\leftstackrelawd>0pt}}%
{\kern-.5\leftstackrelawd}{}\mathrel{\mathop{#2}\limits^{#1}}}

\usepackage{dsfont}
\newcommand{\bONE}{\mathds{1}}

\makeatletter
\DeclareRobustCommand{\cev}[1]{%
  {\mathpalette\do@cev{#1}}%
}
\newcommand{\do@cev}[2]{%
  \vbox{\offinterlineskip
    \sbox\z@{$\m@th#1 x$}%
    \ialign{##\cr
      \hidewidth\reflectbox{$\m@th#1\vec{}\mkern4mu$}\hidewidth\cr
      \noalign{\kern-\ht\z@}
      $\m@th#1#2$\cr
    }%
  }%
}
\makeatother

\newtheorem{theorem}{Theorem}[section]
\newtheorem{prop}[theorem]{Proposition}

\newtheorem{cor}[theorem]{Corollary}

\newtheorem{lemma}[theorem]{Lemma}

\newtheorem{Ftheorem}[theorem]{Formal Theorem}

\newtheorem{assumption}[theorem]{Assumption}
\theoremstyle{definition}
\newtheorem{definition}[theorem]{Definition}
\newtheorem{property}[theorem]{Property}

\usepackage{textcomp}
\newtheorem{remarkx}[theorem]{Remark}
\newtheorem{examplex}[theorem]{Example}
\newenvironment{remark}
{\pushQED{\qed}\remarkx} 
{\popQED\endremarkx}
\newenvironment{example}
{\pushQED{\qed}\examplex} 
{\popQED\endexamplex}

\numberwithin{equation}{section}

\usepackage{fancyhdr}
\usepackage{datetime}

\newcounter{ex}
\def\theex{\Alph{ex}}

\usepackage[colorlinks=true, linkcolor=black, citecolor=black, urlcolor=black]{hyperref}

\newcommand{\NEW}[1]{{#1}}

\begin{document}

\title{Cosh gradient systems and tilting}
\author{Mark A. Peletier\thanks{\href{mailto:m.a.peletier@tue.nl}{m.a.peletier@tue.nl}}}
\affil{Department of Mathematics and Computer Science\\
Institute for Complex Molecular Systems (ICMS)\\
Eindhoven University of Technology}
\author{Andr\'e Schlichting\thanks{\href{mailto:a.schlichting@uni-muenster.de}{a.schlichting@uni-muenster.de}}}
\affil{Institute for Analysis and Numerics\\
University of M\"unster}
\date{\empty}
\maketitle

\begin{abstract}	
We review a class of gradient systems with dissipation potentials of hyperbolic-cosine type. We show how such dissipation potentials emerge in large deviations of jump processes, multi-scale limits of diffusion processes, and more. We show how the exponential nature of the cosh derives from the exponential scaling of large deviations and arises implicitly in cell problems in multi-scale limits. 

We discuss in-depth the role of \emph{tilting} of gradient systems. Certain classes of gradient systems are \emph{tilt-independent}, which means that changing the driving functional does not lead to changes of the dissipation potential. Such tilt-independence separates the driving functional from the dissipation potential, guarantees a clear modelling interpretation, and gives rise to strong notions of gradient-system convergence. 

We show that although in general many gradient systems are tilt-independent, certain cosh-type systems are not. We also show that this is inevitable, by studying in detail the classical example of the Kramers high-activation-energy limit, in which a diffusion converges to a jump process and the Wasserstein gradient system converges to a cosh-type system. We show and explain how the tilt-independence of the pre-limit system is lost in the limit system. 
This same lack of independence can be recognized in classical theories of chemical reaction rates in the chemical-engineering literature. 

We illustrate a similar lack of tilt-independence in a discrete setting. For a class of `two-terminal' fast subnetworks, we give a complete characterization of the dependence on the tilting, which strongly resembles the classical theory of equivalent electrical networks. 
\end{abstract}


\setcounter{tocdepth}{2}
\tableofcontents

\bigskip
\noindent\textbf{Notation.}
\nobreak
\begin{small}
\begin{longtable}{lll}
$\nodes, \edges$ & abstract topological spaces of nodes and edges & Sec.~\ref{ss:intro-evol-eqs} \\
$\sfC(\cdot)$, $\sfC^*(\cdot)$ & Legendre pair of cosh-type functions &\eqref{eqdef:C-C*}\\
$\sfC(\cdot|\cdot)$ & perspective function of $\sfC$ & Sec.~\ref{ss:perspective-function}\\
$\capacity_{\sfa\sfb}$ & effective capacity of terminal graph & Sec. \ref{sss:structure-of-tilt-dependence}, ~\eqref{eqdef:effective:capacity} \\
$\CCs$ & combination of $\sfC$ and $\sfC^*$ & \eqref{eqdef:CCs-intro}, Sec.~\ref{ss:cell-formula}\\
$\CE(0,T)$ & pairs $(\rho,j)$ satisfying the continuity equation & Sec.~\ref{sss:ct-eq}\\
$\anabla$, $\nabla$, $\ona$ & (abstract, classical, graph) gradient & Sec.~\ref{ss:intro-evol-eqs} \\
$\scrD_{\nodes}, \scrD_{\edges}$ & abstract domain and codomain of $\anabla$  & Sec.~\ref{ss:intro-evol-eqs} \\
$\calD^T(\rho,j)$ & dissipation functional & Def.~\ref{def:EDP:sol}\\
$\adiv$, $\div$, $\odiv$ & (abstract, classical, graph) divergence & Sec.~\ref{ss:intro-evol-eqs} \\
$\rmD_2$, $\partial_2$ 
  & derivative and subdiff. with respect to $2^{\mathrm {nd}}$ arg. 
  & \eqref{eq:flux-GS-intro} \\
$\calE$ & driving energy function of gradient system & Def.~\ref{def:GradSystCE} \\
$\sfF$ & class of tilts & Def.~\ref{def:GS-with-tilting}, \eqref{eq:Kramers:tilts}\\
$\sfF_{\mathrm{Pot}}$ & class of \emph{potential} tilts of the form $\calF(\rho) = \int F\rho$ & \eqref{eqdef:MCs:potential-tilts}\\
$\eta(\cdot|\cdot)$ & relative-entropy density & \eqref{eqdef:eta}\\
$\speck_{\sfx\sfy},k_{\sfx\sfy}$ & edge weights  & Ex.~\ref{ex:heat-flow}\\
$\RelEnt(\cdot|\cdot)$ & relative entropy & \eqref{eqdef:RelEnt}\\
$\calI^T(\rho,j)$ & energy-dissipation functional & Def.~\ref{def:EDP:sol} \\
$\RateFunc(\rho,j)$ & ``level-2.5'' flux large-deviation rate function & Sec.~\ref{sss:cosh-from-ldp-intro}, \ref{ss:ldp-GS}, \ref{ss:tilting-LDPs}\\
$\sfL$ & rate function density & \eqref{eqdef:L} \\
$\Lambda(\cdot,\cdot),\Lambda_{-1}(\cdot,\cdot)$ &(harmonic) logarithmic mean & \eqref{eqdef:log-mean}, \eqref{eqdef:HarmLogMean}\\
$\calM(X)$, $\calM_{\geq0}(X)$  & finite (non-negative) Borel measures on $X$   & \\
$\calM_{\mathrm{loc}}(X)$, $\calM_{\mathrm{loc},\geq0}(X)$  & locally finite (non-negative) Borel measures on $X$   & Sec.~\ref{sss:ct-eq} \\
$\ProbMeas^+(\Omega)$, $\ProbMeas(\Omega)$ & (strictly positive) probability measures in $\Omega$ \\
$\calR(\rho,j)$ & dissipation potential as functions of $(\rho,j)$  & Def.~\ref{def:GradSystCE} \\
$\scrR(\rho,\dot\rho)$ & dissipation potential as functions of $(\rho,\dot\rho)$ & Rem.~\ref{rem:non-ct-eq-GS}\\
$\Redge$ & dissipation potential edge density & Rem.~\ref{rem:unitlting:tilt-GS}, Sec.~\ref{ss:char-tilt-indep-GS}\\
$T_\# \mu$ & push-forward of $\mu$ under $T$, defined as $\mu\circ T^{-1}$\\
$\sfx\sfy\in\edges$ & edge from $\sfx\in\nodes$ to $\sfy\in\nodes$ & Sec.~\ref{ss:intro-evol-eqs}\\
$X_T$  & $[0,T]\times X$\\
$\calY$ & $\calY = H^{-1}(0,1)\times \R\times\R$ & Sec.~\ref{ss:cell-formula}\\
$\PartSum$ & partition sum normalizing probability measure \\
\end{longtable}
\end{small}

\newpage

\section{Introduction}
\label{s:intro}

A \emph{gradient system} is an evolution equation with an additional variational structure. The gradient systems of this paper involve evolution equations that can formally be written as 
\begin{equation}
\label{eq:intro-GF}
\dot \rho = \rmD_\xi \scrR^*(\rho,-\rmD_\rho \calE(\rho)),
\end{equation}
which  describe the time evolution of the state $\rho\in \sfZ$ in terms of a driving functional~$\calE:\sfZ \to \R$ (typically an energy or an entropy) and a \emph{dissipation potential} $\scrR^* = \scrR^*(\rho,\xi)$. Here $\rmD_\xi$ and $\rmD_\rho$ denote derivatives with respect to $\xi$ and $\rho$. Such a gradient system is characterized by the triple $(\sfZ,\calE,\scrR^*)$.

Gradient structures are important aspects of evolution equations for a number of reasons. In a mathematical sense, when an evolution equation has a gradient structure, then this can be used to construct weak-solution concepts~\cite{MielkeTheilLevitas02,Dal-MasoDeSimoneMora06,AmbrosioGigliSavare08,MielkeRoubicek15,PeletierRossiSavareTse22}, well-posedness results (ibidem), \emph{a priori} and \emph{a posteriori} estimates~\cite{OttoVillani2000,BakryGentilLedoux14}, methods to prove stability and convergence~\cite{SandierSerfaty04,Serfaty11,Mielke16a}, and many other tools and properties. 

Often the components $\calE$ and $\scrR^*$ also have a clear physical interpretation. As illustrated by~\eqref{eq:intro-GF}, the derivative $\rmD_\xi \scrR^*(\rho,\,\cdot\,)$ maps the `force' $-\rmD_\rho \calE(\rho)$ to the `rate of change' $\dot\rho$. This force--to--rate map is similar in spirit to Fourier's law that maps temperature gradient to heat flux, or viscosity relations that map stress to strain rate (see Section~\ref{ss:KineticRelations}). This interpretation gives rise to the modelling of systems via the gradient structure~\cite{Doi11,PeletierVarMod14TR} to guarantee thermodynamic consistency.

\medskip

Historically, dissipations $\scrR$ were quadratic functions of $\xi$, which generate linear force--to--rate maps; subsequent extensions allowed for rate-independent systems, for which $\scrR^*$ is a singular, $\{0,\infty\}$-valued convex function of $\xi$.

In contrast, this paper revolves around a class of dissipation potentials $\scrR^*$ for which the relation $\xi\mapsto \scrR^*(\rho,\xi)$ is of the \emph{hyperbolic cosine} type. These dissipations all have as main building block the Legendre pair of functions $\sfC,\sfC^*:\R\to\R$, 
\begin{equation}
\label{eqdef:C-C*}\sfC(s) := 2s\log \frac{s+ \sqrt{s^2+4}}{2} - 2\sqrt{s^2+4} + 4, 
\qquad
\sfC^*(\xi) := 4\Bigl(\cosh\frac\xi 2 - 1\Bigr).
\end{equation}
Both functions are smooth, strictly convex, even, and superlinear at infinity, and their derivatives are each other's inverse:
\begin{align}\label{eq:C'}
	\sfC'(s) = 2\arsinh\frac s2, \qquad (\sfC^*)'(\xi) = 2\sinh\frac\xi2.
\end{align}

The first purpose of this paper is to present a number of examples and properties of gradient systems with such `cosh-type' dissipations. In doing so we also aim to explain how this type of structure arises in multi-scale limits and large-deviation results. 

The second purpose arises from modelling considerations. In the modelling of gradient systems there is a common assumption that the driving functional and dissipation can be chosen independently, or put differently, that changing the energy (by `tilting' it) does not change the dissipation. This assumption, called `tilt-independence', is essential in the modelling of gradient systems, since it allows the modeller to make independent choices for `energetics' and `kinetics'. It has been observed~\cite{FrenzelLiero21,MielkeMontefuscoPeletier21} that such tilt-independence need not be conserved through singular limits. In this paper we study tilt-independence and tilt-dependence in  more detail; we show how for cosh-type gradient structures tilt-\emph{dependence} is natural. This has major {consequences} for the modelling of gradient systems.

\paragraph{Outline.}

The structure of this paper is  unconventional. Section~\ref{s:intro} is a high-level review of gradient systems with cosh-type dissipation, together with a detailed description of tilting. 
All the gradient systems in this paper are evolution equations for \emph{measures}, and we start in Section~\ref{ss:intro-evol-eqs} by describing the `continuity-equation' structure of such equations. In Section~\ref{ss:flux-GS-intro} we define the corresponding cosh-type gradient structures. In these two sections we also introduce a number of example systems that serve as illustrations. 

In Section~\ref{ss:contraction-intro} we explain how the functions $\sfC$ and $\sfC^*$ arise in multi-scale limits and in large-deviation principles; this clarifies the origin of these functions. In Section~\ref{s:ModellingTilting} we discuss tilting, tilt-independence, and tilt-dependence. In Section~\ref{ss:partial_conclusion} we provide a partial conclusion of the review part and in Section~\ref{ss:history} we give some pointers to the history of gradient structures in general and cosh dissipations in particular.

\medskip

The rest of the paper can be considered as appendices for Section~\ref{s:intro}, in which we give details, explore specific issues, and deepen the discussion. 

In Section~\ref{s:GS} we recall a number of concepts in the theory of gradient systems,  formalize the concept of a continuity equation, and define various types of gradient-system convergence. In Section~\ref{s:props-of-CCstar} we prove some properties of the dual pair $(\sfC,\sfC^*)$. Section~\ref{s:Kramers} is devoted to the first main example, the high-activation-energy limit of Kramers' equation; we give a full proof of this result, for a very broad class of tilting functions, in order to illustrate the impact of tilt-dependence. In Section~\ref{s:thin-membrane} we revisit the example of a thin membrane~\cite{LieroMielkePeletierRenger17,FrenzelLiero21} and similarly investigate the tilt-dependence of the limiting system.

In Section~\ref{s:tilting} we further explore the various meanings of \emph{tilting}: besides the tilting of energies and gradient systems mentioned in Section~\ref{s:ModellingTilting} we also discuss tilting of random variables and Markov processes in Section~\ref{ss:tilting:MCs}, and of sequences of these with a large-deviation principle in Section~\ref{ss:tilting-LDPs}.
In the case of simple jump processes (Example~\ref{ex:heat-flow} below) we first characterize all possible detailed balance jump rates in Section~\ref{ss:mod-jump-rates}.
In this setting, we observe that chemical reactions lead to tilt-dependent gradient systems in Section~\ref{ss:case-study-chem-reactions-tilt-dependent}.
The main result in Section~\ref{ss:char-tilt-indep-GS} characterizes and discusses tilt-independent gradient structures for a large range of detailed-balance jump rates. 

Finally, in Section~\ref{s:ha} we study a class of networks in a fast-reaction limit; in the limit these networks behave as a two-terminal subnetwork, with quasistatic equilibration inside the subnetwork. We show how the dynamics of such subnetworks are described by gradient systems, and that the limit process makes these gradient systems tilt-dependent.

%
%
%
%
%

\subsection{Evolution equations for measures}
\label{ss:intro-evol-eqs}

The evolution equations of this paper all are evolution equations for \emph{measures}. Such equations appear in many models. One class of examples describes the evolution of concentrations of molecules, animals, people, agents, or other objects; such concentrations are non-negative and often conserved. These are naturally represented by non-negative  measures. 

A second class of measure-valued evolutions describes the evolution in time of the law of a Markov process. This law is a non-negative measure on the state space, and this measure satisfies the Forward-Kolmogorov evolution equation. 

\medskip

Evolution equations for measures $\rho$ can often be written in the abstract `continuity-equation' form
\begin{equation}
\label{eq:ct-eq-abstract-intro}
\partial_t \rho + \adiv j = 0.
\end{equation}
Here $j$ is a `flux', and the operator $\adiv$ depends on the context; it can be the usual divergence on $\R^d$, but may also be a more general operator. 
This equation characterizes the admissible evolutions of $\rho$ as those that are generated by fluxes $j$ through the divergence $\adiv$. 

In this paper we will use a terminology suggested by the graph context. 
We  consider a set of `nodes' or `vertices' $\nodes$ and a set of `edges' $\edges$.
The unknown $\rho$ will always be a non-negative measure on $\nodes$, and $j$ a (possibly signed) measure on $\edges$. 
The sets $\nodes$ and $\edges$  are connected by a gradient operator $\anabla: \scrD_\nodes \to \scrD_\edges$, where $\scrD_\nodes$ and $\scrD_\edges$ are suitable topological spaces of functions on nodes $\nodes$ and edges $\edges$. The negative dual of $\anabla$ is $\adiv: \scrD_\edges' \to \scrD_\nodes'$, which maps elements of the dual of $\scrD_\edges$ to the dual of $\scrD_\nodes$, i.e.
\begin{equation*}
{}_{\scrD_\nodes}\dual{f}{\adiv j}_{\scrD_\nodes'} = -{}_{\scrD_\edges}\dual{\anabla f}{j}_{\scrD_\edges'}, 
\qquad\text{for } f\in \scrD_\nodes , \ j\in \scrD_\edges'.
\end{equation*}
This allows us to treat local, discrete, and nonlocal equations with the same structure. 
\NEW{Note that the duality structure of the abstract gradient and divergence implies `natural' or `no-flux' boundary conditions.}

In the terminology of port-Hamiltonian systems, the combination of $\adiv$ and $\anabla$ defines a \emph{Dirac structure} of the `transformer' type (see e.g.~\cite[Ch.~XV]{Paynter61} or~\cite[\S 2.2.1]{Van-Der-SchaftJeltsema14}).
In the terminology of Variational Modelling~\cite{PeletierVarMod14TR}, the flux $j$ is an element of the `process space', and the operator $\adiv$ is the corresponding `process--to--tangent' map.

\begin{table}[!ht]
\footnotesize\centering
	\begin{tabular}{c|cccc}\toprule
		& $\nodes$ & $\edges$ & $\anabla f$ & $\adiv j$ \\
		\midrule 
		\textbf{\ref{ex:heat-flow}} &finite set 
		  & $\nodes \times \nodes$
		  & $(\gnabla f)_{\sfx\sfy} := f_{\sfy} - f_{\sfx}$ 
		  & $\ds\gdiv j(\sfx) := \sum_{\sfy:\, \sfx\sfy\in \edges } j_{\sfx\sfy}-\sum_{\sfy:\, \sfy\sfx\in \edges}j_{\sfy\sfx}$ \\[5\jot]
		\textbf{\ref{ex:FP}} & $\R^d$ & $\R^d \times \set{1,\dots,d}$ 
		  & $(\partial_{i}f(x))_{i=1}^d $ 
		  & $\ds \div j(x) := \sum_{i=1}^d \partial_{i} j(x)$\\[5\jot]
		\textbf{\ref{ex:chem-reactions}} & finite set 
		  & $\R_{\geq0}^\nodes \times \R_{\geq0}^\nodes$
		  & $\ds (\gnabla f)_{\alpha\beta} :=\sum_{\sfy\in \nodes} \beta_\sfy f_{\sfy} - \sum_{\sfx \in \nodes} \alpha_\sfx f_{\sfx}$ 
		  & $\ds\gdiv j(\sfx):= 
		    \sum_{\alpha\beta\in \edges} ( \alpha_{\sfx} -  \beta_{\sfx}) j_{\alpha\beta}$ \\[5\jot]
		\textbf{\ref{ex:Boltzmann}} & $\R^d$ & $\R^{2d} \times \R^{2d}$ 
		  & $\left\{\hbox{\begin{tabular}{@{}l@{}}
		    $(\gnabla f)_{(\sfv\sfv_*)(\sfv'\sfv'_*)} :=$\\
		       $\quad f_{\sfv'} + f_{\sfv'_*} - f_{\sfv} - f_{\sfv_*}$
		       \end{tabular}}\right\}$
		  & $\ds\gdiv j(\sfv) := \iiint \sum_{i=0}^3 j^{\sigma_i} (\sfv\dx\sfv_*,\dxx{\sfv'}{\sfv'_*})$\\
		\bottomrule
	\end{tabular}
	\caption{Examples of continuity equations in various settings.  In row  \textbf{\ref{ex:Boltzmann}}, we consider the involutions $\set*{\sigma_i}_{i=0}^3$ ($\sigma_0$ being the identity) of $(\sfv\sfv^*)(\sfv'\sfv'_*)$ together with a suitable choice of the sign (see Example~\ref{ex:Boltzmann} on the Boltzmann equation below).}
	\label{table:divs}
\end{table}

\bigskip

\noindent
\textbf{\textit{Examples.}}
In this paper we use a number of examples to illustrate various aspects of cosh gradient systems, including the continuity-equation structure. We start below by defining the evolution equations and the corresponding gradient and divergence operators. The gradient structure of each of these systems will then be introduced in Section~\ref{ss:flux-GS-intro}.

The reader who wishes to get a quick overview  might study only Examples~\ref{ex:heat-flow} and~\ref{ex:FP}, and leave the others to a later stage.

\bigskip

\noindent

\refstepcounter{ex}
\subsubsection*{Example~\theex. Heat flow on a graph.}\label{ex:heat-flow}
Let $\nodes$ be a finite set; we denote its elements by $\sfx$ and $\sfy$, and for a function $f:\nodes\to\R$ we write both $f(\sfx)$ and $f_\sfx$ for the value at $\sfx$. We set $\edges := \nodes\times\nodes$ for the set of all directed edges, and write $\sfx\sfy$ for the edge from $\sfx$ to $\sfy$. Given a set of non-negative \emph{rates} $\kappa_{\sfx\sfy}$ of transitions from $\sfx$ to $\sfy$, the \emph{heat flow} on $\nodes$ is the equation for the measure $\rho\in \calM_{\geq0}(\nodes)$,
\begin{equation}
\label{eq:heat-flow-intro}
\partial_t \rho_\sfx  = \sum_{\sfy\in\nodes} \bigl[ \rho_\sfy\kappa_{\sfy\sfx} - \rho_\sfx \kappa_{\sfx\sfy}\bigr].
\end{equation}
This equation is of the form~\eqref{eq:ct-eq-abstract-intro}, where $\adiv$ is  the \emph{graph divergence} $\odiv$,
\begin{equation*}
\adiv := \gdiv, \qquad (\gdiv j)(\sfx) := 
\sum_{\sfy:\, \sfx\sfy\in \edges } j_{\sfx\sfy}-\sum_{\sfy:\, \sfy\sfx\in \edges}j_{\sfy\sfx},
\end{equation*}
and the corresponding \emph{graph gradient} operator $\gnabla$ is given  by
\begin{equation}
\label{eqdef:gnabla}
(\gnabla f)_{\sfx\sfy} := f_{\sfy} - f_\sfx,
\qquad\text{for } f:\nodes \to \R.
\end{equation}
This choice of gradient and divergence is given in row~\ref{ex:heat-flow} of Table~\ref{table:divs}.

Equation~\eqref{eq:heat-flow-intro} is recovered by choosing the flux $j=j(\rho)$ 
as
\[ 
j_{\sfx\sfy} := \rho_\sfx \kappa_{\sfx\sfy},
\]
which represent unidirectional fluxes from $\sfx$ to $\sfy$. 
Alternatively, the same equation also is recovered by choosing $j$ as the skew-symmetrization of the unidirectional fluxes, that is
\begin{equation}
\label{eq:heat-flow-flux-intro}
j_{\sfx\sfy} := \frac{\rho_\sfx \kappa_{\sfx\sfy} - \rho_\sfy\kappa_{\sfy\sfx}}{2}.
\end{equation}
Both choices lead to the same divergence $\odiv j$, and therefore to the same equation~\eqref{eq:heat-flow-intro}.

Equation~\eqref{eq:heat-flow-intro} can also be interpreted as the evolution equation for the law of a Markov jump process $X_t$ on $\nodes$ with jump rates $\kappa$ (see Section~\ref{sss:cosh-from-ldp-intro}) or a simple monomolecular reaction network (see~Example~\ref{ex:chem-reactions}). 

Although we assume above that $\nodes$ is a finite set, `heat flows' can be defined on general topological spaces, since they only require a well-defined kernel. Such equations and their various gradient structures are investigated in more detail in~\cite{Erbar2014,PeletierRossiSavareTse22}.

\refstepcounter{ex}
\subsubsection*{Example~\theex. The Fokker-Planck equation.}\label{ex:FP}
Fix  $D\in \R^{d\times d}$ and $V:\R^d\to\R$. The diffusion-advection equation
\begin{equation}
	\label{eq:FP-intro}
	\partial_t \rho(x) - \div \bigl[D(\nabla \rho(x) + \rho(x)\nabla V(x))\bigr] =0 \qquad \text{in }\R^d
\end{equation}
can be written as
\begin{equation}
	\label{eq:FP-rho-j}
	\partial_t \rho + \div j=0, \qquad
	j = -D(\nabla \rho + \rho \nabla V).
\end{equation}
Equation~\eqref{eq:FP-intro} has many interpretations, such as the evolution of concentrations of diffusing chemical species or of the law of a diffusion $X_t$ on $\R^d$ (see~\cite{Risken84} and Section~\ref{s:Kramers}).

\emph{Notation.} A gradient $\nabla f$ is usually considered to be a vector-valued function, and similarly `$\div$' is usually defined on vector-valued functions. For later comparison we slightly abuse notation by considering $\nabla$ and $\div $ as operating with scalar-valued objects, by setting
\[
\nabla f(x,i) := (\nabla f(x))_i = \partial_{x_i} f(x)
\qquad\text{and}\qquad
(\div j)(x) := \sum_{i=1}^n \partial_{x_i}j(x,i).
\]
Defining $\edges := \R^d\times\{1,\dots,d\}$, it follows that $\nabla$ maps a scalar function on  $\nodes := \R^d$ to a scalar function on~$\edges$, and similarly $\div$ maps a scalar function on $\edges$ to a scalar function on~$\nodes$. 
Following the same convention the matrix $D\in\R^{d\times d}$ maps  a vector field $v = v(x,i)$  (a scalar function on~$\edges$) to a new vector field as 
\[
(Dv)(x,i) := \sum_{j=1}^d D_{ij} v(x,j).
\]
With these conventions, the formulation~\eqref{eq:FP-rho-j} can be interpreted both in its classical meaning and in the modified meaning above, and the two interpretations are  equivalent.

\refstepcounter{ex}
\subsubsection*{Example~\theex. Chemical reactions.}\label{ex:chem-reactions}
Let $\nodes = \{1,\dots,n\}$ represent a set of species $X_1,\dots,X_n$, and consider  reactions between these species of the form 
\begin{equation}
\label{eq:reaction-equation}
\alpha_1 X_1 + \dots + \alpha_nX_n
\xleftrightharpoons{\quad}
\beta_1 X_1 + \dots + \beta_n X_n \qquad\text{for } \alpha\beta\in \edges .
\end{equation}
Here $\alpha,\beta \in \R_{\geq0}^\nodes$ are stoichiometric coefficients and the set of all reactions of this type is a finite subset $\edges\subset \R_{\geq0}^\nodes \times \R_{\geq0}^\nodes$.  

To each reaction $\alpha\beta\in \edges$ we associate fluxes $j_{\alpha\beta}$. The change in the concentration $\rho_\sfx$ of the species $\sfx\in \nodes$ is given in terms of the continuity equation
\begin{equation}
	\label{eq:cont-eq-chem-reactions}
 	\partial_t \rho_\sfx + (\adiv j)(\sfx) := \partial _t \rho_\sfx + \sum_{\alpha\beta\in\edges} ( \alpha_{\sfx} -  \beta_{\sfx}) j_{\alpha\beta}  = 0.
\end{equation}
The corresponding gradient is 
\[
(\anabla f)_{\alpha\beta} =  (\gnabla f)_{\alpha\beta} :=\sum_{\sfy\in \nodes} \beta_\sfy f_{\sfy} - \sum_{\sfx \in \nodes} \alpha_\sfx f_{\sfx}.
\]

When all reactions are mono-molecular, i.e.\ both sides in~\eqref{eq:reaction-equation} have exactly one non-zero term, the divergence in~\eqref{eq:cont-eq-chem-reactions} coincides with the graph divergence of Example~\ref{ex:heat-flow}. Hence, one can also view the reaction network~\eqref{eq:reaction-equation} in the general case as a directed weighted hypergraph~\cite{Berge1989,GalloLongoPallottino1993,KlamtHausTheis2009,FlammStadlerStadler2015,JostMulas2019}.
Other interpretations of the reaction mechanism~\eqref{eq:reaction-equation} are coagulation-fragmentation mechanism as found in the Smoluchowski equation~\cite{Smoluchowski1916}, the Becker-D\"oring model~\cite{BD1935,BCP86} or exchange-driven growth models~\cite{NK2003,S-EDG2018}, where the growth of clusters is described by some physical mechanism of recombination and split-up.

\subsubsection*{Example~\ref{ex:FP}+\ref{ex:heat-flow}. Fokker-Planck equations with internal state changes.}
\customlabel{ex:FP-heat-flow}{\ref{ex:FP}+\ref{ex:heat-flow}}

This example highlights how two of the previous examples can be combined, as for instance in the case of Fokker-Planck equations (Example~\ref{ex:FP}) with linear reactions (Example~\ref{ex:heat-flow}). The set $\nodes$ consists of pairs $(x,\sfx)\in \R^d\times \ol\nodes$, where $\ol\nodes$ is a finite set (for instance corresponding to species or internal states), and the set of edges is the disjoint union of `continuous' and `discrete' edges from the two setups:
\begin{equation}\label{defn:Ex:B+A:edges}
	\edges := \edges^{\mathrm c} \sqcup \edges^{\mathrm d} ,
\qquad\edges^{\mathrm c} := \underbrace{\R^d\times \{1,\dots,d\}}_{\text{continuous edges}}\times \ol \nodes, 
\qquad
\edges^{\mathrm d} := \R^d \times \underbrace{\ol\nodes\times \ol\nodes}_{\text{discrete edges}}.
\end{equation}
Note that the node set of the `other' variable is added to the edge set, and plays the role of a parameter.

The gradient operator is the union of the two gradients, 
\begin{equation*}
\anabla f(e) := \begin{cases}
	(\partial_{x_i}f)(x,i,\sfx)  & \text{if } e = (x,i,\sfx)\in \edges^{\mathrm c}	,\\
	(\ona f)(x,\sfx\sfy) & \text{if }e= (x,\sfx,\sfy)\in \edges^{\mathrm d}.
\end{cases}
\end{equation*}
Fluxes defined on $\edges = \edges^{\mathrm c}\sqcup \edges^{\mathrm d}$ similarly take two types of arguments:
\[
j: \edges\to\R, \qquad
j(e) = \begin{cases}
	j^{\mathrm c}(e) & \text{if }e\in \edges^{\mathrm c},\\
	j^{\mathrm d}(e) & \text{if }e\in \edges^{\mathrm d},	
\end{cases}
\]
and the corresponding divergence is the sum of the two parts
\[
(\adiv j)(x,\sfx) = (\div j^{\mathrm c} )(x,\sfx) + (\odiv j^{\mathrm d})(x,\sfx).
\]
The continuity equation then reads
\begin{equation}
	\label{eqdef:div-RD}
	\partial_t \rho(x,\sfx) + 
		(\div j^{\mathrm c})(x,\sfx) + (\odiv j^{\mathrm d})(x,\sfx)
	=0.
\end{equation}
A typical example, which we will also revisit below, is the case of\/ $\ol\nodes = \{\sfa,\sfb\}$:
\begin{subequations}
	\label{eq:RD-intro}	
	\begin{align}
		\partial_t \rho_\sfa (x) &= \div \bigl[D(\nabla \rho_\sfa(x) + \rho_\sfa(x)\nabla V_\sfa(x))\bigr] + \kappa_{\sfb\sfa}\rho_\sfb(x) - \kappa_{\sfa\sfb} \rho_\sfa(x),\\
		\partial_t \rho_\sfb (x) &= \div \bigl[D(\nabla \rho_\sfb(x) + \rho_\sfb(x)\nabla V_\sfb(x))\bigr] + \kappa_{\sfa\sfb}\rho_\sfa(x) - \kappa_{\sfb\sfa} \rho_\sfb(x).
	\end{align}
\end{subequations}

\subsubsection*{Example~\ref{ex:FP}+\ref{ex:chem-reactions}. Fokker-Planck equation with chemical reactions.}
\customlabel{ex:FP+chem-reactions}{\ref{ex:FP}+\ref{ex:chem-reactions}}
A similar combination can be made of Fokker-Planck equations with chemical reactions (Example~\ref{ex:chem-reactions}). 

The set $\nodes$ again consists of pairs $(x,\sfx)\in \R^d\times \ol\nodes$, where $\ol\nodes$ is the set of species; the set of edges now is the disjoint union with chemical-reaction edges:
\begin{equation}\label{eqdef:edges:B+C}
\edges := \edges^{\mathrm c} \sqcup \edges^{\mathrm d} ,
\qquad\edges^{\mathrm c} := \underbrace{\R^d\times \{1,\dots,d\}}_{\text{continuous edges}}\times \ol \nodes, 
\qquad
\edges^{\mathrm d} := \R^d \times \underbrace{\R_{\geq0}^{\ol\nodes} \times \R_{\geq0}^{\ol\nodes}}_{\text{chem-reaction edges}}.
\end{equation}
The gradient operator is the union of the two gradients, 
\[
\anabla f(e) := \begin{cases}
	(\partial_{x_i}f)(x,i,\sfx)  & \text{if } e = (x,i,\sfx)\in \edges^{\mathrm c}	,\\
	(\ona f)(x,\alpha\beta) & \text{if }e= (x,\alpha\beta)\in \edges^{\mathrm d}.
\end{cases}
\]
Fluxes $j\in \scrD_\edges'$ consist of two components $j^{\mathrm c}\in\scrD_{\edges^{\mathrm c}}$ and $j^{\mathrm d}\in\scrD_{\edges^{\mathrm d}}$, and the corresponding divergence is the sum of the two parts
\[
(\adiv j)(x,\sfx) = (\div j^{\mathrm c} )(x,\sfx) + (\odiv j^{\mathrm d})(x,\sfx) := \sum_{i=1}^d \partial_{x_i} j^{\mathrm c}((x,i),\sfx) + \sum_{\alpha\beta\in \edges^{\mathrm d}} \bra*{\alpha_{\sfx} - \beta_{\sfx}} j(x,\alpha\beta).
\]
The continuity equation takes the same form as~\eqref{eqdef:div-RD} and reads
\begin{equation*}
	\partial_t \rho(x,\sfx) + 
		(\div j^{\mathrm c})(x,\sfx) + (\odiv j^{\mathrm d})(x,\sfx)
	=0.
\end{equation*}

\refstepcounter{ex}
\subsubsection*{Example~\theex. The spatially homogeneous Boltzmann equation.}\label{ex:Boltzmann}

The spatially homogeneous Boltzmann equation describes the evolution of a density $f(v)$ of particles with velocity $v\in \nodes := \R^d$. The density evolves by collisions that exchange energy and momentum according to 
\begin{equation}
\label{eq:Boltzmann}
\partial_t f(v) = Q(f,f)(v) := \iint\limits_{\R^d\times\bbS^{d-1}} \bigl[f(v')f(v_*')\kappa(v',v_*',-n) - f(v)f(v_*)\kappa(v,v_*,n)\big]  \dxx n{v_*}.
\end{equation}
Here the pair $(v',v_*')$ is determined by $(v,v_*,n)\in \R^d\times\R^d\times\bbS^{d-1}$ through the perfect elastic collision law
\begin{equation}
\label{eq:collision-law}
\NEW{v'= v'(v,v_*,n) := v - \bra*{(v-v_*)\cdot n}_+ n, \qquad v_*' = v'_*(v,v_*,n):= v_* + \bra*{(v-v_*)\cdot n}_+ n,}
\end{equation}
where $n$ is the normal of the plane of specular reflection parametrized by the  vector pointing from  particle $v$ to particle $v_*$ and $(\cdot)_+$ is the positive part. In particular, we consider an elastic collision only to happen if $(v-v_*) \cdot n >0$. Note, that this implies the involutionary identities
\begin{equation}\label{eq:Boltz:involution}
	v'\bigl(v'(v,v_*,n),v'_*(v,v_*,n),-n\bigr) = v\quad\text{and}\quad 	v'_*\bigl(v'(v,v_*,n),v'_*(v,v_*,n),-n\bigr) = v_* ,
\end{equation}
that express the fact that after switching the roles of incoming and outgoing velocities we arrive at the same configuration.
The collision law~\eqref{eq:collision-law} gives rise to the conservation of velocity and momentum
\begin{equation}\label{eq:Boltz:conservation:law}
	v + v_* = v' + v'_*  \qquad\text{and}\qquad \abs{v}^2 + \abs{v_*}^2 = \abs{v'}^2 + \abs{v'_*}^2 . 
\end{equation}
In particular, the mass, center of mass, and second moment of $f$ are preserved along the Boltzmann evolution:
\begin{equation}\label{eq:Boltz:conservations}
	\pderiv{}{t} \set*{\int f(v)\dx v, \int v f(v) \dx v, \int \abs{v}^2 f(v) \dx v} = 0 .
\end{equation}

The collision kernel $\kappa(v,v_*,n)$ characterizes the rate at which a pair of particles with incoming velocities $v$ and $v_*$ transforms into a pair of particles with outgoing velocities $v'$ and $v_*'$ undergoing an elastic collision with angle parameter $n$. A typical expression is $\kappa(v,v_*,n)=\bra*{(v-v_*)\cdot n}_+$, which models hard-sphere collisions~\cite{Boltzmann1872,Boltzmann1964}.
By indistinguishability of the particles, we demand the symmetry condition $\kappa(v_*,v,n) = \kappa(v,v_*,-n)$.

We use the slightly uncommon form of the flux in~\eqref{eq:Boltzmann} and collision law~\eqref{eq:collision-law} for two reasons. First, we want to highlight how the Boltzmann equation resembles an uncountable set of chemical reactions of the form
\begin{equation*}
	\set*{v} + \set*{v_*} \xrightleftharpoons[\kappa(v',v'_*,-n)]{\kappa(v,v_*,n)} \set*{v'} + \set*{v'_*} \qquad\qquad (v,v_*,n) \in \R^d\times \R^d \times \bbS^{d-1},
\end{equation*}
where the backward reaction is thanks to~\eqref{eq:Boltz:involution}.
Secondly, we want to stress that the  classical form of the collision term 
\[
  Q(f,f)= \iint\limits_{\R^d\times\bbS^{d-1}} \bigl[f(v')f(v_*') - f(v)f(v_*)\big] \kappa(v,v_*,n)  \dxx n{v_*}
\]
follows from an assumption  that the collision kernel $\kappa$ satisfies  the symmetry property
\begin{equation}\label{eq:Boltz:kapppa-symmetry}
\kappa(v,v_*,n) = \kappa(v',v_*',-n) \qquad\text{for all } (v,v_*,n) \in \R^d\times \R^d \times \bbS^{d-1}.
\end{equation}
This property can be considered an assumption of \emph{detailed balance}; see the continuation of Example~\ref{ex:Boltzmann} in the next section. 

We now construct two continuity-equation structures for the Boltzmann equation, which each has their advantages. 
The collisions can be interpreted as jumps from one point in $\nodes\times\nodes$ to another, where $\nodes := \R^d$. This suggests a first formulation close to that of Example~\ref{ex:heat-flow}, with edge space $\edges := (\nodes\times\nodes)\times (\nodes\times\nodes)$ and edge elements $(v_1v_2)(v_3v_4)$, where we for notational ease also write $v_1v_2\;v_3 v_4\in \edges$. A second continuity-equation structure follows instead from the suggestion that is implicit  in~\eqref{eq:Boltzmann}, to parametrize the edges by $\wh \edges := \nodes\times\nodes\times\bbS^{d-1}$, with edge elements $vv_*n\in\wh\edges$.

For the first formulation, in terms of the edge set $\edges := (\nodes\times\nodes)\times (\nodes\times\nodes)$, we use a measure-valued framework and set $\scrD_\nodes=C_{\mathrm b}(\nodes)$ and $\scrD_\edges=C_{\mathrm b}(\edges)$. The `graph' gradient $\ona: C_{\mathrm b}(\nodes)\to C_{\mathrm b}(\edges)$ and divergence $\odiv:\calM(\edges)\to \calM(\nodes)$ are given by
\begin{subequations}
\begin{align}
\label{eqdef:grad-Boltzmann}
(\ona g)(v_1v_2\;v_3v_4)  &:= g(v_3) + g(v_4) - g(v_1)-g(v_2) ,
&& \text{for } v_1v_2\;v_3v_4\in \edges;\\
\int_\nodes g(v)(\gdiv j)(\dx v) &:= \iiiint_\edges g(v) \sum_{k=0}^3 (-1)^k j^{\sigma_k} (\dxx{v}{v_2}\,\dxx {v_3}{v_4}), &&\text{for } g \in C_{\mathrm b}(\nodes) .
\label{eqdef:div-Boltzmann}
\end{align}
\end{subequations}
Here $j^\sigma$ is the push-forward of $j$ under the permutation $\sigma$ of four elements, defined by 
\[
\dual{j^\sigma}g := \dual j {g\circ \sigma} \qquad\text{for any }g.
\]
The $\sigma_k$ in~\eqref{eqdef:div-Boltzmann} are the four permutations of a pair of pairs:
\[
\sigma_0(ab\;cd) = ab\;cd,\quad 
\sigma_1(ab\;cd) = cd\;ab,\quad 
\sigma_2(ab\;cd) = ba\;cd,\quad 
\sigma_3(ab\;cd) = cd\;ba.
\]

For the second formulation, in terms of the edge set $\wh \edges := \nodes\times\nodes\times\bbS^{d-1}$, note that the collision law~\eqref{eq:collision-law} defines a map $T:\widehat\edges\to \edges$   by
\begin{equation}\label{eqdef:Boltzmann-map-T}
T(v,v_*,n ) := \bigl(v,v_*,v'(v,v_*,n),v_*(v,v_*,n)\bigr)=
\bigl( v, v_*, v-n \cdot (v-v_*)\, n , v_* + n \cdot (v-v_*) \,  n \bigr)  .
\end{equation}
Again using a measure-valued framework, with $\scrD_\nodes=C_{\mathrm b}(\nodes)$ and $\scrD_{\wh\edges}=C_{\mathrm b}(\wh\edges)$, we obtain the gradient $\wh\nabla$ by pulling back the  gradient $\gnabla$ in~\eqref{eqdef:grad-Boltzmann} under the map $T$,
\begin{align*}
		(\widehat\nabla g)(v v_*\;n)  &:= g(v'(v,v_*,n)) + g(v_*'(v,v_*,n)) - g(v)-g(v_*) ,
		&& \text{for }(v, v_*,n)\in \widehat\edges 
\end{align*}
The divergence $\wh\div$ of any $\widehat\jmath\in \calM(\widehat \edges)$ is obtained as negative dual of $\widehat\nabla$ and can be characterized by using suitable push-forwards of $\widehat\jmath$ under the symmetries $\zeta: v \leftrightarrow v_*$ and switching from incoming to outgoing velocities $\tau:(v,v_*) \leftrightarrow (v',v_*')$. The result is
\begin{align*}
(\wh\div \wh j)(v) &= \iint\limits _{\R^d\times \bbS^{d-1}} (\wh \jmath^{,\mathrm{bw}} - \wh \jmath ^{\mathrm{\,fw}})(v,\dxx {v_*} n), \qquad\text{with}\\
\wh \jmath^{\mathrm{\,fw}} &= \wh\jmath + \zeta_\#\wh\jmath, \qquad
\wh \jmath^{\mathrm{\,bw}} = \tau_\# \wh \jmath^{\mathrm{\,fw}}.
\end{align*}

There is a fundamental difference between the two continuity equations: the first one allows jumps from any $v_1v_2\in \nodes\times\nodes$ to any $v_3v_4\in \nodes\times\nodes$, while the second formulation restricts the jumps  to those of the form~\eqref{eq:collision-law}; in particular, the first continuity equation may violate conservation of mass and energy~\eqref{eq:Boltz:conservation:law} while the second preserves both.

\medskip
Equation~\eqref{eq:Boltzmann} can be written in both forms,  
\begin{equation*}
\partial_t f + \odiv j = 0 \qquad\text{or}\qquad
\partial_t f + \wh\div \wh\jmath = 0.
\end{equation*}
Equation~\eqref{eq:Boltzmann} reduces to the second form above with the choice
\begin{equation}
\label{eq:choice-wh-jmath-Boltzmann}
\wh\jmath (\dxxx v {v_*} n) = f(v)f(v_*) \kappa(v,v_*,n) \dxxx v {v_*} n.
\end{equation}
Note that by~\eqref{eq:collision-law}  collisions only happen if $(v-v_*)\cdot n >0$; for each $(v,v_*,n)$ the expression 
\[
\wh \jmath^{,\mathrm{bw}} - \wh \jmath ^{\mathrm{\,fw}}
=   \tau_\# \wh\jmath + \tau_\#(\zeta_\#\wh\jmath) - \wh\jmath - \zeta_\#\wh\jmath
\]
therefore reduces to $\tau_\# \wh\jmath - \wh\jmath  $ if $(v-v_*)\cdot n>0$ and to $\zeta_\#\wh\jmath -  \zeta_\#\wh\jmath$ if $(v-v_*)\cdot n<0$; this yields the expression~\eqref{eq:Boltzmann}. The first form of the continuity equation is then obtained by taking $j := T_\# \wh \jmath$ for the $\wh\jmath$ given in~\eqref{eq:choice-wh-jmath-Boltzmann}.

There is a massive literature on the Boltzmann equation; see e.g. the monographs~\cite{Cercignani88,CercignaniIllnerPulvirenti1994,Saint-Raymond09}.

%
%
%

\subsection{Fluxes generated by dissipation potentials}
\label{ss:flux-GS-intro}

In this paper we focus on those evolution equations for measures that are generated by {gradient systems}. The following definition makes this precise.

\begin{definition}[Gradient system in continuity-equation format]\label{def:GradSystCE}
A gradient system in continuity-equation format is a quintuple $(\sfV,\sfE,\anabla,\calE,\calR)$ with the following properties:
\begin{enumerate}
\item $\sfV$ and $\sfE$ are topological spaces, and $\scrD_\nodes$ and $\scrD_\edges$ are topological spaces of functions on~$\nodes$ and~$\edges$;
\item $\anabla$ is a linear map from  $\scrD_\nodes$ to $\scrD_\edges$,  with negative dual $\adiv$;
\item $\calE:\calM_{\geq0}(\sfV)\to\R$;
\item \label{def:GradSystCE:DP}
$\calR$ is a \emph{dissipation potential}, which means that for each $\rho\in \calM_{\geq0}(\sfV)$, $j \mapsto \calR(\rho,j)$ is a convex lower semicontinuous functional on $\scrD'_\edges$ satisfying $\min \calR(\rho,\cdot) = \calR(\rho,0) = 0$. 
\end{enumerate}
\end{definition}
The dual dissipation potential $\calR^*$ is defined by duality,
\[
\calR^*(\rho,\Xi) := \sup_{j}\set*{ \dual j \Xi - \calR(\rho,j) },
\qquad \text{for } \Xi\in \scrD_\edges.
\]
and is again a convex lower semicontinuous function with $\min \calR^*(\rho,\cdot) = \calR^*(\rho,0) = 0$. Given $\calR^*$,  $\calR$ can similarly be reconstructed from $\calR^*$ by duality.

The subdifferentials $\partial_2\calR(\rho,j)$ and $\partial_2\calR^*(\rho,
\Xi)$ play an important role in this paper. We use $\partial_2$ to indicate the subdifferential in terms of the second variable, which is a set; when the function is differentiable and the subdifferential reduces to a single point, we write $\rmD_2$ instead. Therefore, if $\calR^*(\rho,\cdot)$ is differentiable, then  $\partial_2 \calR^*(\rho,\Xi) = \{\rmD_2\calR^*(\rho,\Xi)\}$.
We have the following equivalence:
\begin{equation}
\label{eq:subdiff-characterization}
j\in \partial_2 \calR^*(\rho,\Xi)
\quad\Longleftrightarrow\quad
\Xi\in \partial_2 \calR(\rho,j)
\quad\Longleftrightarrow\quad
\calR(\rho,j) + \calR^*(\rho,\Xi) = \dual j\Xi .
\end{equation}
This formulation requires that $\calE$ is differentiable; we discuss this in Section~\ref{ss:formal-rigorous}.

\medskip

The \emph{gradient flow in continuity equation format} defined by such a gradient system is the evolution equation obtained by the combination  
\begin{equation}
\label{eq:flux-GS-intro}	
\partial_t \rho + \adiv j = 0
\qquad \text{and}\qquad
j\in  \partial_2 \calR^*\bigl(\rho;\mathopen-\anabla \rmD\calE(\rho)\bigr).
\end{equation}

\begin{remark}[Relation with other gradient system concepts]
\label{rem:non-ct-eq-GS}
The gradient systems in continuity-equation format that we use in this paper may appear to be a subclass of more general concepts of gradient systems. In~\cite{Mielke16a,LieroMielkePeletierRenger17}, for instance, a gradient system is a triple $(\mathscr Z, \mathscr E,\mathscr R)$ consisting of a space~$\mathscr Z$, a functional $\mathscr E$ on~$\mathscr Z$, and a dissipation potential~$\mathscr R$; then $\mathscr R(z,\cdot)$ is a convex functional of `time derivatives' (tangents) $\dot z\in\mathscr Z$, and the dual $\mathscr R^*(z,\cdot)$ is a functional of `forces' (co-tangents) $\xi\in \mathscr Z^*$.

In fact, the continuity-equation structure is not more restrictive, but only more explicit. To see this, note first that the trivial choice $\edges=\nodes$, $\anabla = \mathrm{Id}$ turns a gradient system $(\mathscr Z, \mathscr E,\mathscr R)$ into a gradient system $(\mathscr Z, \mathscr Z, \mathrm{Id}, \mathscr E,\mathscr R)$ in continuity-equation format without actually changing it; this shows that imposing the continuity-equation format does not entail any loss of generality. 

More importantly, however, the continuity-equation structure often leads to simpler modelling arguments and more explicit expressions.
To understand this, note that a gradient system $(\sfV,\sfE,\anabla,\calE,\calR)$ in continuity-equation format can be reformulated in the more general form $(\mathscr Z, \mathscr E,\mathscr R)$ by setting
\begin{equation}
\label{eq:relation-classical-conteq}
\mathscr Z := \calM_{\geq0}(\sfV), \qquad \mathscr E(\rho) := \calE(\rho), \qquad \mathscr R^*(\rho,\xi) := \calR^*(\rho,\anabla \xi).
\end{equation}
Using well-known characterizations of duals of compositions (e.g.~\cite[11.23]{Rockafellar-Wets98}  or~\cite[Prop.~6.1]{MaasMielke20}) one formally finds
\begin{equation}
\label{eq:scrR-calR}
\mathscr R(\rho,\dot \rho) = \inf_{\adiv  j	 = -\dot \rho} \calR(\rho,j).
\end{equation}
This expression for $\scrR(\rho,\dot \rho)$ illustrates the main reason for adopting the continuity-equation structure: while $\calR(\rho,j)$ often is explicit, as in the many examples below, the  minimization in~\eqref{eq:scrR-calR} typically can not be performed analytically, and $\scrR(\rho,\dot\rho)$ only has the implicit definition~\eqref{eq:scrR-calR}.

Specifying $\calR(\rho,j)$ instead of $\scrR(\rho,\dot\rho)$ also simplifies modelling of gradient systems: the `flux' $j$ can be a general combination of the various processes that contribute to $\dot \rho$, such as diffusion and reaction in a reaction-diffusion equation (see Examples~\ref{ex:FP-heat-flow} and~\ref{ex:FP+chem-reactions}  below), with corresponding separate dissipation potentials $\calR$. The continuity-equation structure also corresponds to the `process space' approach introduced  in~\cite{PeletierVarMod14TR}.
\end{remark}

\begin{remark}[Freedom in choosing the set of edges $\edges$]
For a given node set $\nodes$, one can often choose different sets of edges $\edges$ while generating the same evolution equation. This is implicit in the formulation~\eqref{eq:relation-classical-conteq}, where $\calR^*(\rho,\anabla \xi)$ depends on the choice of $\edges$ and of $\anabla$, but the potential $\scrR^*(\rho,\xi)$ does not. Practically, given two sets of edges $\edges_{a,b}$ and corresponding operators $\anabla_{\!a,b}$ and $\adiv_{a,b}$, the following two evolutions are formally identical,
\begin{equation}
\label{eq:two-evolutions}
\partial_t \rho = \begin{cases}
 -\adiv_a \rmD_2\calR^*_a(\rho,-\anabla	_{\!a} \rmD\calE(\rho))\\
 -\adiv_b \rmD_2\calR^*_b(\rho,-\anabla	_{\!b} \rmD\calE(\rho)),
\end{cases}
\end{equation}
whenever $\calR_a^*(\rho,\anabla_{\!a} \varphi) = \calR_b^*(\rho,\anabla_{\!b}\varphi)$ for all $\varphi\in C_{\mathrm c}^\infty(\nodes)$. This follows by remarking that 
\[
\pderiv{}{t} \calR^*_a(\rho,\anabla_{\!a}(\varphi + t\psi))\Big|_{t=0}
= \Big\langle {\rmD_2\calR^*_a(\rho,\anabla_{\!a}\varphi)}\,,\,{\anabla_{\!a}\psi}\Big\rangle
= \Big\langle {-\adiv_a\rmD_2\calR^*_a(\rho,\anabla_{\!a}\varphi)}\,,\,{\psi}\Big\rangle,
\]
and therefore the two evolutions in~\eqref{eq:two-evolutions} coincide as soon as $\calR_a^*(\rho,\anabla_{\!a} \varphi) = \calR_b^*(\rho,\anabla_{\!b}\varphi)$. 

The symmetric appearance of the operators $\anabla_{\!a}$ and $\adiv_{a}$ in the map
\[ 
  \varphi \mapsto -\adiv_a\rmD_2\calR^*_a(\rho,\anabla_{\!a}\varphi)
\]
is very similar to the classical linear-algebra formula $O^T\!\!AO$ for the transformation of a matrix $A$ under an orthogonal coordinate change $x=Oy$. In fact, for symmetric positive $A$  one can recover $x\mapsto Ax$ as the derivative of $x\mapsto \tfrac 12 \|x\|_A^2$, and    $y \mapsto O^T\!\!AOy$ as the derivative of the transformed quadratic function $y\mapsto \tfrac12 \|Oy\|_A^2=\tfrac 12 \|x\|_A^2$. In this sense the joint appearance of $\anabla_{\!a}$ and $\adiv_a$ can be recognized as a direct consequence of the `coordinate transform' characterized by $\anabla_{\!a}$. \"Ottinger~\cite{Ottinger19} discusses this aspect in more detail.
\end{remark}

\begin{remark}[Freedom in scaling the energy $\calE$]	
\label{rem:scaling-energy-in-GS}
For any gradient system there is an inherent scale invariance of the following type: for $\lambda>0$ define
\[
\calE^\lambda := \lambda \calE, \qquad
\calR^\lambda(\rho,j) := \lambda \calR(\rho,j),
\qquad\text{and}\qquad
\calR^{\lambda,*}(\rho,\Xi) := \lambda \calR^*(\rho,\lambda^{-1} \Xi).
\]
Then $\calE^\lambda$, $\calR^\lambda$, and $\calR^{\lambda,*}$ generate the same evolution equation as $\calE$, $\calR$, and $\calR^{*}$.
\end{remark}

\bigskip
\noindent 
\textbf{\textit{Examples.}}
Each of the examples in the previous section can be written as a gradient system in continuity-equation format under a suitable detailed-balance condition.

\subsubsection*{Example~\ref{ex:heat-flow}. Heat flow on a graph (continued).} 

We fix a measure $\pi\in\ProbMeas^+(\nodes)$ and a edge function $\speck:\edges\to[0,\infty)$ with the symmetry $\speck_{\sfx\sfy} = \speck_{\sfy\sfx}$.
We then define the following gradient system: 
\begin{subequations}
	\label{eq:ER-heat-flow-intro}
	\begin{align}
		\calE(\rho) &:= \RelEnt(\rho|\pi), \label{eqdef:E-jump}\\
\calR^*(\rho,\Xi) &:= \sum_{\sfx\sfy\in \edges} \sigma_{\sfx\sfy}(\rho)\sfC^*(\Xi_{\sfx\sfy}) \qquad\text{with}\qquad
\sigma_{\sfx\sfy}(\rho) := \frac12 {\speck_{\sfx\sfy}}\sqrt {\rho_\sfx\rho_\sfy},
\label{eqdef:Rstar-jump}\\
\calR(\rho,j) &:=  \sum_{\sfx\sfy\in \edges} \sigma_{\sfx\sfy}(\rho)\sfC\Bigl(\frac{j_{\sfx\sfy}}{\sigma_{\sfx\sfy}(\rho)}\Bigr).
		\label{eqdef:R-jump}
	\end{align}
\end{subequations}
The Legendre pair of functions $\sfC$ and $\sfC^*$ was defined in~\eqref{eqdef:C-C*}.
Here and later in this paper the relative entropy $\RelEnt$ of two measures $\mu,\nu\in\calM_{\geq 0}(\Omega)$ on a set $\Omega$, e.g. a subset of $\nodes$, $\R^d$ or some discrete set, is defined as
\begin{equation}\label{eqdef:RelEnt}
	\RelEnt(\mu|\nu) := \begin{cases}
		\ds \int_\Omega \bra*{ u \log u - u + 1}\, d\nu  
		& \text{if }\mu \ll \nu \text{ and }\mu = u \nu,\\
		+\infty & \text{otherwise.}	
	\end{cases}
\end{equation}

We calculate the flux $j$ generated by $\calE$ and $\calR^*$ in~\eqref{eq:ER-heat-flow-intro}. 
From $\bigl(\rmD \calE(\rho)\bigr)_\sfx = \log(\rho_\sfx/\pi_\sfx)$ and   the definition of the graph gradient~\eqref{eqdef:gnabla} we have
%
%
\[
\bra*{\gnabla \rmD \calE(\rho)}_{\sfx\sfy} = \log \frac{\rho_\sfy}{\pi_\sfy} - \log \frac{\rho_\sfx}{\pi_\sfx}.
\]
Observing by~\eqref{eq:C'} that
\[
\bra[\big]{\rmD_2  \calR^*(\rho,\Xi)}_{\sfx\sfy} =  \speck_{\sfx\sfy} \sqrt{\rho_\sfx \rho_\sfy} \, \sinh\bra*{ \tfrac{1}{2}\Xi_{\sfx\sfy}} ,
\]
we find for the flux $j$ generated by $\calR^*$ and $\calE$,
\begin{align}
	\notag	
	j_{\sfx\sfy} = \bra[\big]{\rmD_2 \calR^*\bra*{\rho,-\gnabla\rmD\calE(\rho)}}_{\sfx\sfy} 
	&=  \speck_{\sfx\sfy} \,\sqrt{\rho_\sfx\rho_\sfy} \; \sinh\Bigl(-\frac12 \log\frac{\rho_\sfy}{\pi_\sfy} + \frac12 \log \frac{\rho_\sfx}{\pi_\sfx}\Bigr)\\
	&= \frac12  \speck_{\sfx\sfy} \sqrt{\rho_\sfx\rho_\sfy} \, \biggl( \sqrt{\frac{\rho_\sfx\pi_\sfy}{\rho_\sfy\pi_\sfx}} - \sqrt{\frac{\rho_\sfy\pi_\sfx}{\rho_\sfx\pi_\sfy}}\biggr)\notag\\
	&= \frac12  \speck_{\sfx\sfy} 
		\biggl( \rho_\sfx\sqrt{\frac{\pi_\sfy}{\pi_\sfx}}  
		- \rho_\sfy\sqrt{\frac{\pi_\sfx}{\pi_\sfy}}\biggr).
	\label{eq:deriv-eq-heatflow-intro}
\end{align}

Comparing~\eqref{eq:deriv-eq-heatflow-intro} with~\eqref{eq:heat-flow-flux-intro} we observe that the two coincide if the rates $\kappa$ satsisfy the relation
\begin{equation}
\label{eq:rel-kappa-speck}	
\kappa_{\sfx\sfy} = \speck_{\sfx\sfy} \sqrt{\frac{\pi_\sfy}{\pi_\sfx}} 
\qquad \text{for all }\sfx\sfy\in \edges.
\end{equation}
This leads to the following definition.
\begin{definition}[Detailed balance]\label{def:DBC}
The pair $(\kappa,\pi)$ with $\kappa:\edges\to[0,\infty)$ and $\pi\in\ProbMeas^+(\nodes)$  satisfies the condition of \emph{detailed balance} if there exists  a symmetric function $\speck:\edges\to[0,\infty)$ such that~\eqref{eq:rel-kappa-speck} holds, or equivalently, if 
\begin{equation}
\label{eq:def:graph:DBC}
\pi_{\sfx} \kappa_{\sfx\sfy} = \pi_{\sfy} \kappa_{\sfy\sfx}
\qquad \text{for all }\sfx\sfy\in \edges.
\end{equation}
For future use we set $k_{\sfx\sfy} := \pi_{\sfx} \kappa_{\sfx\sfy} = \pi_{\sfy} \kappa_{\sfy\sfx}$, and we have the relations
\[
\speck_{\sfx\sfy} = \sqrt{\kappa_{\sfx\sfy}\kappa_{\sfy\sfx}}
\qquad \text{and}\qquad
k_{\sfx\sfy} = \speck_{\sfx\sfy} \sqrt{\pi_\sfx\pi_\sfy}.
\]
\end{definition}

Consequently, if $(\kappa,\pi)$ satisfies detailed balance, then the gradient system~\eqref{eq:ER-heat-flow-intro} induces the equation~\eqref{eq:heat-flow-intro} with rates $\kappa$ given by~\eqref{eq:rel-kappa-speck}.
When equation~\eqref{eq:heat-flow-intro} is interpreted as evolution equation for the law of a Markov process~$X_t$ on $\nodes$, detailed balance is equivalent to the reversibility of the Markov process.  

\begin{remark}[Single and double directed edges]\label{rem:edges-single-double}
In the context of flows on graphs $(\nodes,\edges)$, there are two common choices for the set of edges $\edges$:
\begin{enumerate}
\item Double directed edges, which come in forward-backward pairs:\\
If $\sfx\sfy\in\edges$ then also $\sfy\sfx\in \edges$;
\item Single directed edges, which characterize the connection in terms of  one directed edge:\\
If $\sfx\sfy\in\edges$ then $\sfy\sfx$ is \emph{not} an element of $\edges$.
\end{enumerate}
See Figure~\ref{fig:double-single-edges} for an example.
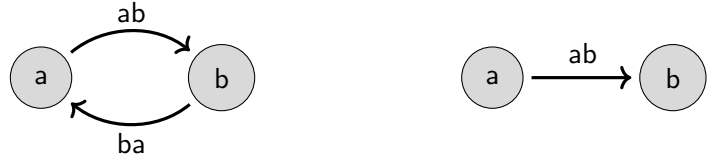
\begin{figure}[ht]
\centering
{\small
\begin{tikzpicture}[scale=0.6, shorten >=3pt, shorten <=3pt]
\tikzstyle{every node}=[draw,shape=circle,fill=black!15];
\node (vv1) at (10,0) {$\ \sfa\ $};
\node (vv3) at (14,0) {$\ \sfb\ $};
\tikzstyle{every node}=[];

\draw[->, very thick] (vv1) to [out=40, in=140] (vv3);
\draw[->, very thick] (vv3) to [out=220, in=-40] (vv1);

\node[above] at (12,1) {${\sfa\sfb}$};
\node[below] at (12,-1) {${\sfb\sfa}$};

\tikzstyle{every node}=[draw,shape=circle,fill=black!15];
\node (zv1) at (20,0) {$\ \sfa\ $};
\node (zv3) at (24,0) {$\ \sfb\ $};
\tikzstyle{every node}=[];

\draw[->, very thick] (zv1) to  (zv3);

\node[above] at (22,0.1) {${\sfa\sfb}$};
\end{tikzpicture}
}
\caption{Two versions of a simple two-state system, with two edges or a single edge.}
\label{fig:double-single-edges}
\end{figure}
In the first case of double directed edges sometimes the space of fluxes is limited  to either only \emph{non-negative} unidirectional fluxes ($j_{\sfx\sfy}\geq0$) or only \emph{skew-symmetric} fluxes ($j_{\sfx\sfy} = -j_{\sfy\sfx}$). Skew-symmetric fluxes can be considered as {net} flux over the edge, and can be obtained by considering the skew-symmetrization of the unidirectional fluxes.

The continuity-equation setup of this paper is meaningful for any choice of edges and any choice of flux, and in the developments below we use both single and double directed edges, and both non-negative and skew-symmetric fluxes. Our choices are determined by what is the most adequate environment in the given situation. 
\end{remark}

\begin{remark}[Conversion between single and double edges]\label{rem:edges-conversion}
Sometimes it is useful to be able to convert between single-edged and double-edged descriptions. Consider the simple two-state system of Figure~\ref{fig:double-single-edges}. Consider a gradient system for the double-edged setup on the left, 
\begin{alignat*}2
\edges = \{\sfa\sfb,\sfb\sfa\}, \qquad
\calR^*(\rho,\Xi) &:= \sum_{\sfx\sfy\in \{\sfa\sfb,\sfb\sfa\}} \sigma_{\sfx\sfy}(\rho)\sfC^*(\Xi_{\sfx\sfy}),
&\qquad&\text{for }\Xi_{\sfa\sfb},\Xi_{\sfb\sfa}\in \R,\\
\calR(\rho,j) &:=  \sum_{\sfx\sfy\in \{\sfa\sfb,\sfb\sfa\}} \sigma_{\sfx\sfy}(\rho)\sfC\Bigl(\frac{j_{\sfx\sfy}}{\sigma_{\sfx\sfy}(\rho)}\Bigr),
&\qquad&\text{for }j = (j_{\sfa\sfb},j_{\sfb\sfa})\in \R^2,
\end{alignat*}
and a similar system for the right-hand side, 
\begin{alignat*}2
\ol \edges = \{\sfa\sfb\}, \qquad
\ol \calR^*(\rho,\xi) &:= \ol\sigma(\rho)\sfC^*(\xi),
&\qquad&\text{for }\xi\in \R,\\
\ol\calR(\rho,\rmj) &:=  \ol\sigma(\rho)\sfC\Bigl(\frac{\rmj}{\sigma(\rho)}\Bigr),
&\qquad&\text{for }\rmj \in \R.
\end{alignat*}
The following lemma describes the recipe for converting one into the other.
\begin{lemma}
Assume that 
\[
\ol\sigma(\rho)  = \sigma_{\sfa\sfb}(\rho) + \sigma_{\sfb\sfa}(\rho).
\]
If\/ $\Xi_{\sfa\sfb} = -\Xi_{\sfb\sfa} = \xi$, then $\ol\calR^*(\rho,\xi) = \calR^*(\rho,\Xi)$. Similarly, if $j_{\sfa\sfb} = -j_{\sfb\sfa} = \rmj$, then $\ol\calR(\rho,\rmj)= \calR(\rho,j)$. 

In addition, if $(\rho,j)$ is a solution of the gradient flow generated by $\edges$ and $\calR$, then $(\rho,\rmj)$ is a solution of the gradient flow generated by $\ol\edges$ and $\ol\calR$, where $\rmj = j_{\sfa\sfb}$. 
\end{lemma}
\noindent
The proof is a simple verification; the correspondence between gradient-flow solutions arises because $\Xi = -\gnabla \rmD \calE(\rho)$ is by construction skew-symmetric.
\end{remark}

\subsubsection*{Example~\ref{ex:FP}. The  Fokker-Planck equation (continued).}
The stationary states of equation~\eqref{eq:FP-intro} are multiples of the probability measure
\begin{equation}\label{eq:def:FP:pi}
\pi (\dx x) := \frac1{\PartSum} \exp(-V(x)) \dx x\qquad\text{with}\qquad \PartSum:= \int \exp(-V(x)) \dx{x} .
\end{equation}
With this $\pi$, we set $\calE(\rho) := \RelEnt(\rho|\pi)$ similar to~\eqref{eqdef:E-jump} in Example~\ref{ex:heat-flow}.
Equation~\eqref{eq:FP-intro} itself takes the form of~\eqref{eq:flux-GS-intro} when we define
\begin{equation}
\calR^*(\rho,\Xi) := \frac12\int_{\R^d}  |\Xi(x)|_{D}^2\,\rho(\dx x), 
\text{where we write }|v|_D^2 := v^T D v.
\label{eqdef:R-FP}
\end{equation}
Indeed, if we assume that $\rho$ is Lebesgue absolutely continuous, then $u = \dx\rho/\dx\pi = \PartSum\ee^{V}\!\rho$, and writing 
\begin{equation}
\label{eq:ex-FP-Xi}	
\Xi = -\nabla \rmD\calE(\rho) = -\nabla \log u = -\nabla V - \nabla \rho/\rho,
\end{equation}
we find that~\eqref{eq:flux-GS-intro} reduces to 
\begin{equation}
\label{eq:deriv-eq-FP}
j =  \rmD_2 \calR^*(\rho,\Xi) = \rho D\Xi =  -D(\nabla\rho + \rho V).
\end{equation}
In this way the Fokker-Planck equation~\eqref{eq:FP-intro} can be written in the form of~\eqref{eq:flux-GS-intro}.
\NEW{Note that the detailed balance condition from Definition~\ref{def:DBC} in the present situation can be rephrased as the solution to the zero-flux equation $j[\rho] = 0$, which thanks to the non-degeneracy of $D$ is equivalent to $\nabla \rho = -\rho V$ and hence $\pi$ from~\eqref{eq:def:FP:pi} is the unique detailed balance probability measure in this case.}

\subsubsection*{Example~\ref{ex:chem-reactions}. Chemical reactions (continued).}

We follow the classical theory of chemical kinetics (see e.g.~\cite{Feinberg1972,Laidler1987,Connors1990,GorbanKarlinZmievskiiDymova2000,Grmela2010}) to specify the rates associated to the reaction~\eqref{eq:reaction-equation} following Arrhenius and assuming mass-action kinetics. For doing so, we associate to each species $\sfx\in\nodes$ a chemical energy $E_{\sfx}$ and to each reaction $\alpha\beta \in \edges\subset \R_{\geq0}^\nodes \times \R_{\geq0}^\nodes$ an absolute activation energy $E^{\mathrm{act}}_{\alpha\beta}$ and a kinetic pre-factor $D_{\alpha\beta}$. For reactions at temperature $T>0$, we obtain the net fluxes $j_{\alpha\beta}$ for $\alpha\beta\in\edges$ as skew-symmetrization of unidirectional fluxes $\vec\jmath$ given by 
\begin{equation}\label{eq:chem-reactions:flux}
	j_{\alpha\beta} := \frac{\vec\jmath_{\alpha\beta} - \vec \jmath_{\beta\alpha}}{2} \qquad\text{with}\qquad \vec\jmath_{\alpha\beta} := D_{\alpha\beta} \exp\bra[\big]{-\upbeta (E^{\mathrm{act}}_{\alpha\beta}-\alpha\ip E)} \rho^{\alpha} ,
\end{equation}
where we use the notation $\upbeta = 1/k_{\mathrm B}T$, $\rho^\alpha = \prod_{\sfx\in \nodes} \rho_\sfx^{\alpha_\sfx}$ and $\alpha\ip E= \sum_{\sfx\in \nodes} \alpha_\sfx E_\sfx$ for the total chemical potential energy of the complex specified by $\alpha\in \R_{\geq0}^\nodes$.

Under symmetry assumptions\footnote{The general conditions are also called \emph{Wegscheider conditions} after~\cite{Wegscheider1901} and are discussed for instance in~\cite[\S 2.2]{MaasMielke20}} on the activation energy $E^{\mathrm{act}}_{\alpha\beta}=E^{\mathrm{act}}_{\beta\alpha}$ and kinetic pre-factor $D_{\alpha\beta}=D_{\beta\alpha}$  the reaction rate system consisting of the continuity equation~\eqref{eq:cont-eq-chem-reactions} and the constitutive law for the fluxes~\eqref{eq:chem-reactions:flux} has a gradient structure. 
We first write the resulting system in a more convenient form. For doing so, we introduce the overall activity of a reaction $\alpha\beta\in \edges$ defined by
\begin{equation*}
	k_{\alpha\beta} :=  D_{\alpha\beta} \ee^{-\upbeta E^{\mathrm{act}}_{\alpha\beta}}.
\end{equation*}
Moreover, we define an equilibrium measure obtained from the chemical potential energy by
\begin{equation*}
 \pi_{\sfx} = \ee^{-\upbeta E_\sfx} .
\end{equation*}
With these definitions, the reaction rate equation consisting of \eqref{eq:cont-eq-chem-reactions} and~\eqref{eq:chem-reactions:flux} becomes
\begin{equation}\label{eq:chemical-reactions:RRE}
	\partial_t \rho_\sfx + \frac{1}{2}\sum_{\alpha\beta\in \edges} k_{\alpha\beta} \bra*{\frac{\rho^\alpha}{\pi^\alpha} -  \frac{\rho^\beta}{\pi^\beta}}( \alpha_{\sfx} -  \beta_{\sfx})  = 0 . 
\end{equation}

The driving energy is given by the Gibbs energy $\calE(\rho) := \RelEnt(\rho|\pi)$, as in~\eqref{eqdef:E-jump}. The dissipation potential for $\rho \in \calM_{\geq0}(\nodes)$, writing $\rho= u \pi$, and $\Xi : \edges \to \R$ is defined by
\begin{equation}\label{eqdef:chemical-reactions:R}
	\calR^*(\rho,\Xi) =  \frac{1}{2}\sum_{\alpha\beta\in \edges} k_{\alpha\beta}\sqrt{ \frac{\rho^\alpha \rho^\beta}{\pi^\alpha \pi^\beta}} \, \sfC^*(\Xi_{\alpha\beta}) . 
\end{equation}

The formal verification of the gradient flow property comes from similar algebraic manipulations as in Example~\ref{ex:heat-flow}. However, the kinetic rates are in general different from those in Example~\ref{ex:heat-flow}, which is reflected in the fact that even for monomolecular reactions, the prefactors of $\sfC^*$ in~\eqref{eqdef:R-jump} and~\eqref{eqdef:chemical-reactions:R} are different: in~\eqref{eqdef:chemical-reactions:R} the stationary measure $\pi$ appears in the prefactor. This appearance points ahead to the problem that we discuss in more detail in Section~\ref{s:ModellingTilting}: since the stationary measure depends on the energy $\calE$, a perturbation (tilting) of the energy changes not only the force $\Xi_{\alpha\beta}$ but also the prefactor $1/\sqrt{\pi^\alpha \pi^\beta}$. In Section~\ref{s:ModellingTilting} we will see that the situation is even worse: the `constant' $k_{\alpha\beta}$ also depends on tilting. 

We obtain from the definition of the gradient in row $\textbf{\ref{ex:chem-reactions}}$ of Table~\ref{table:divs} for a reaction $\alpha\beta\in \edges$ the expression
\begin{equation*}
	\bra*{ \gnabla \rmD \calE(\rho)}_{\alpha\beta} = \sum_{\sfx\in\nodes} \bra*{\beta_\sfx \log \frac{\rho_\sfx}{\pi_\sfx} - \alpha_\sfx \log \frac{\rho_\sfx}{\pi_\sfx}} = \log \frac{\rho^\beta}{\pi^\beta} - \log \frac{\rho^\alpha}{\pi^\alpha} . 
\end{equation*}
From here, the identification of~\eqref{eq:chemical-reactions:RRE} follows by calculating for a fixed $\alpha\beta\in\edges$
\begin{align*}
	\bra*{\rmD_2 \calR^*\bra*{\rho,-\gnabla \rmD\calE(\rho)}}_{\alpha\beta} &= \frac{1}{2} k_{\alpha\beta}\sqrt{ \frac{\rho^\alpha \rho^\beta}{\pi^\alpha \pi^\beta}} \bra*{ \sqrt{\frac{\rho^\alpha \pi^\beta}{\pi^\alpha \rho^\beta}} -  \sqrt{\frac{\rho^\beta \pi^\alpha}{\pi^\beta \rho^\alpha}}} = \frac{1}{2} k_{\alpha\beta} \bra*{ \frac{\rho^\alpha}{\pi^\alpha} -  \frac{\rho^\beta}{\pi^\beta}} . 
\end{align*} 
Applying the divergence from~\eqref{eq:cont-eq-chem-reactions} we obtain~\eqref{eq:chemical-reactions:RRE}.

\subsubsection*{Example~\ref{ex:FP}+\ref{ex:heat-flow}. Fokker-Planck equations with linear reactions (continued). }

Taking as an example the two-species reaction-diffusion equation~\eqref{eq:RD-intro}, the state space is $\nodes:= \R^d\times \{\sfa,\sfb\}$; given a potential $V:\nodes \to \R$ we define a similar stationary measure and driving functional, 
\[
\pi(\dx x,\sfx) := \frac1{\PartSum} \ee^{-V(x,\sfx)} \dx x,
\qquad\text{and}
\qquad
\calE(\rho) := \RelEnt(\rho|\pi).
\]
Recall that the edge space for Example~\ref{ex:FP}+\ref{ex:heat-flow} consists of continuous edges $\edges^{\mathrm c}$ and discrete edges $\edges^{\mathrm d}$~\eqref{defn:Ex:B+A:edges}, with corresponding functions $\Xi^{\mathrm c}$ and $\Xi^{\mathrm d}$. As dual dissipation potential one can take 
\begin{align}
	\calR^*(\rho,\Xi) &= \calR_{\mathrm {FP}}^*(\rho_\sfa,\Xi^{\mathrm c}_\sfa) 
	+ \calR_{\mathrm {FP}}^*(\rho_\sfb,\Xi^{\mathrm c}_\sfb) 
	+\calR_{\mathrm J}^*(\rho,\Xi^{\mathrm d}),
	\qquad \text{  for $\Xi = (\Xi^{\mathrm c}, \Xi^{\mathrm d})$},\notag\\
	\calR_{\mathrm J}^*(\rho,\Xi^{\mathrm d})
	&=\sum_{\sfx\sfy\in \{\sfa\sfb,\sfb\sfa\}}\int_{\R^d} \frac12 \speck\sqrt{\rho_\sfx\rho_\sfy}(\dx x) \sfC^*(\Xi_{\sfx\sfy}^d(x)).
	\label{eqdef:R*FP+reactions}
\end{align}
Here $\calR_{\mathrm {FP}}(\rho,\Xi^{\mathrm c})$ is the Fokker-Planck potential defined in~\eqref{eqdef:R-FP},  $\rho_{\sfa,\sfb}$ are the measures $\rho(\cdot,\sfa)$ and $\rho(\cdot,\sfb)$, and $\Xi^{\mathrm c}_{\sfa,\sfb}$ similarly are	 the values of $\Xi^{\mathrm c}$ at $\sfx=\sfa$ and $\sfx=\sfb$. 

The integral in~\eqref{eqdef:R*FP+reactions} can be understood as follows. Since $\rho_\sfx$ and $\rho_\sfy$ are non-negative measures on $\R^d$, the combination $\sqrt{\rho_\sfx\rho_\sfy}$ again is a non-negative measure\footnote{Indeed, if~$\rho_\sfx$ and~$\rho_\sfy$ have a density $v_{\sfx}$ and $v_{\sfy}$ with respect to some measure $\mu$, then $\sqrt{\rho_\sfx\rho_\sfy} := \sqrt{v_\sfx v_\sfy}\,\mu$. This construction is independent of the choice $\mu$.} on $\R^d$. If $\rho_\sfx$ and ~$\rho_\sfy$ each consist of a single Dirac at $x_0\in \R^d$ times a constant, then the integral coincides with $\calR^*(\rho,\Xi^{\mathrm d})$ as defined in~\eqref{eqdef:Rstar-jump}. The integral expression in~\eqref{eqdef:R*FP+reactions} can therefore be seen as a natural extension to reactions that are distributed in space, with the measure $\sqrt{\rho_\sfx\rho_\sfy}$ determining the local level of `activity' of the reaction.

\medskip
Again assuming for simplicity that $\rho$ is Lebesgue absolutely continuous, we write $\rho = u\pi$ and find
\[
\anabla \rmD\calE(\rho)(x,\sfx\sfy) = (\anabla \log u(x,\sfx\sfy))
= \biggl(\frac{\nabla_x u}u(x,\sfx) \,,\; 
(\ona \log u)(x,\sfx\sfy)\biggr).
\]
Setting $j = (j^{\mathrm c},j^{\mathrm d}) = \rmD_2\calR(\rho,-\anabla \rmD\calE(\rho))$, 
we find with calculations similar to those in~\eqref{eq:deriv-eq-FP} and~\eqref{eq:deriv-eq-heatflow-intro}
\begin{align}
	j^{\mathrm c}(x,\sfx) &= 	\rmD_2\calR_{\mathrm {FP}}^*\bra*{\rho_\sfx,  - \frac{\nabla u_\sfx}{u_\sfx}} 
	= D\bigl(\nabla_x\rho_\sfx + \rho_\sfx \nabla_x V_\sfx)(x),
	\qquad \sfx = \sfa,\sfb,\notag\\
	j^{\mathrm d}(\dx x,\sfa\sfb) &= \frac12\, \speck\, \sqrt{\rho_\sfa\rho_\sfb}(\dx x) 
	{\sfC^*}'\bra*{-\gnabla \log u(x,\sfa\sfb)} \notag\\
	&= \frac1{2Z} \,\speck \,\ee^{-\frac12V(x,\sfa)-\frac12 V(x,\sfb)}\sqrt{u_\sfa u_\sfb}(x) 
	{\sfC^*}'\bra[\big]{ \log u(x,\sfa)-\log u(x,\sfb)}\dx x\notag \\
	&=   \frac12 \bra*{\kappa_{\sfa\sfb}(x)\rho_\sfa(\dx x) - \kappa_{\sfb\sfa}(x)\rho_\sfb(\dx x)}
	\qquad\text{with } \kappa_{\sfx\sfy}(x) :=  \speck\, \ee^{\frac12(V(x,\sfx)-V(x,\sfy))}.
	\label{eq:flux-ex-B+A}
\end{align}
This shows that the equations $\partial_t\rho + \adiv j=0$ and $j = \rmD_2\calR(\rho,-\anabla \rmD\calE(\rho))$ reduce to a version of~\eqref{eq:RD-intro}.

\begin{remark}[`Wrong' gradient structures]
	While the choices of $\calE$ and $\calR^*$ above are intuitively appealing,  we show in Section~\ref{s:Kramers} that the coarse-graining procedure based on the Kramers high activation-energy limit leads to a different gradient structure. 
\end{remark}

\subsubsection*{Example~\ref{ex:FP}+\ref{ex:chem-reactions}. Fokker-Planck equation with chemical reactions (continued).}

Similarly to the previous example, upon combining the gradient structure of the Fokker-Planck equation (Example~\ref{ex:FP}) with that of the chemical reactions (Example~\ref{ex:chem-reactions}) one might postulate the following structure.

The state space is $\nodes = \R^d\times \ol\nodes$, where $\ol\nodes$ is a finite set of species,  and the edge set $\edges = \edges^{\mathrm c} \sqcup \edges^{\mathrm d}$ is defined in~\eqref{eqdef:edges:B+C}.
Again we choose a potential $V:\nodes\to\R$ and define a stationary measure and driving energy
\[
\pi(\dx x,\sfx) := \frac1{\PartSum} \ee^{-V(x,\sfx)} \dx x,
\qquad\text{and}
\qquad
\calE(\rho) := \RelEnt(\rho|\pi).
\]
As dual dissipation potential for $\Xi = (\Xi^{\mathrm c}, \Xi^{\mathrm d}): \edges^{\mathrm c}\times \edges^{\mathrm d} \to \R$, we take  (writing $\rho = u\pi$)
\begin{equation*}
	\calR^*(\rho,\Xi) = 
	  \sum_{\sfx\in\ol\nodes} \calR_{\mathrm {FP}}^*(\rho_\sfx,\Xi^{\mathrm c}_\sfx) 
	+ \int_{\R^d} \frac{1}{2}\sum_{\alpha\beta\in \edges^{\mathrm d}} k_{\alpha\beta}\sqrt{ u^\alpha(x) u^\beta(x)} \;
	 \sfC^*\bra*{\Xi^{\mathrm d}_{\alpha\beta}(x)} \dx x. 
\end{equation*}
Here $\calR_{\mathrm {FP}}(\rho,\Xi^{\mathrm c})$ is as defined in~\eqref{eqdef:R-FP}, 
$\rho_{\sfx}$ indicates the measure $\rho(\cdot,\sfx)$, and $\Xi^{\mathrm c}_\sfx$ and $\Xi^{\mathrm d}_{\alpha\beta}$ denote the reductions
\[
\Xi^{\mathrm c}_\sfx(x) := \Xi^{\mathrm c}(x,\sfx) \qquad \text{and}\qquad 
\Xi^{\mathrm d}_{\alpha\beta} (x) := \Xi^{\mathrm d} (x,\alpha\beta).
\]
Note that the prefactor of $\sfC^*$  is the density
\[
u^\alpha(x) u^\beta(x)
:= {\prod_{\sfx\in\ol\nodes}{u(x,\sfx)}^{\alpha_\sfx} }
   {\prod_{\sfy\in\ol\nodes}{u(x,\sfy)}^{\beta_\sfy} }.
\]

\begin{remark}[Again a `wrong' gradient structure]
Just as in the case of  Example~\ref{ex:FP}+\ref{ex:chem-reactions}, the intuitive choices of $\calE$ and $\calR$ above lead to dependence of the reaction rates on chemical potentials that is at odds with classical chemical-reaction modelling; see Section~\ref{ss:Kramers-activity-interpretation}. 
\end{remark}

\subsubsection*{Example~\ref{ex:Boltzmann}. The Boltzmann equation (continued).}

The Boltzmann equation~\eqref{eq:Boltzmann} with symmetry assumption~\eqref{eq:Boltz:kapppa-symmetry} has a three-parameter family of detailed-balance stationary measures given by the Maxwell distributions.
Indeed, if we also assume the non-degeneracy property $\kappa(v,v_*,n)>0$ for $v,v_*\in\R^d$ and $n\in \bbS^{d-1}$ such that $(v-v_*)\cdot n>0$,
then we can regard~\eqref{eq:Boltz:kapppa-symmetry} as an assumption of  \emph{detailed-balance} with respect to any stationary state $\pi \in \calM_{\geq 0}(\R^d)$ that satisfies
\begin{equation*}
  \pi(v) \pi(v_*) = \pi(v') \pi(v'_*) \qquad\text{for a.e. } v,v_* \in \R^d , n\in \bbS^{d-1}: (v-v_*) \cdot n >0 .
\end{equation*}
The stationary solutions in the class of $L^1$-densities
are the parametrized family of Maxwellian distributions (see~\cite[\S 3.1]{CercignaniIllnerPulvirenti1994} or~\cite[Theorem 1]{Villani2008}) given by
\begin{equation}\label{eqdef:BoltzmannMaxwell}
	\pi(v) = \frac{1}{\PartSum} \exp\bra*{ - \frac{\abs*{v-u}^2}{2 T}} .
\end{equation}
Hereby, the constants $\PartSum>0$, $u\in\R^d$, and $T>0$ correspond to the conserved quantities in~\eqref{eq:Boltz:conservations} and can be for instance chosen to match the initial datum.

Any of the Maxwellian distributions in~\eqref{eqdef:BoltzmannMaxwell} can be used as reference state for the relative entropy as driving functional for the gradient flow description. The usual choice is to consider the non-normalized relative entropy with respect to the Lebesgue measure, which can be understood as the limit $T\to \infty$ in~\eqref{eqdef:BoltzmannMaxwell}.
For any choice of $\pi$ in this class, the Boltzmann equation~\eqref{eq:Boltzmann} arises from the gradient structure that is given for $\Xi:\edges \to \R$ by 
\begin{subequations}
\label{eq:gradient-structure-Boltzmann}	
\begin{align}
\calE(f) &:= \RelEnt(f|\pi)\\
\calR^*(f,\Xi) &= \iiiint_\edges \sqrt{f(v_1)f(v_2)f(v_3)f(v_4)} 
\,\sfC^*\bigl(\Xi(v_1v_2\, v_3v_4)\bigr)
\,\overline \kappa(\dxx{v_1}{v_2}\,\dxx{v_3}{v_4})\\
&= \iiint_{\widehat\edges}
\sqrt{f(v)f(v_*)f(v')f(v_*')} \,\sfC^*\bigl(\Xi(vv_*\, v'v_*')\bigr)
\kappa(v,v_*,n)\dxx v{v_*}\,\dx n.
\label{eqdef:Boltzmann-Rstar}
\end{align}
\end{subequations}
Here, the two expressions above show  $\calR^*$  equivalently as an integral over $\widehat\edges$ against the kernel~$\kappa$ and as an integral over $\edges$ against the push-forward $\overline \kappa := T_\#\kappa$ of $\kappa$ under the collision map $T:\widehat\edges\to \edges$ in~\eqref{eqdef:Boltzmann-map-T}. In~\eqref{eqdef:Boltzmann-Rstar} the pair $v'v_*'$ is assumed to be characterized by $vv_*\,n\in\widehat\edges$ as in~\eqref{eq:collision-law}.

The crucial identity for the identification of the driving flux is obtained from a similar observation as in~\eqref{eq:deriv-eq-heatflow-intro} in Example~\ref{ex:heat-flow} from the identity
\begin{equation*}
  \sqrt{f(v)f(v_*)f(v')f(v_*')} \;{\sfC^*}'\bra*{\widehat\nabla \log f} = f(v')f(v_*') - f(v)f(v_*) 
\end{equation*}
implying that $\rmD_2 \calR^*\bigl(f,-\widehat\nabla\rmD\calE(f)\bigr) = \widehat\jmath$ in~\eqref{eq:choice-wh-jmath-Boltzmann}.
	
This particular gradient structure~\eqref{eq:gradient-structure-Boltzmann} for the Boltzmann equation seems to have been explicitly identified for the first time by Grmela~\cite[Eq.~(A7)]{Grmela93} (see also~\cite[Eq.~(23)]{Grmela2002} and~\cite[Eq.~(79)]{Grmela2010}); the corresponding large-deviation rate functions appeared in early results such as~\cite{Leonard95a}, but the complete rigorous proof for the LDP is currently still open (see also~\cite{Bouchet2020,BasileBenedettoBertiniOrrieri2021,Heydecker2021} for recent advances regarding the LDP and~\cite{BodineauGallagherSaint-RaymondSimonella20TR} for the inhomogeneous case). A gradient system with a quadratic dual dissipation potential was identified in~\cite{Ottinger97} and studied mathematically in~\cite{Erbar2016TR}. For the linear Boltzmann equation the work \cite{BasileBenedettoBertini2020} contains a gradient structure based on the linearization of the dissipation potential~$\calR^*$ in~\eqref{eqdef:Boltzmann-Rstar}.

\begin{remark}[Other gradient structures]
\label{rem:other-gradient-structures}
For several of the examples discussed above, other gradient structures have been identified and studied than the cosh-based ones presented here. A well-studied class of gradient structures on discrete spaces has a quadratic dependence on~$\Xi$, and is based on the `log-mean' discrete chain rule
\begin{equation}
\label{eqdef:log-mean}
a -b  = {\Lambda(a,b)} (\log a - \log b),
\qquad\text{with}\qquad
	\Lambda(a,b) := 
	\begin{cases}
		 \dfrac{a-b}{\log a - \log b} , & \text{for } a\ne b; \\
	 	a , & \text{for } a=b, 
	\end{cases}
\end{equation}
to formulate the linear expression $a-b$ as a linear function of the graph gradient  $\log a-\log b$  of a relative entropy. 
We already mentioned the quadratic structure identified by \"Ottinger for the Boltzmann equation~\cite{Ottinger97}, and Chow-Huang-Li-Zhou, Maas,  Mielke, and many others developed the mathematical theory for this type of gradient structure in the case of the discrete-space heat flow of Example~\ref{ex:heat-flow}~\cite{ChowHuangLiZhou12,Maas11,Mielke11}. Discrete analogues of porous-medium-type equations are studied in~\cite{ErbarMaas2014} and lead to more general mean functions than the `log-mean' above. Also the chemical reaction networks of Example~\ref{ex:chem-reactions} possess a quadratic gradient structure, which is investigated in~\cite{Mielke11,MaasMielke20}.

The elastic Boltzmann equation with restitution coefficient strictly smaller one has only Dirac measures as stationary states and therefore satisfies a trivial detailed balance condition. The work~\cite{EspositoGvalaniSchlichtingSchmidtchen2021TR} proposes a gradient structure for it and studies its relation to the aggregation equation.

Gradient structures on metric graphs, which can be seen inbetween those of Example~\ref{ex:heat-flow} and~\ref{ex:FP} are recently studied in~\cite{ErbarForkertMaasMugnolo2021,BurgerHumpertPietschmann2021}.
\end{remark}

\subsection{Cosh-type dissipations arise by minimization}
\label{ss:contraction-intro}

As mentioned in the first part of the introduction, most dissipation potentials in the literature are of quadratic type.   Such quadratic potentials have linear derivatives,  generate formal Riemannian  structures, and are linked to many powerful functional inequalities. Therefore they have several advantages over non-quadratic  dissipation potentials such as the cosh-type ones of this paper. 

Nonetheless, we argue in this paper that cosh-type dissipations are `natural', because they arise `naturally' in multi-scale limits and in large-deviation principles. In this section we give an overview, and the details are provided in Sections~\ref{s:Kramers} and~\ref{s:ha}.

\subsubsection{Contraction formulas involving \texorpdfstring{$\sfC$}{C} and \texorpdfstring{$\sfC^*$}{C*}}

The cosh-type dissipation functions $\sfC$ and $\sfC^*$ arise as the result of minimization. In the case of the high-activation-energy limit in Kramers' equation, the minimization problem resembles the cell problem in homogenization of systems with rapidly oscillating parameters~\cite{Hornung97,CioranescuDonato99}. The basic identity characterizes the particular combination 
\begin{equation}
\label{eqdef:CCs-intro}
\CCs(j;\alpha,\beta;k):= k\sqrt{\alpha\beta}  \bra*{\sfC\bra*{\frac{ j}{k\sqrt{\alpha\beta}}} 
   + \sfC^*\bra*{\log \frac\beta\alpha}}
\end{equation}
as a minimization problem over interpolations $w$ of boundary values $\alpha$ and $\beta$:
\begin{multline}
\CCs(j;\alpha,\beta;k) \\
=
\inf_w \biggl\{
   \int_0^1 \bra*{ \frac{ j^2}{2k w(x)} +2k \abs[\big]{\partial_x \sqrt{w(x)}}^2 }\dx x:
   \sqrt{w}\in H^1(0,1), \ w(0)=\alpha, \ w(1) = \beta
   \biggr\} .
\label{eq:CCs:cell-problem}
\end{multline}
In this formula the parameters $\alpha, \beta>0$ have the interpretation of rescaled densities at opposite ends of an edge in a graph, and $k>0$ is an overall jump rate along that edge. 
The integral on the right-hand side is a transformation of the ``$\calR + \calR^*$'' dissipation function for the one-dimensional Fokker-Planck equation, in which one recognizes the quadratic dependence on the flux~$j$ and on the spatial derivative $w'$. 
We give a rigorous definition of the function~$\CCs$ for a larger range of parameters in Section~\ref{ss:cell-formula}. 

While $\CCs$ and $\sfC$ depend on their arguments in a non-quadratic manner, the right-hand side of~\eqref{eq:CCs:cell-problem} shows that these non-quadratic functions do have a quadratic origin. 
This underlying quadratic nature of the functional is explored in Section~\ref{s:ha}, where we show a connection to the potential-theoretic capacity. This can be formally observed from the fact that the second term on the right-hand side of~\eqref{eq:CCs:cell-problem} containing the spatial derivative of $w$ is the Dirichlet form of $\sqrt{w}$. 
With this interpretation, the value of $\CCs(0;0,1;k)$ is proportional to the effective conductance of a conducting wire with conductance $k$. 
This interpretation manifests itself for instance in series and parallel laws satisfied by $\CCs$ (see Corollaries~\ref{cor:N:series} and~\ref{cor:N:parallel})
and a general connection to the effective conductivity under contractions in the application to two-terminal networks (see Section~\ref{sss:structure-of-tilt-dependence} and Section~\ref{ss:ha:EDP:discussion}). 
Interestingly, while it is  useful to have explicit formulas for $\sfC$ and $\sfC^*$, the above-mentioned properties of these functions can best be observed and also proved on the level of the cell problem on the right-hand side of~\eqref{eq:CCs:cell-problem}. 

A second minimization problem leading to cosh dissipations comes from the \emph{contraction principle} in the theory of large deviations (see~\cite[\S III.5]{DenHollander00} or~\cite[\S 4.2]{DemboZeitouni98}). In this context the basic identity is  (for $\alpha,\beta>0$)
\begin{equation}
\label{eq:formula-1}	
\frac{1}{2} \CCs\bra*{2j,\alpha,\beta;k} - j \log \frac{\beta}{\alpha}
= \inf\Bigl\{ \eta(a | \alpha k) + \eta(b|\beta k):\ a,b\geq 0, \ \frac{b-a}{2} = j\Bigr\}.
\end{equation}
Here the relative entropy density $\eta:\R_{\geq0}^2\to[0,\infty]$ is defined by
\begin{equation}
\label{eqdef:eta}
\eta(a|b) := \begin{cases}
	a\log \dfrac ab - a + b & \text{if $a,b>0$; }\\
	b &\text{if }a = 0, b\geq 0 ;\\
	+\infty &\text{if $a>0$, $b=0$.}
\end{cases}
\end{equation}
In~\eqref{eq:formula-1} the parameters $a$ and $b$ represent forward and backward unidirectional density fluxes over an edge,  $\alpha k$ and $\beta k$ are the reference fluxes, and the quantity $\eta(a|\alpha k) + \eta(b|\beta k)$ can be interpreted as a rate function for independent forward and backward jumps (see Theorem~\ref{t:ldp-intro} below). The identity~\eqref{eq:formula-1} then shows how the cosh structure arises from minimizing the rate function under constrained skew-symmetric flux~$j$.

\medskip

Both minimization problems can be interpreted as describing an edge in a graph, and optimizing a functional under constrained net flux over the edge. In the second case~\eqref{eq:formula-1} the edge is given; in the continuous-space setting of the first case~\eqref{eq:CCs:cell-problem}, the edge and the accompanying nodes are generated by the high-activation-energy limit, in which the measure concentrates onto the wells of the potential and the flux concentrates onto an `edge' between those wells. In this case mass conservation forces the flux to become constant along the edge, which is reflected in the fact that in~\eqref{eq:CCs:cell-problem} the flux $j$ does not depend on $x$. Here the cosh-structure results from constrained minimization of both the flux and the scaled density~$w$ along the edge.

\begin{remark}[Common origin]
Although we describe these two origins of the cosh structures as different, there is overlap at the mathematical level, since many large-deviation results can be considered to be theorems of homogenization; see e.g.\ the discussion in~\cite[Ch.~11]{FengKurtz06}. 
\end{remark}

\subsubsection{Cosh-type dissipations from multi-scale limits}
\label{sss:Kramers-intro}

We now describe the Kramers high-activation-energy limit in some detail. 
Consider the Fokker-Planck equation of the type~\eqref{eq:FP-intro} in one space dimension, 
\begin{equation}
\label{eq:FP-Arnrich}
\partial_t \rho = \tau_\e \partial_x \Bigl(\partial_x \rho + \frac1\e \rho\,\partial_x H\Bigr)
\qquad \text{on }\R.
\end{equation}
This equation describes the evolution of the law $\rho = \rho(t,\cdot)$ of a particle $X_t$ that diffuses in a potential landscape described by the potential $H/\e$. 
The function $H$ is assumed to be of the double-well type, with wells of equal depth at  $x=a$ and $x=b$ and a single local maximum at $x=c\in (a,b)$ (see Figure~\ref{fig:DoubleWell}). This system is a microscopic model for the simple reaction $A\leftrightharpoons B$. 

As described in Example~\ref{ex:FP} of Section~\ref{ss:flux-GS-intro}, this equation is a gradient system with $\nodes=\R$, $\edges=\R$, $\anabla=\nabla$,  driving functional  $\calE(\rho)=\RelEnt(\rho|\pi)$   with respect to the invariant probability measure $\pi_\e(\dd x) := \PartSum_\e^{-1} \ee^{-H(x)/\e}\, \dx x$, and  dissipation potential  
\[
\mathcal R^*_\e(\rho,\Xi) := \frac{\tau_\e}2 \int_\R | \Xi(x)|^2 \, \rho(\dd x).
\]
This model has two parameters. The parameter $\e$ characterizes the height of the mountain pass that separates the two wells, known as the \emph{activation energy}; in the limit $\e\to0$ this height is large, and transitions from one well to the other are exponentially rare. 
The parameter~$\tau_\e$ is the according exponential time scale at which such transitions happen; it converges to ~$+\infty$ as $\e\to0$, and  its appearance in~\eqref{eq:FP-Arnrich} causes the evolution to capture the behaviour at this slow time scale.

\medskip

In~\cite{LieroMielkePeletierRenger17} it was shown that as $\e\to0$ the gradient system \emph{EDP-converges} (see Section~\ref{ss:convergence-of-GS}) to a limit system $(\wt\nodes,\wt\edges,\gnabla, \wt\calE, \wt\calR)$, defined by
\begin{alignat*}{2}
&\wt \nodes := \set*{a,b} \text{ and } \wt \edges := \{ab\}, &\qquad &	\gnabla \varphi (ab) := \varphi(b)-\varphi(a),\\
&\wt\calE(\rho_0) := \RelEnt(\rho_0|\pi_0), && \pi_0 := \gamma^a \delta_a + \gamma^b \delta_b,\\
&\wt\calR^*(\rho_0,\Xi) := k \sqrt{u_a u_b} \,\sfC^*(\Xi) 
\quad \text{for }\Xi\in \R, && u_a := \frac{\dx \rho_0}{\dx \pi_0}(a) \text{ and } u_b := \frac{\dx \rho_0}{\dx \pi_0}(b) ,
\end{alignat*}
where $\gamma^a,\gamma^b>0$ with $\gamma^a+\gamma^b=1$ and the rate parameter $k$ are given in terms of the limit behaviour of $\pi_\eps$ and $\tau_\e$, respectively. Related variational convergence results were proved in~\cite{PeletierSavareVeneroni10, ArnrichMielkePeletierSavareVeneroni12, HerrmannNiethammer11} and non-variational convergence of the equations was proved in~\cite{EvansTabrizian16, SeoTabrizian20}, also in higher dimensions.

Since the measure $\pi_0$ is supported on the two wells $x=a,b$, a solution $\rho_0$ of the limit system also is concentrated on $x=a,b$. This limit system therefore is equivalent to the heat flow of Example~\ref{ex:heat-flow} on the two-point state space $\wt\nodes= \{a,b\}$, and the corresponding evolution equation for $\rho_0 = u_a\gamma^a \delta_a + u_b\gamma^b \delta_b$ takes the form
\[
\gamma^a \partial _t u_a = -\gamma^b \partial _t u_b = k(u_b - u_a).
\]

In the limit $\e\to0$  we therefore recover the `mass-action' model of the reaction $A\leftrightharpoons B$. This convergence  is a template for the derivation of rates of more general chemical reactions from more microscopic models; see e.g.~\cite[\S 14.4]{Nitzan06}, \cite[Ch.~16]{Peters17}, or Section~\ref{s:Kramers}.

\medskip
To summarize, in this limit the quadratic gradient structure $(\R,\R,\nabla,\calE_\e,\calR_\e)$ converges to a cosh-type gradient structure $(\{a,b\},\{ab\},\gnabla,\wt\calE,\wt\calR)$. The cosh functions $\sfC$ and $\sfC^*$ appear via `cell problems' of the form of~\eqref{eq:CCs:cell-problem}, leading to formal expressions for $\wt\calR$ and~$\wt\calR^*$,
\[
\wt\calR(\rho_0,j) + \wt\calR^*(\rho_0,-\ona \log u) 
= \CCs(j;u_a,u_b;k)
= k\sqrt{u_au_b} \bra*{\sfC\bra*{\frac{j}{k\sqrt{u_au_b}}}
   + \sfC^*\bra*{\log \frac{u_b}{u_a}}}.
\]
In Section~\ref{s:Kramers} we study this emergence of the cosh structure in detail, and also incorporate the effect of tilting of the gradient system. In Section~\ref{s:thin-membrane} we describe a very similar outcome of a thin-membrane limit.

\subsubsection{Cosh-type dissipations from large-deviation principles}
\label{sss:cosh-from-ldp-intro}

The gradient systems of this paper also have a strong connection to Markov jump processes; in this section we describe how the gradient system of Example~\ref{ex:heat-flow} arises in this way.

Consider $n$ i.i.d.\ particles $X^i$ jumping between the points of a finite set $\nodes$, with jump rates given by a kernel $\kappa:\edges:= \nodes\times \nodes\to [0,\infty)$. 
With probability one, a realization of each process has a countable number of jumps in the time interval $[0,\infty)$, and we write  $t^i_k$ for the $k^{\mathrm{th}}$ jump time of $X^i$. We can assume that $X^i$ is a c\`adl\`ag function of time.

We next define the empirical measure $\rho^n$ and the empirical one-way flux $j^n$ by
\begin{align*}
&\rho^n: [0,T]\to \calM^+(V), 
 &\rho^n_x(t) &:= \frac1n \sum_{i=1}^n \delta_{X^i_t}(x) = \frac1n \#\{i: X_t^i = x\},
\\
&j^n\in \calM^+((0,T)\times \edges), &\qquad
j^n_{xy}(\dd t)&:= \frac1n \sum_{i=1}^n \sum_{k=1}^\infty \delta_{t^i_k}(\dd t) \delta_{(X^i_{t-},X^i_{t})}(x,y),
\end{align*}
where 
$X^i_{t-}$ is the left limit (pre-jump state) of $X^i$ at time~$t$. Equivalently, $j^n$ is defined by
\[
\langle j^n, \varphi\rangle := \frac1n \sum_{i=1}^n \sum_{k=1}^\infty \varphi\bigl(t_k^i,X^i_{t_k^i-}X^i_{t_k^i}\bigr), 
\qquad \text{for }\varphi\in \CB([0,T]\times \edges).
\]

We assume that the process is irreducible. This implies that there is a unique invariant measure $\pi\in \ProbMeas(\nodes)$ for each $X_t^i$, which gives rise to a corresponding invariant measure $\varPi^n\in\ProbMeas(\ProbMeas(\nodes))$ for $\rho^n$.

The gradient system $(\nodes,\edges,\ona,\calE,\calR)$ in~\eqref{eqdef:E-jump} has its origin in the following large-deviation result. See also Section~\ref{ss:ldp-GS}.

\begin{theorem}[Large-deviation principles for the process $(\rho^n,j^n)$ (e.g.~{\cite[Th.~3.1]{Renger18}})]\label{t:ldp-intro}
Fix $\rho_\circ \in \ProbMeas(\nodes)$ and choose a sequence of empirical measures 
\[
\rho^n_\circ = \frac1n \sum_{i=1}^n \delta_{x^i_\circ}  \stackrel *\longrightharpoonup  \rho_\circ\qquad\text{as }n\to\infty.
\] 
Let $X_t^i$ be the Markov processes with initial datum $X^i_{t=0} = x_\circ^i$, and construct $\rho^n$ and $j^n$ as described above.
\begin{enumerate}
\item The invariant measures $\varPi^n$ of the process $\rho^n$ satisfy a large-deviation principle
\[
 \Prob\bra*{ \varPi^n \approx \rho} \sim \exp\bra*{ -n \calE(\rho)} \quad\text{as } n\to \infty , \qquad\text{with}\qquad  \calE(\rho) := \RelEnt(\rho|\pi) . 
\]
\item 
The random pair $(\rho^n,j^n)$ satisfies a large-deviation principle with rate function $\RateFunc$, i.e.
\[
\Prob\bra*{\,(\rho^n,j^n)\approx (\rho,j) \;\Big| \;\rho^n(0) = \rho_\circ^n}
\sim \exp\bra*{-n\RateFunc(\rho,j)} \qquad \text{as }n\to\infty,
\]
where 
\begin{equation}
\label{eqdef:J}
\RateFunc(\rho,j) := \begin{cases}
 \ds \int_0^T \sum_{\sfx\sfy\in \edges}
\eta\bra[\big]{j_{\sfx\sfy}(t)\,| \,\rho_{\sfx}(t)\kappa_{\sfx\sfy}}\dx t
    &\text{if $(\rho,j):$ $\partial_t \rho + \odiv j=0$, $\rho(t=0) =\rho_\circ$}, \\
  +\infty & \text{otherwise}.
\end{cases}
\end{equation}
\end{enumerate}
\end{theorem}
\noindent

We now use the detailed-balance condition~\eqref{eq:def:graph:DBC} to rewrite the expression~\eqref{eqdef:J} of the large-deviation rate function $\RateFunc$. 
For a pair $(\sfx\sfy, \sfy\sfx)$ of forward and backward edges, we obtain from the contraction formula~\eqref{eq:formula-1}  that
\begin{align}
\notag
\inf\Bigl\{  &\eta(j_{\sfy\sfx}| \rho_\sfy\kappa_{\sfy\sfx}) + \eta(j_{\sfx\sfy}| \rho_\sfx\kappa_{\sfx\sfy})  :
     \frac{j_{\sfx\sfy} - j_{\sfy\sfx}}2 = \tilde\jmath_{\sfx\sfy}\Bigr\}\\
&= \frac12\sqrt{\rho_\sfx\rho_\sfy \kappa_{\sfx\sfy}\kappa_{\sfy\sfx}}
\biggl\{ \sfC\Bigl(\frac {2\tilde\jmath_{\sfx\sfy}} {\sqrt{\rho_\sfx\rho_\sfy \kappa_{\sfx\sfy}\kappa_{\sfy\sfx}}}\Bigr)
+ \sfC^*\Bigl(\log \frac{\rho_\sfx\kappa_{\sfx\sfy}}{\rho_\sfy\kappa_{\sfy\sfx}}\Bigr)\biggr\}
- {\tilde\jmath_{\sfx\sfy}} \log \frac{\rho_\sfx\kappa_{\sfx\sfy}}{\rho_\sfy\kappa_{\sfy\sfx}},
\notag
\end{align}
in terms of the `skew-symmetrized' flux $\tilde\jmath_{\sfx\sfy} = (j_{\sfx\sfy}-j_{\sfy\sfx})/2$. Note that skew-symmetrization preserves the divergence: for any flux $j:\edges\to\R$, we have $\odiv \tilde\jmath = \odiv j$. 

The detailed-balance condition~\eqref{eq:def:graph:DBC} states that there exists a symmetric edge function~$k:\edges\to [0,\infty)$ such that $\pi_\sfx\kappa_{\sfx\sfy} = k_{\sfx\sfy}$. This property leads to the simplifications 
\[
\sqrt{\rho_\sfx\rho_\sfy \kappa_{\sfx\sfy}\kappa_{\sfy\sfx}}
= k_{\sfx\sfy} \sqrt{\frac{\rho_\sfx}{\pi_\sfx}\frac{\rho_\sfy}{\pi_\sfy}} 
\qquad\text{and}\qquad
\frac{\rho_\sfx\kappa_{\sfx\sfy}}{\rho_\sfy\kappa_{\sfy\sfx}}
= \frac {u_\sfx}{u_\sfy} \quad \text{with} \quad u_\sfx = \frac{\rho_{\sfx}}{\pi_\sfx},
\]
Note that whenever $u_\sfx,u_\sfy>0$,  
\[
\log \frac{u_\sfx}{u_\sfy} = -\log \frac{u_\sfy}{u_\sfx}
= -\gnabla_{\sfx\sfy}\rmD \calE(\rho), \qquad\text{with}\qquad \calE(\rho) = \calH(\rho|\pi).
\]

The calculations above assume that $\rho_\sfx>0$ for each $\sfx$. When one or both of $(\rho_\sfx,\rho_\sfy)$ are zero, the expressions need to be adapted; the lemma below summarizes the situation. 
\begin{lemma}
Setting for $\alpha,\beta, k \geq 0$ and $j\in \R$ 
\begin{align*}
\sfL(j;\alpha,\beta;k) &:= \inf\Bigl\{ \eta(a|\alpha k) + \eta(b|\beta k):\ a,b\geq 0, \ \frac{b-a}2=j\Bigr\},
\end{align*}
we have 
\begin{align}
\sfL(j;\alpha,\beta;k)&=\begin{cases}
\ds\frac{k}{2} \sqrt{\alpha\beta}\Bigl\{\sfC\Bigl(\frac {2j}{k \sqrt{\alpha\beta}}\Bigr) + \sfC^*\bra[\Big]{\log \frac\beta\alpha} \Bigr\} - j \log	\frac{\beta}{\alpha} 
  & \text{if }\alpha,\beta>0\\
\eta(2j|\beta k) &\text{if }\alpha = 0, \ \beta \geq0, \ j\geq0	\\
\eta(-2j|\alpha k) &\text{if }\alpha \geq0, \ \beta =0	, \ j\leq 0\\
+\infty  &\text{otherwise}.
\end{cases}
\label{eqdef:L}
\end{align}
\end{lemma}
\noindent
The proof is a straightforward calculation,  and as a consequence we obtain the characterization of gradient systems obtained from large-deviation principles.
\begin{cor}
Assume that the detailed-balance condition~\eqref{eq:def:graph:DBC} holds. 
Then the contraction $\wt\RateFunc$ of $\RateFunc$ over fluxes on pairs of forward and backward edges, 
\begin{align*}
\wt \RateFunc	(\rho,j) := \inf_{\hat \jmath} \Bigl\{ \RateFunc(\rho,\hat \jmath\,): \hat\jmath_{\sfx\sfy} - \hat\jmath_{\sfy\sfx} = j_{\sfx\sfy} - j_{\sfy\sfx} \Bigr\}
\end{align*}
satisfies $\wt\RateFunc(\rho,j) = \RateFunc(\rho,\tilde\jmath\,)$ with $\tilde\jmath_{\sfx\sfy} = (j_{\sfx\sfy}-j_{\sfy\sfx})/2$, and it has the characterization
\begin{subequations}
\begin{equation}
\label{eq:l:char-ldp-gs:L}	
\RateFunc(\rho,\tilde\jmath\,) =  \frac12 \int_0^T \sum_{	\sfx\sfy\in \edges} \sfL\bra[\big]{\tilde\jmath_{\sfx\sfy}(t);u_{\sfx}(t),u_{\sfy}(t); k_{\sfx\sfy}}\dx t 
\end{equation}
\text{with}
\begin{equation*}
	u_\sfx(t) := \frac{\rho_\sfx(t)}{\pi_{\sfx}} 
	\qquad\text{and}\qquad
	k_{\sfx\sfy} =\speck_{\sfx\sfy} \sqrt{\pi_\sfx\pi_\sfy}.
\end{equation*}
Whenever $\rho_t(\sfx)>0$ for all $\sfx,t$ we have the alternative representation
\begin{equation}
\RateFunc(\rho,\tilde\jmath) = \frac12 \calE(\rho_T) - \frac12 \calE(\rho_0)
  + \frac12 \int_0^T \Bigl[ \calR(\rho_t,\tilde\jmath_t) 
     + \calR^*\bigl(\rho_t, -2\ona \rmD\tfrac12 \calE(\rho_t)\bigr)\Bigr]
     \dx t.
\label{eq:l:char-ldp-gs:ERR}
\end{equation}
\end{subequations}
Here $\calE$, $\calR$, and $\calR^*$ are as in Example~\ref{ex:heat-flow} in Section~\ref{ss:flux-GS-intro}. 

In expressions~\eqref{eq:l:char-ldp-gs:L} and~\eqref{eq:l:char-ldp-gs:ERR}, the pairs $(\rho,j)$ and $(\rho,\tilde\jmath)$ are assumed to satisfy the continuity relation $\partial_t \rho =-\odiv j = -\odiv\tilde\jmath$; if they do not, then $\RateFunc(\rho,j) = \RateFunc(\rho,\tilde\jmath\,) = +\infty$.
\end{cor}

The expression~\eqref{eq:l:char-ldp-gs:ERR} is a rescaling of the `EDP functional'  in~\eqref{eqdef:EDP}; this rescaling is a version of the scale invariance mentioned in Remark~\ref{rem:scaling-energy-in-GS} with $\lambda=1/2$. Curves $(\rho,\wt \jmath\,)$ satisfying $\RateFunc(\rho,\wt\jmath\,)=0$ are solutions of the gradient-flow equations (see Section~\ref{ss:formal-rigorous})
\[
\partial_t \rho + \odiv \wt\jmath = 0 
\qquad\text{and}\qquad
\wt\jmath = \rmD_2\calR^*\bra*{\rho,-\gnabla\rmD\calE(\rho)}.
\]

\begin{proof}
The identity~\eqref{eq:l:char-ldp-gs:L} follows from the arguments above. The factor $1/2$ before the summation arises from the double counting of forward and backward edges.

In the case of positive densities $\rho_\sfx$, the final term $s\log (\beta/\alpha)$ in $\sfL$ becomes  an exact differential, since 
\begin{align*}
-\frac12 \sum_{\sfx\sfy\in\edges} \tilde\jmath_{\sfx\sfy}(t) \log \frac{u_\sfx(t)}{u_\sfy(t)}
&= \frac12 \sum_{\sfx\sfy\in\edges} \tilde\jmath_{\sfx\sfy}(t) \ona_{\sfx\sfy} \rmD\calE(\rho(t))
= -\frac12 \sum_{\sfx\in\nodes} \rmD\calE(\rho(t))(\sfx) \, \bigl(\odiv \tilde\jmath\bigr)(\sfx)\\
&= \frac12 \sum_{\sfx\in\nodes} \rmD\calE(\rho(t))(\sfx) \, \partial_t \rho_\sfx(t)
= \frac12 \frac{\dd}{\dx t} \calE(\rho(t)).
\end{align*}
After integration in time we find the expression~\eqref{eq:l:char-ldp-gs:ERR}.
\end{proof}

\bigskip
The appearance of the cosh dissipations $(\calR,\calR^*)$ can therefore be traced back to three ingredients:
\begin{itemize}
\item Independent forward and backward jump fluxes, with large-deviation behaviour characterized by $j_{\sfx\sfy}\mapsto \eta(j_{\sfx\sfy}|\rho_\sfx\kappa_{\sfx\sfy})$;
\item Contraction over forward and backward fluxes on the same edge;
\item Identification of the term $j\log \beta/\alpha$ in~\eqref{eqdef:L} as an exact differential, because of detailed balance.
	This is effectively a chain rule, and a one-sided version of this chain rule is essential for the variational characterization of gradient flows (see Property~\ref{property:ChainRuleLowerBound} below). 
\end{itemize}
Recently these observations have been generalized in various ways beyond detailed balance; see e.g.~\cite{KaiserJackZimmer18,PattersonRengerSharma21TR}.

\subsection{Modelling and tilting}
\label{s:ModellingTilting}

Gradient systems and `tilting' both have natural interpretations in the context of modelling, and these interpretations are intertwined. This is the topic of this section.

\subsubsection{Tilting of energies}
\label{sss:tilting:energies}

Before discussing the tilting of gradient structures, 
we  first consider static systems defined by a single `energy' functional~$\calE:\sfX\to\R$, and in which the `solutions of the system' are defined to be  the minimizers of the energy $\calE$.
In this setting, \emph{tilting} is a standard operation: `tilting' the system with an additional potential $\calF$ amounts to replacing $\calE$ by $\calE+\calF$, and a solution of the tilted system is by definition a minimizer of $\calE+\calF$. 
As an example, let $u\in \R$ be the displacement of a spring with energy $\calE(u) := ku^2/2$ and spring constant $k>0$. 
To this system, we apply an external load $f\in \R$ on the spring, thus generating a `tilt' $\calF(u) := -fu$. Then the combination of spring and external load is described by the functional 
\[
\calE(u) + \calF(u) = \frac k2 u^2 - fu .
\]
The constitutive relationship of the loaded system is obtained as the minimizer $u^\calF$ of the combined functional and satisfies 
\[
\calE'(u^\calF) = -\calF'(u^\calF)
\qquad\text{or}\qquad
ku^\calF = f,
\]
in which we recognize the usual force-displacement relation of the spring. 
In this way, the tilting induces a, possibly multi-valued, natural mapping 
\begin{equation}\label{e:tilt:energies:map}
	\calF \mapsto u^{\calF} := \argmin_u \set{\calE(u) + \calF(u)} .
\end{equation}
This principle extends to any system determined by an energy, and to any conservative force. 

\medskip

The mapping~\eqref{e:tilt:energies:map} also explains why tilting combines well with Gamma-convergence of the energy functionals.
If $\calE_n\stackrel\Gamma\longrightarrow \calE$, and if $\calF$ is continuous, then $\calE_n+\calF \stackrel\Gamma\longrightarrow \calE+\calF$.
Since Gamma-convergence implies convergence of minimizers to minimizers, this observation implies that the single proof of Gamma-convergence $\calE_n\stackrel\Gamma\longrightarrow \calE$ also guarantees convergence of solutions of all loaded systems $\calE_n+\calF$ and the associated mappings as defined in~\eqref{e:tilt:energies:map}.
This aspect is one of the main reasons for choosing Gamma-convergence when proving convergence of variational problems. 
Note that this observation extends to tilts $\calF_n$ that converge continuously~\cite[Prop.~6.20]{DalMaso93}. 
\medskip

Tilting of energies also interacts naturally with \emph{composition} of systems. 
If $\sfX_1$ and $\sfX_2$ are state spaces of systems described by energies $\calE_1$ and $\calE_2$, then the trivial, `non-interacting' composition of the two systems is given by the sum 
\[
\sfX := \sfX_1\times \sfX_2 \qquad\text{and}\qquad
\calE(x_1,x_2) := \calE_1(x_1) + \calE_2(x_2).
\]
Tilting $\calE_1$ and $\calE_2$ by $\calF_1$ and $\calF_2$ naturally transfers to tilting of $\calE$ by $\calF(x_1,x_2) := \calF_1(x_1) + \calF_2(x_2)$, and $\calE+\calF $ is minimized by the pair $(x_1^{\calF_1},x_2^{\calF_2})$. 

Less trivial, `interacting' compositions of the two systems can be generated by adding energies depending on both $x_1$ and $x_2$, 
\[
\calE(x_1,x_2) := \calE_1(x_1) + \calE_2(x_2) + \calE_{12}(x_1,x_2).
\]
For fixed $x_2$, the additional term $x_1\mapsto \calE_{12}(x_1,x_2)$ functions as a tilting of $\calE_1$, and this is reflected in the stationarity condition for $x_1$, 
\[
\rmD \calE_1(x_1) = -\rmD_1\calE_{12}(x_1,x_2).
\]
In this way, the combination of two systems can be naturally interpreted as a tilting of each of the two systems by the other.

\subsubsection{Kinetic relations}
\label{ss:KineticRelations}

Many gradient systems have strong connections to modelling: typically the driving functional~$\calE$ is an energy, entropy, or free energy, and also the dissipation potential~$\calR$ and its dual $\calR^*$ have a clear modelling interpretation, which are referred to as the \emph{energetics} and \emph{kinetics}, respectively.

To describe this interpretation of dissipation potentials we recall the concept of a \emph{kinetic relation}~\cite{MielkeMontefuscoPeletier21}. 
Mathematically, a kinetic relation $\KR_\rho$ at a point $\rho\in \calM_{\geq0}(\sfV)$ is a subset of a product space $\scrD'_\edges\times\scrD_\edges$. 
From the point of view of modelling, it characterizes the relationship between forces $\Xi\in \scrD_\edges$ and fluxes $j\in \scrD'_\edges$, in the sense that the pair $(j,\Xi)\in\scrD'_\edges\times \scrD_\edges$ is considered `admissible' if and only if $(j,\Xi)\in \KR_\rho$. 
In combination with an energy functional~$\calE$, the kinetic relation $\KR_\rho$ defines an evolution or even a possible set of evolutions by the three equations
\begin{equation}
\label{eq:evol-eq-KR}	
\partial_t \rho+\adiv j = 0, \qquad 
\Xi = -\anabla \rmD\calE(\rho), \qquad
(j,\Xi) \in \KR_\rho.
\end{equation}
Such kinetic relations appear throughout science, often under the name `constitutive relation'; for instance,
\begin{enumerate}
\item Fick's law $j_D = - D RT \nabla c =  D c (-RT \nabla \log c )$ relates the diffusive flux $j_D$ to the concentration gradient $\nabla c$. In the following, we pursue the thermodynamic point of view, in which  the chemical potential $ \mu$ is the driving force, which is the derivative of the Gibbs energy. In the case of simple diffusion, we have $\mu=RT  \log c$, where $R$ is the gas constant and $T$ the temperature~\cite[Ch.~4,5]{PeletierVarMod14TR}, but if the system contains additional phenomena such as electrostatic interaction or molecular crowding, this expression may be different. This can also be witnessed in the expression for $\Xi$ in~\eqref{eq:ex-FP-Xi}, which contains an additional term involving $V$.

In particular, note that by writing $j_D = -D c \nabla \mu$, the leading coefficient $D c$ depends on the current state $c$, which also is observed in the Fokker-Planck example in~\eqref{eqdef:R-FP}.
\item Fourier's law $j_H = -k \nabla T = k T^2 \nabla (1/T)$ relates the heat flux $j_H$ to the temperature gradient $\nabla T$. Again, a thermodynamic approach leads to the chemical potential $\mu = -(1/T)$~\cite{PeletierRedigVafayi14} with the state dependent coefficient $k T^2$.
\item A power-law viscosity relation $j_S = \eta |\Xi/\Xi_0|^p \sign(\Xi)$ relates the shear force $\Xi$ to the shear rate $j_S$.
\item The Coulomb friction law $\Xi\in \Xi_0 \sign (j_S)$, relating the sliding velocity $j_S$ to the shear force~$\Xi$, can be considered to be the $p\to\infty$ limit of the power-law viscosity relation given a reference shear force $\Xi_0$.
\end{enumerate}
Each of these relations defines a subset of the space of all pairs $(j,\Xi)$. The Coulomb friction example shows that this relation need not be linear or bijective. The examples also illustrate that many kinetic relations have parameters that depend on the state,  as in the case of the coefficients $Dc$ and $kT^2$ in Fick's and Fourier's law. 

\medskip

Dissipation potentials define such a kinetic relation.
Given a pair $(\calR,\calR^*)$ of dissipation potentials, the corresponding kinetic  relation between $j$ and $\Xi$ is defined via the three equivalent formulations
\begin{equation}
\label{eq:KR-for-GS}
j\in \partial_2 \calR^*(\rho,\Xi)
\quad\Longleftrightarrow\quad
\Xi\in \partial_2 \calR(\rho,j)
\quad\Longleftrightarrow\quad
\calR(\rho,j) + \calR^*(\rho,\Xi) = \dual j\Xi,
\end{equation}
and indeed the formulation~\eqref{eq:flux-GS-intro} coincides with~\eqref{eq:evol-eq-KR} whenever $\KR_\rho$ is given by~\eqref{eq:KR-for-GS}. The kinetic relations defined by dissipation potentials are necessarily \emph{dissipative}: if $(j,\Xi)\in \KR_\rho$ is given by~\eqref{eq:KR-for-GS}, then
\begin{equation}
\label{ineq:dissipative}
\dual j \Xi =  \calR(\rho,j) + \calR^*(\rho,\Xi) \stackrel{\NEW{\text{Def.~\ref{def:GradSystCE}(\ref{def:GradSystCE:DP})}}} \geq 0.
\end{equation}
In port-Hamiltonian parlance such a kinetic relation is an \emph{energy-dissipating} or  \emph{resistive} element~\cite[\S2.4]{Van-Der-SchaftJeltsema14}. In the context of gradient systems the inequality~\eqref{ineq:dissipative} implies that energy decreases along solutions (see~\eqref{eq:EDP:estimate}).

\medskip
In all the examples above, the coefficients $D$, $k$, $\eta$, and $\Xi_0$ are independent of the driving forces $-\nabla \log c$, $-RT\nabla (1/T)$, and $\Xi$, which also is natural from a philosophical point of view. Note that still these coefficients might depend on the current \emph{state}; this is different from depending on the local \emph{force}.

In the case of a kinetic relation generated by a dissipation potential $\calR$ as in~\eqref{eq:KR-for-GS}, this force-independence corresponds to the statement that $\calR = \calR(\rho,j)$ is independent of the driving force $\Xi$. In writing $\calR$ as a function of $(\rho,j)$ and not as a function of $(\rho,j,\Xi)$ this non-dependence on $\Xi$ appears to be obvious; but as we shall see below, this is not always the case. This brings us to the question of how $\calR$ depends on changes in $\calE$.

\subsubsection{Tilting of gradient systems: tilt-dependence and tilt-independence}

In a gradient system driven by an energy $\calE$, it is natural to encode additional loads or effects similarly by adding a tilting functional $\calF$ to $\calE$. How about the dissipation potential?
The rest of this section is devoted to this question:
\begin{quote}
	\emph{If the driving energy $\calE$ is tilted by adding $\calF$, does the dissipation potential $\calR$ change? If so, how does it change?	}
\end{quote}
To facilitate the discussion, we augment the definition of a gradient system $(\nodes,\edges, \anabla,\calE,\calR)$ with a family $\sfF$ of admissible tilts $\calF:\calM_{\geq0}(\nodes)\to \R$, such as for instance $\calF^V(\rho) = \int V \dx\rho$ in the Fokker-Planck example. 
\begin{definition}[Gradient systems with tilting]
\label{def:GS-with-tilting}
	A tuple $(\nodes,\edges,\anabla,\calE,\sfF,\calR)$ is a \emph{tilt gradient system} if 
	\begin{enumerate}
		\item $\nodes$ and $\edges$ are topological spaces, and $\scrD_\nodes$ and $\scrD_\edges$ are topological spaces of functions on $\nodes$ and $\edges$;
		\item $\anabla$ is a linear map from  $\scrD_\nodes$ to $\scrD_\edges$,
 with negative dual $\adiv$;
		\item $\calE$ is a function $\calE:\calM_{\geq0}(\nodes)\to\R$;
		\item $\sfF$ is a set of functions $\calF :\calM_{\geq0}(\nodes)\to\R$,
		\item For each $\calF\in \sfF$, $\calR(\cdot,\cdot;\calF)$ is a dissipation potential (see Definition~\ref{def:GradSystCE}).
	\end{enumerate}
	For each $\calF\in \sfF$, the tilted gradient system generates the evolution defined by 
	\[
	\partial_t \rho + \adiv j =0 \qquad\text{and}\qquad
	j \in \partial_2 \calR^*\bigl(\rho,\mathopen -\anabla \rmD(\calE+\calF)(\rho);\,\calF\bigr).
	\]
	Tilt-\emph{independence} then is the situation that
	\[
	\calR(\cdot,\cdot;\calF_1)=\calR(\cdot,\cdot;\calF_2)
	\qquad \text{for all }\calF_1,\calF_2\in \sfF,
	\]
	i.e.\ that the dissipation $\calR(\cdot,\cdot;\calF)$ does not depend on $\calF$. 
\end{definition}

\subsubsection*{Example~\ref{ex:heat-flow} (continued). Tilt-independent gradient structure of discrete heat flow}

We revisit Example~\ref{ex:heat-flow}, and extend the cosh gradient structure $(\nodes,\edges,\gnabla,\calE,\calR)$ obtained in~\eqref{eq:ER-heat-flow-intro} with a family of tilts. 
Recall that $\pi$ is a fixed measure on $\nodes$, $\speck$ a symmetric edge function, and 
\[
\calE(\rho) := \RelEnt(\rho|\pi), 
\qquad
\calR(\rho,j) := \sum_{\sfx\sfy\in \edges} \sigma_{\sfx\sfy}(\rho) \sfC\bra*{\frac{j_{\sfx\sfy}}{\sigma_{\sfx\sfy}(\rho)}}, 
\qquad\text{with}\quad \sigma_{\sfx\sfy}(\rho) = \frac12 \speck_{\sfx\sfy} \sqrt{\rho_\sfx\rho_\sfy} .
\]


To extend the system to a tilt gradient system, we first define the set of \emph{potential tilts} $\sfF_{\mathup{Pot}} = \set*{ \calF^F_{\mathup{Pot}} | F:\nodes \to \R}$, with 
\begin{equation}
	\label{eqdef:MCs:potential-tilts}
	\calF^F_{\mathup{Pot}}:\calM_{\geq0}(\nodes) \to \R, \qquad
	\calF^F_{\mathup{Pot}}(\rho) := \sum_{\sfx\in\nodes} \rho_\sfx F_\sfx.
\end{equation}
We then define the tilt gradient structure $\bra*{\nodes,\edges,\gnabla,\calH(\cdot|\pi) ,\sfF_\Pot,\ol\calR}$ by extending $\calR$ trivially to tilts, namely by setting
\[
\ol\calR(\rho,j;\calF) := \calR(\rho,j).
\]
With this definition, this tilt gradient system is tilt-independent. 
A similar calculation as in~\eqref{eq:deriv-eq-heatflow-intro} shows that $\ol\calR$ induces the flux 
\begin{equation*}
	j^F_{\sfx\sfy} := \frac{1}{2} \bra*{ \rho_\sfx \kappa_{\sfx\sfy}^F - \rho_{\sfy}\kappa_{\sfy\sfx}^F} \qquad\text{where}\qquad  \kappa_{\sfx\sfy}^F := \kappa_{\sfx\sfy} \ee^{\frac{1}{2} ( F_{\sfx}- F_{\sfy})} . 
\end{equation*}

This example illustrates how tilt-independence selects a particular way of modifying the jump rates $\kappa$. 
We generalize this example in Section~\ref{sss:tilt-indepdent-GS-MC}, where we  provide a characterization of all families of detailed-balance jump kernels $\kappa^F$ possessing a tilt-independent gradient structure (see Corollary~\ref{cor:MC:tilt:kappa:norm}).

\subsubsection{Tilt-dependence introduced by limits}

Tilt-independence is natural in many modelling situations. For instance, in a system consisting of a drop of syrup sliding down a slope, if we increase the forcing by blowing, then we expect the viscosity of the syrup to remain the same; the kinetic relation (that characterizes  viscous flow) is independent of the external forcing. 

One might even assume that tilt-independence is natural in \emph{all} situations. In fact, this is not the case, and this is one of the main messages of this paper. Tilt-independence need not be preserved through limits: if a sequence of tilt-independent `gradient systems with tilting' (as in Definition~\ref{def:GS-with-tilting}) converges to a limit system,
	then the limit system may be tilt-\emph{dependent}. This was observed in~\cite{FrenzelLiero21,MielkeMontefuscoPeletier21}, and in that latter reference alternative convergence concepts were introduced to remedy this: under such alternative concepts of gradient-system convergence, tilt-independence is preserved.
	
However, in this paper we show that in many cases tilt-dependence is unavoidable. A prime example is the Kramers high-activation-energy limit, which we already mentioned in Section~\ref{sss:Kramers-intro}. In Section~\ref{s:Kramers} we study a more general case, in which we combine `Kramers-type reaction-diffusion' in a variable $y\in \Upsilon$ with `standard diffusion' in a variable $x\in \Omega$. We show that a sequence of gradient systems with tilting $(\nodes_\e,\edges_e\, \anabla_\e,\calE_\e,\sfF,\calR_\e)$ converges to a limit system with tilting $(\wt\nodes,\wt\edges, \wt\anabla,\wt\calE,\sfF,\wt \calR)$. Although the pre-limit dissipation potential $\calR_\e$ is tilt-independent, 
\[
\calR_\e(\rho,j; \calF) := \iint\limits_{\Omega\times \Upsilon} \biggl[\frac1{2m_\Omega} \abs*{\frac{\dx j^x}{\dx\rho}(x,y)}^2 + \frac1{2\e\tau_\e} \abs*{\frac{\dx j^y}{\dx\rho}(x,y)}^2\biggr]\, \rho(\dxx xy),
\]
the limit potential $\calR$ does depend on the tilt $\calF(\rho)$ (see~\eqref{eqdef:tildeRR*}):
\begin{equation}
\label{eqdef:R0-intro}
\wt \calR(\rho,j; \calF) :=\iint\limits_{\Omega\times \Upsilon} \frac1{2m_\Omega} \abs*{\frac{\dx j^x}{\dx\rho}}^2\dx \rho
+ \int\limits_{\Omega} \sigma(x;\rho,\calF) \sfC\bra*{\frac{\ol\jmath(x)}{\sigma(x;\rho,\calF)}} \dx x.
\end{equation}
This dependence on $\calF$ is encoded in the activity function 
\begin{equation}
\label{eq:def-sigma-intro}
\sigma(x;\rho,\calF) := \frac{\abs{m_\Upsilon}}{\abs*{\Omega}} \sqrt{\frac{\rho_a(x)\rho_b(x)}{\gamma^a \gamma^b}}\; \exp{\tfrac12 \bigl(F^\rho(x,a)+F^\rho(x,b)- 2F^\rho(x,c)\bigr)},
\end{equation}
in which $F^\rho = \rmD\calF(\rho)$ and $\rho_{a,b}(x)$ is the Lebesgue density of $\rho|_{y=a,b}$ at $x$. See Section~\ref{s:Kramers} for full details.

\subsubsection{Structure of tilt-dependence in the examples}
\label{sss:structure-of-tilt-dependence}

The form of the dependence of $\sigma$ in~\eqref{eq:def-sigma-intro} on the tilt $\calF$  turns out to be fairly general. In Sections~\ref{s:thin-membrane} and~\ref{s:ha} we study two other examples of tilt-dependence, also arising as a limit of tilt-independent gradient systems. 
Remarkably, in each of the three examples the pair of dissipation potentials has the form
\begin{align*}
\wt\calR(\rho,j;\calF) &= \int_{\edges} \sigma(\dx e;\rho,\calF)\, \sfC\bra*{\frac{\dx j}{\dx \sigma(\,\cdot\,;\rho,\calF)}(e)},\\
\wt\calR^*(\rho,\Xi;\calF) &= \int_{\edges} \sigma(\dx e;\rho,\calF)\, \sfC^*(\Xi(e)),
\end{align*}
where the tilt $\calF$ appears only in the factor~$\sigma$. 
The factor $\sigma$ has the interpretation of a parameter that determines a `global rate' or `activity' of an `edge', which scales the forward and the backward flux over the edge in the same way. In all three examples the `edge' does not exist in the pre-limit system, but emerges in the limit. 

Moreover, in all the considered examples, there is only one single edge emerging in the limit and hence the above integrals become an evaluation along this single edge. 
In detail, the emerging `activity' $\sigma$ has the following dependence on the tilt $F:=\rmD \calF(\rho)$:
\newcommand{\linespreading}{5mm}
\begin{align*}
\sigma(x;\rho,\calF) &= \alpha \sqrt{\frac{\dx\rho}{\dx\pi_0}(x,a)\frac{\dx\rho}{\dx\pi_0}(x,b)}
  \;\exp{\tfrac12 \bigl(F^\rho(x,a)+F^\rho(x,b)- 2F^\rho(x,c)\bigr)},\\
\noalign{\rule[-\linespreading]{0pt}{\linespreading} 
   \hfill\text{(Kramers limit, }\eqref{eq:char-sigma-measure-Kramers})}
\sigma(\rho,\calF) &= \sqrt{\frac{\dx\rho}{\dx \pi}(0^-)\frac{\dx\rho}{\dx\pi}(1^+)} 
  \bra*{\int_0^1 \frac1{a_*(s)}\ee^{V_*(s)}
  \exp\tfrac12 \bra[\big]{2F(s) - F(0) - F(1)}\dx s}^{-1}
  \\
\noalign{\rule[-\linespreading]{0pt}{\linespreading} 
   \hfill\text{(thin-membrane limit, }\eqref{eq:def:thin-membrane:sigma})}\\
\sigma(\rho,\calF) &= \sqrt{\frac{\rho_{\sfa} \,\rho_{\sfb}}{\pi_{\sfa}^0\, \pi_{\sfb}^0}} \capacity^F_{\sfa\sfb} \text{ with } 
\capacity^F_{\sfa\sfb} := 
\inf_{h:\nodes\to \R}\set[\bigg]{ \frac{1}{2} \sum_{\sfx\sfy\in \edges} k_{\sfx\sfy}^0 \ee^{-(F_\sfx+F_\sfy)/2} \abs*{\gnabla_{\sfx\sfy} h }^2 \,\bigg|\, h_\sfa=1, h_\sfb=0} \\
\noalign{\hfill\text{(two-terminal networks, }\eqref{eq:ha:effective-R})}
\end{align*}
Comparing these, we notice the following similarities and differences:
\begin{itemize}
\item In all three examples, the tilting $F$ enters the formula for $\sigma$ exponentially (and with a factor $\frac12$). 
\item All three are invariant under adding a constant function to the tilt, i.e.\ under replacing $\calF(\rho)$ by $\calF(\rho) + c$. This is a natural property, since adding a constant to the energy does not change the pre-limit evolution, or the pre-limit EDP functional (see Def.~\ref{def:EDP:sol}); therefore the same should hold in the limit. 
\item All three involve the \emph{average tilt} ${\tfrac12 (F_-+ F_+)}$ over the separate tilts $F_\pm$ at the two ends of the new `edge' (in the Kramers case $\tfrac12 (F^\rho(b)+F^\rho(a))$, in the membrane case $\tfrac12 (F(1)+F(0))$, and in the two-terminal network case implicitly in the definition of $\capacity_{\sfa\sfb}^F$ through $\tfrac12(F_\sfb+F_\sfa)$).
This average tilt cannot be written as a function of the \emph{difference} $F_+-F_-$. This fact shows that the limit is tilt-dependent, even if the `missing' tilt information is disregarded (see the next point). 
\item In all three cases, $\sigma$ also depends on values of the tilt at states $x\in \nodes_1\subseteq\nodes\setminus\nodes_0$ that are `missing' from the final reduced state space $\nodes_0\subset \nodes$:
in Kramers' case the state $y\in \nodes_1:=\set{c}$ is missing,  
in the thin-membrane case all values $s\in \nodes_1:=(0,1)$, 
and for the two-terminal networks all non-terminal nodes $\sfx\in \nodes_1:=\nodes\setminus\set{\sfa,\sfb}$.
These states are `missing' because the limit energy $\calE_0$ is infinite whenever $\rho$ takes non-zero values on these states. 
Therefore, no differentiation of $\calE_0+\calF$ is possible with respect to these values, implying that differentiation of $\calE_0+\calF$ can not recover the value of $F=\rmD\calF$ at these points.
This makes the tilt values `inaccessible', but since they do influence the evolution, 
they necessarily appear somewhere in~$\calR_0^*$; it turns out that they appear in~$\sigma$.
\item In all three examples, the tilt-dependence on those inaccessible values $\nodes_1\subseteq \nodes\setminus\nodes_0$ of $F$ is `exponentially harmonic': in Kramers' case it is $\exp(-F^\rho(x,c))$,
in the thin-membrane case the harmonic average $\int_0^1 \exp(-F(s))\dx{s}$,
and for the two-terminal networks the capacity $\capacity_{\sfa\sfb}^F$ depends on $\exp\bra{-(F_\sfx+F_\sfy)/2}$.
As a result, we obtain that changing the tilt solely on the inaccessible nodes $\nodes_1$ by a constant $\alpha\in \R$, that is $F\mapsto F^{1,\alpha}|_{\nodes_1}:= F|_{\nodes_1} +\alpha$, we obtain the relation $\sigma(\rho,\calF^{1,\alpha})=\ee^{-\alpha} \sigma(\rho,\calF)$, where $\calF^{1,\alpha}:= \int_\nodes F^{1,\alpha}\dx\rho$. 
\end{itemize}

\begin{remark}[`De-tilting' tilt-dependent gradient structures]\label{rem:unitlting:tilt-GS}

In each of the examples above we observed that $\sigma$ depends on tilt values at `missing' parts of the state space---parts of the state space $\nodes$ that have become inaccessible in the limit. 
In cases where this \emph{does not} happen, it may be possible to remove the tilt-dependence from the gradient structure while preserving the induced evolution equation~\eqref{eq:flux-GS-intro}. We now describe an example of this. 

Let $(\nodes,\edges,\anabla,\calE,\sfF,\calR)$ be a tilt-dependent gradient system. 
Assume that the domain of each admissible tilt is contained in the domain of the energy, that is
\begin{equation}\label{ass:untilting:domain}
	\forall \calF\in \sfF: \dom \calF \subseteq \dom \calE
\end{equation}
and that the dual dissipation potential $\calR^*$ has the structure
\begin{equation}\label{ass:untilting:F-dependence}
	\forall \calF\in \sfF: \qquad \calR^*(\rho,\Xi;\calF) = \int_{\edges} \Redge^*\bra{e,\rho,\Xi(e); -\anabla\rmD\calF(e)}\dx{e} ,
\end{equation}
for some $\Redge^*: \edges \times \calM_{\geq 0}(\nodes) \times \R \times \R \to [0,\infty)$.

Next, define $\wt\Redge^*:  \edges \times \calM_{\geq 0}(\nodes) \times \R \to [0,\infty)$  by
\begin{equation}\label{eqdef:untilted:Psis}
	\wt\Redge^*(e,\rho,\Xi) := \int_0^{\Xi} \rmD_3\Redge^*\bra*{e,\rho,\xi;\xi+\anabla\rmD\calE(e)} \dx\xi.
\end{equation}
We now make the additional assumption that the corresponding functional
\begin{equation*}
	\wt\calR^*(\rho,\Xi;\calF) := \int_{\edges} \wt\Redge^*\bra{e,\rho,\Xi(e)}\dx{e}
\end{equation*}
is a valid dissipation potential (after Definition~\ref{def:GradSystCE}).

In that case the gradient structure $(\nodes,\edges,\anabla,\calE,\sfF,\wt\calR)$ is \emph{tilt-independent}
and \emph{evolution-equivalent} to $(\nodes,\edges,\anabla,\calE,\sfF,\calR)$, by which we mean that  \emph{for every tilt $\calF\in \sfF$} the induced evolution equation coincides with that of the original system $(\nodes,\edges,\anabla,\calE,\sfF,\calR)$:
\begin{equation}\label{eqdef:evolution-equivalent}
	\forall \rho \in \calM_{\geq0}(\nodes):\quad
	  \rmD_2\calR^*\bra[\big]{\rho,-\anabla\rmD(\calE+\calF)(\rho);\calF(\rho)} = \rmD_2\wt\calR^*\bra[\big]{\rho,-\anabla\rmD(\calE+\calF)(\rho)}.
\end{equation}
See Section~\ref{sss:MC:quadratic:untilting} for an explicit situation in which the `log-mean' quadratic gradient structure from~\cite{ChowHuangLiZhou12,Maas11,Mielke11} for Example~\ref{ex:heat-flow} is `de-tilted' in the above sense. The  pair $(\wt\calR,\wt\calR^*)$ that results from this de-tilting procedure is the cosh structure~\eqref{eq:ER-heat-flow-intro}; this is to be expected, since for this system the cosh structure is the only tilt-independent one (see Section~\ref{ss:char-tilt-indep-GS}).

Note that such a `de-tilting' technique does not provide any information about the origin of the resulting tilt-independent gradient structure. At the same time, it is possible that the tilt-independent structure arises from microscopic models in multi-scale limits or from stochastic dynamics via large deviations results.
\end{remark}

\subsubsection{Consequences of tilt-dependence}

Tilt-\emph{dependence} has a number of consequences. 
As mentioned above, it complicates modelling at macroscopic level:
if tilt-dependence has to be assumed, then the driving functional and the dissipation have to be chosen in a self-consistent manner. 

One instance of this is \emph{composition} of systems. 
We described in Section~\ref{sss:tilting:energies} how separate systems defined by energies can be combined by postulating a joint energy, and how this joint energy appears to the separate subsystems to be a form of tilting.
In the context of tilt-\emph{independence} it is natural to do the same for gradient systems. 
However, if tilt-independence can not be assumed, then this practice may yield incorrect results; in Section~\ref{ss:Kramers-activity-interpretation} we revisit Examples~\ref{ex:FP}+\ref{ex:heat-flow} and~\ref{ex:FP}+\ref{ex:chem-reactions} from this point of view. 
This example illustrates the problem; how to solve this problem, i.e. how  to choose energy and dissipation together reliably, is a question for future work.


\medskip

A second consequence of tilt-dependence is that the technique of `reverse engineering' of the kinetic relation from potentials becomes unavailable.
An example of successful reverse engineering is the identification of the Wasserstein-entropy gradient structure of the Fokker-Planck equation. The authors in~\cite{JordanKinderlehrerOtto98} consider a family of Fokker-Planck equations~\eqref{eq:FP-intro}, parametrized by potentials $V:\R^d \to \R$. In particular, by identifying the term $\div \bra*{ D\rho \nabla V}$ as the advection caused by a potential energy $\calF(\rho) = \int V \dx{\rho}$, the form of the Wasserstein metric tensor is enforced, indicated by the symbol $\stackrel{!}{=}$, through the identity
\begin{equation*}
	\div \bra*{D \rho \nabla V} = - \div \bra*{D \rho \nabla (-\rmD\calF)} \stackrel{!}{=} -\div (\rmD_2\calR^*(\rho, - \nabla \rmD\calF)) .
\end{equation*}
This leads to the identification of $\calR^*(\rho,\nabla\xi)$ as $\frac{1}{2} \int \abs*{\nabla\xi}^2_D \dx{\rho}$. In this way one recovers for this example both the notion of the gradient $\anabla = \nabla$ and the kinetic relation $\KR_\rho$ encoded by the graph of $\Xi \mapsto \rmD_2 \calR(\rho,\Xi)$. In a second step the driving functional $\calE$ is identified as the entropy $\calE(\rho ) =  \int \rho \log \rho \dx{x}$. 

This type of reasoning via reverse engineering is based on an implicit assumption of `tilt-independence' of the underlying gradient structure.
In the case of the Fokker-Planck equation, 
other arguments lead to the same gradient system~\cite{AdamsDirrPeletierZimmer2011,AdamsDirrPeletierZimmer2013}, 
justifying this result by different means.
The implicit `tilt-independence` property is also used in~\cite[\S 2.1]{BrunaBurgerRanetbauerWolfram2018TR} to identify possible asymptotic gradient flow structures for nonlinear Fokker-Planck equations.

In general, however, it appears that we need to be careful with this reverse-engineering approach. 
The consequence for modelling of phenomenological models on macroscopic level is that the kinetic relation should be verified by more microscopic models, where the assumption of tilt-independence is justified. Alternatively, one could take the point of view from inverse problems and verify the kinetic relation with suitable measurements (see~\cite{BurgerPietschmannWolfram2013}, where this is done for crowd dynamics).

\subsubsection{Philosophy of tilt-dependence and tilt-independence}
\label{sss:philosopy-of-tilt-dependence}

Hidden in Definition~\ref{def:GS-with-tilting} is a philosophical point: whether a given gradient system is tilt-independent or tilt-dependent depends on more information than that contained in the gradient system itself.
This `more information' is mathematically codified by the dependence of $\calR$ on $\calF$.
These are some examples of how that dependence can be obtained:
\begin{enumerate}
\item One can {postulate} a dependence of $\calR$ on $\calF$ on the basis of modelling arguments. The example above of the drop of syrup is an instance of this, in which $\calR$ is assumed to be independent of $\calF$. In Section~\ref{ss:case-study-chem-reactions-tilt-dependent} we discuss  a postulate of non-trivial dependence: the classical theory of chemical reactions leads to a specific dependence of $\calR$ on $\calF$.
\item One can also postulate a form of dependence or independence by arguments of simplicity or symmetry. The method of `reverse engineering' mentioned above is based on an assumption of tilt-independence, and in some cases one can characterize all tilt-independent gradient systems; Mielke and Stephan~\cite{MielkeStephan2020} give a result of this type, and in Section~\ref{ss:char-tilt-indep-GS} we discuss a similar result.
\item 
One can derive the dependence of $\calR$ on $\calF$ by following the dependence through a limit; given a dependence of pre-limit objects, the dependence of the limit object follows as a consequence of the convergence. This is the philosophy that we follow in the examples of the Kramers high-activation-energy limit (Section~\ref{s:Kramers}), the thin-membrane limit (Section~\ref{s:thin-membrane}), and the fast-reaction limit in graphs (Section~\ref{s:ha}); see the next section.
\item \label{i:avoid-tilt-dependence-by-restricting-tilts}
The previous point suggests a fourth possibility: one can postulate a limited independence, among a subclass of tilts. In the Kramers example~(\ref{eqdef:R0-intro}--\ref{eq:def-sigma-intro}), for instance, one observes that the assumption
\[
F^\rho(x,y) = F^\rho(x), \qquad \text{i.e. $F^\rho(x,y)$ is independent of $y$,}
\]
is sufficient to make the dependence on $\calF$ or $F^\rho$ disappear. Since the $x$-coordinate represents spatial position of a particle and the $y$-coordinate represents some internal degree of freedom, such an assumption on $F^\rho$ has a natural interpretation:  the source of energy represented by $F^\rho$ depends on spatial position but not on internal state. 
\end{enumerate}

\begin{remark}[Tilt-dependence is unavoidable]
In~\cite{MielkeMontefuscoPeletier21} the concepts of `tilt-EDP convergence' and `contact-EDP convergence' were introduced, for which the limit system is tilt-independent by construction. In that paper it was shown by example that the same sequence of gradient systems may converge both in the `simple' EDP sense (Definition~\ref{defn:EDP-convergence} below) and in the contact-EDP sense, and that the limit systems may differ.

This raises the question whether for the examples of this paper the tilt-dependence of the limit can be `removed' by considering a different concept of convergence. 
In Section~\ref{ss:Kramers-not-contact-EDP-conv} we investigate this situation for the Kramers high-activation-energy limit, and we show that the sequence of gradient systems does not converge in the sense of either tilt- or contact-EDP convergence.  In fact, the tilt-dependence of classical chemical-reaction modelling (Section~\ref{ss:case-study-chem-reactions-tilt-dependent}) strongly suggests that the tilt-dependence is unavoidable.
\end{remark}

\subsection{Partial conclusion}
\label{ss:partial_conclusion}

In the pages above we have introduced the main messages of this paper, which we investigate in more detail in the remaining sections:

\begin{itemize}
	\item 
Cosh-type dissipations emerge through large deviations and coarse-graining, and in particular through the minimization problems that are implicit in such theories. In Section~\ref{s:Kramers} we prove the emergence of the cosh structure in the Kramers high-activation energy limit, and in Section~\ref{s:thin-membrane} we review it for a thin-membrane limit.

\item Cosh-type dissipations may be  stable under taking limits. This has been observed in for instance~\cite[\S 3.3]{LieroMielkePeletierRenger17} and~\cite[\S 5]{MielkeStephan2020}, and  we generalize this in Section~\ref{s:ha} to fast-reaction limits in general two-terminal networks.

\item Gradient systems may respond to tilting in different ways. In the simplest  case the energy is tilted and the dissipation potential is unchanged, and we called this situation \emph{tilt-independence}. We characterize the tilt-independent jump processes in Section~\ref{ss:char-tilt-indep-GS}.

\item Tilt-independence can change into tilt-\emph{dependence} when taking limits, especially as a result of `loss of state space'. We illustrate this with the Kramers limit in Section~\ref{ss:Kramers:limit-pb}, the thin-membrane example in Section~\ref{s:thin-membrane}, and the two-terminal networks in Section~\ref{ss:ha:EDP:discussion}. These examples also suggest a particular form of tilt-dependence in this type of situation, through an `activity' function $\sigma $ that depends on the tilt in a particular way (see Section~\ref{sss:structure-of-tilt-dependence}).
\end{itemize}

\subsection{Bibliographic comments}
\label{ss:history}

\emph{Kinetic relations: Quadratic, 1-homogeneous, and exponential.} 
Gradient structures in applications have often been either quadratic or one-homogeneous. Quadratic dissipations correspond to linear  kinetic relations and have a long history, going back at least to Rayleigh~\cite{Rayleigh13} and Onsager~\cite{Onsager31,OnsagerMachlup1953}. During a large part of this history the term `gradient flows' even was synonymous with the restricted class of Hilbert-space gradient systems, in which the dissipation potential is the squared norm. More recently, gradient systems with other quadratic dissipation potentials have been studied, with the Wasserstein gradient systems as most famous example~\cite{JordanKinderlehrerOtto97,JordanKinderlehrerOtto98,Otto01,AmbrosioGigliSavare08}. The quadratic structure implies that such potentials generate a metric on the state space, at least formally, and this property is the basis of De Giorgi's metric-space interpretations~\cite{DeGiorgiMarinoTosques80,MarinoSacconTosques89} and the rigorous theory of Ambrosio, Gigli, and Savar\'e~\cite{AmbrosioGigliSavare08}. 

A separate, well-studied class of dissipations is {1-homogeneous} in the flux; such potentials give rise to systems that are {rate-independent}, which means that their evolution is slaved to the evolution of external forces. For this class the formulation in terms of a dissipation potential appears to go back to Moreau~\cite{Moreau70}, and the specific nature of rate independence has given rise to a number of different solution concepts and accompanying theory~\cite{MielkeTheilLevitas02,Dal-MasoDeSimoneMora06,MielkeRossiSavare12a, MielkeRoubicek15}.


The first example of an \emph{exponential} kinetic relation, such as the cosh-type relations, appears to be in the work on chemical reactions of Marcelin~\cite[equation~(1)]{Marcelin15}. He reformulated the law of mass action, which is a product of powers of concentrations (or more precisely `activities') as an exponential function of the weighted sum of the corresponding chemical potentials. Grmela~\cite[Eq.~(A7)]{Grmela93} formulated this kinetic relation as the derivative of a cosh-type dissipation potential.

The theory of gradient structures is related to the theory of `doubly nonlinear parabolic equations' of the form $A\partial_t u + Bu \ni f$, in which $A$ and $B$ are maximal monotone operators (see e.g.~\cite{AltLuckhaus1983,ColliVisintin90}, \cite[Ch.~III]{Visintin96}, or~\cite[Ch.~11]{Roubicek13}). The trio Mielke--Rossi--Savar\'e has developed far-reaching generalizations that deal with general, i.e.\ non-quadratic, non-1-homogeneous dissipations (e.g.~\cite{MielkeRossiSavare09,MielkeRossiSavare12,MielkeRossiSavare13,Mielke16a}).

\medskip

\emph{Onsager reciprocity.}
The early work by Onsager on reciprocity relations~\cite{Onsager31} was focused on the concept of symmetry of linear operators. It is not obvious how to generalize this concept to nonlinear operators such as those arising in nonlinear kinetic relations, and there has been a large amount of discussion on this topic; see e.g. \cite{Casimir45,Ziegler58,Gyarmati70,HurleyGarrod82,Garcia-ColinRio-Correa84,GallavottiCohen95,Gallavotti96,MaesNetocny07,Seifert12,ReinaZimmer15}. 

Mielke, Renger, and one of us~\cite{MielkePeletierRenger16} proposed to re-interpret the concept of symmetry of a positive linear operator $L$ as the property that $L$ can be written as a derivative (of the non-negative quadratic functional $x\mapsto \tfrac12 x^TLx$). In such a context a nonlinear operator~$N$ can be considered to be  `generalized symmetric' if it is the derivative of a convex functional~$\Phi$. This naturally leads to a generalization of `Onsager reciprocity' as the property that the macroscopic equations are a generalized gradient flow; Onsager's original case is recovered when the dissipation potentials $\calR$ and $\calR^*$ are quadratic.

This re-interpretation meshes well with the other half of Onsager's original contribution. Onsager proved the symmetry of a macroscopic kinetic relation by considering the consequences of `microscopic reversibility', a property that reduces to stochastic reversibility (detailed balance) in the case of microscopic Markov processes. The derivation of gradient systems from large-deviation principles (Theorem~\ref{ftheorem:LDP}) directly generalizes this: stochastic reversibility of the stochastic processes implies that their macroscopic limits are described by a gradient system, and this result transparently allows for both quadratic and non-quadratic gradient systems.

\medskip
\emph{Variational Modelling, tilt-independence, and tilt-dependence.}
The modelling of systems using gradient structures is known under various different names, such as \emph{Onsager's variational principle}~\cite{Doi11}, the \emph{Energetic Variational Approach}~\cite{HyonEisenbergLiu11}, or \emph{Variational Modelling}~\cite{PeletierVarMod14TR}. In this approach, modelling choices are formulated in terms of the components of gradient structures (energies, dissipations, continuity-equation structures, and such), and equations are derived from these. It has close connections to other approaches such as \emph{Maximum Entropy Production Principle}, \emph{Minimum Entropy Production Principle}, and a variational principle formulated by Ziegler~\cite{Ziegler58}; see the review~\cite{MartyushevSeleznev06} for a discussion.

In all uses of this method that are known to us, tilt-independence is implicitly assumed, and choices are made for energies and dissipations without discussion of possible relations between them. Most applications (e.g.~\cite{Doi11,HyonEisenbergLiu11,HyonFonsecaEisenbergLiu12,HorngLinLiuEisenberg12,XuDiDoi16,ZhouDoi18,Torres-SanchezMillanArroyo19}) focus on continuum-mechanical processes without chemical reactions, for which the assumption of tilt-independence is reasonable. Arroyo and co-authors~\cite{ArroyoWalaniTorres-SanchezKaurin18} do consider chemical reactions, and use a quadratic dissipation potential of the log-mean type (Remark~\ref{rem:other-gradient-structures}). They observe that the dissipation potential necessarily depends on the `chemical potential'; this implies tilt-dependence, although the authors do not discuss this in those terms. 

It appears that~\cite{FrenzelLiero21} and~\cite{MielkeMontefuscoPeletier21} were the first to discuss tilt-dependence and tilt-independence as such. We discuss Frenzel \& Liero's example~\cite{FrenzelLiero21} in Section~\ref{s:thin-membrane} and the convergence concepts of~\cite{MielkeMontefuscoPeletier21} in Section~\ref{ss:Kramers-not-contact-EDP-conv}. Mielke and Stephan~\cite{MielkeStephan2020} showed that requiring tilt-independence may automatically lead to cosh-type gradient systems, and in Section~\ref{ss:char-tilt-indep-GS} we give a result of  similar type under weaker  assumptions on the gradient system.

\subsection*{Acknowledgments}
The authors would like to thank Giuseppe Savar\'e, Chun Yin Lam, the members of the `Wednesday morning session' at Eindhoven University of Technology, and the members of the Research Group ``Partial Differential Equations'' at WIAS for many helpful comments. 

The authors very much appreciate several comments, questions, and an extensive lists of minor typos and mistakes from the three anonymous referees.

This work is funded by the Deutsche Forschungsgemeinschaft (DFG, German Research Foundation) under Germany's Excellence Strategy EXC 2044--390685587, Mathematics M\"unster: Dynamics--Geometry--Structure.

\section{Gradient systems}
\label{s:GS}

\subsection{Basic assumptions and notation}
\label{ss:assumptions-notation}

We use both measures on $\R^d$ and their Lebesgue density; we write `$\rho(\dx x)$' for the measure and `$\rho(x)$' for the Lebesgue density. In this way a symbol such as `$\rho$' can mean either, but the context will make clear which is intended. 

We will always assume that $\nodes$ is a topological space. This includes the case of a finite set~$\nodes$ (e.g.~Example~\ref{ex:heat-flow}), in which case we equip it with the discrete topology. We also assume that~$\edges$ is a locally compact Hausdorff  space. The space $\calM_{\mathrm {loc}}(\edges)$ is the set of locally finite Borel measures, i.e.\ the Borel measures $\mu$ on $\edges$ such that each point $e\in \edges$ has a neighbourhood $U$ such that $|\mu|(U) < \infty$, or equivalently, such that $|\mu|(K)<\infty$ for each compact set $K\subset\edges$.

The \emph{wide} topology on the set of locally finite Borel measures $\calM_{\mathrm{loc}}(X)$ on a locally compact space $X$ is the weak topology generated by $C_{\mathrm c}(X)$; the \emph{narrow} topology on the set of finite Borel measures is the weak topology generated by $C_{\mathrm b}(X)$. We will denote convergence in these topologies by $\stackrel*\longrightharpoonup$; the situation will make it clear which type is intended.

The following properties of $\sfC^*$ are used many times:
\begin{alignat}2
\label{eq:C-star-logpq}
\sqrt{pq}\, \sfC^*\bra[\big]{\log p - \log q} &= 2\bra*{\sqrt p - \sqrt q}^2,&\qquad &\text{for all $p,q>0$},\\
\sqrt{pq}\,{\sfC^*}'(\log p -\log q) & = p-q,&\qquad &\text{for all $p,q>0$}. \notag
\end{alignat}
Note that the right-hand sides are well-defined for $p,q\geq0$.

\subsection{Formal and rigorous formulations of gradient systems}
\label{ss:formal-rigorous}

Many of the discussions in Section~\ref{s:intro} were non-rigorous. In order to establish properties of gradient systems rigorously, we need proper definitions of the components of a gradient system and the resulting equations. We start with the concept of a solution. 

Let $(\sfV,\sfE,\anabla,\calE,\calR)$ be a gradient system in continuity-equation format (Definition~\ref{def:GradSystCE}). In~\eqref{eq:flux-GS-intro} we gave the corresponding evolution equation as 
\begin{equation}
\label{eq:flux-GS}	
\partial_t \rho + \adiv j = 0
\qquad \text{and}\qquad
j\in  \partial_2 \calR^*\bigl(\rho;\mathopen-\anabla \rmD\calE(\rho)\bigr).
\end{equation}
This `definition' of a solution needs specification in a number of aspects, and we now discuss these one by one. 

\medskip
\subsubsection{The continuity equation}
\label{sss:ct-eq}
	We give the equation $\partial_t\rho + \adiv j = 0$ a rigorous meaning in a measure-valued framework by adopting a formulation familiar from parabolic weak-solution theory (see e.g.~\cite[\S I.3]{LSU1968}) which incorporates the initial datum into the definition. Similar definitions can be found in~\cite{DolbeaultNazaretSavare2008,Erbar2014,ErbarFathiLaschosSchlichting16,Erbar2016TR,EspositoPatacchiniSchlichtingSlepcev2021,PeletierRossiSavareTse22,PeletierSchlottke21TR}.

Recall that $\nodes$ and $\edges$ are assumed to be topological spaces.
The definition of the continuity equation below assumes that the domain $\scrD_\nodes$ of the operator~$\anabla$ contains the set $C^1_{\mathrm c}(\nodes)$; this is the case for all the examples studied in this paper.
Note that in the context of martingale optimal transport~\cite{HuesmannTrevisan2019}, second order continuity equations occur. 
\begin{definition}[Continuity equation]\label{def:continuity-equation}
	For $T>0$ a pair~$(\rho(t,\cdot),j(t,\cdot))_{t\in [0,T]}$ satisfies the continuity equation  $\partial_t\rho + \adiv j = 0$ if:
	\begin{enumerate}
		\item \label{def:CE-part1-continuity}
		For each~$t\in[0,T]$, $\rho(t,\cdot)\in\calM_{\geq0}(\nodes)$, and 
		the map~$t\mapsto \rho(t,\cdot)$ is continuous with respect to the wide topology on~$\calM_{\geq0}(\nodes)$, that is
		\begin{equation*}
		 \forall \varphi\in C_{\mathrm c}(\nodes): \quad  [0,T] \ni t \mapsto \int_\nodes \varphi(x) \rho(t,\dx x) \text{ is continuous.}
		\end{equation*}
		\item \label{def:CE-part2-j}
		For each~$t\in[0,T]$, $j(t,\cdot)\in\calM_{\mathrm{loc}}(\edges)$, and 
		the map~$t\mapsto j(t,\cdot)$ is measurable with respect to the wide topology on~$\calM_{\mathrm{loc}}(\edges)$, that is
		\[
		  \forall \varphi\in C_{\mathrm c}(\edges) :\quad  [0,T]\ni t \mapsto \int_\edges \varphi(e) j(t,\dx e) \text{ is measurable},
		\]
		and the joint measure $|j|\in \calM([0,T]\times \edges)$ characterized by 
		\begin{equation*}
		\int_{t\in A} \abs*{j(t,\cdot)}(B)\, \dd t
		\qquad \text{for }A\subset [0,T], \ B\subset \edges \text{ open,}
		\end{equation*}
		is locally finite on $[0,T]\times \edges$. 
	
		\item \label{def:CE-part3-weak-eq}
		The pair solves $\partial_t\rho + \adiv j = 0$ in the sense that  for any $\varphi\in C_{\mathrm c}^1((0,T)\times \nodes)$,
		\begin{equation}
		\label{eq:weak-form-CE}
		 \int_0^T\!\!\int_{\nodes} \partial_t \varphi(t,x) \, \rho(t,\dx x) \dx t 
		  +\int_0^T\!\!\int_{\edges} \anabla \varphi(t,e) \, j(t,\dx e)\dx t = 0.
		\end{equation}
	\end{enumerate}
We denote by~$\CE(0,T)$ the set of all pairs~$(\rho,j)$ satisfying the continuity equation.
\qed
\end{definition}

\begin{remark}[Continuity of $\rho$]
Points~\ref{def:CE-part2-j} and~\ref{def:CE-part3-weak-eq} together imply that $t\mapsto \rho(t)$ can be replaced by a curve $t\mapsto \wt\rho(t)$ that is continuous in some weak topology defined by $\anabla$, such that $\wt\rho$ coincides with $\rho$ at Lebesgue almost every $t$. 
When $\anabla$ is a bounded linear operator as in Example~\ref{ex:heat-flow} or the classical gradient as in Example~\ref{ex:FP}, for instance, we find that $\wt\rho$ is widely continuous  (see e.g.~\cite[Lemma~4.4]{PeletierRossiSavareTse22} and~\cite[Lemma 8.1.2]{AmbrosioGigliSavare08}). 

With the additional condition of wide continuity in part~\ref{def:CE-part1-continuity} above, we require that $\rho$ already is this continuous representative, and in particular that the value at time $t=0$ connects continuously to $\rho(t)$ for $t>0$. This last point is important because of the role of initial data in \emph{a priori} estimates (e.g.~\eqref{est:initial-energy-D}).
\end{remark}

\subsubsection{Singularities of \texorpdfstring{$\calE$}{E} and \texorpdfstring{$\calR$}{R}}
\label{sss:formal-rigorous-singularities}
A next step in giving a rigorous definition of~\eqref{eq:flux-GS} is the occurrence of singularities: $\rmD\calE(\rho)$ might not exist, and if it does exist, it might not generate an admissible argument to $\calR^*$, that is in general $\anabla\rmD\calE(\rho)$ will not be in $\scrD_\edges$, by either lack of regularity or possible singularities.
In the case of Example~\ref{ex:heat-flow}, for instance, $\rmD\calE(\rho)(\sfx) = \log \rho_\sfx$, which is meaningless if $\rho_\sfx=0$; as a result, the corresponding expression
\[
\calR^*(\rho,-\ona\rmD\calE(\rho)) = 
\sum_{\sfx\sfy\in \edges} \sigma_{\sfx\sfy}(\rho)\sfC^*\bra*{-\ona \rmD\calE(\rho)_{\sfx\sfy}}
= \sum_{\sfx\sfy\in \edges}\frac12 {\speck_{\sfx\sfy}}\sqrt {\rho_\sfx\rho_\sfy} \sfC^*\bra[\big]{\log\rho_\sfx - \log{\rho_\sfy}}
\]
is formally also meaningless. However, because of the identity \eqref{eq:C-star-logpq},
we can rewrite
\begin{equation}\label{eq:Rstar:Hellinger}
\calR^*\bra[\big]{\rho,-\ona\rmD\calE(\rho)} 
= \sum_{\sfx\sfy\in\edges} \speck_{\sfx\sfy} \bra*{\sqrt{\rho_\sfx} - \sqrt{\rho_\sfy}}^2 ,
\end{equation}
and the right-hand side in~\eqref{eq:Rstar:Hellinger} is well-defined for all non-negative measures~$\rho$; the right-hand side can therefore be consdidered a rigorous replacement for the left-hand side.

This phenomenon happens in many systems: the term $\calR^*(\rho,-\anabla\rmD\calE)$ is not well-defined as it stands, but it admits a rigorous reformulation. 
In the metric-space gradient-flow framework of Ambrosio, Gigli, and Savar\'e~\cite{AmbrosioGigliSavare08}, for instance, the rigorous formulation takes the form of a `strong upper gradient', and in~\cite{PeletierRossiSavareTse22} a generalization is constructed for a wide class of non-quadratic gradient systems modelling jump processes. In this section we continue with  formal expressions of the type `$\calR^*(\rho,-\anabla\rmD\calE)$', while keeping in mind that a rigorous proof should be based on such a  reformulation. 

\subsubsection{Variational formulation}
A final problem with~\eqref{eq:flux-GS} arises when we want to pass to the limit in this equation: the nonlinearities of $\partial_2\calR^*$ and $\calE$ interact in a way that complicates convergence proofs. Various concepts have been developed that facilitate this, such as those based on time-discretized approximations~\cite{AmbrosioSavareZambotti09,Braides14} or the evolutionary variational inequality~\cite{Savare07, DaneriSavare10TR, Mielke16a}. 

In this paper we focus on the solution concept known as the \emph{Energy-Dissipation Principle (EDP)}~\cite{Mielke16a} or  \emph{curves of maximal slope}~\cite{DeGiorgiMarinoTosques80,AmbrosioGigliSavare08}.
The equivalence mentioned in~\eqref{eq:subdiff-characterization},
\begin{equation*}
j\in \partial_2 \calR^*(\rho,\Xi)
\quad\Longleftrightarrow\quad
\Xi\in \partial_2 \calR(\rho,j)
\quad\Longleftrightarrow\quad
\calR(\rho,j) + \calR^*(\rho,\Xi) = \dual j\Xi,
\end{equation*}
formally implies for any pair $(\rho,j)$ solving the continuity equation~\eqref{eq:ct-eq-abstract-intro} that
\begin{equation}\label{eq:EDP:estimate}
	\int_0^T\Bigl[ \calR(\rho,j)+\calR^*(\rho,-\anabla \rmD\calE(\rho))\Bigr]\dx t
	\geq -\int_0^T \int_\edges 
	\anabla \rmD\calE(\rho)(e)\, j(t,\dx e) \dx t
	= \calE(\rho(0))-\calE(\rho(T)).
\end{equation}
Again formally, equality holds if and only if $j \in \partial_2\calR^*(\rho,-\anabla \rmD\calE(\rho))$; this implies that equality in~\eqref{eq:EDP:estimate} can be used as a definition of solutions.

\begin{definition}[Solutions in the EDP sense]\label{def:EDP:sol}
	For each $T>0$ define the dissipation functional
	\begin{equation}\label{eqdef:ED:dissipation}
	\calD^T(\rho,j) := 
	\begin{cases}
		\ds \int_0^T \Bigl[ \calR(\rho,j)+\calR^*(\rho,-\anabla \rmD\calE(\rho))\Bigr]\dx t & \text{if }(\rho,j)\in \CE(0,T),\\
		+\infty &\text{otherwise}.
	\end{cases}
	\end{equation}
	and whenever $\calE(\rho(0))<\infty$ define the \emph{Energy-Dissipation} functional 
	\begin{equation}\label{eqdef:ED:functional}	
		\calI^T(\rho,j):= \calE(\rho(T))-\calE(\rho(0)) + \calD^T(\rho,j).
	\end{equation}
	A curve $(\rho,j)\in \CE(0,T)$ is an \emph{EDP solution} of the gradient system $(\sfV,\sfE,\anabla,\calE,\calR)$ provided that
	\begin{equation}\label{eqdef:EDP}	
	 	\calI^\tau(\rho,j) \leq 0 , \qquad  \text{for all } \tau\in (0,T] .
	\end{equation}
\end{definition}

For this definition to be meaningful, 
the formal inequality~\eqref{eq:EDP:estimate} should hold, meaning that the functional $\calI^\tau(\rho,j)\geq 0$ is non-negative; we formulate this as a property of a gradient system.
\begin{property}[Chain-rule lower bound]\label{property:ChainRuleLowerBound}
The gradient system $(\sfV,\sfE,\anabla,\calE,\calR)$ satisfies the \emph{chain-rule lower bound} if for all $(\rho,j)\in \CE(0,T)$ such that $\calI^T(\rho,j) + \calE(\rho(0)) < \infty$ and all $\tau\in [0,T]$ we have 
\begin{equation}\label{eq:ChainRuleLowerBound}
	\calE(\rho(\tau)) - \calE(\rho(0)) \geq \int_0^\tau \Dual{ \anabla\pra*{\rmD \calE(\rho(t))}(e)}{\, j(t,\dx e)} \dx t 
\end{equation}
and it particular
\begin{equation*}
	\calI^\tau(\rho,j) \geq 0 .
\end{equation*}
\end{property}
Similarly to the discussion in Section~\ref{sss:formal-rigorous-singularities}, the right-hand side in~\eqref{eq:ChainRuleLowerBound} is not obviously well-defined, since  $\rmD\calE(\rho)$ might not exist in any obvious interpretation. A proof of Property~\ref{property:ChainRuleLowerBound} therefore also has to show that the control of $\calI^T$ and $\calE(\rho(0))$ allows one to give meaning to the right-hand side; we do this for instance in Lemmas~\ref{l:CRLB-Kramers-epsilon} and \ref{l:CRLB-Kramers-limit}. For Example~\ref{ex:heat-flow}, for instance, this was done in~\cite[Cor.~4.20]{PeletierRossiSavareTse22}.

If Property~\ref{property:ChainRuleLowerBound} is satisfied, we obtain that EDP solutions to a gradient system $(\sfV,\sfE,\anabla,\calE,\calR)$ are characterized as elements of the zero-locus set $\set{\calI^T=0}$.


\subsubsection{Gradient systems which arise from large deviation principles}
\label{ss:ldp-GS}

A large class of gradient structures for evolution equations arise as the result of taking a deterministic limit of a sequence stochastic processes. The abstract result below illustrates the typical case, where both the gradient structure and the EDP definition of solutions arise from large-deviation principles. See~\cite{AdamsDirrPeletierZimmer2011,AdamsDirrPeletierZimmer2013,ErbarMaasRenger2015} for further discussion on this topic.

The space $D([0,T];\sfZ)$ is the Skorokhod space of c\`adl\`ag curves in $\sfZ$ (see e.g.~\cite[\S 3.5]{EthierKurtz09}).
\begin{Ftheorem}[{\cite[Prop.~3.7]{MielkePeletierRenger14}}]
\label{ftheorem:LDP}
Suppose that $Z^n$ is a sequence of continuous-time Markov processes in a space 
$\sfZ$ that are reversible with respect to their stationary measures
$\mu^n\in \ProbMeas(\sfZ)$. Assume that the following two
large-deviation principles hold: 
\begin{enumerate}
\item The invariant measures $\mu^n$ satisfy a large-deviation
  principle with rate function $\scrE:\sfZ\to[0,\infty]$, i.e.
\begin{equation*}
\mu^n \sim \exp\bigl(-n\scrE\bigr), \qquad \text{as }n\to\infty;
\end{equation*}
\item The time courses on $[0,T]$ of $Z^n$ satisfy a large-deviation principle in
  $D([0,T];\sfZ)$ with rate function $\RateFunc:D([0,T];\sfZ)\to[0,\infty]$,
  i.e.
\begin{equation}
\label{ldp:time-courses}
\Prob\bigl(Z^n \approx z \,\big|\, 
Z^n_0 \approx z(0)\bigr) \sim \exp\bigl(-n\RateFunc(z)\bigr),
\qquad \text{as }n\to\infty.
\end{equation}
\end{enumerate}

Then $\RateFunc$ can be written as 
\begin{equation}
\label{eq:I-GGS}
\RateFunc(z) = \frac12 \scrE(z(T)) - \frac12 \scrE(z(0)) + \frac12 \int_0^T \bigl[ \scrR(z,\dot z) + \scrR^*\bigl(z,-2\rmD \tfrac12 \scrE(z)\bigr)\bigr]\, dt,
\end{equation}
for some symmetric dissipation potential $\scrR$.
\end{Ftheorem} 

\noindent
This result suggests a number of observations:
\begin{enumerate}
\item The functional $\RateFunc$ in~\eqref{eq:I-GGS}  has the same structure as the functional $\calI^T$ introduced in~\eqref{eqdef:ED:functional}, up to factors $2$ (see Remark~\ref{rem:scaling-energy-in-GS}). Therefore  $\RateFunc$ is non-negative, and the equation $\RateFunc(z)=0$ can be considered a variational formulation of the equation $\dot z \in \partial_2 \scrR^*\bra*{z,-\rmD\scrE(z)}$. 
\item If the sequence of processes $Z^n$ is tight (which typically is proved as part of establishing~\eqref{ldp:time-courses}), then~\eqref{ldp:time-courses} implies that  a sequence of realizations $Z^n$  converges 
  (along subsequences) almost surely to a curve $z$ satisfying $\RateFunc(z)=0$. This identifies the gradient structure $(\sfZ,\scrE,\scrR)$ as the characterization of the limiting behaviour of~$Z^n$. This also gives the rate functional $\scrE$ of the invariant measure an after-the-fact interpretation as the functional that drives the gradient-flow evolution. 
\item The appearance of the functional $\RateFunc$ (or $\calI^T$) as rate functional in~\eqref{ldp:time-courses} gives the definition of EDP solutions a foundation in the upscaling of more microscopic systems. 
\item The connection between gradient structure and large-deviation principles also suggests design rules for the modelling of systems using gradient systems~\cite{Doi11,PeletierVarMod14TR}.
\item A central role in the proof is played by the stochastic reversibility, which gives rise to an exact differential of the rate functional $\scrE$. The difference $\tfrac12 (\scrE(z(T))-\scrE(z(0)))$ in~\eqref{eq:I-GGS} arises once the chain rule Property~\ref{property:ChainRuleLowerBound} holds with equality. 
\end{enumerate}
\noindent 
Many of the examples of cosh-type dissipations of this paper arise in this way, formally or rigorously, from large-deviation principles.


\begin{remark}[Rigorous characterization of $\RateFunc$]
Since the Skorokhod space $D([0,T];\sfZ)$ contains elements $z$ that are not even continuous in $t$, the definition~\eqref{eq:I-GGS} does not make sense as it stands; this is another instance of the problems described earlier in this section. In practice this is solved in two different ways, depending on the situation. In the first case, one proves that~$\RateFunc$ has a superlinear dependence on $\dot z$, by which finiteness of $\RateFunc$ implies that~$z$ is absolutely continuous (see~\cite[Lem.~A.5]{MielkePeletierRenger14} for an example). In the second case, when there is only linear control of~$\dot z$, discontinuities are unavoidable, and one adapts~\eqref{eq:I-GGS} following the theory of rate-independent systems (see e.g.~\cite[Def.~3]{BonaschiPeletier16} for an example, and~\cite{MielkeRoubicek15} for the general theory).
%
\end{remark}

\subsection{Convergence of gradient systems}
\label{ss:convergence-of-GS}

\subsubsection{EDP convergence for systems in continuity-equation format}


The use of the EDP solution concept of the previous section to prove convergence of gradient systems was pioneered by Sandier and Serfaty~\cite{SandierSerfaty04}, and has been  extended by Serfaty to metric spaces~\cite{Serfaty11} and by Liero, Mielke, Peletier, and Renger to cases where energetic effects are transferred to the limiting dissipation potentials~\cite{LieroMielkePeletierRenger17}. Mielke, Montefusco, and Peletier introduced additional, stronger convergence concepts that incorporate tilting~\cite{MielkeMontefuscoPeletier21}. Peletier and Schlottke also used this approach to establish convergence of a sequence of gradient systems to a limiting system that is no longer of gradient type, but still a variational evolution~\cite{PeletierSchlottke21TR}. Mielke~\cite{Mielke16a} discusses this method in the broad context of convergence of gradient systems. 

\bigskip
Let $(\nodes,\edges,\anabla,\calE_\eps,\calR_\eps)$ for $\eps>0$ be a family of gradient systems in continuity-equation format. 
Here, we impose a fixed node and edge set $(\nodes,\edges)$ and fixed gradient operator $\anabla$ linking both for all $\eps$ (see Remark~\ref{rem:comments:EDPreduce} why this is in general not a restriction). 
We start with the concept of convergence in the fixed continuity equation induced by $(\nodes,\edges,\anabla)$, which consists of the minimal assumptions to pass to the limit in~\eqref{eq:weak-form-CE} in Def.~\ref{def:continuity-equation} of $\CE(0,T)$. 
\begin{definition}[Convergence in $\CE$]
\label{def:converge-in-CE} 
A family $(\rho_\eps,j_\eps)\in \CE(0,T)$ converges to $(\rho_0,j_0)\in \CE(0,T)$  as $\eps\to 0$ if
\begin{enumerate}
\item $\rho_\eps$ converges widely to $\rho_0$ on $[0,T]\times \nodes$;
\item $j_\eps$ converges widely to $j_0$ on $[0,T]\times \edges$.
\end{enumerate}
\end{definition}
Note that wide convergence of $\rho_\e$ and $j_\e$ does not imply that the limit $\rho_0$ has finite mass or is continuous in time; therefore the requirement $(\rho_0,j_0)\in \CE(0,T)$ in Def.~\ref{def:converge-in-CE}  above implicitly places additional restrictions on the sequence $(\rho_\e,j_\e)$.  Also, the convergence of Def.~\ref{def:converge-in-CE} does not imply $\rho_\eps(t)\to \rho_0(t)$ for all $t\in [0,T]$. For all these properties, additional compactness is required, which typically follows from bounds on $\calE_\e$, $\calR_\e$, and $\calR^*_\e$.

\medskip
The definition of EDP convergence for gradient systems of~\cite{LieroMielkePeletierRenger17} is formulated in terms of the dissipation functions~\eqref{eqdef:ED:dissipation} for the $\e$-indexed systems,

\begin{equation*}
\calD_\eps^T(\rho_\eps,j_\eps) := 
\begin{cases}
	\ds \int_0^T \Bigl[ \calR_\eps(\rho_\eps,j_\eps)+\calR^*_\eps(\rho_\eps,-\anabla \rmD\calE_\eps(\rho_\eps))\Bigr]\dx t & \text{if } (\rho_\eps,j_\eps)\in\CE(0,T),\\
	+\infty &\text{otherwise}.
\end{cases}
\end{equation*}
\begin{definition}[EDP convergence for gradient systems in continuity-equation format]\label{defn:EDP-convergence}
A sequence $(\nodes,\edges,\anabla,\calE_\eps,\calR_\eps)$ \emph{EDP-converges} to a limit gradient structure in continuity-equation format $(\nodes,\edges,\anabla,\calE_0,\calR_0)$ as $\eps\to 0$ if 
\begin{enumerate}
\item $\calE_\eps \stackrel\Gamma\longrightarrow \calE_0$ with respect to the wide topology on $\calM_{\geq0}(\nodes)$;
\item 
For any sequence  $ (\rho_\eps,j_\eps)\longrightarrow (\rho_0,j_0)$ in $\CE(0,T)$ (see Definition~\ref{def:converge-in-CE}) with 
\[
\sup_{\e>0,t\in[0,T]} \calE(\rho_\e(t))<\infty
\]
and for any $\tau\in (0,T]$, we have 
\begin{equation}\label{defn:EDP-convergence:seq:a}
	\liminf_{\eps\to 0} \calD_\eps^\tau(\rho_\eps,j_\eps) \geq \calD_0^\tau(\rho_0,j_0) ,
\end{equation}
where $\calD_0^\tau$ is given by 
\begin{equation}\label{eqdef:D0}
	\calD_0^T(\rho_0,j_0) := 
	\begin{cases}
		\ds \int_0^T \Bigl[ \calR_0(\rho_0,j_0)+\calR^*_0(\rho_0,-\anabla \rmD\calE_0(\rho_0))\Bigr]\dx t & \text{if } (\rho,j)\in\CE(0,T),\\
		+\infty &\text{otherwise}.
	\end{cases}
\end{equation}
\end{enumerate}
\end{definition}

\begin{remark}[Comments on Definition~\ref{defn:EDP-convergence}]
\begin{enumerate}
\item Definition~\ref{defn:EDP-convergence} imposes the same continuity equation for $(\rho_\eps,j_\eps)$ for all~$\eps> 0$. 
This is convenient for passing to the limit, since the continuity equation acts as a fixed linear constraint, which works well with lower semicontinuity and convexity arguments. 
\item Since the EDP solutions at $\e>0$ are defined by the condition $\calI_\e^T(\rho_\e,j_\e)\leq0$, a minimal requirement for any concept of evolutionary $\Gamma$-convergence is that along sequences $(\rho_\eps,j_\eps)\in \CE(0,T)$ with $(\rho_\eps,j_\eps)\to (\rho_0,j_0)\in \CE(0,T)$ we have 
\begin{equation}
\label{conv:liminf:calI}
	\liminf_{n\to \infty} \calI_\eps^T(\rho_\eps,j_\eps) \geq \calI^T_0(\rho_0,j_0) . 
\end{equation}
However, the convergence~\eqref{conv:liminf:calI} does not ensure that the limit functional~$\calI_0^T$ is again of the form~\eqref{eqdef:ED:functional} (see~\cite{MielkeMontefuscoPeletier21,PeletierSchlottke21TR} for examples). This explains the need for additional conditions on $\calE_\eps$ and $\calD_\eps^T$, including the requirement that $\calD_0^T$ is of the form~\eqref{eqdef:D0}. 

\item The pioneering paper~\cite{SandierSerfaty04} and its follow-up~\cite{Serfaty11} assumed $\Gamma$-$\liminf$ bounds separately for the integrals of $\calR_\eps$ and $\calR^*_\eps$, which is a stronger requirement than the $\Gamma$-$\liminf$ bound~\eqref{defn:EDP-convergence:seq:a} for the sum~$\calD_\e^T$. 
As observed in~\cite{LieroMielkePeletierRenger17,MielkeStephan2020}, for limits in which part of the energy landscape `migrates' to the dissipation, such individual $\Gamma$-$\liminf$ estimates are unable to capture this migration, resulting in a gap in the $\Gamma$-$\liminf$ bound of the joint functional $\mathcal{D}_\eps^T$. By considering the sum $\calD_\e^T$ of the two terms one can prove a sharp bound.

\item Many similar EDP-based convergence results are specified in terms of $\Gamma$-convergence of $\calD_\e^T$, which requires a recovery-sequence property in addition to the lim-inf estimate of~\eqref{defn:EDP-convergence:seq:a}; see e.g.~\cite{SandierSerfaty04,Mielke16a,LieroMielkePeletierRenger17,MielkeMontefuscoPeletier21}. As is usual when using $\Gamma$-convergence, the recovery-sequence property functions as a guarantee that the lim-inf estimate is not unnecessarily weak. 

In this paper we  take a different route, following~\cite{Serfaty11}, by interpreting the chain-rule lower bound (Property~\ref{property:ChainRuleLowerBound}) as providing this guarantee instead. For instance, since no chain-rule lower bound is known for the Kramers problem with tilting, we prove it here (Lemma~\ref{l:CRLB-Kramers-limit}).

\end{enumerate}
\end{remark}

\subsubsection{EDP convergence with contracted limit continuity equation}
\label{sss:EDP-conv-contracted}

In practice, a modification is needed of Definition~\ref{defn:EDP-convergence} above. 
For many examples, and especially in multi-scale limits, the reconstruction of the building blocks $\mathcal{R}_0$ and $\mathcal{R}_0^*$ from~$\calD_0^\tau$ as in~\eqref{eqdef:D0} can be non-explicit and not suited for further analysis.
To arrive at explicit functionals $\mathcal{R}_0$ and $\mathcal{R}_0^*$, it is necessary to use information and compactness properties obtained from the bound $\sup_\eps \calI^T_\eps <\infty$ that is implicit in~\eqref{conv:liminf:calI}. In some situations, such as multi-scale limits, one can infer from this bound a contraction of the node and edge set, since one actually obtains  $\supp \rho_0 = [0,T]\times \wt\nodes$ and $\supp j_0 \in [0,T]\times \wt\edges$ for some $\wt\nodes \subseteq \nodes$ and $\wt\edges \subseteq \edges$. Then it is possible to restrict the limit structure of the continuity equation to $(\wt\nodes,\wt\edges,\wt\anabla)$ for a suitable restricted operator $\wt\anabla$. 
This situation will be encountered in Sections~\ref{s:Kramers}, \ref{s:thin-membrane},  and~\ref{s:two-terminal-reduction}. 
In this process it can happen that the gradient operator changes its nature from local to nonlocal (this happens for instance in Definitions~\ref{defn:Kramers:CE} and~\ref{defn:Kramers:CE0}).

Hence we suggest an additional definition of EDP convergence for gradient systems that contains such a  contraction step, in which we only ask for the contracted continuity equation, denoted by $\wt\CE$, in the limiting gradient structure $(\wt\nodes,\wt\edges,\wt\anabla,\wt\calE_0,\wt\calR_0)$. 
We use for any $\rho\in\calM_{\geq 0}(\nodes)$ and a node subset $\wt\nodes\subset \nodes$ the notation $\rho\lfloor\wt\nodes \in \calM_{\geq 0}(\wt\nodes)$ to denote the  measure restricted to $\wt\nodes$.
\begin{definition}[EDP convergence with contracted limiting continuity equations]\label{defn:EDP-convergence:reduce}
A sequence $(\nodes,\edges,\anabla,\calE_\eps,\calR_\eps)$ \emph{EDP-converges} to a limit gradient structure in continuity-equation format $(\wt\nodes,\wt\edges,\wt\anabla,\wt\calE,\wt\calR)$ if 
\begin{enumerate}
\item For some $\calE_0: \calM_{\geq 0}(\nodes)\to \R$ we have $\calE_\eps \stackrel\Gamma\longrightarrow \calE_0$ with respect to the wide topology on $\calM_{\geq0}(\nodes)$;
\item \label{def:EDP-convergence:E:restriction} $\calE_0$ has an alternative representation in terms of an $\wt\calE:\calM_{\geq0}(\wt\nodes) \to \R\cup\{+\infty\}$ as 
\[
\calE_0(\rho)= \begin{cases}
\wt\calE(\rho\lfloor\wt\nodes) & \text{if }\rho(\nodes\setminus \wt\nodes) = 0\\
+\infty & \text{otherwise}.
\end{cases}
\]
\item \label{defn:EDP-convergence:reduce:sequence}
There exists a  functional $\calD_0^\tau$ on $\CE(0,T)$ 
such that 
for any sequence  $ (\rho_\eps,j_\eps)\longrightarrow (\rho_0,j_0)$ in $\CE(0,T)$ (see Definition~\ref{def:converge-in-CE}) with $\sup\limits_{\e>0,t\in[0,T]} \calE(\rho_\e(t))<\infty$  and any $\tau\in (0,T]$ we have 
\begin{equation*}
	\liminf_\eps \calD_\eps^\tau(\rho_\eps,j_\eps) \geq \calD_0^\tau(\rho_0,j_0) .
\end{equation*}
\item There exists a functional $\wt\calD^\tau$ on $\wt\CE(0,T)$ such that whenever $\calE_0(\rho_0(t))<\infty$ for all $t$, there exists a $\tilde\jmath$ such that for $\wt\rho(t) := \rho_0(t)\lfloor\wt\nodes$ we have $(\tilde\rho,\tilde\jmath)\in \wt\CE(0,T)$ and
\begin{align*}
	\calD_0^\tau(\rho_0,j_0) &= \wt\calD^\tau(\tilde\rho,\tilde\jmath).
\end{align*}
$\wt\calD^\tau$ has the structure
\[
\wt\calD^\tau(\tilde\rho,\tilde\jmath) := 
\begin{cases}
\ds\int_0^\tau \Bigl[ \wt\calR(\tilde \rho,\tilde\jmath)+\wt\calR^*(\tilde \rho,-\wt\anabla \rmD\wt\calE(\tilde\rho))\Bigr]\dx t & \text{if }(\tilde\rho,\tilde\jmath)\in \wt\CE(0,T),\\
		+\infty &\text{otherwise}.
\end{cases}
\]
\end{enumerate}
\end{definition}

\begin{remark}[Comments on Definition~\ref{defn:EDP-convergence:reduce}]\label{rem:comments:EDPreduce}\leavevmode
\begin{enumerate}
	\item Point~\ref{def:EDP-convergence:E:restriction} in Definition~\ref{defn:EDP-convergence:reduce} shows how the limiting energy functional $\calE_0$ enforces the restriction of $\rho_0$ to $\wt\nodes$; an example we will encounter below is $\calE_0(\rho) = \RelEnt(\rho|\pi_0)$, with $\supp\pi_0 \subsetneqq \nodes$. From there the restrictions from $\edges$ to $\wt\edges$ typically follow  from how $\edges$ was constructed from $\nodes$. For instance, if $\edges \subset \nodes \times \nodes$, then $\wt\edges = \edges \cap (\wt\nodes \times \wt\nodes)$ is a possible choice (see for instance Section~\ref{s:two-terminal-reduction}). In Section~\ref{s:Kramers}  we will also encounter the situation where $\nodes$ is continuous and $\wt\nodes$ is discrete, which has the consequence that the local gradient $\anabla$ morphs in to a nonlocal gradient $\wt\anabla$ in the limit.
	\item Definition~\ref{defn:EDP-convergence:reduce} is well suited to this paper because the emergence of cosh-type gradient structure is closely tied to multi-scale limits, in which a contraction from continuous to discrete or nonlocal state spaces occurs. In comparison, Definition~\ref{defn:EDP-convergence} covers well the case of discrete--to--continuum limits (see e.g.~\cite{DisserLiero2015,Schlichting2019}), mean-field limits (see e.g.~\cite{ErbarFathiLaschosSchlichting16}), and homogenization (see e.g.~\cite{DondlFrenzelMielke19,GladbackKopferMaas2020,GladbackKopferMaasPortinale2020,HraivoronskaTse2022TR}).
	In these cases the problem statement usually involves $\eps$-dependent node and edge sets $(\nodes_\eps,\edges_\eps)$ leading in general also to an $\eps$-dependent continuity equation $\CE_\eps$ given in terms of a suitable $\anabla_{\!\eps}$. Those problems are covered by Definition~\ref{defn:EDP-convergence} with the help of a suitable \emph{projection} map $\Pi_\eps: \CE_\eps(0,T)\to \CE(0,T)$, where $\CE$ is a continuity equation induced by a common structure $(\nodes,\edges,\anabla)$.
	
\item In Definition~\ref{defn:EDP-convergence:reduce} the restriction of $\rho_0$ to $\calM_{\geq0}(\wt\nodes)$ is forced by the energy alone: $\calE_0(\rho_0(t))<\infty$ is sufficient to restrict $\rho(t)$ to $\calM_{\geq0}(\wt\nodes)$.
In other cases, however, such restriction comes from the dissipation functional $\calD_0^T$ instead of from $\calE_\e$. In the examples of~\cite{MielkePeletierStephan21,Stephan21}, for instance, the fast-reaction nature enters not through $\calE_\e$ but through $\calR_\e$ and $\calR^*_\e$, and finiteness of $\calD_0^T(\rho_0,j_0)$ implies that $\rho_0(t)$ lies in a smaller set for almost all $t\in[0,T]$. EDP convergence in such situations can be defined in a similar manner. \qedhere
\end{enumerate}
\end{remark}
With this definition, the convergence of solutions to solutions becomes a question of simple verification of the necessary sequential lower semicontinuity for the involved functionals.

\begin{lemma}[Solutions converge to solutions]
\label{l:solns-converge-to-solns}
Let $(\nodes,\edges,\anabla,\calE_\eps,\calR_\eps)$ \emph{EDP converge} to $(\wt\nodes,\wt\edges,\wt\anabla,\wt\calE,\wt\calR)$ according to Definition~\ref{defn:EDP-convergence:reduce}. Let $(\rho_\eps,j_\eps)\to (\rho_0,j_0)$ as in point~\ref{defn:EDP-convergence:reduce:sequence} of Definition~\ref{defn:EDP-convergence:reduce}, and assume in addition
\begin{enumerate}
	\item \label{l:solns-converge-to-solns:1}
	pointwise-in-time limits:
	$\rho_\eps(t) \weaksto \rho_0(t)$ widely for all $t\in(0,T]$ and $(\rho_0,j_0)\in\CE(0,T)$;
	\item \label{l:solns-converge-to-solns:2}
	well-prepared initial datum: 
	$\calE_\eps(\rho_\eps(0)) \to \calE_0(\rho_0(0))$;
	\item \label{l:solns-converge-to-solns:3}
	EDP solution property: 
	for all $\tau\in(0,T]$ we have 
	\[
	\calI_\eps^\tau(\rho_\eps,j_\eps)= \calE_\eps(\rho_\eps(\tau))- \calE_\eps(\rho_\eps(0)) + \calD_\eps^\tau(\rho_\eps,j_\eps)\leq 0 .
	\]
\end{enumerate}
Then for all $\tau\in(0,T]$,
\begin{equation}
\label{eq:EDP-limit}
\tilde\calI(\tilde\rho,\tilde\jmath) = \tilde\calE(\tilde\rho(\tau)) - \tilde\calE(\tilde\rho(0)) + \wt\calD^\tau(\tilde\rho,\tilde \jmath) \leq 0 ,
\end{equation}
that is, $(\tilde\rho,\tilde\jmath)\in\wt\CE(0,T)$ is an EDP solution of $(\wt\nodes,\wt\edges,\wt\anabla,\wt\calE,\wt\calR)$.
\end{lemma}
\begin{proof}
By $\calE_\eps \stackrel\Gamma\longrightarrow \calE_0$ and assumption~\ref{l:solns-converge-to-solns:1} we have $\liminf_{\eps\to 0}\calE_\eps(\rho_\eps(t))\geq \calE_0(\rho_0(t)) = \wt\calE(\tilde\rho(t))$ for all $t\in (0,T]$.
From point~\ref{defn:EDP-convergence:reduce:sequence} of Definition~\ref{defn:EDP-convergence:reduce} we obtain $\liminf_\eps\calD_\eps^\tau (\rho_\eps,j_\eps) \geq \calD_0^\tau(\rho_0,j_0) = \wt\calD_0^\tau(\tilde\rho,\tilde\jmath)$ for all $\tau$. The inequality~\eqref{eq:EDP-limit} then follows from assumptions~\ref{l:solns-converge-to-solns:2} and~\ref{l:solns-converge-to-solns:3}.
\end{proof}
\begin{remark}[Well-preparedness of the initial data]
\label{rem:well-preparedness-initial-data}
Assumption~\ref{l:solns-converge-to-solns:2} in Lemma~\ref{l:solns-converge-to-solns} above hides a non-trivial assumption, which becomes clearer by recalling the time-continuity of $[0,T]\ni t\mapsto \rho_0(t)$ implied by $(\rho_0,j_0)\in\CE(0,T)$ and writing the convergence as 
\[
\lim_{\e\to0} \calE_\e\bra*{\lim_{t\downarrow0}\rho_\e(t)} = \calE_0\bra*{\lim_{t\downarrow0} \lim_{\e\to0} \rho_\e(t)}.
\]
This form shows that there actually are two independent assumptions:
\begin{enumerate}
\item $\rho_\e(0) = \lim_{t\downarrow0} \rho_\e(t)$ should be a recovery sequence for the convergence $\calE_\e\stackrel\Gamma\longrightarrow\calE_0$, i.e.\ $\rho_\e(0)\to\ol\rho$ with $\calE_\e(\rho_\e(0)) \to \calE_0(\ol\rho)$;
\item $\ol\rho = \lim_\e\lim_t \rho_\e(t)$ may be different from  $\rho_0(0) = \lim_t\lim_\e \rho_\e(t)$, but the two energy values should be the same: $\calE_0(\ol\rho) = \calE_0(\rho_0(0))$.
\end{enumerate}
The first assumption is the classical `well-preparedness' assumption that appears in many evolutionary $\Gamma$-convergence results; see e.g.~\cite[\S 3.3]{Mielke16a} for a discussion. 

The second assumption is a weak equicontinuity assumption: it states that the convergence $\rho_\e\to\rho_0$ should be strong enough to ensure preservation of \emph{part} of the continuity of $t\mapsto \rho_\e(t)$, namely that part that determines the value of $\calE_0$. For an example of this, see~\cite{PeletierRenger21}, where the dissipation bounds only control \emph{some} components of the measure $\rho_\e$; correspondingly, the limit energy in~\cite{PeletierRenger21} is defined as a minimum over the non-controlled degrees of freedom (see the definition of $\wt{\amsfontsmathcal I}_0^0$ in ~\cite[Th.~1.2]{PeletierRenger21}), thus recovering the lower bound inequality.

In many cases, the second assumption even requires full preservation of continuity, i.e.\ $\ol\rho = \rho_0(0)$. This can be recognized as follows.
If $\lim_\e\lim_t \rho_\e(t) \not = \lim_t\lim_\e \rho_\e(t)$, then there is a rapid initial transient in $\rho_\e$. If $\rho_\e$ is a solution, then this transient is  driven by a decrease in energy during that short initial transient; however,  the condition  $\calE_\e(\rho_\e(0)) \to \calE_0(\rho_0(0))$ implies that the size of the decrease during the transient converges to zero. In many cases, therefore, the condition  $\calE_\e(\rho_\e(0)) \to \calE_0(\rho_0(0))$ prevents fast initial transients from occurring, and equicontinuity is preserved.
\end{remark}

Definitions~\ref{defn:EDP-convergence} and~\ref{defn:EDP-convergence:reduce} give rise to corresponding definitions for tilt gradient systems:

\begin{definition}[EDP convergence for tilt gradient systems]\label{defn:EDP-tilting}
	A tilt gradient system $(\nodes,\edges,\anabla,\calE_\eps,\calR_\eps,\sfF)$ EDP-converges to a tilt gradient system $(\nodes,\edges,\anabla,\calE_0,\calR_0,\sfF)$ if for all $\calF\in \sfF$ the gradient system $(\nodes,\edges,\anabla,\calE_\eps+\calF,\calR_\eps(\cdot,\cdot;\calF))$ EDP-converges to $(\nodes,\edges,\anabla,\calE_0+\calF,\calR_0(\cdot,\cdot;\calF))$.
	
	The tilted version of the EDP convergence with contracted limit is defined similarly.
\end{definition}
The definition above is especially interesting if either the pre-limit structure or the limit structure has a tilt-independent dissipation potential. In this paper we mainly are  interested in the situation where tilt-independent systems converge to a tilt-dependent limit. 

\begin{remark}[Tilt- and contact-EDP convergence]
In~\cite{MielkeMontefuscoPeletier21} two variants of EDP convergence were introduced that guarantee that  tilt-independent sequences have tilt-indep\-en\-dent limits; they are called \emph{tilt-EDP convergence} and \emph{contact-EDP convergence}. The relation with the convergence concepts  above is as follows: Tilt-EDP convergence is the special case of Definition~\ref{defn:EDP-tilting} in which both $\calR_\e$ and $\calR_0$ are tilt-independent; contact-EDP convergence is a weaker concept, in which again  $\calR_\e$ and $\calR_0$ are tilt-independent, but the convergence of $\calD_\e^T$ to $\calD_0^T$ only is required on the so-called \emph{contact set}. We give more details in Section~\ref{ss:Kramers-not-contact-EDP-conv}, and we also show that the Kramers limit of Section~\ref{s:Kramers} is neither tilt- nor contact-EDP convergent. 
\end{remark}

\begin{remark}[Convergent sequences of tilts]
	Definition~\ref{defn:EDP-tilting} can be strengthened if the set $\sfF$ itself is a  space with a convergence concept. Then we can ask for the stronger statement that for every $\set{\calF_\e}_{\e>0}\subset \sfF$ such that $\calF_\e \to \calF_0 \in \sfF$ the gradient system $(\nodes,\edges,\anabla,\calE_\eps+\calF_\e,\calR_\eps(\cdot,\cdot;\calF_\e))$ EDP-converges to $(\nodes,\edges,\anabla,\calE_0+\calF_0,\calR_0(\cdot,\cdot;\calF_0))$. In Section~\ref{s:Kramers} this situation is discussed in detail.
\end{remark}

\section{Properties of the \texorpdfstring{$\sfC$-$\sfC^*$}{cosh} gradient structure}
\label{s:props-of-CCstar}

In this section, we collect some identities and variational problems involving the functions $\sfC$ and $\sfC^*$ from~\eqref{eqdef:C-C*}. 

\subsection{Perspective functions}\label{ss:perspective-function}

\begin{definition}[Perspective functions]\label{defn:perspective-function}
Let $f:\R^d\to[0,\infty]$ be convex, lower semicontinuous, and superlinear, i.e.\ $\lim_{|x|\to\infty} f(x)/|x|=+\infty$. The \emph{perspective} function of $f$ is the function $f(\,\cdot\,|\,\cdot\,): \R^n\times [0,\infty) \to (-\infty,\infty]$ given by
\[
f(a|b) := \begin{cases}
\ds b f\bra*{\frac ab} & \text{if }b>0,\\
f(0) & \text{if } b = 0 \text{ and } a=0 , \\
+\infty & \text{if $b=0$ and $a\not=0$}.
\end{cases}
\]
\end{definition}
\noindent
See~\cite[\S 5]{Rockafellar1970}, \cite[IV.2.2]{HiriartUrrutyLemarechal1993}, or \cite{Combettes18}; the name `perspective' function appears to have been coined by Lemar\'echal~\cite{Combettes18}.
We already encountered many perspective functions explicitly and implicitly in the previous sections, such as the relative entropy $\RelEnt$ in~\eqref{eqdef:RelEnt} and its density $\eta$~\eqref{eqdef:eta}, and the dissipation functional $\calR$ in~\eqref{eqdef:R-jump}.
We collect some properties of perspective functions, for which a proof can be found in~\cite[Lemma 2.3]{PeletierRossiSavareTse22}.
\begin{lemma}[Properties]\label{lem:prop:perspective}
For any convex, lower semicontinuous, and superlinear function~$f:\R^d \to [0,\infty]$,
its perspective function $(a,b) \mapsto f(a| b)$ is convex, lower semicontinuous, and positively $1$-homogeneous in the pair~$(a,b)\in \R^d \times [0,\infty)$. It has the dual formulation
\begin{equation}
\label{eq:dual-perspective}
f(a|b) = \sup_{\xi\in\R^d}\, \pra*{\xi\ip a - b f^*(\xi)}, \qquad\text{for }(a,b)\in\R^d \times [0,\infty),
\end{equation}
where $f^*$ is the Legendre dual of $f$.
\end{lemma}

Let $\Omega$ and $\calY$ be  topological spaces, and let $g:\Omega\times \calY \to[0,\infty)$ be  lower semicontinuous,  positively one-homogeneous,  and  convex in the second variable. Define the functional $\calG_g:\calM(\Omega;\calY)\to[0,\infty]$ by 
\[
\calG_g (\mu) := \int_\Omega g\bra*{x,\frac{\dx\mu}{\dx\gamma}(x)}\gamma(\dx x)
\]
for any $\gamma\in \calM_{\geq0}(\Omega)$ such that $\mu\ll \gamma$ (e.g.\ $\gamma= |\mu|$. This definition is independent of the choice of $\gamma$ because of the one-homogeneity of $g$. 

In particular, if $f\geq0$ is as in Lemma~\ref{lem:prop:perspective}, then define the functional $\calF_f :\calM(\Omega)\times \calM_{\geq0}(\Omega)\to [0,\infty]$  by
\[
\calF_f(\mu|\nu) := \int_\Omega f\Bigl(\frac{\dx\mu}{\dx\gamma}(x)\Big|\frac{\dx \nu}{\dx\gamma}(x)\Bigr) \gamma(\dx x),
\]
where $\gamma$ is any measure such that $|\mu|\ll \gamma$ and $|\nu|\ll \gamma$ (e.g.\ $\gamma = |\mu| + |\nu|$). The functional $\calF_f$ then is well-defined, positive, and independent of the choice of $\gamma$.

\subsection{Cell formula}
\label{ss:cell-formula}

In both the Kramers high-activation limit (Section~\ref{s:Kramers}) and the two-terminal networks (Section~\ref{s:ha}) we encounter the function $\CCs$, which has as formal definition
\begin{multline*}
	\CCs(j;\alpha,\beta;k) := \inf \Biggl\{
	\int_0^1 \pra[\bigg]{ \frac{j^2(z)}{2k(z)u(z)} + 2 k(z)\abs[\big]{\partial_z\sqrt {u(z)}}^2}\dx z: \ u \in C^1([0,1]),\\[-4\jot]
	u(0) = \alpha, \ u(1) = \beta, \ u>0 \text{ on }(0,1)\Biggr\}.
\end{multline*}
The expression on the right-hand side is a generalization of~\eqref{eq:CCs:cell-problem} to the case of non-constant~$k$.

The definition above does not work well for $\alpha=0$ or $\beta=0$, and we therefore  give a rigorous definition of $\CCs$ by using duality. 	Let $\calY := H^{-1}(0,1)\times \R\times \R$ be the Hilbert space with norm $\|(j,\alpha,\beta)\|^2_\calY:= \|j\|_{H^{-1}(0,1)}^2  + |\alpha|^2 + |\beta|^2$. 
Define $\CCs:\calY\times L^1(0,1) \to [0,\infty]$ by
	\begin{multline}
		\label{eqdef:calN}
		\CCs(j;\alpha,\beta;k) := \inf \Biggl\{
		\sup_{\varphi\in H^1_0(0,1)}\dual j\varphi 
		+ \int_0^1 k(z)\pra*{-\frac1{2} {v^2}\varphi^2 + 2 \abs*{v'}^2}\dx z:
		\ v\in H^1(0,1),\\
		v(0) = \sqrt{{\alpha}}, \ v(1) = \sqrt{{\beta}}, \ v>0 \text{ on }(0,1)\Biggr\}.
	\end{multline}
The value of $\CCs(j;\alpha,\beta;k)$ is set to $+\infty$ if $\alpha<0$ or $\beta<0$.
\begin{lemma}
	\label{l:props-N}
	The function $\CCs$ has the following properties:
	\begin{enumerate}
		\item \label{l:props-N:one-homogeneous}
			The function $\CCs$ is even in ${j}$ and has the following joint positive one-homogeneity properties:
			For $\alpha,\beta \geq 0$, $\lambda>0$, $j\in H^{-1}(0,1)$ and $k\in L^1(0,1)$ we have
			\begin{equation}\label{eq:G:one-homogeneous}
				\CCs(\lambda j, \lambda\alpha,\lambda\beta; k) 
				=\lambda \CCs(j,\alpha,\beta;k) 
				= \CCs(\lambda j, \alpha, \beta;\lambda k) ,
			\end{equation}
			and in particular $\CCs(j, \lambda\alpha,\lambda\beta; k) = \CCs(j, \alpha,\beta; \lambda^{-1} k)$.
		\item \label{l:props-N:char-const-j}
			If $j$ and $k>0$ are both constant on $(0,1)$, then for all $\alpha,\beta>0$,
			\begin{align}
				\CCs(j,\alpha,\beta;k) &= 
				\sigma\bra*{\sfC\bra*{\frac{j }\sigma} + \sfC^*\bra*{\log \beta-\log \alpha}}
				\qquad\text{with}\qquad
				\sigma = k\sqrt {\alpha\beta},
				\label{eqdef:N:explicit} \\
				&=\sfC(j |\sigma) + 2k\bra*{\sqrt{\alpha}-\sqrt\beta}^2, \label{eqdef:N:explicit:rigrous}
			\end{align}
			and the second expression above applies for all $\alpha,\beta\geq0$. 
		\item \label{l:props-N:infty}
			In particular, if $j$ is constant and non-zero, $k>0$,  and $\alpha\beta=0$, then $\CCs(j,\alpha,\beta;k) = +\infty$.
		\item \label{l:props-N:harmonic-mean}
			If $j$ is constant and $k: [0,1]\to(0,\infty)$ with $\frac{1}{k}\in L^1(0,1)$ we have
			\begin{equation}\label{eqdef:G:harmonic-mean}
				\CCs(j,\alpha,\beta;k) = \CCs(j,\alpha,\beta;k^*) \quad\text{with}\quad k^*:= \bra*{\int_0^1 \frac{1}{k(z)} \dx{z}}^{-1} . 
			\end{equation}
		\item  \label{l:props-N:est-B}
			Define the $[-\infty,\infty]$-valued function $B:[0,\infty)^2 \times \R\to[-\infty,\infty]$ by 
			\begin{equation}
				\label{eqdef:B}
				B(\alpha,\beta,j) := \begin{cases}
					j(\log \beta - \log \alpha) & \text{if }\alpha,\beta>0;\\
					+\infty & \text{if $\alpha=0$, $\beta>0$, and $j>0$}\\
				 			& \qquad\qquad  \text{or $\alpha>0$, $\beta=0$, and $j<0$,}\\
					-\infty & \text{if $\alpha=0$, $\beta>0$, and $j<0$}\\
							& \qquad\qquad  \text{or $\alpha>0$, $\beta=0$, and $j>0$,}\\
					0 &\text{if $\alpha=\beta = 0$ or $j=0$}.	
				\end{cases}
			\end{equation}
			For any $\alpha,\beta\in [0,\infty)$, $j\in \R$, and $k>0$, we have the extended-real-number inequality 
			\begin{equation}
				\label{ineq:props-N:BG}		
				|B(\alpha,\beta,j)| \leq \CCs(j,\alpha,\beta;k),
			\end{equation}
			where on the right-hand side $j$ and $k$ are considered as constants. 
		\item \label{l:props-N:dual-char-1}
			For all $k>0$ and $\alpha,\beta \geq 0$, and $j\in \R$ we have the dual characterization
			\begin{equation}\label{eq:props-N:dual-char-1}
				\CCs(j;\alpha,\beta;k) = 
				\sup_{\zeta\in \R} \pra*{\zeta j + 2k\bra[\big]{\alpha +\beta -2\sqrt{\alpha\beta}\cosh\tfrac12\zeta}},
			\end{equation}
			and if the supremum is finite, then it is achieved as a maximum.
	\end{enumerate}
\end{lemma}
\begin{proof}
	The even dependence on $j$ in~\eqref{eqdef:calN} is a consequence 
	of the sign of $\varphi$ only entering the pairing $\skp{j,\varphi}$, 
	but not the second term. 
	The one-homogeneity~\eqref{eq:G:one-homogeneous} straightforwardly follows from the definition~\eqref{eqdef:calN}.
		
	\medskip
	For part~\ref{l:props-N:char-const-j}, setting $\tilde\jmath = j/k$  and $w = v^2$ we rewrite for $\alpha,\beta>0$
	\begin{align*}
		\inf_v \set*{\int_0^1 \pra*{\frac{ j^2}{2k v^2} + 2k{v'}^2}\dx z: \sqrt\alpha\stackrel v\rightsquigarrow\sqrt\beta}
		&= k\inf_w \set*{ \int_0^1 \pra*{\frac{\tilde \jmath ^2+{w'}^2}{2w}}\dx z: \alpha\stackrel w \rightsquigarrow \beta}\\
		&\leftstackrel{(*)}= k \sqrt{\alpha\beta} \bra*{\sfC\bra[\Big]{\frac{\tilde\jmath}{\sqrt{\alpha\beta}}} 
			+ \sfC^*\bra*{\log \beta-\log \alpha}}\\
		&= \sfC\bra[\big]{\,j\,\big| \, k\sqrt{\alpha\beta}} 
			+ 2k\bra[\big]{\sqrt{\alpha}-\sqrt\beta}^2,
	\end{align*}
	where the identity marked $(*)$ follows from~\cite[Prop.~A.1]{LieroMielkePeletierRenger17}, and the final identity from~\eqref{eq:C-star-logpq}. 
	
	If $j=0$ and $\alpha\beta=0$, then the value of $\CCs$ follows from a direct calculation of the minimizer in~\eqref{eqdef:calN}.
	If $j\not=0$ and $\alpha\beta=0$, then the value $+\infty$ (part~\ref{l:props-N:infty}) follows from the definition of the perspective-function of $\sfC$ (Def.~\ref{defn:perspective-function}).
	
	\medskip
	Part~\ref{l:props-N:harmonic-mean} can be reduced to the case of constant $k$ from part~\eqref{l:props-N:char-const-j} by introducing a change of coordinate with the help of
	\[
		z = Z(y) = k^* \int_0^y \frac{1}{k(x)} \dx{x} .
	\]
	Note that $Z(0)=0$, $Z(1)=1$ and $Z'(y) = k^*/k(y)$, 
	which in the definition of $\CCs$ in~\eqref{eqdef:calN} provides immediately the claim~\eqref{eqdef:G:harmonic-mean}.

	\medskip
	For part~\ref{l:props-N:est-B} we assume that $\CCs(j,\alpha,\beta;k)<\infty$, 
	since otherwise there is nothing to prove. 
	If $j=0$, then $B(j,\alpha,\beta)=0$, and inequality~\eqref{ineq:props-N:BG} is trivially satisfied. 
	If $j\not=0$, then by part~\ref{l:props-N:infty} we have $\alpha\beta>0$ and therefore $B(\alpha,\beta,j)$ is finite. 
	The inequality~\eqref{ineq:props-N:BG} then follows from 
	\[
		\abs*{\,j \log\frac\beta\alpha }
		= \sigma \abs*{\,\frac{j}{{\sigma}}\cdot \log\frac\beta\alpha }
		\leq \sigma \,\bra* {\sfC\bra*{\frac{j}\sigma } + \sfC^*\bra[\Big]{\log\frac\beta\alpha}}
		= \CCs(j,\alpha,\beta;k)
	\]
	with $\sigma = k\sqrt{\alpha\beta}$. A similar property is proved in more generality in~\cite[Cor.~4.20]{PeletierRossiSavareTse22}.

	\medskip
	For part~\ref{l:props-N:dual-char-1}, note that by~\eqref{eq:dual-perspective},
	\[
	\sfC(j|\sigma) = \sup_{\xi\in\R} \,\pra[\big]{\xi j - \sigma\sfC^*(\xi)}
	\qquad \text{for }j\in \R, \ \sigma\geq 0.
	\]
	Then the characterization~\eqref{eq:props-N:dual-char-1} follows by adding $2k(\sqrt\alpha-\sqrt\beta)^2$ and rearranging. 
\end{proof}

\begin{cor}[Series law]\label{cor:N:series}
	The function $\CCs$ satisfies the following series law:
	For $\alpha,\beta \geq 0$, $j\in \R$ and $k^1,k^2:[0,1]\to (0,\infty)$ we have 
	\begin{equation}\label{eq:props-N:addition} 
		\inf_{\gamma>0} \bra*{\CCs(j,\alpha,\gamma;k^1)+ \CCs(j,\gamma,\beta;k^2)}
		= \CCs(j,\alpha,\beta;k)
	\end{equation}
	with 
	\begin{equation*}
		k(z) :=\frac{1}{2} 
		\begin{cases} 
			k^1(2z) , & z\in (0,1/2) \\
			k^2(2z-1) , & z\in (1/2,1)
		\end{cases} .
	\end{equation*}
	In particular, if $k^1,k^2\in (0,\infty)$ are constant, then we have
	\begin{equation}\label{eq:N:series}
		\inf_{\gamma\geq 0}\bra*{ \CCs(j,\alpha,\gamma;k^1)+\CCs(j,\gamma,\beta;k^2) } = \CCs\bra*{j,\alpha,\beta;\bra*{\frac{1}{k^1} + \frac{1}{k^2}}^{-1}},
	\end{equation}
\end{cor}
\begin{proof}
	We rewrite the left-hand side of~\eqref{eq:props-N:addition}
	by using the variational definition~\eqref{eqdef:calN}. Note that since $j$ is constant, the supremum in~\eqref{eqdef:calN} over $\varphi\in H^1_0(0,1)$ can equivalently be taken over $\varphi\in L^1(0,1)$.	
	
	Starting with the left-hand side as an optimization over $\gamma$ and over $v^1,\varphi^1$ and $v^2,\varphi^2$, thanks to the shared boundary conditions $v^1(1)=\sqrt\gamma=v^2(0)$, 
	we can equivalently optimize over $v\in H^1(0,1)$ satisfying $v(0)=\sqrt{\alpha}$ and $v(1)=\sqrt{\beta}$; $v$ is related to $v^1,v^2$ by $v^1(z) = v(z/2)$ and $v^2(z)=v((z+1)/2)$. 
	Similarly, optimizing over $\varphi^1, \varphi^2 \in L^1(0,1)$ is equivalent to optimizing over $\varphi\in L^1(0,1)$
	with the relations $\varphi^1(z) = \varphi(z/2)/2 = \varphi \circ L(z)/2$ and
	$\varphi^2(z) = \varphi((z+1)/2)/2 = \varphi \circ R(z)/2$ for $z\in (0,1)$ 
	with $L:[0,1]\to [0,1/2]$ given by $L(z):=z/2$ and $R:[0,1]\to [1/2,1]$ given by $R(z):=(z+1)/2$. 
	Using those definitions, we obtain first by construction of $\varphi^1$ and $\varphi^2$ the identity
	\[
	\skp{j,\varphi^1}+\skp{j,\varphi^2} = \frac{1}{2}\skp{j,\varphi\circ L} + \frac{1}{2}\skp{j,\varphi\circ R} =  \skp{j,\varphi} .
	\]
	Similarly, we obtain
	\begin{align*}
		\MoveEqLeft\int_0^1 k^1(z)\pra*{-\frac1{2} {v^1(z)^2}\varphi^1(z)^2 + 2 \abs*{\partial_z v^1(z)}^2} \dx{z}
		+ \int_0^1 k^2(z)\pra*{-\frac1{2} v^2(z)^2\varphi^2(z)^2 + 2 \abs*{\partial_z v^2(z)}^2} \dx z \\
		&= \int_0^1 2 k\circ L(z) \pra*{ -\frac1{2} {v(L(z))^2} \frac{1}{4} \varphi(L(z))^2 + 2 \frac{1}{4} \abs*{v'\circ L(z)}^2} \dx{z}\\
		&\qquad + \int_0^1 	2 k\circ R(z)\pra*{-\frac1{2} v(R(z))^2 \frac{1}{4} \varphi(R(z))^2 + 2 \frac{1}{4} \abs*{v'\circ R(z)}^2} \dx z \\
		&= \int_0^1 k(z) \pra*{-\frac1{2} {v(z)^2}\varphi(z)^2 + 2 \abs*{v'(z)}} \dx{z} ,
	\end{align*}
	where we applied the push-forwards $L$ and $R$ 
	and note that $L_\sharp \dx{z}|_{[0,1]} = 2 \dx{z}|_{[0,1/2]}$ and $R_\sharp \dx{z}|_{[0,1]} = 2 \dx{z}|_{[1/2,1]}$. 
	
	We then use Property~\ref{l:props-N:harmonic-mean} from Lemma~\ref{l:props-N} to obtain 
	\[
		k^* = \bra*{ \int_0^{\frac{1}{2}}  \frac{2}{k_1} \dx{z} +  \int_{\frac{1}{2}}^1  \frac{2}{k_2} \dx{z}}^{-1} = \bra*{\frac{1}{k^1} + \frac{1}{k^2}}^{-1} .   \qedhere
	\]
\end{proof}
\begin{cor}[Parallel law]\label{cor:N:parallel}
	Let $k^i>0$ for some index set $i\in I\subseteq \N_0$ s.t. $\sum_{i\in I} k^i=: k <\infty$. Then, for any $j\in \R$ and $\alpha,\beta\geq 0$ we have
	\begin{align}\label{eq:N:parallel}
		\min_{\substack{j^i \in \R\\ \sum_{i\in I} j^i =j}} \sum_{i\in I} \CCs\bra*{j^i;\alpha,\beta;k^i}
		=  \CCs\bra*{j;\alpha,\beta;k},
	\end{align}
	with the $\min$ attained with $j^i = j \frac{k^i}{k}$ being the minimizer. 
\end{cor}
\begin{proof}
	If $\alpha\beta=0$ and $j=0$, then the right-hand side is $0$, in which case $j^i=0$ for $i\in I$ is the minimizer. If $\alpha\beta=0$  $j\ne 0$, then both the left-hand and right-hand side are infinite and any $j^i$ satisfying $\sum_{i\in I} j^i =j \ne 0$ is a minimizer.
	
	For the case $\alpha,\beta>0$, by the representations~\eqref{eqdef:N:explicit} and~\eqref{eqdef:N:explicit:rigrous} from Lemma~\ref{l:props-N}  
	the term involving~$\sfC^*$ trivially follows by addition.
	Using the joint convexity and one-homogeneity of the perspective function $(j,\sigma)\mapsto \sfC(j|\sigma)$ and Jensen's inequality for one-homogeneous functionals we also obtain
	\[
	  \sum_{i\in I} \sfC\bra[\big]{j^i \big| \sqrt{\alpha \beta} \, k^i} \geq \sfC\bra[\big]{ j \big| \sqrt{\alpha\beta} \, k} .
	\]
	The equality case is readily observed.
\end{proof}

\subsection{Lower semicontinuity of integrals of \texorpdfstring{$\CCs$}{C}}
\label{ss:lsc-G}

Let $A$ be a Lebesgue measurable subset of $\R^m$. For Lebesgue measurable  $j:A\to H^{-1}(0,1)$, $u,v:A\to[0,\infty)$, we can define a $\calY$-valued measure $\mu$ by 
\[
\dual \varphi \mu
:= \int_A \bra*{\dual{\varphi^1(x)}{ j(x)}
  + \varphi^2(x)u(x) 
  + \varphi^3(x)v(x) }\dx x
\]
for any $\varphi\in C_b\bra[\big]{A;H^1_0(0,1)\times \R^2}$.
Note that for fixed measurable  $k:A\to(0,\infty)$, $\CCs(\cdot,\cdot,\cdot;k)$ is non-negative, jointly one-homogeneous, convex, and lower semicontinuous in the first three variables. 
Following Section~\ref{ss:perspective-function},  the integral of $\CCs(j,u,v;k)$ can then be written as 
\[
\int_A \CCs\bra*{j(x),u(x),v(x);k(x)}\dx x
= 
\int_A \CCs\bra*{\frac{\dx\mu}{\dx|\mu|}(x); k(x)}|\mu|(\dx x).
\]

This  way of writing is useful when $\mu_n$ converges to a limit $\mu$ which might be singular with respect to the Lebesgue measure. The convexity and lower-semicontinuity properties of~$\CCs$ imply that the integral of $\CCs$ satisfies a lower bound, and in that case 
part~\ref{l:props-N:infty} of Lemma~\ref{l:props-N} gives a tool to disprove such singularity. This is the content of the following lemma.

\begin{lemma}[Lower semicontinuity of integrals of $\CCs$]
\label{l:G-lsc}
Let $A\subset \R^m$ be compact. Assume that $\mu_n$ is a sequence of $\calY$-valued measures on $A$ such that $\mu_n\weakto \mu$ in the sense that 
\[
\text{for all }\varphi\in C(A), \qquad
	\int_A \varphi(x) \mu_n(\dx x) \longrightharpoonup \int_A \varphi(x) \mu(\dx x) 
	\qquad\text{in }\calY.
\]
Assume in addition that $k_n,k\in L^1(A;L^1(0,1))$ with $k_n\geq k_0>0$ for all $n$, and $k_n\to k$ in $L^1(A;L^1(0,1))$. Then
\begin{enumerate}
\item \label{l:G-lsc:part1}
We have 
\[
\liminf_{n\to\infty} \int_A \CCs\bra*{\frac{\dx \mu_n}{\dx |\mu_n|}(x); k_n	(x)}|\mu_n|(\dx x)
\geq \int_A \CCs\bra*{\frac{\dx \mu}{\dx |\mu|}(x); k(x)}|\mu|(\dx x).
\]
\item \label{l:G-lsc:part2}If 
$\mu^{2,3}$ is Lebesgue absolutely continuous, then $\mu^1$  is an $H^{-1}(0,1)$-valued Lebesgue absolutely continuous measure, and setting $\mu =: (j,u,v)\dx x$ we have 
\begin{equation}
\label{eq:l:G-lsc:ac-equality}
\int_A \CCs\bra*{\frac{\dx \mu}{\dx |\mu|}(x); k	(x)}|\mu|(\dx x)
=
\int_A \CCs\bra*{j(x),u(x),v(x); k(x)}\dx x.
\end{equation}
\end{enumerate}

\noindent
Analogous results hold for the integral 
\begin{equation}
\label{eq:l:G-lsc:Z}\mu \mapsto \int_A \sum_{\sfx\sfy\in \edges} 
  \CCs\bra*{
    \frac{\dx\mu^1_{\sfx\sfy}}{\dx|\mu|}(x),
    \frac{\dx\mu^2_{\sfx}}{\dx|\mu|}(x),
    \frac{\dx\mu^2_{\sfy}}{\dx|\mu|}(x); k_{\sfx\sfy}}|\mu|(\dx x)
\end{equation}
in which $\mu$ is $\calZ$-valued, with 
\[
\calZ := H^{-1}(0,1)^\edges \times \R^\nodes.
\]
\end{lemma}

\begin{proof}
We prove Part~\ref{l:G-lsc:part1} by applying the general lower-semicontinuity result Lemma~\ref{l:Reshetnyak} in Appendix~\ref{ss:lower-bound}. For this we set
\[
f_n,f: A\times \calY \to [0,\infty], 
\qquad
f_n(x,\xi) := \CCs(\xi;k_n(x)),
\quad
f(x,\xi) := \CCs(\xi;k(x)).
\]
and we now show that $f_n$ and $f$ have the lower-semicontinuity property~\eqref{l:Reshetnyak:ass-lsc}.

We consider a sequence $\xi_n = (j_n,\alpha_n,\beta_n)$ in $\calY$, 
converging weakly to $(j,\alpha,\beta)$, such that  $\CCs(j_n,\alpha_n,\beta_n;k_n)$ is bounded from above. 
Because $k_n\geq k_0>0$, $\CCs$ bounds $\|v'\|^2_{L^2}$ from above, 
and consequently the infimum in~\eqref{eqdef:calN} is achieved by a $v_n\in H^1(0,1)$ for each $n$. 
By the same bound we can extract a subsequence (without changing notation) such that 
$v_n\rightharpoonup v$ in $H^1$ and $v_n\to v$ in $C_{\mathrm b}([0,1])$; 
by the compactness of the trace operator the traces of the limit~$v$ coincide with the limits $\alpha$ and $\beta$ of the traces of $v_n$.

We then calculate for any $\varphi\in H^1_0(0,1)$ and any $\psi\in L^2(0,1)$, 
\begin{align*}
	\liminf_{n\to\infty} \CCs(j_n,\alpha_n,\beta_n;k_n)
	&\geq \liminf_{n\to\infty}\ \dual {j_n}\varphi 
	+ \int_0^1 k_n\pra*{-\frac1{2} {v_n^2}\varphi^2 + 2 \abs*{v_n'}^2}\dx z\\
	&\geq \liminf_{n\to\infty}\ \dual {j_n}\varphi 
	+ \int_0^1\pra*{-\frac{k_n}{2} {v_n^2}\varphi^2 + 4v_n'\psi - 2\frac{\psi^2}{k_n}}\dx z\\
	&= \dual {j}\varphi 
	+ \int_0^1 \pra*{-\frac{k}{2} {v^2}\varphi^2 + 4v'\psi - 2\frac{\psi^2}k}\dx z.
\end{align*}
Taking the supremum over $\varphi$ and $\psi$ we regain
\begin{align*}
	\liminf_{n\to\infty} \CCs(j_n,\alpha_n,\beta_n;k_n)
	&\geq \sup_{\varphi\in H^1_0(0,1)}\dual {j}\varphi 
	+ \int_0^1 k\pra*{-\frac{1}{2} {v^2}\varphi^2 + 2\abs{v'}^2}\dx z
	\geq \CCs(j,\alpha,\beta;k).
\end{align*}
With this lower semicontinuity property,  part~\ref{l:G-lsc:part1} follows from Lemma~\ref{l:Reshetnyak}.

\medskip
To prove part~\ref{l:G-lsc:part2}, we decompose $|\mu|$ and $\mu$ into Lebesgue absolutely continuous and singular parts:
\[
|\mu|(\dx x)  =: |\mu|^{\mathrm{ac}}(x)  \dx x + |\mu|^\perp(\dx x), 
\quad
\mu^{\mathrm{ac}}(x) := \frac{\dx \mu}{\dx|\mu|} (x) |\mu|^{\mathrm{ac}}(x), 
\quad
\mu^\perp(\dx x) := \frac{\dx \mu}{\dx|\mu|} (x)|\mu|^\perp(\dx x).
\]
Separating the regular and singular part in the integral we have
\begin{align*}
\MoveEqLeft  \int_A \CCs\bra*{\frac{\dx{\mu}}{\dx{|\mu|}}(x);k(x)} |\mu|(\dx x)
	\\
&=\int_{A}  \CCs\bra*{\frac{\dx{\mu}}{\dx{|\mu|}}(x);k(x)} |\mu|^{\mathrm{ac}}(x)\dx x 
	+\int_A  \CCs\bra*{\frac{\dx{\mu}}{\dx{|\mu|}}(x);k(x)} |\mu|^\perp(\dx x).
\end{align*}
Since the second and third components of $\mu^\perp$ vanish by assumption, 
part~\ref{l:props-N:infty} of Lemma~\ref{l:props-N} implies that the first component of $\mu^\perp$ also vanishes, leading to the expression~\eqref{eq:l:G-lsc:ac-equality}.

The generalisation to~\eqref{eq:l:G-lsc:Z} is straightforward and we omit the details.
\end{proof}

\section{Kramers' high-activation-energy limit}
\label{s:Kramers}

In a seminal paper in 1940, Hendrik A.~Kramers~\cite{Kramers1940} described various modelling approaches for the calculation of chemical reaction rates. He proposed the use of a Brownian particle in a chemical potential landscape, where the reaction event corresponds to the escape of the particle from one energy well into another. 

When the energy barrier between the wells (the `activation energy') is large with respect to the noise, the behaviour of the particle simplifies; the particle spends long times in one well before rapidly jumping to a different well. This behaviour has been termed \emph{metastable}, and is observed in many processes. Metastability is the subject of the monograph by Bovier and Den Hollander~\cite{BovierDenHollander16}, and the review by Berglund, which focuses in particular on Kramers' problem~\cite{Berglund13}. 

Mathematically, this metastability can be characterized in  a number of different ways, in terms of the well-to-well transition time (see~\cite[Ch.~4]{FreidlinWentzell98} or~\cite{BovierEckhoffGayrardKlein04}), the spectrum of the generator~\cite{HelfferKleinNier04,HelfferNier06}, or the convergence of the Fokker-Planck equations~\cite{EvansTabrizian16,SeoTabrizian20}. 

In this section we revisit this system, taking the viewpoint of the gradient system that describes the evolution of the law of the particle, as in~\cite{PeletierSavareVeneroni10, ArnrichMielkePeletierSavareVeneroni12,  LieroMielkePeletierRenger17}. Our main result is a rigorous derivation of the $\e\to0$ limit of the gradient systems that includes tilting. This allows us to follow the impact of the tilt through the limit.

\subsection{Setting}
\label{ss:Kramers-setting}



\paragraph{Geometry.}
Consider a compact $d$-dimensional rectangle\footnote{The only place we use the rectangular geometry of $\Omega$ is in proving the chain-rule inequality (Lemmas~\ref{l:CRLB-Kramers-epsilon} and~\ref{l:CRLB-Kramers-limit}).} $\Omega\subset \R^d$ and a compact interval $\Upsilon\subset \R$ . The particle is assumed to have a spatial coordinate $X_t\in \Omega$ and a `chemical' coordinate $Y_t\in \Upsilon$; the total state space is $\nodes := \Omega\times \Upsilon$. For any set $A$, we often write $A_T:= [0,T]\times A$; therefore $\Omega_T := [0,T]\times \Omega$ and $\nodes_T := [0,T]\times \nodes$. Note that by the compactness of $\nodes$, $\Omega$, and $\Upsilon$, the wide and narrow convergences coincide; we write $\stackrel*\longrightharpoonup$ for both.

\paragraph{Energy.}

At $\e>0$ the `chemical' potential landscape is described by a function $H:\Upsilon\to\R$ with the following properties.
\begin{assumption}
\label{ass:H}
$H\in C^2(\Upsilon;\R)$ satisfies
\begin{enumerate}
\item $H$ has exactly two minima at $y=a$ and $y=b>a$, both at value zero;
\item $a,b\in \Int\Upsilon$;
\item $H$ has a global maximum at $y=c$, $a<c<b$;
\item $H$ is strictly less than $H(c)$ on all sets bounded away from $y=c$.
\label{ass:H:global-bound}
\end{enumerate}
\end{assumption}
\begin{figure}[ht]
	\centering
	\setlength{\unitlength}{0.5\textwidth}
	\begin{picture}(1,0.52092241)%
		\put(0,0){\includegraphics[width=\unitlength]{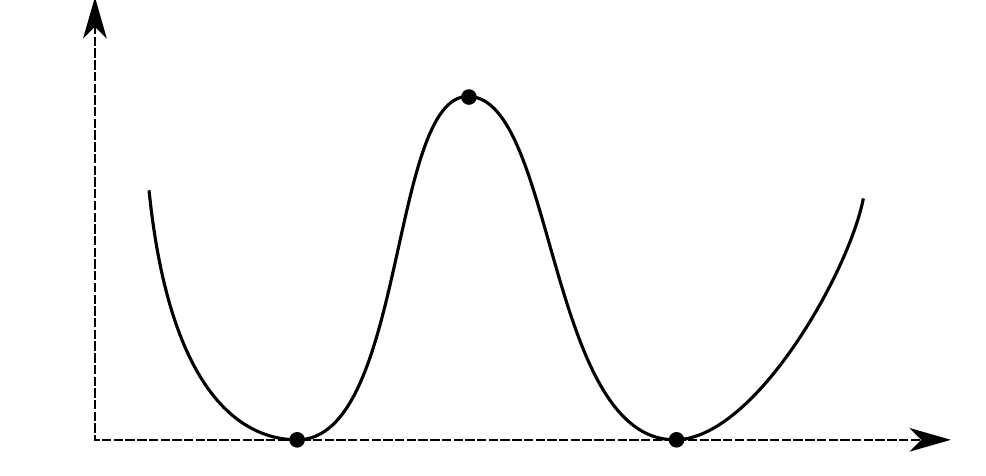}}%
		\put(0.29,-0.02){{$a$}}%
		\put(0.48,-0.02){{$c$}}%
		\put(0.68,-0.02){{$b$}}%
		\put(1,0.02){{$x$}}%
		\put(-0.04,0.4){{$H(x)$}}%
	\end{picture}%
	\caption{Illustration of the potential $H$.}
	\label{fig:DoubleWell}
\end{figure}
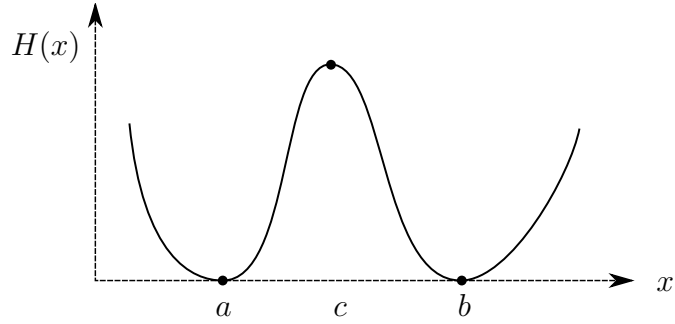
The potential $H$ is scaled by a factor $1/\e$; this scaling encodes the high-activation-energy limit.
Hence, the potential $H/\e$ defines the stationary measure
\begin{equation*}
	\pi_\eps(\dxx xy) = \frac{1}{\PartSum_\eps} \ee^{-\frac{H(y)}{\eps}}\dxx xy \qquad\text{with}\qquad \PartSum_\eps := \int_\nodes   \ee^{-\frac{H(y)}{\eps}} \dxx xy .
\end{equation*}
As $\e\to0$, $\pi_\e$ converges to the singular measure 
\begin{equation}
\label{eqdef:pi_0}
\pi_\e \longweaksto \pi_0(\dxx xy) = \weight^a \dx x \, \delta_a(\dx y) + \weight^b \dx x \, \delta_b(\dx y), \qquad\text{for }\weight^a+\weight^b=\frac1{|\Omega|}.
\end{equation}
By Watson's Lemma we have the explicit expression
\begin{equation*}
	\weight^a = \frac1{|\Omega|}\; \frac{H''(a)^{-\tfrac{1}{2}}}{H''(a)^{-\tfrac{1}{2}}+H''(b)^{-\tfrac{1}{2}}}
	\qquad\text{and}\qquad
	\weight^b = \frac1{|\Omega|}\; \frac{H''(b)^{-\tfrac{1}{2}}}{H''(a)^{-\tfrac{1}{2}}+H''(b)^{-\tfrac{1}{2}}}
\end{equation*}
We often consider the $y$-measure separately and write
\[
\pi_\e^y(\dx y) := \pi_\e^y(y)\dx y := \frac1{\PartSum_\e} \ee^{-H(y)/\e}\dx y
\longweaksto \weight^a \delta_a(\dx y) + \weight^b \delta_b(\dx y).
\]
With this notation  $\pi_\e = \pi_\e^y(\dx y) \dx x$. Note that the measure $\pi_\e$ is normalized, but $\pi_\e^y$ is not: $|\pi_\e^y| = |\Omega|^{-1}$.

As in Example~\ref{ex:FP}, the driving functional is 
\begin{equation*}
\calE_\e(\rho) := \RelEnt(\rho|\pi_\e).
\end{equation*}
Because of the convergence~\eqref{eqdef:pi_0}, $\calE_\e$ $\Gamma$-converges to the limiting functional (see e.g.~\cite[Lemma~6.2]{AmbrosioSavareZambotti09})
\[
\calE_0(\rho) := \RelEnt(\rho|\pi_0).
\]

\paragraph{Continuity equation.}

Following the Fokker-Planck Example~\ref{ex:FP}, the edge space at $\e>0$ is the union of `edges in $x$' and `edges in $y$',
\[
\edges := \edges^x \sqcup \edges^y   \qquad\text{with} \quad \edges^x := \nodes \times \{1,\dots,d\} 
\quad\text{ and }\quad \edges^y := \nodes ,
\]
and the gradient is the usual gradient in $\R^{d+1}$, which we also split into parts:
\[
\nabla = (\nabla_x,\partial_y).
\]
Definition~\ref{def:continuity-equation} of the continuity equation coincides with the well-known distributional one; we incorporate the no-flux boundary condition into this definition.

\begin{definition}[Continuity Equation for $(\nodes,\edges,\nabla)$]\label{defn:Kramers:CE}

%
Let $(\rho(t))_{t\in[0,T]}\subset \calM_{\geq0}(\nodes)$ and $j=(j^x,j^y)$ with $(j^x(t))_{t\in[0,T]}\subset \calM(\nodes;\R^d)$ and $(j^y(t))_{t\in [0,T]}\subset \calM(\nodes)$.
The pair~$(\rho,j)$ solves the continuity equation $\partial_t \rho + \div j=0$, denoted with $(\rho,j)\in \CE(0,T)$, if 
\begin{enumerate}
\item $t\mapsto \rho(t)$ is narrowly continuous;
\item The map~$t\mapsto j(t)$ is measurable with respect to the narrow topology on~$\calM(\nodes)$;
\item 
For any  $\varphi\in C^1_{\mathrm c}((0,T)\times\R^d\times \R)$ we have
\begin{equation}
\label{eq:Kramers:weak-form-CE}
\hspace{-2em}\int\limits_{[0,T]\times\R^{d+1}} \Bigl[\partial_t \varphi(t,x,y) \, \rho(t,\dxx xy)\dx t
  + \bra*{\nabla_x \varphi(t,x,y) , \partial_y \varphi(t,x,y)}\, \bra*{j^x,j^y}(t,\dxx xy)\dx t \Bigr] = 0 .
\end{equation}
\end{enumerate}
\end{definition}
Note that the test functions $\varphi$ in~\eqref{eq:Kramers:weak-form-CE} are defined on $\R^{d+1}$, and therefore $j$ satisfies a weak version of the no-flux boundary condition $j\ip n = 0$ on $[0,T]\times\partial \nodes$. By choosing $\varphi= \varphi(t)$ on~$\nodes_T$ we find that mass is conserved: $\rho(t,\nodes) = \rho(0,\nodes)$ for all $t$. We therefore restrict ourselves in the following to normalized initial  measures $\rho(t=0) \in \ProbMeas(\nodes)$.

\paragraph{Contracted limiting continuity equation.}
Since $\pi_0$ is supported on $\Omega\times \{a,b\}$,  measures~$\rho$ with finite {limiting} energy $\calE_0(\rho)<\infty$ also are supported on $\Omega\times \{a,b\}$, and these can be written as 
\begin{equation}
\label{char:rho0}
\rho(\dxx xy) = \rho(\dx x,a) \delta_a(\dx y) + \rho(\dx x,b) \delta_b(\dx y)
\quad \text{for some }\rho(\cdot,a), \rho(\cdot,b)\in \calM_{\geq0}(\Omega).
\end{equation}
This is an example of the general phenomenon described in Section~\ref{sss:EDP-conv-contracted}: finiteness of the limiting energy $\calE_0(\rho)$ implies that $\rho$ is supported on a strict subset 
\begin{equation}
\label{eqdef:Kramers:contracted-nodes}
\wt\nodes := \Omega\times \{a,b\}\;\subset \;\Omega\times \Upsilon =: \nodes,
\end{equation}
and as in Definition~\ref{defn:EDP-convergence:reduce} we can trivially write
\begin{equation}
\label{eqdef:Kramers:wt-calE}
\calE_0(\rho) = \begin{cases}
 \wt\calE\bra*{\rho\lfloor \wt\nodes} := \RelEnt\bra*{\rho\lfloor \wt\nodes \,\big|\,\pi_0\lfloor\wt\nodes } &\text{if } \rho(\nodes\setminus \wt\nodes) = 0,\\
 +\infty &\text{otherwise},
\end{cases}
\end{equation}
The restriction of $\supp\rho$ to $\wt\nodes$  also implies that the $y$-flux $j^y$ has a special form (see Lemma~\ref{lem:Kramers:limit:jy}): If $(\rho,j)\in \CE(0,T)$ is such that $\rho$  has the structure $\eqref{char:rho0}$ at each time $t$ and $j^x\ll \rho$, then we have  
\begin{equation}
	\label{char:j0}
	j^y(t,\dxx xy) = \overline \jmath (t,\dx x) \bONE_{[a,b]}(y)\dx y
	\qquad \text{for some $(\overline \jmath(t))_t\subset \calM(\Omega)$.}
\end{equation}	
For pairs $(\rho,j)\in \CE(0,T)$ satisfying~(\ref{char:rho0}-\ref{char:j0}), Definition~\ref{defn:Kramers:CE} reduces to the following concept of
contracted continuity equation.
\begin{definition}[Contracted continuity equation]\label{defn:Kramers:CE0}
Let $(\rho(t))_{t\in[0,T]}\subset \calM_{\geq0}(\nodes)$ and $j=(j^x,\ol\jmath)$ with $(j^x(t))_{t\in[0,T]}\subset \calM(\nodes;\R^d)$ and $(\ol\jmath(t))_{t\in [0,T]}\subset \calM(\Omega)$.
The pair~$(\rho,j)$ solves the continuity equation on $\wt\nodes_T$ if for any  $\varphi\in C_{\mathrm c}^1((0,T)\times\R^d\times \{a,b\})$ we have 
\begin{multline}
\label{eq:weak-form-CE0}
\int_0^T \int_{\R^d} \sum_{y=a,b}\bra[\big]{\partial_t \varphi(t,x,y)\, \rho(t,\dx x,y)
  + \nabla_x \varphi(t,x,y)\,j^x(t,\dx x,y) } \dd t \\
+ \int_0^T \int_{\R^d} \bigl(\varphi(t,x,b)-\varphi(t,x,a)\bigr)\overline\jmath(t,\dx x) \dx t = 0.
\end{multline}
We write $(\rho,j)\in \wt\CE(0,T)$.
\end{definition}
\noindent 
Indeed, the following lemma is simple to check. 
\begin{lemma}
For pairs  $(\rho,j)$ satisfying~(\ref{char:rho0}-\ref{char:j0}), the two continuity equations are equivalent: $(\rho,j)\in \CE(0,T)\iff (\rho,j)\in \wt\CE(0,T)$.
\end{lemma}
Note how the third term in~\eqref{eq:weak-form-CE0} has become an integral against the discrete gradient $(\gnabla\varphi)_{ab} := \varphi_b-\varphi_a$.
In Definition~\ref{defn:Kramers:CE0} we therefore recognize the continuity equation generated by $(\wt\nodes, \wt\edges,\wt\anabla)$ where
\begin{subequations}
\label{eqdef:Kramers:contracted-edges-anabla}
\begin{align}
\wt\edges &:= \underbrace{\Omega\times \{a,b\}\times \{1,\dots,d\}}_{\wt\edges^x,\text{ continuous edges}} 
  \quad \sqcup \underbrace{\Omega\times \{ab\}}_{\wt \edges^y, \text{ discrete edges}}\\
\wt\anabla \varphi(e) &:= \begin{cases}
	\partial_{x_i} \varphi(x,a) & \text{if }e = (x,a,i)\in \wt\edges^x\\
	\partial_{x_i} \varphi(x,b) & \text{if }e = (x,b,i)\in \wt\edges^x\\
	\varphi(x,b)-\varphi(x,a) & \text{if }e = (x,ab)\in \wt\edges^y
\end{cases}
\end{align}
\end{subequations}
This setup is the same as~\eqref{defn:Ex:B+A:edges} of Example~\ref{ex:FP}+\ref{ex:heat-flow}. 
These definitions fix the notions of $(\nodes,\edges,\anabla)$ and $(\wt\nodes,\wt\edges,\wt\anabla)$ that underlie the contracted EDP convergence (after Definition~\ref{defn:EDP-convergence:reduce}) that we prove below.

\paragraph{Dissipation potentials.}
We next specify the dissipation potentials for the gradient structure, which are the ones~\eqref{eqdef:R-FP} from the Fokker-Planck equation in Example~\ref{ex:FP} adapted to the current setting. For $\rho\in \calM_{\geq0}(\nodes)$, $j = (j^x,j^y)\in \calM(\nodes;\R^d\times \R)$,  and  $\Xi = (\Xi^x,\Xi^y)\in C(\nodes;\R^d\times \R)$, we set 
\begin{align*}
\calR_\e(\rho,j) 
&=  \frac{1}{2} \int_\nodes \biggl[\frac1{m_\Omega} \abs*{\frac{\dx j^x}{\dx\rho}(x,y)}^2 + \frac1{\tau_\e} \abs*{\frac{\dx j^y}{\dx\rho}(x,y)}^2\biggr]\, \rho(\dxx xy)\\
\calR_\e^*(\rho,\Xi) &=  \frac{1}{2} \int_\nodes \pra*{ m_\Omega |\Xi^x (x,y)|^2 +  \tau_\e \abs*{ \Xi^y (x,y)}^2 }\,\rho(\dxx xy).
\end{align*}
We implicitly set $\calR_\e$ to $+\infty$ if  $j\not\ll \rho$.

Note how the two terms in $\calR^*$ are scaled by two parameters $m_\Omega$ and $\tau_\e$. The parameter~$m_\Omega$ is fixed; the parameter $\tau_\e$ is chosen to scale as the typical time of transition between the two wells of $H$,
\begin{equation}\label{eq:def:Kramers:tau}
\tau_\e := \frac{m_\Upsilon \PartSum_\e}{|\Omega|} \int_a^b \ee^{H(y)/\e}\dx y \quad \xrightarrow{\e\to 0}\infty,
\end{equation}
where $m_\Upsilon$ can be chosen freely and hence will return in the formulation of the limit. 
From Watson's Lemma we deduce that 
\[
\tau_\e \sim C\, \e \, \ee^{H(c)/\e}, \qquad\text{with}\qquad C := \frac{m_\Upsilon}{\abs{\Omega}}\, 2\pi \bra*{H''(a)^{-\frac12} + H''(b)^{-\frac12}}
  \abs*{H''(c)}^{-\frac12} .
\]

\bigskip
\paragraph{Induced evolution equations.}
With these definitions, similar manipulations to those in~\eqref{eq:deriv-eq-FP} from Example~\ref{ex:FP} lead to the Fokker-Planck equation with no-flux boundary conditions
\begin{subequations}
\label{eq:FP-Kramerslimit}
\begin{align}
\partial_t \rho - \div_x (m_\Omega \nabla_x \rho )
  - \tau_\e \div_y  \bra*{\nabla_y \rho + \frac1\e\rho \nabla_y H} &=0
  \qquad \text{in } \nodes,\\
\nabla_x\rho \cdot n &= 0 \qquad \text{on }\partial\Omega\times \Upsilon,\\
\bra*{\nabla_y\rho + \frac1\e\rho\nabla_y H}\cdot n &= 0 \qquad \text{on }\Omega\times \partial\Upsilon. 
\end{align}
\end{subequations}

\paragraph{Interpretation as  Fokker-Planck equation of a diffusion process.}
The equation~\eqref{eq:FP-Kramerslimit} also has an interpretation as the Fokker-Planck equation or Forward-Kolmogorov equation for a stochastic differential equation. If $(X_0,Y_0)$ is distributed according to some $\rho^\circ\in \ProbMeas(\nodes)$, then the corresponding solution $\rho(t)$ at time $t$ is the law of a stochastic process $(X_t,Y_t)$. This process is a diffusion in $\nodes$ with drift given by $\e^{-1}\partial_y H(y)$, and with reflecting boundary conditions on $\partial \nodes$. The (square) diffusivity of the diffusion in $X_t$ is the constant $2m_\Omega$, while the diffusivity of the diffusion in $Y_t$ is  $2\tau_\e\to\infty$.

\paragraph{Adding tilting.}
In order to understand the role of tilting,  we consider the $\e$-indexed sequence together with a set of possible tilts, as described in Section~\ref{s:ModellingTilting}. In the context of a singular limit such as the one at hand, the philosophy is that the singularity of the limit is characterized by the sequence of energies $\calE_\e$ and dissipation potentials $\calR_\e$, and that any applied tilt `should not interfere with the singular limit'. This translates into choosing a class of additional potentials (`tilts') $\calF$ on the space $\ProbMeas(\nodes)$, which are differentiable with respect to the total-variation norm on $\ProbMeas(\nodes)$, that is,
\begin{equation}\label{eq:Kramers:tilts}
	\sfF := \set[\Big]{\;\calF \in C^1(\ProbMeas(\nodes)): \;\rmD\calF(\rho)\in C_{\mathrm b}(\nodes)\quad \text{with } \sup_{\rho\in\ProbMeas(\nodes)} \norm*{\rmD\calF(\rho)}_{C_{\mathrm b}(\nodes)}<\infty\;}.
\end{equation}
In addition, we equip $\sfF$ with a topology ensuring convergence within $\sfF$ and a uniform modulus of continuity in the $y$-variable (see Assumption~\ref{ass:F} below).
Typical examples are families of potential energies $\calF(\rho) = \int V(x,y) \rho(\dxx{x}{y})$ with $V\in C_{\mathrm b}^1(\nodes)$ or interaction energies $\calF(\rho) =\frac{1}{2} \iint W(x_1-x_2,y_1-y_2) \rho(\dxx{x_1}{y_1})\rho(\dxx{x_2}{y_2})$ with $W\in  C^1_{\mathrm b}(\R^{d+1}\times \R^{d+1})$.

Given a sequence $(\calF_\e)_\e\subset \sfF$, the components of the tilted gradient system are
\begin{align*}
	\calE_\e^\calF(\rho) &:= (\calE_\e+\calF_\e)(\rho),
\\
\calD_\e^T(\rho,j;\calF_\e) &:= 
\begin{cases}
  \ds\int_0^T \bigl[ \calR_\e(\rho,j) 
    + \underbrace{\calR_\e^*(\rho,-\nabla \rmD\calE_\e^\calF(\rho))}_{\mathclap{\text{formal}}} \,\bigr]\, \dx t
		      & \text{if }(\rho,j)\in \CE(0,T)\\
	+\infty & \text{otherwise.}
			\end{cases}
\end{align*}
As discussed in Section~\ref{ss:formal-rigorous}, the second term in $\calD_\e^T$ above is only formal, since $\dx\rho/{\dx {\pi_\e}}$ may vanish, and therefore $\rmD\calE_\e^\calF(\rho) = \log (\dx\rho/{\dx {\pi_\e}}) + \rmD\calF_\e(\rho)$ may not be well defined; in addition this expression may not have sufficient regularity for the operator $\nabla$ to be applied. We therefore replace the term by exploiting the formal chain rule $u|\nabla (\log u + f)|^2  = 4\ee^{-f}|\nabla \sqrt {u\ee^{f}}|^2$,  to obtain the following rigorous definition:
\begin{subequations}
\label{eq:FP-Kramerslimit:D:tilted}
\begin{align}
\calD_\e^T(\rho,j;\calF_\e) &:= \calD_\e^{T,x}(\rho,j;\calF_\e)
                                  +\calD_\e^{T,y}(\rho,j;\calF_\e)                                   \notag \\
&:= \int_0^T \int_\nodes \set[\bigg]{ \frac1{2m_\Omega} \abs*{\frac{\dx j^x}{\dx\rho}}^2  \rho(t,\dxx xy) 
  +  2m_\Omega \ee^{-F_\e^\rho}\abs*{\nabla_x \sqrt {u^\calF}}^2\pi_\e(\dxx xy)}\dx t
\label{eqdef:D-part1}\\
&\qquad +  \int_0^T \int_\nodes \set[\bigg]{ \frac1{2\tau_\e} \abs*{\frac{\dx j^y}{\dx\rho}}^2  \rho(t,\dxx xy)
  +  2\tau_\e \ee^{-F_\e^\rho} \abs*{\partial_y \sqrt {u^\calF}}^2\pi_\e(\dxx xy ) }\dx t,
\label{eqdef:D-part2}
\end{align}
\end{subequations}
where we write $u := \dx\rho/\dx{\pi_\e}$, $F_\e^\rho := \rmD\calF_\e(\rho_\e)$, and $u^\calF := u \ee^{F_\e^\rho}$, and we set $\calD^T_\e=+\infty$ whenever $\rho\not\ll \pi_\e$, $j\not\ll \rho$, or $u^\calF$ is not weakly differentiable.

The system that we have now constructed satisfies the lower-bound Property~\ref{property:ChainRuleLowerBound}, and in fact the stronger statement
\begin{lemma}[Chain rule]
\label{l:CRLB-Kramers-epsilon}
Let $\calF\in \sfF$.
For all $(\rho,j)\in \CE(0,T)$ with $\calE_\e(\rho(0))<\infty$, 
\[
\abs*{\calE_\e^\calF(\rho(T)) - \calE_\e^\calF(\rho(0))}  \leq \calD_\e^T(\rho,j;\calF) ,
\]
and in particular,
\begin{equation*}
\calI_\e(\rho,j;\calF) := \calE_\e^\calF(\rho(T)) - \calE_\e^\calF(\rho(0)) + \calD_\e^T(\rho,j;\calF) \geq 0.
\end{equation*}
\end{lemma}
\noindent
The proof is with minor changes similar to the proof of Lemma~\ref{l:CRLB-Kramers-limit} below (see Appendix~\ref{app:CRLB}) and we omit it.
If $\calI_\e(\rho,j;\calF)=0$, then $\rho$ is a weak solution of 
	\begin{align*}
		\partial_t \rho - m_\Omega\div_x ( \nabla_x \rho + \rho \nabla_x F_\e^\rho)
		- \tau_\e \div_y  \bra*{\nabla_y \rho +\rho \nabla_y \bra*{ \frac1\e H + F_\e^\rho}} &=0
		\quad \text{in } \nodes,\\
		( \nabla_x \rho + \rho \nabla_x F_\e^\rho)\cdot n &= 0 \quad \text{on }\partial\Omega\times \Upsilon,\\
		\bra*{\nabla_y \rho +\rho \nabla_y \bra*{ \frac1\e H + F_\e^\rho}} \cdot n &= 0 \quad \text{on }\Omega\times \partial\Upsilon. 
	\end{align*}
Note that in comparison to~\eqref{eq:FP-Kramerslimit}, the system may no longer be linear, because $F_\e^\rho$ may depend on~$\rho$. 

\paragraph{Rescaling.}
Following~\cite{ArnrichMielkePeletierSavareVeneroni12,LieroMielkePeletierRenger17,PeletierSchlottke21TR} we introduce a rescaling in the $y$-direction, which desingularizes the system in the limit $\e \to 0$.
Define the map $z_\e: \Upsilon \to \hUpsilone := z_\e(\Upsilon)$ by 
\[
z_\e(y) := \frac{m_\Upsilon \PartSum_\e}{\tau_\e |\Omega|} \int_a^y \ee^{H( y')/\e}\dx y'
\;\stackrel{\eqref{eq:def:Kramers:tau}}=\; \frac{\int_a^y \ee^{H( y')/\e}\dx y'}{\int_a^b \ee^{H( y')/\e}\dx y'} .
\]
The following properties are straightforward to verify.
\begin{lemma}
\label{l:asymp-behaviour-ze}
The function $z_\e$ has the following properties:
\begin{enumerate}
\item $z_\e$ maps the wells at $y=a,b$ to $z=0,1$;
\item $z_\e(y)$ is uniformly bounded in $y$ and $\e$, and converges to $\bONE_\{y>c\}$, uniformly on any set $\{y:|y-c|>\delta\}$;
\item $z_\e^{-1}(z)$ converges to $c$ for $0<z<1$;
\item Because of part~\ref{ass:H:global-bound} of Assumption~\ref{ass:H}, $\hUpsilone = z_\e(\Upsilon)$ is bounded in $\e$ and such that $\forall \e >0: \hUpsilone \subseteq \hUpsilon$ for some bounded interval $\hUpsilon$;
\item We have the identity (in terms of the Lebesgue density $\pi_\e(x,y)$ of the measure $\pi_\e$)
\begin{equation}
\label{eq:z'tau}
\tau_\e z_\e'(y) \pi_\e(x,y) = \frac{m_\Upsilon}{|\Omega|} .
\end{equation}

\end{enumerate}
\end{lemma}
 Define the corresponding 
transformation $\Psi_\e:\nodes\to \wh\nodes_\eps := \Omega\times \hUpsilone$ by
\[
\Psi_\e(x,y) := (x,z_\e(y)) \qquad\text{with Jacobian}\quad \rmD\Psi_\e(x,y) = \diag(\mathrm I_d,z_\e'(y)).
\]
We transform the measure $\pi_\e$ and a pair $(\rho,j)\in \CE(0,T)$ on $\nodes$ to new objects on $\wh\nodes_\eps$ by
\begin{subequations} 
\label{eqdef:hat-transformed-objects}
\begin{align}
\hpi_\e &:= (\Psi_\e)_\# \pi_\e,\\
\hrho &:= (\Psi_\e)_\# \rho,\\
\hj &:=\bigl( \rmD\Psi_\e \circ \Psi_\e^{-1}\bigr)(\Psi_\e)_\# j,\\
&\quad\text{or equivalently }\hj^x = (\Psi_\e)_\# j^x \text{ and } \hj^y = (z_\e'\circ z_\e^{-1})\;  (\Psi_\e)_\# j^y,\notag \\
&\quad\text{or in terms of densities:}\quad \hj^x(x,z_\e(y))z_\e'(y) = j^x(x,y) \text{ and } \hj^y(x,z_\e(y)) = j^y(x,y). \notag
\end{align}
\end{subequations}
We extend these measures by zero to the $\eps$-independent domain  $\hnodes := \Omega\times \hUpsilon\supseteq \hnodes_\e$.
With this transformation, $(\hrho,\hj)$ satisfies the continuity equation both on $\hnodes_{\e T}:= [0,T] \times \hnodes_\e$ and on $\hnodesT:= [0,T]\times \hnodes$ (see Appendix~\ref{ss:Kramers:append:rescaling:CE}).

The functions $u$ and $F_\e^\rho$ transform as 
\[
\hu (x,z_\e(y)) := \frac{\dx \hrho}{\dx \hpi_\e} (x,z_\e(y)) = u(x,y), 
\qquad\text{and}\qquad
\hF_\e^\rho(x,z_\e(y)) := F_\e^\rho(x,y),
\]
which implies that the transformation $y\rightsquigarrow z$ and the multiplication by $\ee^{F_\e^\rho}$ commute:
\[
\hu^\calF := \widehat{u^\calF }= \hu\, \ee^{\hF_\e^\rho}.
\]
We now rewrite $\calD^T_\e$ in~\eqref{eq:FP-Kramerslimit:D:tilted} as function of the transformed variables $\hrho$ and $\hj$.
By following the details in Appendix~\ref{ss:Kramers:append:rescaling:D}, we obtain the rescaled expression
\begin{subequations}
\begin{align}
\calD_\e^T (\rho,j;\calF_\e) &= \wh \calD_\e^T(\hrho,\hj;\calF_\e) := \wh\calD_\e^{T,x}(\hrho,\hj;\calF_\e)
+ \wh\calD_\e^{T,y}(\hrho,\hj;\calF_\e) := \notag\\
&:= \int_{\hnodes_{\e T}} \set[\bigg]{\frac1{2m_\Omega} \abs*{\frac{\dx \hj^x}{\dx\hrho}}^2 \hrho(t,\dxx xz)
  + 2m_\Omega \ee^{-\hF_\e^\rho}\abs*{\nabla_x \sqrt {\hu^\calF}}^2\hpi_\e(\dxx xz)}\dx t
  \label{eqdef:hD-part1}\\
&\qquad {}+  \int_{\hnodes_{\e T}} 
\set[\bigg]{\frac{|\Omega|}{2m_\Upsilon} \ee^{\hF_\e^\rho} \frac{|\hj^y|^2}{\hu^\calF}  
  + \frac{2m_\Upsilon}{|\Omega|} \ee^{-\hF_\e^\rho}\abs*{\partial_z \sqrt{\hu^\calF}}^2 }\dxxx xz t.
  \label{eqdef:hD-part2}
\end{align}
\end{subequations}

\subsection{Main results and discussion}
\label{ss:Kramers:main-results}

Since the setup is general enough to allow for an $\e$-dependent family of tilts $(\calF_\e)_\e\subset \sfF$ from~\eqref{eq:Kramers:tilts}, we  specify a convergence concept $\calF_\e \longrightarrow \calF_0$ for passing to the limit in $\e\to 0$.
\begin{assumption}[Assumption on the tilts]
	\label{ass:F}
	Let $\calF_\e,\calF_0\in \sfF$ satisfy $\calF_\e\longrightarrow\calF_0$ 
	in the following sense:
	\[
	\text{for all }\rho_\e\weaksto\rho_0, \qquad
	\abs*{\calF_\e(\rho_\e) - \calF_0(\rho_0)} + 
	\norm*{\rmD\calF_\e(\rho_\e)
		- \rmD\calF_0(\rho_0)}_{C_{\mathrm b}(\nodes)} \longrightarrow 0.
	\]
	In addition, $\rmD\calF_\e$ is assumed to have a uniform $y$-modulus of continuity, i.e.\ there exists a continuous function $\moc:[0,\infty)\to[0,\infty)$ with $\moc(0)=0$  such that 
	\begin{equation}
		\label{eqdef:moc}
		|\rmD\calF_\e(\rho)(x,y)- \rmD\calF_\e(\rho)(x,y')|\leq \moc(|y-y'|)
		\qquad\text{for all $\e$, $\rho$, and $x$, and all $y,y'\in \Upsilon$.}
	\end{equation}
\end{assumption}
From now on, we fix a sequence of tilts~$(\calF_\e)_\e \subset \sfF$ satisfying Assumption~\ref{ass:F}. 
Then, we find from the uniform convergence of $\rmD\calF_\e(\rho_\e) = F_\e^\rho$ and Lemma~\ref{l:asymp-behaviour-ze} that 
\begin{equation*}
\forall z\in [0,1]: \qquad 
\hF_\e^\rho(t,x,z) = F_\e^\rho(t,x,z_\e^{-1}(y)) \longrightarrow 
\begin{cases}
	F^\rho(t,x,a) & \text{if }z=0,\\
	F^\rho(t,x,c) & \text{if }0<z<1,\\
	F^\rho(t,x,b) & \text{if }z = 1.
\end{cases}
\end{equation*}

The first main result is a compactness theorem that we state here in simplified form (see Theorem~\ref{t:compactness}):
\begin{theorem}[Compactness]
	Let $\e_n$ be a sequence that converges to zero.
	Let $(\calF_{\e_n})_n$ satisfy Assumption~\ref{ass:F}, and let the sequence $(\rho_{\e_n},j_{\e_n})_n\subset \CE(0,T)$ satisfy the uniform dissipation bound
\begin{equation}
	\label{est:initial-energy-D}
	\sup_{n} \calE_{\e_n}^\calF(\rho_{\e_n}(t=0)) + \calD^T_{\e_n}(\rho_{\e_n},j_{\e_n};\calF_{\e_n}) < \infty.
\end{equation}	
Then there exists $\rho\in \calM_{\geq0}(\wt\nodes_T)$, and $j=(j^x,\overline\jmath)$ with $j^x\in \calM(\wt\nodes_T;\R^d)$, $\overline \jmath\in \calM(\Omega_T)$,  such that along a subsequence $(\rho_{\e_n},j_{\e_n})$ converges in some suitable topology to a limit $(\rho,j)\in \wt\CE(0,T)$.
\end{theorem}

The proof of this theorem is a generalization of that of~\cite[Th.~3.2]{ArnrichMielkePeletierSavareVeneroni12}, to allow for sequences of tilts $\calF_\e$ and the additional dependence on $x$. 

\medskip

The next step is to obtain suitable $\Gamma$-$\liminf$ estimates on the dissipation potentials in~\eqref{eqdef:hD-part1} and~\eqref{eqdef:hD-part2}. For the first dissipation term $\wh \calD_\e^{T,x}(\hrho,\hj;\calF_\e)$, the  compactness and convergence properties established by  Theorem~\ref{t:compactness} are sufficient for a straightforward passing to the limit. 

It is in the second dissipation term $\wh\calD_\e^{T,y}(\rho,j;\calF_\e)$  that the singularity can be observed. 
By truncating the domain from $\wh\nodes$ to $\Omega\times [0,1]$, the integral  in~\eqref{eqdef:D-part2} can be bounded from below by using the functional $\CCs$ from Lemma~\ref{l:props-N} as follows:
\begin{align}
	\calD_\e^{T,y}(\rho_\e,j_\e;\calF_\e) 
	&= \int_{\Omega_T}\int_{\hUpsilon_\e} 
	\set*{\frac{|\Omega|}{2m_\Upsilon} \ee^{\hF_\e^\rho} \frac{|\hj^y_\e|^2}{\hu_\e^\calF}  
		+ \frac{2m_\Upsilon}{|\Omega|} \ee^{-\hF_\e^\rho}\abs*{\partial_z \sqrt{\hu_\e^\calF}}^2 } \dx z \dxx xt \notag \\
	&\geq\int_{\Omega_T}\int_0^1 
	\set*{\frac{|\Omega|}{2m_\Upsilon} \ee^{\hF_\e^\rho} \frac{|\hj^y_\e|^2}{\hu_\e^\calF}  
		+ \frac{2m_\Upsilon}{|\Omega|} \ee^{-\hF_\e^\rho}\abs*{\partial_z \sqrt{\hu_\e^\calF}}^2 }\dx z \dxx xt \notag \\
	& \geq\int_{\Omega_T} \CCs\bra*{\hj^y_\e(t,x,\cdot)\big|_{[0,1]},\hu_\e^\calF(t,x,a), \hu_\e^\calF(t,x,b); \frac{m_\Upsilon}{|\Omega|} \ee^{-\hF_\e^\rho(t,x,\cdot)}}\dxx xt. \label{eq:Kramers:lb:Dy}
\end{align}
In this form, we can use the lower-semicontinuity properties of the function~$\CCs$ of~Lemma~\ref{l:props-N} and  its representation~\eqref{eqdef:N:explicit} to conclude the lower bound and arrive heuristically at
\begin{equation}
\label{ineq:Kramers-lower-bound-formal}
	\liminf_{\eps\to 0} \calD_\e^{T,y}(\rho_\e,j_\e;\calF_\e) \geq  \int_{\Omega_T} \sigma(t,x;\rho,\calF) \pra*{\sfC\bra*{\frac{\ol\jmath(t,x)}{\sigma(t,x;\rho,\calF)}} 
		+ \sfC^*\bra*{\log \frac{u^{\calF}(t,x,b)}{u^{\calF}(t,x,a)}}}\dxx x t. 
\end{equation}
Since it is not  guaranteed that $\sigma>0$ or $u^\calF>0$, we use the formulation~\eqref{eqdef:N:explicit:rigrous} of $\CCs$ to arrive at a rigorous result, which is part of the following theorem.

\begin{theorem}[Lower bound]
	\label{t:lower-bound}
	Assume that the sequence $(\rho_\e,j_\e)$ satisfies the convergence properties of Theorem~\ref{t:compactness}, and let $\calF_\e,\calF_0$ satisfy Assumption~\ref{ass:F}.
	Then the functional $\calD^T_\e$ defined in~\eqref{eq:FP-Kramerslimit:D:tilted} satisfies the $\Gamma$-lower bound
	\[
	\liminf_{\e\to0} \calD^T_\e(\rho_\e,j_\e;\calF_\e) \geq \calD^T_0(\rho_0,j_0;\calF_0),
	\]
	where 
	\begin{align}	\label{eqdef:D0T-Kramers}
		\MoveEqLeft\calD_0^T(\rho,j;\calF) :=  \int_{\nodes_T} \frac1{2m_\Omega} \abs*{\frac{\dx j^x}{\dx\rho}}^2\dx \rho\dx t   +  \int_{\nodes_T} 2{m_\Omega} \ee^{-F^\rho}\abs*{\nabla_x \sqrt {u^\calF} }^2\pi_0(\dxx xy)\dx  t 
		 \\
		&+ \int_{\Omega_T} \pra*{\sfC\bra*{\ol\jmath(t,x)\, \big| \, \sigma(t,x;\rho,\calF)} 
		+ 2\frac{m_\Upsilon}{|\Omega|} \ee^{-F^\rho(t,x,c)}\bra[\big]{\sqrt{u^{\calF}(t,x,b)}-\sqrt{u^{\calF}(t,x,a)}}^2}\dxx x t. \notag
	\end{align}
	Here  $\calD_0^T(\rho,j;\calF) = +\infty$ unless $(\rho,j)\in\wt\CE(0,T)$  with $\rho(t) \ll \pi_0$ for almost all $t\in[0,T]$. 	In these expressions
	\begin{subequations}
	\label{eqdef:sigma:Kramers}
	\begin{align}
	    \notag
		u(t,x,y) &:= \frac{\dx\rho}{\dx{\pi_0}}(t,x,y),\quad u^\calF = u \ee^{F^\rho}, 
		\qquad F^\rho = (\delta\calF/\delta\rho)(\rho),
		\qquad \qquad\text{and}\\
	\sigma(t,x;\rho,\calF)
	&:= \frac{{m_\Upsilon}}{\abs*{\Omega}} \ee^{-F^\rho(t,x,c)} \sqrt{ u^\calF(t,x,a) u^\calF(t,x,b)} \label{eqdef:sigma:Kramers-V1}\\
	\label{eq:char-sigma-measure-Kramers}
	 &=\frac{m_\Upsilon}{|\Omega|} \sqrt{\frac{\dx\rho}{\dx\pi_0}(x,a)\frac{\dx\rho}{\dx\pi_0}(x,b)}
  		\;\exp{\tfrac12 \bigl(F^\rho(x,a)+F^\rho(x,b)- 2F^\rho(x,c)\bigr)}.
	\end{align}
	\end{subequations}
\end{theorem}

To complement the liminf inequality above we verify that the limiting objects $\calE_0$ and $\calD_0^T$ satisfy the chain-rule lower bound Property~\ref{property:ChainRuleLowerBound}.
\begin{lemma}[Chain rule]
\label{l:CRLB-Kramers-limit}
Let $\calF\in \sfF$.
For all $(\rho,j)\in \wt\CE(0,T)$, 
\[
\abs*{ \calE_0^\calF(\rho(T)) - \calE_0^\calF(\rho(0)) }\leq  \calD_0^T(\rho,j;\calF),
\]
and in particular
\[
\calI_0(\rho,j;\calF) := \calE_0^\calF(\rho(T)) - \calE_0^\calF(\rho(0)) + \calD_0^T(\rho,j;\calF) \geq 0.
\]
\end{lemma}
\noindent
The proof is given in Appendix~\ref{app:CRLB}.

\subsection{Gradient-system convergence and the limiting problem}
\label{ss:Kramers:limit-pb}

The lower-bound Theorem~\ref{t:lower-bound} can be interpreted as a convergence of gradient systems to a contracted limiting system, as described by Definition~\ref{defn:EDP-convergence:reduce}. Recall that the contracted continuity-equation triple $(\wt\nodes,\wt\edges,\wt\anabla)$ was defined in~(\ref{eqdef:Kramers:contracted-nodes}--\ref{eqdef:Kramers:contracted-edges-anabla}) and the contracted energy $\wt\calE$ in~\eqref{eqdef:Kramers:wt-calE}.

\begin{cor}\label{cor:Kramers-tilt-EDP-conv}
Fix a modulus of continuity $\moc$ and define the subset of tilts
\[
\sfF^\moc := \set[\Big]{\;\calF \in C^1(\ProbMeas(\nodes)): \;\rmD\calF(\rho)\in C_{\mathrm b}^\moc(\nodes)\quad \text{such that } \sup_{\rho\in\ProbMeas(\nodes)} \norm*{\rmD\calF(\rho)}_{C_{\mathrm b}(\nodes)}<\infty\;},
\]
where $C^\moc_{\mathrm b}(\nodes)$ is defined as 
\[
C^\moc_{\mathrm b}(\nodes):= \set[\big]{ F\in C_{\mathrm b}(\nodes) : \;\abs{F(x,y)-F(x,y')} \leq \moc(|y-y'|) \text{ for all } x\in \Omega \text{ and }y,y'\in \Upsilon}
\]
Then the tilt gradient system $(\nodes,\edges,\nabla, \calE_\e, \sfF^\moc, \calR_\e)$ converges in the sense of Definition~\ref{defn:EDP-tilting} to the contracted system $(\wt\nodes,\wt\edges,\wt\anabla,\wt\calE,\sfF^\moc,\wt\calR)$ defined by the dissipation potentials
\begin{subequations}
\label{eqdef:tildeRR*}
\begin{align}
\label{eqdef:tildeR}	
\wt\calR\bra[\big]{\rho,j;\calF} 
  &:= \frac1{2m_\Omega}\int_\Omega 
    \bra[\bigg]{\abs*{\frac{\dx j^x_a}{\dx \rho_a }}^2\dx \rho_a 
          + \abs*{\frac{\dx j^x_b}{\dx \rho_b }}^2\dx\rho_b}
      + \int_\Omega \sfC\bra*{\ol \jmath(x)|\sigma(x;\rho,\calF)}\dx x\\
\wt\calR^*\bra[\big]{\rho,\Xi;\calF} 
  &:= \frac{m_\Omega}{2}\int_\Omega 
    \bra*{\abs*{\Xi^x_a}^2\dx \rho_a 
          + \abs*{\Xi^x_b}^2\dx\rho_b}
      + \int_\Omega \sigma(x;\rho,\calF)\sfC^*(\Xi^y(x))\dx x.
\end{align}
\end{subequations}
In these expressions we write $\rho_a$ and $\rho_b$ for the restrictions of $\rho$ to $\Omega\times\{a\}$ and $\Omega\times\{b\}$. We similarly write $j=(j^x_a, j^x_b,\ol\jmath)$ and $\Xi=(\Xi^x_a, \Xi^x_b,\Xi^y)$.
\end{cor}
\begin{proof}
The only condition of Definitions~\ref{defn:EDP-convergence:reduce} and~\ref{defn:EDP-tilting} that remains to be verified is the characterization of $\calD_0^T(\rho,j)$ in~\eqref{eqdef:D0T-Kramers} in terms of the contracted dissipations $\wt\calR$ and $\wt\calR^*$.  

For $(\rho,j)\in \wt\CE(0,T)$ the measure $\rho$ is of the form~\eqref{char:rho0} for almost all $t$; for such measures~$\rho$ we calculate that
\[
\rmD\wt\calE(\rho)(x,y) = \log \frac{\dx \rho}{\dx \pi_0}(x,y)  = 
\begin{cases}
\ds u_a(x) := \frac1 {\weight^a }\frac{\dx \rho_a}{\dx x}(x)	 & \text{for }y=a\\[2\jot]
\ds u_b(x) := \frac1{\weight^b} \frac{\dx \rho_b}{\dx x}(x)	 & \text{for }y=b\\[2\jot]
\text{undefined} & \text{for } y\not=a,b.
\end{cases}
\]
For given tilt $\calF$, writing  $F^\rho := \rmD\calF(\rho)$ and $u^\calF := u\ee^{F^\rho}$ as in Sections~\ref{ss:Kramers-setting} and~\ref{ss:Kramers:main-results},  we set 
%
\begin{align*}
\Xi^x (x,y) &:= -\nabla_x \rmD\bra*{\wt\calE(\rho)+\calF(\rho)}(x,y) = -\nabla_x \log u^\calF_y(x) ,\qquad\text{for }y=a,b,\\
\Xi^y (x,ab) &:= -\ona \rmD\bra*{\wt\calE(\rho)+\calF(\rho)}(x)= \log u_a^\calF(x) - \log u_b^\calF(x) .
\end{align*}
We then have
\begin{align*}
\wt\calR^*\bra[\big]{\rho,\Xi;\calF} (e)
= \frac{m_\Omega}{2} &\int_\Omega 
    \bra*{\abs*{\nabla_x \log u_a^\calF}^2\dx \rho_a 
          + \abs*{\nabla_x \log u_b^\calF }^2\dx\rho_b}\\
+ &\int_\Omega \sigma(x;\rho,\calF)\sfC^*\bra[\big]{\log u_a^\calF(x) - \log u_b^\calF(x) }\dx x.
\end{align*}
As in Section~\ref{ss:Kramers-setting} we observe that the first integral above is  a formal version of the second  integral in~\eqref{eqdef:D0T-Kramers}. The second integral above similarly is a formal version of the fourth integral in~\eqref{eqdef:D0T-Kramers} by the identity $\sqrt{pq}\sfC^*(\log p/q) = 2(p-q)^2$.

The first and third integral in~\eqref{eqdef:D0T-Kramers} are equal to the expression of $\wt\calR$ in~\eqref{eqdef:tildeR}.
\end{proof}
\begin{remark}[The limiting system from Corollay~\ref{cor:Kramers-tilt-EDP-conv} is tilt-dependent]
	As discussed already in Section~\ref{sss:structure-of-tilt-dependence}, the structure of $\sigma$ and $\wt\calR^*$ means that the limiting gradient system is tilt-dependent. 
\end{remark}

We next give a formal derivation of the equation induced by the limiting gradient system $(\wt\nodes,\wt\edges,\wt\anabla,\wt\calE,\sfF^\moc,\wt\calR)$. 
We calculate 
\begin{align*}
\rmD_2\wt\calR^*\bra[\big]{\rho,\Xi;\calF} (e)
&= \begin{cases}
m_\Omega \Xi^x_a\rho_a 	 & \text{for }e=(x,a)\\
m_\Omega \Xi^x_b\rho_b 	 & \text{for }e=(x,b)\\
\ds\sigma(x;\rho,\calF)(\sfC^*)'(\Xi^y(x)) & \text{for }e=(x,ab).
\end{cases}
\end{align*}
and in the case $\Xi = -\wt\anabla \rmD(\wt\calE+\calF)(\rho) = -\log u^\calF$ for $\rho = u^{\calF} \pi_0 = u \ee^{F^\rho}\pi_0$ as above we deduce
\begin{align*}
\rmD_2\wt\calR^*\bra[\big]{\rho,\Xi;\calF} (e)
&= \begin{cases}
 -m_\Omega \weight^a \bra*{\nabla_x u_a + u_a \nabla_x F^\rho_a}	 & \text{for }e=(x,a)\\
 -m_\Omega \weight^b \bra*{\nabla_x u_b + u_b \nabla_x F^\rho_b}	 & \text{for }e=(x,b)\\
 \sigma(x;\rho,\calF)(\sfC^*)'\bra[\Big]{\log \frac{u_a^\calF}{u_b^\calF}(x)} & \text{for }e=(x,ab).
\end{cases}
\end{align*}

Using $\sqrt{pq}\,{\sfC^*}'(\log p -\log q) = p-q$ and~\eqref{eqdef:sigma:Kramers-V1} this last expression can be written as 
\[
\rmD_2\wt\calR^*\bra[\big]{\rho,\Xi;\calF} (x,ab) 
= r(x) := \frac{m_\Upsilon}{|\Omega|} \ee^{-F^\rho(x,c)}\; \bra*{u^\calF_a(x) - u_b^\calF(x)}.
\]

The limiting equation therefore is
\begin{align*}
\partial_t \rho_a &= m_\Omega \weight^a\div_x  \bra[\big]{\nabla_x u_a + u_a \nabla_x F^\rho_a}
  - r,\\
\partial_t \rho_b &= m_\Omega \weight^a\div_x \bra[\big]{\nabla_x u_b + u_b \nabla_x F^\rho_b}
+r,
\end{align*}
which can also be written in terms of the measures  $\rho_{a,b}$ as 
\begin{subequations}\label{eq:Kramers-limit-equation-measures}
\begin{align}
\partial_t \rho_a &= m_\Omega \div_x\bra[\big]{ \nabla_x \rho_a + \rho_a \nabla (V_a + F_a^\rho)} 
  - \frac{m_\Upsilon}{|\Omega|} \bra*{\frac{1}{\weight^a} \ee^{V_a + F^\rho_a-F^\rho_c}\rho_a - \frac{1}{\weight^b} \ee^{V_b + F^\rho_b-F^\rho_c}\rho_b},\\
\partial_t \rho_b &= m_\Omega \div_x \bra[\big]{ \nabla_x \rho_b + \rho_b \nabla (V_b + F_b^\rho)}
+\frac{m_\Upsilon}{|\Omega|} \bra*{\frac{1}{\weight^a} \ee^{V_a + F^\rho_a-F^\rho_c}\rho_a - \frac{1}{\weight^b} \ee^{V_b + F^\rho_b-F^\rho_c}\rho_b}.
\end{align}
\end{subequations}

\subsection{Interpretation of \texorpdfstring{$\sigma$}{sigma} as `activity'}
\label{ss:Kramers-activity-interpretation}

This pair of equations~\eqref{eq:Kramers-limit-equation-measures} above  has the same structure as~\eqref{eq:RD-intro}--\eqref{eq:flux-ex-B+A} of Example~\ref{ex:FP}+\ref{ex:heat-flow}, and this allows us to compare the jump rates in the two sets of equations:
\begin{subequations}
\begin{alignat}3
\label{eq:kappas-Kramers-and-ExBA-1}
\text{Equations~\eqref{eq:Kramers-limit-equation-measures}}: &&\quad 
\kappa_{ab} &= \frac1{\weight^a}\ee^{V_a + F^\rho_a-F^\rho_c} &\quad\text{and}\quad
\kappa_{ba} &= \frac1{\weight^b}\ee^{V_b + F^\rho_b-F^\rho_c},\\
\text{Eqns~\eqref{eq:RD-intro}--\eqref{eq:flux-ex-B+A}}: &&\quad 
\kappa_{\sfa\sfb} &=  k_{\sfa\sfb} \ee^{\frac12(V_\sfa+F^\rho_\sfa-V_\sfb-F^\rho_\sfb)}&\quad\text{and}\quad
\kappa_{\sfb\sfa} &=  k_{\sfa\sfb} \ee^{\frac12(V_\sfb+F^\rho_\sfb-V_\sfa-F^\rho_\sfa)}
\label{eq:kappas-Kramers-and-ExBA-2}
\end{alignat}
\end{subequations}
The equations~\eqref{eq:RD-intro}--\eqref{eq:flux-ex-B+A} do not contain any tilting, and therefore in the version~\eqref{eq:kappas-Kramers-and-ExBA-2} of~\eqref{eq:RD-intro}--\eqref{eq:flux-ex-B+A}   we conjectured the dependence on $F^\rho$ by replacing $V$ by $V+F^\rho$. We also omitted the unimportant global rate constant $m_\Upsilon/|\Omega|$.

\medskip

The differences between the two sets of jump rates are notable in two ways:
\begin{enumerate}
\item The value of the tilt $F^\rho$ in $y=c$ influences the rates~\eqref{eq:kappas-Kramers-and-ExBA-1} but not~\eqref{eq:kappas-Kramers-and-ExBA-2};
\item In~\eqref{eq:kappas-Kramers-and-ExBA-1} the jump rate from $a$ to $b$  depends on $F^\rho_a$ (and $F^\rho_c$), but not on $F^\rho_b$, while in~\eqref{eq:kappas-Kramers-and-ExBA-2} the rate depends on $F^\rho_\sfa$ and $F^\rho_\sfb$.
\end{enumerate}
If we set $\scrV := V + F^\rho$, and disregard the lack of a value of $V$ at $y=c$, then we can formulate these relations as
\begin{subequations}
\begin{alignat}2
\label{eq:kappas-Kramers-and-ExBA-bis-1}
\text{Equations~\eqref{eq:Kramers-limit-equation-measures}}: &&\qquad 
\kappa_{ab} &\sim \ee^{\scrV_a - \scrV_c}\\
\text{Equations~\eqref{eq:RD-intro}}: &&\qquad 
\kappa_{\sfa\sfb} &\sim \ee^{\frac12\bra*{\scrV_\sfa - \scrV_\sfb}}
\end{alignat}
\end{subequations}
The  Kramers description explains the dependence in~\eqref{eq:kappas-Kramers-and-ExBA-bis-1}: $\scrV_a-\scrV_c$ can be interpreted as the energy barrier for a particle to leave the well at $a$, and the $a\rightsquigarrow b$ rate depends on this difference (and not on the value of~$\scrV$ at~$b$). 

 We can also recognize in the dependence on $\scrV$ in~\eqref{eq:kappas-Kramers-and-ExBA-bis-1} and~\eqref{eq:kappas-Kramers-and-ExBA-1} the classical modelling of rates of chemical reactions (see Example~\ref{ex:chem-reactions} and Section~\ref{ss:case-study-chem-reactions-tilt-dependent}). Again, this is natural considering the understanding of chemical reactions as transitions between wells in potential-energy landscapes, corresponding to the Kramers limit.

\bigskip

The differences between~\eqref{eq:kappas-Kramers-and-ExBA-1} and~\eqref{eq:kappas-Kramers-and-ExBA-2} illustrate a challenge in modelling with gradient systems. 
The rates~\eqref{eq:kappas-Kramers-and-ExBA-1} and~\eqref{eq:kappas-Kramers-and-ExBA-2} arise from two different sets of assumptions:  
In constructing Example~\ref{ex:FP}+\ref{ex:heat-flow} out of Examples~\ref{ex:FP} and~\ref{ex:heat-flow} (leading to~\eqref{eq:kappas-Kramers-and-ExBA-2}) we {assumed} that the `correct' expression for the prefactor~$\sigma$ in~$\calR_J^*$ is the geometric mean of the measures
\[
\sigma(\dx x) = \sqrt{\rho_\sfa\rho_\sfb}(\dx x)\qquad\text{(modulo constants)}.
\]
The expressions~\eqref{eqdef:sigma:Kramers}, however, suggest to view the prefactor as 
\begin{equation*}
\sigma(\dx x) = \sqrt{\sigma_{\sfa\sfb}(x)\sigma_{\sfb\sfa}(x)} \dx x,
\end{equation*}
where `forward' and `backward' \emph{activities} $\sigma_{\sfa\sfb}$ and $\sigma_{\sfb\sfa}$ are defined as 
\[
\sigma_{\sfa\sfb}(x) := C_{\text{saddle}}(x)\frac{\dx\rho}{\dx\pi^\calF}(x,\sfa)
\qquad\text{and}\qquad
\sigma_{\sfb\sfa}(x) := C_{\text{saddle}}(x)\frac{\dx\rho}{\dx\pi^\calF}(x,\sfb),
\]
where $\dx \pi^\calF = \ee^{-F^\rho}\dx \pi \sim \ee^{-F^\rho-V}\dx x$ is the tilted invariant measure.
Indeed, such a choice leads to reaction rates  (again using $\sqrt{pq}\,{\sfC^*}'(\log p-\log q) = p-q$)
\begin{align*}
\sqrt{\sigma_{\sfa\sfb}\sigma_{\sfb\sfa}} {\sfC^*}'\bra[\Big]{-\ona\log \frac{\dx\rho}{\dx\pi^\calF}}
&= C_{\text{saddle}}(x) \bra[\Big]{\frac{\dx\rho}{\dx\pi^\calF}(x,\sfa) - \frac{\dx\rho}{\dx\pi^\calF}(x,\sfb)}\\
&\sim C_{\text{saddle}}(x) \bra[\Big]{\frac1{\gamma^\sfa}\rho_\sfa\ee^{V_\sfa+F^\rho_\sfa} - \frac1{\gamma^\sfb}\rho_\sfb\ee^{V_\sfb+F^\rho_\sfb}}(x)
\end{align*}
This reproduces the dependence  of~\eqref{eq:kappas-Kramers-and-ExBA-1} on $V$ and $F^\rho$ at $\sfa$ and $\sfb$. It also reproduces the dependence of~\eqref{eq:kappas-Kramers-and-ExBA-1} on~$F^\rho_c$, if we interpret the factor $C_{\text{saddle}}$ as modulating the height of the saddle by
\[
C_{\text{saddle}}(x) := C \ee^{-F^\rho(x,c)}.
\]
The importance of  activities such as $\sigma_{\sfa\sfb}$ and $\sigma_{\sfb\sfa}$ is well known in the chemical-reaction literature (see e.g.~\cite[p.~8]{Pekar05} or~\cite[\S 2.1]{Peters17}). 

\medskip
 Summarizing, the discussion above suggests that the prefactor $\sigma$ should be considered as a geometric mean of \emph{activities} instead of a geometric mean of \emph{concentrations} or \emph{measures}.
 As we observed above, a consequence of modelling $\sigma$ as geometric mean of activities is that the ensuing gradient structure is not tilt-independent, and it also follows that $\sigma$ may depend on $\rho$ through the tilting function $F^\rho$.

\subsection{The Kramers sequence  is not contact-EDP convergent}
\label{ss:Kramers-not-contact-EDP-conv}

In~\cite{MielkeMontefuscoPeletier21} two alternative convergence concepts were introduced for sequences of gradient systems which are specifically aimed at preserving tilt-independence. If a sequence of tilt-independent gradient systems converges in either of these senses (called \emph{tilt-EDP convergence} and \emph{contact-EDP convergence}), then the limit also is tilt-independent. 

In this section we show that the sequence of Kramers $\e$-systems discussed above does not converge in either of these senses; this is a confirmation that the tilt-dependence in the limit is a true property of the behaviour of these systems, and not something that can be remedied by changing the convergence concept. 

The definition of contact-EDP convergence in~\cite{MielkeMontefuscoPeletier21} is based on the setup of gradient systems that was mentioned in Remark~\ref{rem:non-ct-eq-GS}, in which the dissipation potentials $\scrR$ and $\scrR^*$ are functions of rate of change $\dot \rho$ and variational derivative $\xi = -\rmD\calE(\rho)$ rather than of flux~$j$ and of the gradient $\Xi = \anabla \xi$.  
In order to connect these setups we convert the functional $\calD_\e^T = \calD_\e^T(\rho,j)$ into a function of the time-course of $\rho$ only,
\[
\scrD_\e^T(\rho;\calF) := \inf_{j: (\rho,j)\in\CE} \calD_\e^T (\rho,j;\calF).
\]

In~\cite{MielkeMontefuscoPeletier21}, tilt-EDP and contact-EDP convergence are defined rigorously for systems on finite-dimensional manifolds; in more general situations,  such as those of this paper, case-dependent adaptations need to be made to make statements rigorous. For this reason we give only a formal definition:
\begin{definition}[Tilt-EDP and contact-EDP convergence]
\label{def:tilt-contact-EDP-convergence}
A sequence $(\calE_\e,\scrR_\e)$ converges to a limit $(\calE_0,\scrR_0)$ in either tilt- or contact-EDP sense if:
\begin{enumerate}
\item $\calE_\e\stackrel\Gamma\longrightarrow \calE_0$;
\item For each $\calF$, $\scrD_\e(\,\cdot\, ;\calF) \stackrel\Gamma \longrightarrow \scrD_0(\,\cdot\, ;\calF)$ for some limit functional $\scrD_0$, which can be written as 
\begin{equation}
\label{eq:tilt-contact-conv-N}
\scrD_0(\rho;\calF) = \int_0^T \calN_0\bra[\big]{\rho(t),\dot\rho(t), -\rmD\calF(\rho(t))}\dx t,
\end{equation}
for some functional $\calN_0 = \calN_0(\rho,\dot\rho,\eta)$.  
\item There exists a functional $\calM_0$ such that 
\[
\calM_0(\rho,v,\xi) = \calN_0(\rho,v,\xi+\rmD\calE_0(\rho)),
\]
for some class of functions $\xi:\nodes\to\R$.

\end{enumerate}

\noindent
For tilt-EDP convergence:
\begin{enumerate}[resume,label=\arabic*a.]
	\item $\calM_0$ can be written as
\begin{equation}
\label{eqdef:tilt-EDP-limit-condition}
\calM_0(\rho,v,\xi) = \scrR_0(\rho,v) + \scrR_0^*\bra*{\rho,\xi}
\qquad\text{for all }\rho,v,\xi.
\end{equation}
\end{enumerate}

\noindent
For contact-EDP convergence we define the `contact sets'
\begin{align}
\label{eqdef:contact-set}	
\Contact_{\calM_0}(\rho) &= 
\set[\big]{\,(v,\xi):\  \calM_0(\rho,v,\xi) = \dual v \xi\,},\\
\Contact_{\scrR_0\oplus \scrR_0^*}(\rho) &= 
\set[\big]{\,(v,\xi):\  \scrR_0(\rho,v) + \scrR_0^*\bra*{\rho,\xi}  = \dual v \xi\,}. \notag
\end{align}
For contact-EDP convergence we then require
\begin{enumerate}[resume,label=\arabic*b.]
\addtocounter{enumi}{-1}
\item We have 
\begin{equation}
\label{eqdef:contact-EDP-limit-condition}
\Contact_{\calM_0}(\rho) = \Contact_{\scrR_0\oplus \scrR_0^*}(\rho)\qquad\text{for all }\rho.
\end{equation}
\end{enumerate}
\end{definition}

\begin{remark}[Differences with the EDP-convergence concepts of Section~\ref{ss:convergence-of-GS}]
Both the definitions of tilt-EDP and contact-EDP convergence above and the EDP convergence concepts of Section~\ref{s:GS} build upon the EDP concept of a solution (Definition~\ref{def:EDP:sol}). They also both include dependence on tilting. 

There are important differences, however:
\begin{enumerate}
	\item The tilt-EDP and contact-EDP convergence concepts are constructed with the aim of strengthening the convergence concept to the point that the corresponding limit systems are tilt-independent by construction. This is reflected by the requirement that $\scrR_0$ in~\eqref{eqdef:tilt-EDP-limit-condition} and~\eqref{eqdef:contact-EDP-limit-condition} can not depend on the tilt. 
	\item By contrast, the versions of EDP convergence in Section~\ref{ss:convergence-of-GS} (Definitions~\ref{defn:EDP-convergence} and~\ref{defn:EDP-convergence:reduce}) explicitly allow for dependence of the limit potential $\calR_0$ on the tilt. 
\end{enumerate}
\end{remark}

\medskip

In the remainder of this section we show that the sequence of Kramers gradient systems can not converge in the contact-EDP convergence sense; since  tilt-EDP convergence implies contact-EDP convergence, this disproves both convergence concepts at the same time. 

\bigskip

To start with, assume that the Kramers gradient systems converges in the contact-EDP sense above. We then identify the function $\calN_0$ in~\eqref{eq:tilt-contact-conv-N} by matching  $F^\rho$ to $-\eta$ in~\eqref{eqdef:D0T-Kramers} (see also the formulation~\eqref{ineq:Kramers-lower-bound-formal}), by which we obtain
\begin{multline*}
\calN_0(\rho,v,\eta)  
=  
\inf_{\substack{j = (j^x,\ol \jmath):\\ \adiv(j^x,\ol \jmath) = -v}}\Biggl\{
 \int_{\nodes} \frac1{2m_\Omega} \abs*{\frac{\dx j^x}{\dx\rho}}^2\dx \rho   +  \int_{\nodes} 2{m_\Omega} \ee^{\eta}\abs*{\nabla_x \sqrt {u\ee^{-\eta}} }^2\pi_0(\dxx xy) \\
		\quad\quad\quad\quad\quad\quad  + \int_{\Omega} \ol\sigma(x;\rho,\eta) \pra*{\sfC\bra*{\frac{\ol\jmath(x)}{\ol\sigma(x;\rho,\eta)}} 
		+ \sfC^*\bra*{\log \frac{u(x,b)}{u(x,a)} - \eta(x,b) + \eta(x,a)}}\dx x  \Biggr\},
\end{multline*}
with correspondingly modified rate parameter $\ol\sigma$,  
\[
\ol\sigma(x;\rho,\eta) = 
\frac{{m_\Upsilon}}{\abs*{\Omega}} \sqrt{{u(x,a)u(x,b)}}\; \exp{\tfrac12 \bigl(2\eta(x,c) - \eta(x,a)-\eta(x,b)\bigr)}. 
\]
Here $u = \dx \rho/\dx \pi_0$. This expression now presents two problems:
\paragraph{Problem 1.}
The next step in the verification of contact-EDP convergence is the assumption that we can set 
\begin{equation*}
\calM_0(\rho,v,\xi) := \calN_0(\rho,v,\xi+\rmD\calE_0(\rho)),
\end{equation*}
for an appropriate class of functions $\xi$. 

The problem here is that the value of $\sigma(x;\rho,\eta)$ and therefore $\calN_0(\rho,v,\eta)$ depends on the value of~$\eta$ at $y=c$, and the argument $\eta = \xi+\rmD\calE_0(\rho)$ can not provide this. This is because~$\rho$ is supported on $\Omega\times\{a,b\}$, and therefore $\rmD\calE_0(\rho)$ has no meaning outside of $\Omega\times\{a,b\}$.

%
%
%
%

\paragraph{Problem 2.}
Even if we restrict ourselves to tilts $\calF$ such that $\rmD\calF(\rho)$ vanishes at $y=c$, and therefore $\eta(x,c)=0$ for all $x$,  a problem remains. After making the transformation above, we find
\begin{align*}
\calM_0(\rho,v,\xi)  
&=  
\inf_{j} \set[\Big] {\,\wt\calM_0(\rho,j,\xi): \adiv j = -v\,}, \\
\wt\calM_0(\rho,(j^x,\ol j),\xi) 
&= 
 \int_{\nodes} \frac1{2m_\Omega} \abs*{\frac{\dx j^x}{\dx\rho}}^2\dx \rho   
 +  \int_{\nodes} \frac{m_\Omega}2 \abs*{\nabla_x \xi }^2\pi_0(\dxx xy) \\
&		\quad\quad\quad\quad\quad\quad  + \int_{\Omega} \wt\sigma(x;\xi) \pra*{\sfC\bra*{\frac{\ol\jmath(x)}{\wt\sigma(x;\xi)}} 
		+ \sfC^*\bra*{\ona\xi(x)}}\dx x ,
\end{align*}
with
\[
\wt\sigma(x;\xi) = 
\frac{{m_\Upsilon}}{\abs*{\Omega}} \exp\pra*{-\tfrac12 \bigl( \xi(x,a)+\xi(x,b)\bigr)}. 
\]

The concept of contact-EDP convergence is based on the contact set $\Contact_{\calM_0}$ in~\eqref{eqdef:contact-set}. Using the expression above for $\calM_0$ we can characterize this set more explicitly: 
\begin{alignat*}2
&& \calM_0(\rho,v,\xi) &= \dual v \xi\\
&\iff&\qquad  \inf_{\adiv j = -v} \wt\calM_0(\rho,j,\xi)  &= \dual v \xi\\
&\iff& \wt\calM_0(\rho,j,\xi)  &= \dual j {\anabla\xi} \qquad\text{(and } v = -\adiv j)\\
&\iff& j^x &= \nabla_x \xi \quad\text{and}\quad 
\ol\jmath (x) = \wt\sigma(x;\xi) \,(\sfC^*)'\bra*{\ona \xi(x)},
%
\end{alignat*}
so that 
\begin{equation}
\label{eq:char1-CM0}
\Contact_{\calM_0} = \set[\big]{\,(v,\xi): v = -\div_x j^x - \odiv \ol\jmath, 
\ j^x = \nabla_x \xi, \text{ and }
\ol\jmath (x) = \wt\sigma(x;\xi) \,(\sfC^*)'\bra*{\ona \xi(x)}\,}
\end{equation}
By Definition~\ref{def:tilt-contact-EDP-convergence} the sequence of gradient systems converges in the contact-EDP sense if and only if there exists a dissipation potential $\scrR_0$ such that the contact set $\Contact_{\calM_0}$ of $\calM_0$ coincides with the contact set of $\scrR_0$, i.e. such that $\Contact_{\calM_0}$ can be written as 
\begin{align}
\notag	
\Contact_{\calM_0} (\rho)
&= 
\set[\big]{\,(v,\xi):\  \scrR_0(\rho,v) + \scrR_0^*(\rho,\xi) = \dual v \xi\,}\\
&=\set[\big]{\,(v,\xi):\  v\in \partial_\xi \scrR_0^*(\rho,\xi) \,}.
\label{eq:char2-CM0}
\end{align}
Assume, to force a contradiction, that such an $\scrR_0$ exists. Since $\calM_0(\rho,v,\xi)=+\infty$ unless $v$ can be written as $v = -\adiv j$, Lemma~\ref{l:char-phi-contact-on-subspace} below implies that similarly $\scrR_0(\rho,v) = +\infty$ unless $v$ is a divergence. It then follows that $\scrR_0^*$ depends only on $\anabla \xi$: 
\begin{align*}
\scrR_0^*(\rho,\xi) &= \sup_{v} \dual v\xi - \scrR_0(\rho,v)
= \sup_{j} \dual j{\anabla\xi} - \scrR_0(\rho,v)=: \calR_0^*(\rho,\anabla \xi).
\end{align*}

Since therefore $\xi \mapsto \scrR_0^*(\rho,\cdot)$ depends only on $\anabla \xi$, also the subdifferential $\partial_\xi\scrR_0(\rho,\cdot )$  depends only on $\anabla \xi$. 
However, comparing~\eqref{eq:char2-CM0} with~\eqref{eq:char1-CM0} contradicts this: the prefactor~$\wt\sigma$ depends on $\xi(\cdot,a) + \xi(\cdot,b)$, which is not a function of $\nabla_x \xi$ and $\ona\xi = \xi(\cdot,b) - \xi(\cdot,a)$.
This shows how the particular dependence of the prefactor $\wt\sigma$ on $\xi$ prevents the Kramers limit from converging in the contact-EDP sense. 

This concludes the discussion of tilt- and contact-EDP convergence; it remains to state and prove the following lemma.

\begin{lemma}
\label{l:char-phi-contact-on-subspace}
Let $X$ be a Banach space,  $\varphi:X\to\R\cup \{+\infty\}$ a proper lower semicontinuous convex function, and $\varphi^*:X^*\to\R\cup\{+\infty\}$ its convex  dual. Let $A\subset X$ be a closed linear subspace with the property that
\begin{equation*}
\varphi(x) + \varphi^*(\xi) = \dual x\xi\quad \Longrightarrow \quad x\in A.
\end{equation*}
Then $\varphi(x) = +\infty$ for any $x\in X\setminus A$.
\end{lemma}
The point of this lemma is that whenever a contact set is restricted to a linear subspace, the generating convex function necessarily is equal to $+\infty$ outside of that subspace.

\begin{proof}
Define the convex lower semicontinuous function
\[
\psi(x) := \begin{cases}
  \varphi(x) & \text{if }x\in A\\
  +\infty & \text{otherwise}	.
 \end{cases}
\]
The function $\psi$ satisfies $\psi\geq \varphi$ and therefore $\psi^*\leq \varphi^*$.  If $(x,\xi)$ is in contact for~$\varphi$, then $x\in A$, and therefore
\[
\psi^*(\xi)\leq \varphi^*(\xi) = \dual x\xi - \varphi(x)
= \dual x\xi -\psi(x) \leq \psi^*(\xi),
\]
implying that $(x,\xi)$ also is in contact for $\psi$.
Writing $\Contact_\varphi=\set*{(x,\xi): \varphi(x)+\varphi^*(\xi)=\skp{x,\xi}}$ and similar for $\Contact_\psi$ for the two contact sets, this shows that $\Contact_\varphi\subseteq \Contact_\psi$. Since both contact sets are the graphs of the corresponding subdifferentials, which are maximal cyclical monotone by convexity~\cite{Rockafellar66}, we find $\Contact_\varphi = \Contact_\psi$.

On Banach spaces, the subdifferential of a convex function uniquely characterizes the function itself, up to an additive constant~\cite{Rockafellar66}; therefore $\varphi =\psi+c$ for some $c\in \R$, which implies the assertion.
\end{proof}

\newpage 

\section{A thin-membrane limit}
\label{s:thin-membrane}

Singular limits naturally arise in systems describing materials with strongly contrasting properties. A classic example is the `sandwich structure', in which a thin layer of weakly conducting material is placed between layers of material that conducts well. 
A particular scaling of the properties of the layer causes the thin layer to reduce to a zero-thickness barrier (a `membrane'), and in the limit certain \emph{transmission conditions} arise that connect the flux across this membrane with properties on both sides. This example has been studied in various forms in recent years; see e.g. \cite{Neuss-RaduJager07, GahnNeuss-RaduKnabner16, GahnNeuss-RaduKnabner17, GahnNeuss-RaduKnabner18} for approaches based on convergence of the partial differential equation, and~\cite{LieroMielkePeletierRenger17, FrenzelLiero21, MielkeMontefuscoPeletier21, FrenzelMielke21TA} for approaches based on EDP convergence. 

Here we briefly describe a slightly modified version of the result of~\cite{LieroMielkePeletierRenger17,FrenzelLiero21} and its consequences for tilt-dependence and tilt-independence. The setup is as follows: a layer of thickness $\e$ at $0<\hat x<\e$ is placed between layers $[-1,0]$ and $[\e,1+\e]$. The material properties are defined by a potential $\wh V_\e$ and a mobility $\hat a_\e$, which are assumed to be given by
\begin{equation*}
\hat a_\e(\hat x)=\left\{\begin{array}{cl}
  a_-(\hat x)&\text{for }\hat x<0,\\
  \e a_*(\hat x/\e)& \text{for }\hat x\in [0,\e],\\
  a_+(\hat x{-}\e{+}1) &\text{for }\hat x>\e,
  \end{array} \right. \quad 
\wh V_\e(\hat x)=\left\{\begin{array}{cl}
  V(\hat x)&\text{for }\hat x<0,\\
  V(\hat x/\e)& \text{for }\hat x\in [0,\e],\\
  V (\hat x{-}\e{+}1) &\text{for }\hat x>\e,
\end{array} \right.
\end{equation*}
for some smooth $V:[-1,2]\to\R$.
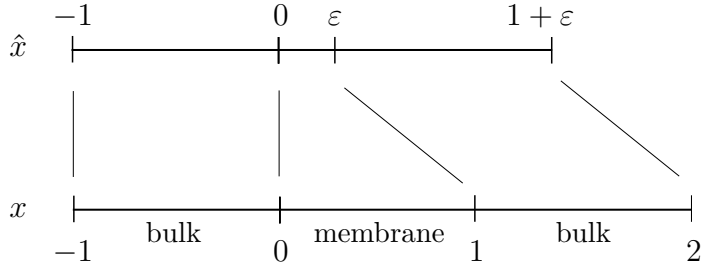
\begin{figure}[ht]
\centering
	\setlength{\unitlength}{0.5\textwidth}
	\begin{picture}(1,0.352092241)%
		\put(0,0){\includegraphics[width=\unitlength]{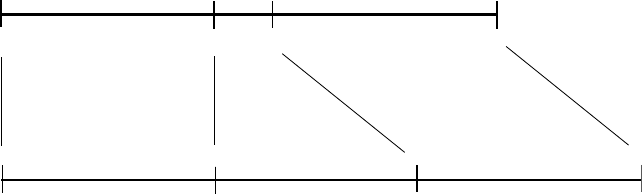}}%
		\put(-0.03,-0.06){{$-1$}}%
		\put(0.325,-0.06){{$0$}}%
		\put(-0.03,0.32){{$-1$}}%
		\put(0.325,0.32){{$0$}}%
		\put(0.64,-0.06){{$1$}}%
		\put(0.413,0.32){{$\e$}}%
		\put(0.7,0.32){{$1+\e$}}%
		\put(0.99,-0.06){{$2$}}%
		\put(-0.1,0.27){{$\hat x$}}%
		\put(-0.1,0.01){{$x$}}%
		\put(0.39,-0.03){\small membrane}
		\put(0.12,-0.03){\small bulk}
		\put(0.78,-0.03){\small bulk}
	\end{picture}%
	\vskip1.5em
	\caption{The geometry of the thin membrane in physical coordinates $\hat x$ (top) and rescaled coordinates $x$ (bottom).}
\end{figure}

The equation describing the evolution is
\begin{equation*}
 \begin{aligned}
 & \partial_t \hat \rho = \partial_{\hat x}\Bigl( \hat a_\e(\hat x) \bigl( \partial_{\hat x} \hat \rho + \hat \rho \wh V_\eps'(\hat x) \bigr) \Bigr)
   \qquad & &\text{in }(-1,1+\e),\\
 & \partial_{\hat x} \hat \rho(t,\hat x) + \hat \rho(t,\hat x) \wh V'_\e(\hat x)=0 \qquad &&\text{at }\hat x=-1,1+\e.
 \end{aligned}
\end{equation*}
The scaling $O(\e)$ of the mobility $\hat a_\e$ inside the membrane will combine with the thickness $\e$ of the membrane to produce a barrier of zero thickness but finite permeability, as we shall see below. The potential $\wh V_\e$ determines the energetic `cost' of being in different parts of the sandwich structure. Note that  the particular choices $\hat a_\e(\hat  x) = \e a_*(x/\e)$ and $\wh V_\e(\hat x) = V(\hat x/\e)$ inside the membrane imply that these parameters can vary across the thickness of the membrane. 

One interpretation for this equation is as follows. The unknown $\hat\rho$ represents the density of some  species of particles that diffuses through the material by hopping from site to site. The parameter $\hat a_\e$ characterizes the hopping rate, and the potential $\wh V_\e$ characterizes the energy levels of the individual hopping sites. The linearity of the prefactor $\hat \rho$ of $\wh V_\e'$  is consistent with a low-density assumption in which there are many more hopping sites than particles. In a higher-density regime, as in the example of the Simple Symmetric Exclusion Process, a prefactor such as $\hat \rho(K-\hat \rho)$ would be expected (see e.g.~\cite{CarrilloLisiniSavareSlepcev10} for a discussion). 

\medskip

To obtain an $\e$-independent geometry we rescale  the membrane to thickness $1$ by setting 
\begin{equation}
\label{eqdef:trans-x-xhat}
x(\hat x) := \left\{
\begin{array}{cl}
  \hat x&\text{for }\hat x<0,\\
  \hat x/\e & \text{for }\hat x\in [0,\e],\\
  \hat x{-}\e{+}1 &\text{for }\hat x>\e,
  \end{array} \right. 
\qquad
a_\e( x)=\left\{\begin{array}{cl}
  a_-( x)&\text{for } x<0,\\
  \e^{-1}a_*( x)& \text{for } x\in [0,1],\\
  a_+( x) &\text{for } x>1,
\end{array} \right. \quad 
\end{equation}
leading to the rescaled equations
\begin{subequations}\label{eq:thin-membrane}
\begin{align}
 & \partial_t \rho = \partial_x\Bigl(  a_\e(x) \bigl( \partial_x \rho + \rho  V'(x) \bigr) \Bigr)\qquad &&\text{in }(-1,2),
\\
 & \partial_{ x}  \rho(t, x) +  \rho(t, x) V'( x)=0 \qquad&&\text{at }x=-1,2.
 \label{eq:thin-membrane-BCs}
\end{align}
\end{subequations}
This equation is a gradient flow in continuity-equation structure for the choices
\begin{alignat*}2
\nodes &:= \edges := [-1,2], &\qquad \calR_\e^*(\rho,\Xi) 
&:= \frac12\int_{-1}^2 a_\e(x) \abs{\Xi(x)}^2\rho(\dx x).
\\
\calE_\e(\rho) &:= \RelEnt(\rho|\pi_\e), &\qquad  \pi_\e(\dx x) &= \frac1{\PartSum_\e} \ee^{-V(x)}
\left\{
  \begin{array}{ll} 
    1 & -1\leq x \leq 0\\
    \e & \phantom- 0<x<1 \\
    1 & \phantom- 1\leq x \leq 2
  \end{array}
  \right\}\dx x.
\end{alignat*}
Note that the low-density assumption is embodied in the choice of Boltzmann entropy and the linear dependence on $\rho$ in $\calR_\e^*$.

\medskip

In the limit $\e\to0$, the structure of the system~\eqref{eq:thin-membrane} changes. To start with, since $\pi_\e$ vanishes on the rescaled membrane $(0,1)$, any sequence of measures $\rho_\e$ with bounded energy $\calE_\e(\rho_\e)$ also converges to zero on $(0,1)$; in modelling terms, the zero-thickness membrane can not contain a positive amount of mass. As a consequence, the flux $j_\e$ in a pair $(\rho_\e,j_\e)\in\CE(0,T)$ satisfies $\partial_x j_\e\to0$ in the sense of distributions on $(0,1)\times (0,T)$. It follows that on the membrane set $(0,1)\times (0,T)$ the limit flux is a function of time only (called $j_\sfm$ below). Similarly to the discussion in Section~\ref{ss:Kramers-setting} this generates a natural limiting combination of node set, edge set, gradient, and divergence:
\begin{align*}
\wt\nodes &= [-1,0]\cup [1,2]  \\
\wt\edges & = [-1,0]\sqcup \{\sfm\}\sqcup [1,2]  \qquad \text{($\sfm$ for `membrane')}\\
\wt\anabla f(x) &= \bra[\big]{f'|_{[-1,0]}, \,f(1)-f(0),\, f'|_{[1,2]}}=: (\nabla_- f ,\,\ona f,\,\nabla_+ f) \\
\wt\adiv \wt\jmath & = (\partial_x j_-)\bONE_{[-1,0]} 
+  (j_-(0) - j_\sfm)\delta_0  + (j_\sfm-j_+(1))\delta_1
+ (\partial_x j_+)\bONE_{[1,2]} \\
&\qquad\qquad \text{for }\wt\jmath = (j_-,j_\sfm, j_+).
\end{align*}
The connection between original flux $j$ and contracted flux $\wt\jmath$ is given by
\begin{equation}
\label{eq:thin-membrane:relation-j-jtilde}
j(t,\dx x) = \begin{cases}
  j_-(t,\dx x)	& \text{if } {-1}\leq x \leq 0\\
  j_\sfm(t)\dx x  & \text{if } \phantom{-} 0< x< 1\\
  j_+(t,\dx x)  & \text{if } \phantom{-} 1\leq x \leq 2,
\end{cases}
\qquad \text{for }\wt\jmath = (j_-,j_\sfm,j_+).
\end{equation}
The corresponding weak form of the continuity equation (including the boundary conditions~\eqref{eq:thin-membrane-BCs}) becomes
\begin{multline}
0 = \int_0^T \!\!\!\!\!\!\int\limits_{\strut(-1,0)\cup (1,2)} \partial_t \varphi(t,x) \rho(t,\dx x)\dx t + \int_0^T \!\!\int_{-1}^0\partial_x \varphi(t,x) j_-(t,\dx x)\dx t
+ \int_0^T \!\!\int_1^2\partial_x \varphi(t,x) j_+(t,\dx x)\dx t \\[-4\jot]
+ \int_0^T \bra*{\varphi(t,1)-\varphi(t,0)}j_\sfm(\dx t) 
\qquad\qquad
\text{for all }
\varphi\in C^1_{\mathrm c}((0,T)\times [-1,2]).
\label{eq:ct-eq-thin-membrane-weak-form}
\end{multline}

By the same techniques as in~\cite{LieroMielkePeletierRenger17,FrenzelLiero21} one can prove the following result, which implies convergence to a contracted gradient system given by $(\wt\nodes,\wt\edges,\wt\anabla,\wt\calE,\sfF,\wt\calR)$ as in Definitions~\ref{defn:EDP-convergence:reduce} and~\ref{defn:EDP-tilting}. We call it a `formal theorem' because the specification of $\wt\calD^T$ is non-rigorous (see the discussion in Section~\ref{ss:formal-rigorous}) and because we do not provide a proof. 

\begin{Ftheorem}
Fix\/ $\calF(\rho) := \int_\nodes F(x)\rho(\dx x)$ for $F\in C_{\mathrm b}(\nodes)$. We have
\[
\calE_\e \stackrel\Gamma\longrightarrow \calE_0(\rho) := \RelEnt(\rho|\pi_0)
\qquad \text{and}\qquad
\calD_\e^T(\cdot;\calF) \stackrel{\Gamma} \longrightarrow
\calD_0^T(\cdot;\calF),
\]
where 
\[
\pi_0(\dx x) := \frac1{\PartSum_0} \ee^{-V(x)} \bONE_{\tilde\nodes}(x)\dx x.
\]
Defining $\wt\calE := \calE_0$, we have the inequality
\[
\calD_0^T(\rho_0,j_0;\calF) 
\geq 
\begin{cases}
\wt\calD^T(\wt\rho,\wt\jmath;\calF)  & \text{if $j_0$ and $\wt\jmath$ are related by~\eqref{eq:thin-membrane:relation-j-jtilde} and $\wt\rho := \rho_0 \lfloor\wt\nodes$
},\\
+\infty &\text{otherwise}.
\end{cases}
\]
Here $\wt\calD^T$ is defined as 
\begin{align}
\wt\calD^T(\rho,j;\calF) 
&= \int_0^T \pra*{\wt\calR(\rho_t,j_t;\calF) + \wt\calR^*\bra[\big]{\rho_t,-\anabla( \rmD\wt\calE(\rho_t)+ F);\calF}}\dx t,\notag\\
\noalign{\noindent where for $\Xi = (\Xi_-, \Xi_\sfm,\Xi_+)\in \bra[\big]{C_{\mathrm b}((-1,0)),\R, C_{\mathrm b}((1,2))}$,}
\wt\calR^*(\rho,\Xi;\calF) 
&= \frac12 \int_{-1}^0 a_-(x)\abs{\Xi_-(x)}^2 \rho(\dx x) 
   + \frac12 \int_1^2 a_+(x)\abs{\Xi_+(x)}^2 \rho(\dx x) 
   + \sigma(\rho,\calF) \sfC^*\bra*{\Xi_\sfm},\notag \\
\sigma(\rho,\calF) &= \bra*{\int_0^1 \frac1{a_*(s)}\ee^{V_*(s)}
  \exp\tfrac12 \bra[\big]{2F(s) - F(0) - F(1)}\dx s}^{-1}
  \sqrt{\frac{\dx\rho}{\dx \pi}(0^-)\frac{\dx\rho}{\dx\pi}(1^+)}\ .
  \label{eq:def:thin-membrane:sigma}
\end{align}
\end{Ftheorem}

The first two terms in $\wt\calR^*$ are the natural limits of the corresponding terms in~$\calR_\e^*$; the final term in~$\wt\calR^*$ represents the transmission across the membrane. 

\begin{remark}[Tilt-dependence]	
The dependence of $\sigma$ above on~$\calF$ implies that the limit system is tilt-dependent; see Section~\ref{sss:structure-of-tilt-dependence} for a discussion of the form of $\sigma$ and its consequences.

Note that there is a choice as to \emph{when} to tilt the system: this could be either before or after applying the transformation~\eqref{eqdef:trans-x-xhat}. Above we applied the tilt to the system \emph{after} transformation, which allows the tilt to depend on the `microscopic' variable in the interval $[0,1]$. On the other hand, when applying the tilt \emph{before} the transformation, a condition of continuity on the tilt at $\hat x=0$ implies that the tilt function is asymptotically constant on the $x$--interval $[0,1]$ after transformation. 
\end{remark}

\begin{remark}[Well-definedness of traces of $\rho$]
The definition of $\sigma$ in~\eqref{eq:def:thin-membrane:sigma} requires the traces of $\dx\rho/\dx\pi$ at $x=0,1$, which need not be well defined for an arbitrary measure $\rho$. In the context of the functional $\wt\calD^T$, however, additional regularity is provided by the term involving~$\wt\calR^*$; this term contains for instance the integral
\[
2\int_0^T \int_{-1}^0 a_-(x) \abs*{\nabla_x \sqrt{\frac{\dx\rho}{\dx\pi}(x)}}^2\dxx xt,
\]
and which leads to a well-defined trace at $x=0$.
\end{remark}

\bigskip

For completeness we give  the limit gradient evolution for $\calF=0$. This evolution is characterized by the equations
\[
\partial_t \wt\rho + \wt\adiv \wt\jmath = 0 ,
\qquad
\wt \jmath \in \partial_2 \wt\calR^*\bra[\big]{\wt\rho,-\wt\anabla\rmD \wt\calE(\wt\rho);0}.
\]
The second condition above reduces to
\begin{align*}
\wt\jmath (t,\dx x) = \begin{cases}
- a_-(x) \pi(\dx x) \partial_x u(t,x) & \text{if }x\in [-1,0]\\
\sigma_0\, (u(t,0) - u(t,1)) & \text{if }x = \sfm\\
- a_+(x) \pi(\dx x) \partial_x u(t,x) & \text{if }x\in [1,2]\\
\end{cases}	
\end{align*}
where $u = \dx\wt\rho/{\dx \pi_0} $ and $\sigma_0 := \bra*{\int_0^1 \frac1{a_*(s)}\ee^{V(s)}\dx s}^{-1}$. This leads to a limiting equation defined by the weak form of the continuity equation~\eqref{eq:ct-eq-thin-membrane-weak-form}, or in strong form for $u$,
\begin{align*}
 & \pi_0\partial_t  u = \partial_{ x}\bra*{a_\pm\pi_0\, \partial_{ x} u }
   \ &&\text{for }x\in [-1,0] \cup [1,2],\\
 & \partial_{ x}  u =0 \qquad&&\text{at } x=-1,2,\\
 & a_-(0)\partial_x u(t,0)=  \sigma_0\, (u(t,1) - u(t,0)), \\
 & a_+(1)\partial_x u(t,1)=  \sigma_0\, (u(t,1) - u(t,0)). \\ 
 \end{align*}

\section{Tilting}
\label{s:tilting}

\emph{Tilting} can mean at least four different things, all of which are relevant to this paper.

\paragraph{A.~Tilting of energies.} 
We described in Section~\ref{sss:tilting:energies} the tilting of an energy by adding another energy: 
$\calE \xmapsto{\text{tilt by }\calF} \calE+\calF$.

\paragraph{A$\!^+$.~Tilting of gradient systems.} 
A particular instance of this is when $\calE$ drives a gradient system; 
then $(\calE,\calF) \xmapsto{\text{tilt by }\calF} (\calE+\calF,\calR(\cdot,\cdot;\calF))$, 
as in the `gradient systems with tilting' of Definition~\ref{def:GS-with-tilting}. 
Here the  dependence of $\calR(\cdot,\cdot;\calF)$ on $\calF$ needs to be specified.
	
\paragraph{B.~Tilting of a Markov process.} 
Markov processes can also be tilted; we discuss this in Section~\ref{ss:tilting:MCs} below.

\paragraph{C.~Tilting of \emph{sequences} of random variables or Markov processes with a large-deviation principle.} 
We discuss this in Section~\ref{ss:tilting-LDPs} below.

\bigskip\noindent
There are various connections between these concepts of tilting, and we now discuss these. 
As mentioned above, there is a choice in how tilting affects the dynamics, and we characterize for finite Markov chains 
all possible choices of detailed-balance jump kernels in Section~\ref{ss:mod-jump-rates}.
As an illustration, we bring the tilt-dependence of chemical reactions
from the Kramers limit problem of Section~\ref{s:Kramers} into this perspective in Section~\ref{ss:case-study-chem-reactions-tilt-dependent}. 
Finally, in Section~\ref{ss:char-tilt-indep-GS}, we characterize tilt-independent gradient structures, and describe how one can sometimes convert a tilt-dependent structure into a tilt-independent one (`de-tilting'). Finally we also connect these gradient structures with various finite-volume discretization schemes. 

%

\subsection{Tilting of Markov processes}\label{ss:tilting:MCs}

We discuss interpretation \textbf B above, the tilting of Markov processes.
Let $X_t$ be a  Markov process on a state space $\nodes$ that is reversible with respect to a stationary measure $\pi$. 
For each $t>0$ we write the law of $X$ on $[0,t]$ as $\bbP_{[0,t]}$, which is a measure on an appropriate set of curves. 
Let $\calF = (\calF_t)_t$ be the filtration generated by~$X$, and write $\Generator$ for the generator of~$X$, 
with domain~$\dom \Generator$; we assume that~$X$ solves the martingale problem for $\Generator$. The reversibility of $X$ with respect to $\pi$ is equivalent to the property that $\Generator$ is self-adjoint in $L^2(\nodes,\pi)$.

A common way of modifying (`tilting') the process $X$ is to modify the law $\bbP_{[0,t]}$ by setting
\[
\wt \bbP_{[0,t]} := M_t \bbP_{[0,t]} \qquad\text{for each }t>0,
\]
where $M$ is a mean-one positive Markovian martingale with respect to $\calF$ and $\bbP_{[0,t]}$. 
The fact that $M$ has mean one is necessary for $\wt\bbP_{[0,t]}$ to be a probability measure, 
the fact that $M$ is adapted makes the modified process $\wt X$ adapted, 
and the fact that $M$ is Markovian ensures that also  $\wt X$ is a Markov process. 
Such a martingale $M$ can be built from functions $F:\nodes\to \R$ with $\ee^{-F/2}\in \dom \Generator$ by setting
\begin{equation}
\label{eqdef:M}
M_t^F := \exp\Bigl\{ -\tfrac 12 F(X_t) + \tfrac12 F(X_0) - \int_0^t \sfH(X_s,F)\dx s\Bigr\},
\end{equation}
where $\sfH$ is the \emph{Fleming-Sheu logarithmic transform}~\cite{Fleming82,Sheu85}
\[
\sfH(x,F) := \ee^{\frac12 F(x)}(\Generator \ee^{-\frac12 F})(x).
\]
If $F$ and $\sfH(\cdot,F)$ are bounded, then this process $M^F$ is a  martingale. 
In that case we can define $X^F$ to be the process with law $M_t^F \bbP_{[0,t]}$ for each $t$. 

The form of $M$ gives an indication of how this tilting changes $X$: 
multiplication with $M$ gives states $x$ with high value of $F(x)$ a lower probability of being visited. 
The following theorem states some properties of the process $X^F$; 
see~\cite{Feng99,ChetriteTouchette15} for similar results. 
\begin{theorem}
Consider the process $X$ described above. 
Assume that  $M^F$ is  a martingale, and  define $X^F$ to be the process with law $M_t^F \bbP_{[0,t]}$ for each $t$. The process $X^F$ then has the following properties.
\begin{enumerate}
\item The process $X^F$  is reversible with respect to the stationary measure $\ee^{-F}\pi$.
\item Assume that $\dom \Generator$ is closed under multiplication.
	The generator~$\Generator^F$ of~$X^F$ then satisfies $\dom \Generator^F = \dom \Generator$ and for any $\varphi\in \dom \Generator$ we have
\begin{equation}\label{eqdef:FlemingSheu}
\Generator^F \varphi = \ee^{F/2}\Generator(\ee^{-F/2} \varphi) - \ee^{F/2} \varphi \Generator\ee^{-F/2}.
\end{equation}
\end{enumerate}
\end{theorem}
\begin{proof}
To calculate the generator of $X^F$, we set 
\[
Y_t := \tfrac12 F(X_0) - \int_0^t \sfH(X_s,F)\dx s
\]
and note that $Y$ is differentiable in time. 
Fix $\varphi\in \dom\Generator$ and use the product rule for differentiation to expand
\begin{align*}
(\Generator^F\varphi)(x) 
&= \frac {\dx{}}{\dx t}\Big|_{t=0} \Expectation_x \varphi(X^F_t)
=\frac {\dx{}}{\dx t}\Big|_{t=0} \Expectation_x\pra*{\ee^{- F(X_t)/2}\ee^{Y_t} \varphi(X_t)}\\
&=\ee^{Y_0} \frac {\dx{}}{\dx t}\Big|_{t=0}\Expectation_x\pra*{\ee^{- F(X_t)/2} \varphi(X_t)}
+ \ee^{-F(x)/2} \varphi(x)\frac {\dx{}}{\dx t}\Big|_{t=0} \Expectation_x{\ee^{Y_t}}\\
&= \ee^{F(x)/2} (\Generator \ee^{-F/2}\varphi)(x) 
  - \ee^{-F(x)/2} \varphi(x) \ee^{Y_0}\sfH(x,F),
\end{align*}
which coincides with~\eqref{eqdef:FlemingSheu}.

\NEW{By the symmetry of $\Generator$ in $L^2(\nodes,\pi)$,} we similarly calculate for $\varphi,\psi\in\dom \Generator$
\begin{align*}
(\varphi,\Generator^F\psi)_{\ee^{-F}\pi} 
&= \int_\nodes \varphi \pra*{\ee^{F/2}\Generator(\ee^{-F/2} \psi) - \ee^{F/2} \psi \Generator\ee^{-F/2}}\ee^{-F}\pi\\
&= \int_\nodes \pra*{\ee^{-F/2}\varphi \Generator(\ee^{-F/2} \psi) - \ee^{F/2} \varphi \psi \Generator\ee^{-F/2}}\pi,
\end{align*}
which is symmetric in $\varphi$ and $\psi$, implying that $\Generator^F$ is reversible with respect to $\ee^{-F}\pi$.
\end{proof}

\noindent 
\textbf{Example~\ref{ex:heat-flow} (continued).} 	
The application of the tilting~\eqref{eqdef:M} 
to a jump process $X$ on a finite set $\nodes$
with jump rates~$\kappa$ and invariant measure~$\pi$
leads to a tilted process $X^F$ with jump rates 
\begin{equation}
\label{eq:tilting-MCs-transformed-kappa-F}
\kappa^F_{\sfx\sfy} := \ee^{\frac12(F_\sfx-F_\sfy)} \kappa_{\sfx\sfy},
\end{equation}
and invariant measure
\[
\pi^F := \frac1 {\PartSum^F}\ee^{-F} \pi, \qquad \PartSum^F := \sum_{\sfx\in \nodes} \ee^{-F_\sfx}\pi_\sfx.
\]
This implies that if equation~\eqref{eq:heat-flow-intro} is generated 
by a gradient system $(\nodes, \edges,\ona, \calE,\calR)$ as in Example~\ref{ex:heat-flow}, 
then the act of tilting replaces $\kappa$ in~\eqref{eq:heat-flow-intro} 
by $\kappa^F$ in~\eqref{eq:tilting-MCs-transformed-kappa-F}, 
and this tilted equation is generated by a gradient system $(\nodes, \edges,\ona, \calE+\calF,\calR)$ 
where $\calF(\rho) := \dual\rho F$. 

In this case, the dissipation potential $\calR$ is the same in the original and in the tilted systems; 
therefore this is an example of a tilt-independent gradient system, 
in the sense of Definition~\ref{def:GS-with-tilting}. 

However, the transformation \eqref{eq:tilting-MCs-transformed-kappa-F} is not the only way 
in which one can change the jump rates $\kappa$ in response to tilting, 
even if one aims to preserve detailed balance and the form $\ee^{-F}\pi$ of the stationary measure. 
We explore this freedom of choice in Section~\ref{ss:mod-jump-rates} below.

\subsection{Tilting of sequences of random variables}
\label{ss:tilting-LDPs}

We now continue with interpretation \textbf C of tilting, the tilting of \emph{sequences} of random variables and processes. 

\paragraph{Without dynamics: Energies from large deviations of sequences.} 
To start with, we disregard dynamics, and we let $X^n$ be a sequence of random variables in $\nodes$ with law $\mathbb P^n\in \ProbMeas(\nodes)$. 
We assume that~$\calE$ characterizes the large deviations of $X^n$ in the limit $n\to\infty$, i.e. 
\[
\Prob (X^n \approx x) \sim \ee^{-n\calE(x)}\qquad\text{as }n\to\infty.
\]
We now tilt $X^n$ by $n\calF$, i.e.\ we define the tilted  sequence $X^{n,\calF}$ by
\[
\Prob(X^{n,\calF} \in A) = \frac1{\PartSum^{n,\calF}} \int_A \ee^{-n\calF(x)}\, \mathbb P_n(dx),
\]
or equivalently $\law X^{n,\calF}  = (\PartSum^{n,\calF})^{-1} \ee^{-n\calF} \law X^n$. 
By Varadhan's and Bryc's lemmas (see e.g.~\cite[Th.~III.17]{DenHollander00} or~\cite[\S 4.3--4.4]{DemboZeitouni98}) 
the tilted sequence~$X^{n,\calF}$ then satisfies a large-deviation principle with rate function $\calE+\calF +\text{constant}$, i.e.
\begin{equation}
	\label{ldp:XF}	
	\Prob (X^{n,\calF} \approx x) \sim \ee^{-n(\calE(x)+\calF(x) + c)}\qquad\text{as }n\to\infty.
\end{equation}
This shows that the exponential tilting with $\ee^{-n\calF}$ leads to a modification of $\calE$ to $\calE+\calF$. 
In other words, interpretations \textbf{C} (tilting of LDPs) and \textbf{A} (tilting of energies) agree with each other, 
if one interprets the large-deviation rate function as an energy. 


\paragraph{With dynamics: Gradient systems from large deviations of Markov processes.}
We now turn to sequences of reversible Markov processes with a large-deviation principle and
how these generate gradient systems.
We now describe a formal proof of this type of result, 
using the example in Section~\ref{sss:cosh-from-ldp-intro} as an illustration.

\medskip
\emph{General structure.} 
Let $Z^n\in \sfZ$ be the Markov process in Theorem~\ref{ftheorem:LDP} for which we want to calculate the large deviations. 
The large deviations of~$Z^n$ can formally be derived in the following way~\cite{FengKurtz06}. 
Let $\Generator^n$ be the generator of $Z^n$.  
For a function $f = f(z)$, define the `nonlinear generator' $\scrH^n$ by\footnote{\NEW{This operator is the generator in the sense of nonlinear semigroup theory~\cite{Miyadera1992} of the nonlinear semigroup 
\[
(V^n_tf)(z) := \frac1n \log \Expectation \pra*{\ee^{nf(Z^n_t)}\Big| Z^n_0 = z}.
\]
As explained by Feng and Kurtz~\cite{FengKurtz06}, classical large-deviation results by Varadhan and Bryc relate the large deviations of a random variable $X^n$ to expressions of the form $\frac1n \log \Expectation \pra*{\ee^{nf(X^n)}}$. The  semigroup above is a generalisation  to the case of Markov processes. In both cases the expectation under tilting at speed~$n$ allows one to extract the large-deviation behaviour.
}}
\begin{equation}
\label{eqdef:Hamiltonian-from-generator}
(\scrH^n f)(z) := \frac1n \ee^{-nf(z)} (\Generator^n \ee^{nf})(z)
\end{equation}
Note how this is a nonlinear, $n$-scaled exponential form of tilting. 
A crucial step then is to show that
\begin{equation}
\label{conv:Hamiltonians}
\scrH^n \xrightarrow{n\to\infty}  \scrH \qquad\text{in an appropriate manner (see below)}.
\end{equation}
If in addition the limit $\scrH$ has the particular structure
\[
(\scrH f)(z) = H(z, \rmD f(z)) \qquad\text{for some }H,
\]
then $Z^n$ satisfies a large-deviation principle with rate function $\scrI$ characterized by $H$, 
\[
\scrI(z) := \sup_{\xi} \set[\Big]{\int_0^T \pra*{\Dual {\dot \xi} z - H(z,\xi)}\dx t\;\Big|\; \xi\in C^1([0,T]\times \sfZ)}.
\]

\medskip
\emph{Illustration.} 
For the example of Section~\ref{sss:cosh-from-ldp-intro}, 
$Z^n = \rho^n$ is the empirical measure of $n$ independent copies 
of a sequence of reversible Markov processes $X^n$ on a state space~$\nodes$. 
For this example we can characterize the generators $\Generator^n$ and $\scrH^n$ 
for the functions $f$ of the form  $f(\rho) := \dual F\rho$, for $F:\nodes\to\R$, by noting that $\rmD f = F$ as follows:
\begin{align}
(\Generator^n f)(\rho) &= \sum_{\sfx\sfy\in \edges} \rho_\sfx \kappa_{\sfx\sfy} \bra*{F_\sfy - F_\sfx},\notag\\
(\scrH^n f)(\rho) &= \sum_{\sfx\sfy\in \edges} \rho_\sfx \kappa_{\sfx\sfy} \bra*{\ee^{F_\sfy - F_\sfx}-1},
\label{eqdef:scrHn-ex-ldp-Sanov}
\end{align}
The convergence~\eqref{conv:Hamiltonians} is trivial for this example, since  $\scrH^n$ is independent of $n$.

For this example, and in fact for all the examples of this paper, 
the continuity-equation structure~\eqref{eq:ct-eq-abstract-intro} leads to a setup,
where $\scrH^n$ and the limit $\scrH$ have slightly more structure in comparison to the abstract formulation above, by
containing an additional gradient:
\[
(\scrH f)(\rho) = H(\rho, \anabla \rmD f(\rho)), \qquad\text{with }
H(\rho,\Xi) = \sum_{\sfx\sfy\in \edges} \rho_\sfx \kappa_{\sfx\sfy} \bra*{\ee^{\Xi_{\sfx\sfy}}-1} .
\]
In addition, $\rho^n$ satisfies a large-deviation principle with rate function $\scrI$ given by
\begin{align*}
\scrI(\rho) &:= \inf_j \set[\Big] {\RateFunc(\rho,j)\;\Big| \; \partial_t \rho + \adiv j = 0}
\qquad \text{with}\\
\RateFunc(\rho,j) &= \sup_\Xi \set[\Big] {\int_0^T \bigl[\Dual {j_t}{\Xi_t}  - H(\rho_t,\Xi_t)\bigr]\, dt \;\Big|\; \Xi:[0,T]\times \edges \to \R}.
\end{align*}
See~\cite[\S 3]{Renger18} for details of this large-deviation principle, 
including the ``level-2.5 large deviations'' for which $\RateFunc$ is the rate function.

\medskip
\emph{Construction of the gradient system.} 
From the Hamiltonian $H$ a gradient system can be derived as follows~\cite{MielkePeletierRenger14} (see also Section~\ref{ss:ldp-GS}). 
The stochastic reversibility implies that $\xi\mapsto H(z,\xi)$ is minimized at $\xi = \tfrac12 \rmD \scrE(z)$, 
where $\scrE$ is the rate function of the invariant measure. 
For the continuity-equation case we have that $\Xi\mapsto H(\rho,\Xi)$ is minimized at $\Xi = \tfrac12 \anabla \rmD \scrE(\rho)$. 
The gradient system then is driven by the functional $\tfrac12 \scrE$, 
and the dissipation potential $\scrR^*$ or $\calR^*$ can be constructed from $H$ and $\tfrac12\scrE$ by
\begin{subequations}
\label{eqdef:Rstar-from-H-general}
\begin{align}
\scrR^*(\rho,\xi) &=  2\bra*{ H\bra[\big]{\rho,\tfrac12(\xi+ \rmD \scrE(\rho))}
     -  H\bra[\big]{\rho,\tfrac12\rmD \scrE(\rho)}} &&\text{(general case)},
     \\
\label{eqdef:Rstar-from-H}	
\calR^*(\rho,\Xi) &=  2\bra*{  H\bra[\big]{\rho,\tfrac12(\Xi + \anabla \rmD \scrE(\rho))}
     -  H\bra[\big]{\rho,\tfrac12\anabla \rmD \scrE(\rho)}} &&\text{(continuity-equation case)}.
\end{align}
\end{subequations}
We already commented on the relation between $\scrR^*$ and $\calR^*$ in Remark~\ref{rem:non-ct-eq-GS}.

For the continuity-equation case it can formally be verified that the functional $\RateFunc(\rho,j)$ is of the form 
\[
\RateFunc(\rho,j)
= \frac12 \scrE(\rho(T))-\frac12 \scrE(\rho(0)) 
+ \frac12 \int_0^T \pra*{\calR(\rho,j) + \calR^*\bra*{\rho,-2\anabla \rmD\tfrac12\scrE(\rho)}}\dx t.
\]
This identifies $\RateFunc$ as the EDP functional $\calI^T$ 
that defines the gradient structure $(\nodes,\edges,\anabla, \scrE, \calR)$ (see Definition~\ref{def:EDP:sol},  and see Remark~\ref{rem:scaling-energy-in-GS} for the factors $2$).

\begin{remark}[The cosh structure and exponential tilting]
The construction of $\calR^*$ from $H$ in~\eqref{eqdef:Rstar-from-H-general} 
gives yet another way of understanding the \emph{exponential} nature of the cosh structure: 
it arises from the exponential tilting~\eqref{eqdef:Hamiltonian-from-generator} 
that underlies the theory of large deviations, 
and which is evident in formulas for the Hamiltonian $H$ such as~\eqref{eqdef:scrHn-ex-ldp-Sanov}.
\end{remark}

\paragraph{Interplay with tilting.}
In some cases the induced gradient system automatically is tilt-independent.
To illustrate this, we repeat the construction above for a tilted version of the process. 
Fix a tilt $\calF:\sfZ\to\R$, and define as in Section~\ref{ss:tilting:MCs} the process $Z^{n,\calF}$ 
via the Fleming-Sheu transformation of the generator $\Generator^n$ in~\eqref{eqdef:FlemingSheu},
\[
\Generator^{n,\calF} f(z) := \ee^{n\calF(z)/2} \Generator^{n} \bra*{ \ee^{-n\calF/2}f}(z)
- \ee^{n\calF(z)/2} f(z) (\Generator^{n} \ee^{-n\calF/2})(z) .
\]
The calculation
\begin{align}
(\scrH^{n,\calF} f)(z) &:= \frac1n \ee^{-nf(z)} (\Generator^{n,\calF} \ee^{nf})(z)\notag\\
&= \frac1n \ee^{-nf(z)+n\calF(z)/2} (\Generator^{n} \ee^{nf-n\calF/2})(z)
  - \frac1n \ee^{n\calF(z)/2} (\Generator^{n} \ee^{-n\calF/2})(z)\notag\\
&= \scrH^n\bra*{f-\tfrac12\calF}(z)
   -\scrH^n\bra*{-\tfrac12\calF}(z)\label{conv:tilted-Hamiltonians-prelimit}\\
&\longrightarrow \scrH\bra*{f-\tfrac12\calF}(z)
   -\scrH\bra*{-\tfrac12\calF}(z)
   \label{conv:tilted-Hamiltonians}\\
&= {H\bra*{z, \rmD f(z)-\tfrac12 \rmD \calF(z)}}
- {H\bra*{z,-\tfrac12 \rmD\calF(z)}},\notag
\end{align}
leads to the definition of a tilted Hamiltonian given by
\[
H^\calF(z,\xi) := H\bra*{z,\xi-\tfrac12\rmD\calF(z)}
- H\bra*{z,-\tfrac12\rmD\calF(z)}.
\]
By~\eqref{ldp:XF}, the tilted invariant measures for $Z^{n,\calF}$ satisfy a large-deviation principle 
with rate function $\scrE^\calF := \scrE + \calF + \text{constant}$, 
and therefore the Hamiltonian $H^\calF$ is minimized at $\tfrac12 \rmD \scrE^\calF(z)$.
Following the prescription~\eqref{eqdef:Rstar-from-H} we then calculate
\begin{align*}
\frac12 \calR^*(z,\xi; \calF) 
&= H^\calF\bra*{z, \tfrac12 \bra*{\xi + \rmD \scrE^\calF(z)}}
    - H^\calF\bra*{z,\tfrac12  \rmD \scrE^\calF(z)} \\
&= H\bra*{z, \tfrac12 \bra*{\xi + \rmD \scrE^\calF(z)} - \tfrac12\rmD\calF(z)}
     - H\bra*{z,- \tfrac12  \rmD\calF(z)}\\
&\qquad - H\bra*{z,\tfrac12  \rmD \scrE^\calF(z) - \tfrac12 \rmD\calF(z)} 
        + H\bra*{z,- \tfrac12  \rmD\calF(z)} \\
&=  H\bra*{z, \tfrac12 \bra*{\xi + \rmD \scrE(z)}}
     - H\bra*{z,\tfrac12  \rmD \scrE(z)} \\
&= \frac12 \calR^*(z,\xi; 0) ,
\end{align*}
which shows, formally, that $\calR^*$ is independent of $\calF$.

\bigskip

However, many large-deviation rate functions of this type are \emph{not} tilt-independent; 
the Kramers limit of Theorem~\ref{t:lower-bound} again is an example. 
This apparent contradiction comes from the assumption of the convergence~\eqref{conv:tilted-Hamiltonians}, 
which in turn is an instance of the convergence in~\eqref{conv:Hamiltonians}.

To understand the problem, note that the definition of this convergence ``$\scrH^n\to\scrH$'' is 
that for each $f\in \dom\scrH$ there exists a sequence $f_n\in \dom\scrH^n$ such that (see e.g.~\cite[Ch.~6]{FengKurtz06})
\[
f_n \to f \qquad \text{and}\qquad \scrH^nf_n \to \scrH f.
\]
In order to deduce from $\scrH^n\to\scrH$ that $\scrH^{n,\calF}\to\scrH^\calF$, we can proceed as follows. 
Assuming for simplicity that $\dom\scrH^\calF = \dom\scrH$, fix $f\in \dom\scrH$, 
and choose $f_n\to f$ such that $\scrH^n f_n \to \scrH f$. 
If the function $\calF$ is sufficiently `nice', 
it is not unreasonable that then also $\scrH^n (f_n -\tfrac12  \calF)\to \scrH (f-\tfrac12  \calF)$.

However, in order to deduce $\scrH^{n,\calF}\to\scrH^\calF$, i.e.\ the convergence~\eqref{conv:tilted-Hamiltonians}, 
we also need that the second term in~\eqref{conv:tilted-Hamiltonians-prelimit} converges, 
\[
\scrH^n\bra*{-\tfrac12 \calF} \to \scrH\bra*{-\tfrac12 \calF}.
\]
This property requires that $\scrH^n$ converges \emph{pointwise} at $-\calF/2$, which is a far  stronger requirement.
Indeed, especially for situations in which small-scale oscillations play a role, such as in homogenization, 
pointwise convergence does not hold, even at functions $-\calF/2$ that are very smooth.

\NEW{As an example, consider the Kramers case of Section~\ref{s:Kramers}. Here the Hamiltonian operator~$\scrH^\e$ has the form (see~\cite[(1.49)]{FengKurtz06} for the calculation)
\begin{multline*}
\scrH^\e \calF(\rho) =\int_\nodes \pra*{
  m_\Omega \Delta_x F^\rho + \tau_\e \bra*{ \Delta_y F^\rho - \frac1\e \nabla_y F^\rho \nabla_y H + \frac12 |\nabla_y F^\rho |^2}} \rho(\dxx xy), \\
\qquad\text{with }
F^\rho(x,y) = \rmD\calF(\rho)(x,y).
\end{multline*}
Since $\tau_\e\to \infty$, this expression diverges for any function $F^\rho$ with non-constant dependence on $y$.  This explains  how in the Kramers case the limit system only is tilt-independent on tilts that are constant in $y$; see also the discussion in Section~\ref{sss:philosopy-of-tilt-dependence}.
}

In conclusion, for `benign' situations in which Hamiltonians converge pointwise, 
the resulting large-deviation rate functions generate tilt-independent gradient systems. 
Without pointwise convergence, however, this may fail, as for instance in the examples of Sections~\ref{s:Kramers},~\ref{s:thin-membrane} and~\ref{s:two-terminal-reduction}.

\subsection{Characterization of detailed-balance jump kernels}
\label{ss:mod-jump-rates}

In Section~\ref{ss:tilting:MCs}, we observed that tilting a Markov process results 
in the particular modification~\eqref{eq:tilting-MCs-transformed-kappa-F} of the jump rates $\kappa$, 
but that also other modifications of $\kappa$ exist that preserve the detailed-balance property of $\kappa$. 
In this section we fully characterize this freedom of choice with Propositions~\ref{prop:Jump:tilt} and~\ref{prop:activity-tilt}.

To fix notation, choose a finite set $\nodes$, a fixed kernel $\kappa$, 
and a fixed measure $\pi$ satisfying the detailed-balance condition~\eqref{eq:def:graph:DBC},
\begin{equation}\label{eq:Jump:DBC}
	\pi_\sfx \kappa_{\sfx\sfy} = \pi_\sfy \kappa_{\sfy\sfx} . 
\end{equation}
To avoid degeneracies we assume that $\pi_\sfx>0$ for all $\sfx\in \nodes$. 
We also choose $\edges$ to be the support of $\kappa$,
\begin{equation*}
	\edges:=\set{\sfx\sfy\in \sfV\times \sfV: \kappa_{\sfx\sfy}>0},
\end{equation*}
which is symmetric ($\sfx\sfy\in \edges \Longleftrightarrow \sfy\sfx\in \edges$) by~\eqref{eq:Jump:DBC} and the assumption $\pi>0$.
Finally, again to avoid degeneracies we assume that $(\nodes,\edges)$ is a connected graph; 
then the measure $\pi$ is the unique equilibrium for the evolution~\eqref{eq:heat-flow-intro} from Example~\ref{ex:heat-flow}.

\medskip

We consider tiltings of both the equilibrium $\pi$ and the jump rate $\kappa$ that are described by a function $f:\nodes\to\R$. 
We say that $f$ induces an admissible tilt of the pair $(\pi,\kappa)$ leading to a new pair $(\pi^f$, $\kappa^f)$ 
provided the tilting satisfies the following two requirements:
\begin{subequations} 
\label{cond:DB-kernels-tilting}
\begin{alignat}2
\pi^f_\sfx  &:= \frac1{\PartSum^f} \ee^{-f_\sfx} \pi_\sfx, 
\qquad \PartSum^f := \sum_{\sfx\in \nodes} \ee^{-f_\sfx} \pi_\sfx 
	&\qquad& \text{(definition of $f$),}\label{eq:Jump:tilt:pi}\\
\pi^f_\sfx \kappa^f_{\sfx\sfy} &= \pi^f_\sfy \kappa^f_{\sfy\sfx}  
	  &\quad & \text{(detailed balance).}
\label{eq:Jump:F:DBC}
\end{alignat}
\end{subequations}
The first requirement above specifies the relationship between $f$ and $\pi$. 
In the case of `Boltzmann statistics' the function~$f$ has an interpretation as a potential energy (see Remark~\ref{rem:MC:general_tilts} below). 
The Markov-process tilting of Section~\ref{ss:tilting:MCs} also is of this form, 
but instead of fixing the form of $\kappa^f$ by choosing the Fleming-Sheu transform that leads to~\eqref{eq:tilting-MCs-transformed-kappa-F}, 
we now allow for any choice of $\kappa^f$ that satisfies~\eqref{cond:DB-kernels-tilting}. 
The second requirement states that the tilted system should also satisfy detailed balance. 

\medskip
The following proposition characterizes \emph{all} positive jump kernels $\kappa^f: \edges\to (0,\infty)$ satisfying these conditions.
\begin{prop}[Tilted jump kernels]\label{prop:Jump:tilt}
The pair $\pi^f\in \calM_{>0}(\nodes)$ and $\kappa^f:\edges \to (0,\infty)$ satisfies~\eqref{cond:DB-kernels-tilting}
if and only if there exists $\theta:\edges \times \R^2 \to (0,\infty)$, jointly symmetric, that is
\begin{equation}\label{eq:Jump:Phi:symmetry}
	\theta_{\sfx\sfy}(a,b)=\theta_{\sfy\sfx}(b,a) \qquad\text{for all } \sfx\sfy\in \edges \text{ and } (a,b)\in \R^2 ,
\end{equation}
such that 
\begin{equation}\label{e:Jump:Phi}
	\kappa^f_{\sfx\sfy} = \kappa_{\sfx\sfy} \ee^{f_\sfx} \theta_{\sfx\sfy}\bra*{\ee^{-f_\sfx},\ee^{-f_\sfy}}
	\qquad\text{for all } \sfx\sfy\in \edges . 
\end{equation}
\end{prop}
\begin{proof}
	First assume that $(\pi^f,\kappa^f)$ satisfy~\eqref{cond:DB-kernels-tilting}. 
	By positivity of $\pi^f$, we can assume that $\kappa^f$ is of the form~\eqref{e:Jump:Phi}
	for some $\theta:\edges \times \R^d\to \R$ not necessarily satisfying the symmetry relation~\eqref{eq:Jump:Phi:symmetry}. 
	Since $(\pi,\kappa)$ also satisfies the detailed balance condition~\eqref{eq:Jump:DBC}, we arrive at the identity
	\begin{align*}
		\pi_\sfx \kappa_{\sfx\sfy} \, \theta(\sfx\sfy;\ee^{-f_\sfx},\ee^{-f_\sfy})
		&= \PartSum^f \pi^f_\sfx \kappa_{\sfx\sfy} \, \theta(\sfx\sfy;\ee^{-f_\sfx},\ee^{-f_\sfy}) \ee^{f_\sfx} \\
		&= \PartSum^f \pi^f_\sfx \kappa^f_{\sfx\sfy} \\
		&= \PartSum^f \pi^f_\sfy \kappa^f_{\sfy\sfx} \\
		&= \PartSum^f \pi^f_\sfy \kappa_{\sfy\sfx} \, \theta(\sfy\sfx;\ee^{-f_\sfy},\ee^{-f_\sfx}) \ee^{f_\sfy} \\
		&= \pi_\sfy \kappa_{\sfy\sfx} \, \theta(\sfy\sfx;\ee^{-f_\sfy},\ee^{-f_\sfx}),
	\end{align*}
	which together with~\eqref{eq:Jump:DBC} shows that $\theta$ satisfies the symmetry condition~\eqref{eq:Jump:Phi:symmetry}.
	
	Similarly, if $\theta$ has the symmetry~\eqref{eq:Jump:Phi:symmetry} and $\kappa^f $ is defined by~\eqref{e:Jump:Phi},
	then we get by again using the detailed balance~\eqref{eq:Jump:DBC} of $(\pi,\kappa)$ the identity
	\begin{equation*}
		 \ee^{-f_\sfx}\pi_\sfx \,\kappa^f_{\sfx\sfy} 
		= \pi_\sfx \kappa_{\sfx\sfy} \, \theta(\sfx\sfy;f_\sfx,f_\sfy)
		=  \pi_\sfy \kappa_{\sfy\sfx} \, \theta(\sfy\sfx;f_\sfy,f_\sfx)
		= \ee^{-f_\sfy}\pi_\sfy\,\kappa^f_{\sfy\sfx} ,
	\end{equation*}
	concluding that $\kappa^f$ satisfies the detailed balance condition~\eqref{cond:DB-kernels-tilting} 
	with stationary measure $\pi^f := (\PartSum^f)^{-1} \ee^{-f}\pi $.
\end{proof}
This Proposition shows that for each pair of edges $\{\sfx\sfy,\sfy\sfx\}\subset \edges$ 
there is the freedom of a choice of a symmetric map $\R^2\to \R$. 
We now characterize this freedom in another way,
by using a discrete Hodge-like decomposition of the rates $\kappa: \edges \to (0,\infty)$ as 
\begin{equation}\label{eqdef:activity-entropy}
	\kappa_{\sfx\sfy} = a_{\sfx\sfy} \ee^{{s_{\sfx\sfy}}/{2}} \quad\text{with}\quad a_{\sfx\sfy} =a_{\sfy\sfx}:=\sqrt{\kappa_{\sfx\sfy}\kappa_{\sfy\sfx}} \quad\text{ and }\quad s_{\sfx\sfy} = - s_{\sfy\sfx}:=\log \frac{\kappa_{\sfx\sfy}}{\kappa_{\sfy\sfx}} .
\end{equation}
It is readily checked that $\kappa$ satisfies the detailed balance condition~\eqref{eq:Jump:DBC} 
if and only if $s_{\sfx\sfy} = V_\sfx-V_\sfy$ for some $V:\nodes \to \R$. 
In this case, we can also write
\begin{equation*}
	 s_{\sfx\sfy} = \log \frac{\pi_\sfx \kappa_{\sfx\sfy}\ \pi_{\sfy}}{\pi_\sfy \kappa_{\sfy\sfx}\ \pi_\sfx} 
	 = \log \frac{\pi_{\sfy}}{\pi_\sfx} = \log \pi_{\sfy} - \log \pi_{\sfx}.
\end{equation*}
Hence, the potential $V$ is given by $V_\sfx =-\log \pi_{\sfx} + \text{constant}$. 
In particular, the equilibrium state of a detailed-balance Markov chain is independent of the choice $a:\edges \to (0,\infty)$. 
From this observation the coefficient $a:\edges \to (0,\infty)$ is also called~\emph{dynamical activity} or \emph{time-symmetric traffic}; 
see the recent review by Maes~\cite{Maes2020} for discussion of this concept in non-equilibrium statistical mechanics. 

Next, we investigate the interplay of the tilting~\eqref{eq:Jump:tilt:pi} obtained 
in Proposition~\ref{prop:Jump:tilt} with the splitting~\eqref{eqdef:activity-entropy}. 
We summarize the statement in the following Proposition, which can be verified by direct calculations.
\begin{prop}[Tilting of dynamical activitiy]\label{prop:activity-tilt}
The kernel $\kappa^f:\edges \to (0,\infty)$ satisfying~\eqref{eq:Jump:F:DBC} and given by~\eqref{e:Jump:Phi} 
has the splitting $\kappa^f_{\sfx\sfy} = a^f_{\sfx\sfy} \ee^{{s^f_{\sfx\sfy}}/{2}}$ 
with $a^f:\edges \to (0,\infty)$ symmetric and $s^f:\edges \to (0,\infty)$ skew-symmetric given by
\begin{equation*}
	a^f_{\sfx\sfy} = a_{\sfx\sfy} \,\theta(\sfx\sfy;\ee^{-f_\sfx},\ee^{-f_\sfy}) \ee^{\frac{f_\sfx+f_\sfy}{2}}
	\qquad\text{and}\qquad s^f_{\sfx\sfy}  = s_{\sfx\sfy} + f_\sfx - f_\sfy . 
\end{equation*}
In particular, the only tilting of the jump kernel that leaves the activity invariant, i.e. $a^f=a$, is given by 
\begin{equation*}
	\theta(\sfx\sfy;\ee^{-f_\sfx},\ee^{-f_\sfy}) = \ee^{-\frac{f_\sfx+f_\sfy}{2}} . 
\end{equation*}
\end{prop}
For later reference, we give some examples of choice for $\theta$ in~\eqref{e:Jump:Phi}.
\begin{example}
\label{ex:particular-DB-parametrisations}
We restrict ourselves to the product setting, where $\theta(\sfx\sfy;a,b) := \omega_{\sfx\sfy}\vartheta(a,b)$ for some symmetric $\omega:\edges\to (0,\infty)$, which acts as a relative conductivity change of the edges.
In the following, we only specify the symmetric function $\vartheta: \R_{\geq 0}^2 \to \R_{\geq 0}$.
\begin{enumerate}
	\item \label{ex:tilt:chem} $\vartheta(a,b) \equiv 1$  leads to the rate kernel
	\begin{align*}
		\kappa^f_{\sfx\sfy} = \omega_{\sfx\sfy}\kappa_{\sfx\sfy} \ee^{f_\sfx} 
		&\qquad\text{and}\qquad a^f_{\sfx\sfy} = \omega_{\sfx\sfy} a_{\sfx\sfy} \ee^{\frac{f_\sfx+f_\sfy}{2}} . 		
	\end{align*}
	\item \label{ex:tilt:sym} $\vartheta(a,b) =  \sqrt{a b}$ leads to the kernel
	\begin{align*}
		\kappa^f_{\sfx\sfy} = \omega_{\sfx\sfy}\kappa_{\sfx\sfy} \ee^{\frac{1}{2}\bra*{f_\sfx-f_\sfy}}  &\qquad\text{and}\qquad a^f_{\sfx\sfy} = \omega_{\sfx\sfy} a_{\sfx\sfy} . 
	\end{align*}
	\item $\vartheta(a,b) = ab$ gives 
	\begin{align*}
		\kappa^f_{\sfx\sfy} = \omega_{\sfx\sfy}\kappa_{\sfx\sfy} \ee^{-f_\sfy} 
		&\qquad\text{and}\qquad a^f_{\sfx\sfy} = \omega_{\sfx\sfy} a_{\sfx\sfy} \ee^{-\frac{f_\sfx+f_\sfy}{2}} . 
	\end{align*}
	\item \label{ex:tilt:Metropolis}
	 The choice $\vartheta(a,b) = \min\set*{a,b}$ leads to the Metropolis-Hastings acceptance and rejection rates~\cite{Metropolis1953,Hastings1970,BinderLandau2015} given by
	\begin{align*}
		\kappa^f_{\sfx\sfy} = \omega_{\sfx\sfy}\kappa_{\sfx\sfy} \ee^{-(f_\sfy - f_\sfx)_+}
		&\qquad\text{and}\qquad
		a^f_{\sfx\sfy} =  \omega_{\sfx\sfy} a_{\sfx\sfy} \ee^{-\frac{\abs{f_\sfx - f_\sfy}}{2}} .
	\end{align*}
	\item 
	Examples~\ref{ex:tilt:sym} and~\ref{ex:tilt:Metropolis} are given in terms of a 
	symmetric positively one-homogeneous function $\vartheta:\R^2_{\geq 0}\to [0,\infty)$, for which we in general get
	\begin{align*}
		\kappa^f_{\sfx\sfy} = \omega_{\sfx\sfy}\kappa_{\sfx\sfy} \,\vartheta\bra*{1,\ee^{f_\sfx-f_\sfy}}
		&\qquad\text{and}\qquad
		a^f_{\sfx\sfy} =  \omega_{\sfx\sfy} a_{\sfx\sfy} \,\vartheta\bra*{\ee^{\frac{f_\sfy - f_\sfx}{2}},\ee^{\frac{f_\sfx - f_\sfy}{2}}} .
	\end{align*}
	It follows immediately from this structure that $\kappa^f$ and $a^f$ are invariant under adding a constant to the tilt: if $\alpha\in \R$ is a constant, then  $\kappa^{f+\alpha}=\kappa^f$ and $a^{f+\alpha}= a^f$.
\end{enumerate}	
These examples highlight the possible dependence of $\kappa^f$ and also the activity $a^f$ on the tilt~$f$.
Example~\ref{ex:tilt:sym} is the form of tilting implicitly selected by the Fleming-Sheu transform of Section~\ref{ss:tilting:MCs}.
		
Example~\ref{ex:tilt:chem} above has the property that the tilt $f$ 
only influences the rate of a jump $\sfx\to\sfy$ through the value of $f$ at the starting point $\sfx$. 
This is a natural property in the context of chemical reactions, 
because rates of elementary, single-direction reactions are assumed to depend on concentrations of the reacting species,
not on the concentrations of the resulting products. 
However, this asymmetric influence changes the overall traffic intensity $a$. 
In the next section we show how chemical-reaction modelling leads to a tilt-dependence of this type.
\end{example}

\subsection{Case study: Tilt-dependence in chemical reactions}
\label{ss:case-study-chem-reactions-tilt-dependent}

Chemical reactions are an example in which natural modelling assumptions lead to gradient systems that are tilt-dependent;
in this section we explore this.

For the single monomolecular reaction
\newcommand\rateab{\vec\jmath^{}_{\sfa\sfb}}
\newcommand\rateba{\vec\jmath^{}_{\sfb\sfa}}
\[
A\xrightleftharpoons[\rateba]{\rateab} B
\]
the general kinetic theory as summarized in Example~\ref{ex:chem-reactions} in Section~\ref{ss:flux-GS-intro} gives 
\begin{align*}
\rateab &= 	D\ee^{-\upbeta (E^{\mathrm{act}} - E_\sfa)} \rho_\sfa, \qquad
\rateba = 	D\ee^{-\upbeta (E^{\mathrm{act}} - E_\sfb)} \rho_\sfb,
\end{align*}
where $D$ and $E^{\mathrm{act}}$ are the rate constant $D_{\alpha\beta}$ 
and the activation energy $E^{\mathrm{act}}_{\alpha\beta}$ of Example~\ref{ex:chem-reactions}, 
and $E_\sfa$ and $E_\sfb$ are the energies of $A$ and $B$. 
In the simplest interpretation, $E^{\mathrm{act}}$ coincides with the energy level of the saddle, which we for now write as~$E_\sfc$.
Hence, in the notation of the previous section,
but using a single directed edge setup $(\nodes=\set{\sfa,\sfb},\edges=\set{\sfa\sfb}$; see Remark~\ref{rem:edges-conversion}) 
we have $\pi_{\sfa} = \ee^{-\upbeta E_{\sfa}}$, $\pi_{\sfb} = \ee^{-\upbeta E_{\sfb}}$, and
\begin{equation}\label{eqdef:kappa:chemical}
	\kappa_{\sfa\sfb} = D\ee^{-\upbeta (E_\sfc - E_\sfa)} \rho_\sfa = a_{\sfa\sfb} \ee^{{s_{\sfa\sfb}}/{2}}
	\quad\text{with}\quad 
	a_{\sfa\sfb} = D \ee^{\frac\upbeta2\bra*{E_{\sfa}+E_{\sfb}-2E_\sfc}}
	\quad\text{and}\quad
	s_{\sfa\sfb} = \upbeta\bra*{E_\sfa - E_\sfb} . 
\end{equation}
The corresponding equation for $\rho_\sfa$, the number of particles of species $A$, 
can then be rewritten as 
\begin{align*}
\partial_t \rho_a 
&= \rateba - \rateab
= D \ee^{-\upbeta E_\sfc}\bra*{\ee^{\upbeta E_\sfb}\rho_\sfb - \ee^{\upbeta E_\sfa}\rho_\sfa}\\
&= D \sqrt{\rho_\sfa\rho_\sfb} \ee^{\tfrac\upbeta2 (E_\sfa+E_\sfb-2E_\sfc)}
  \cdot 2\sinh\frac12 \bra*{\log \frac{\rho_\sfb}{\rho_\sfa} + \upbeta(E_\sfb-E_\sfa)}\\
&= D \sqrt{\frac{\rho_\sfa\rho_\sfb}{\pi_\sfa\pi_\sfb}} \ee^{-\upbeta E_\sfc}
  \cdot 2\sinh\frac12 \bra*{\log \frac{\rho_\sfb\pi_\sfa}{\pi_\sfb\rho_\sfa} }.
\end{align*}
From the chemical kinetics theory~\cite{Laidler1987,Connors1990,Coker2001}, 
it is natural to let tilting act on those energy levels $E_\sfa, E_\sfb$, and $E_\sfc$. 
In this case tilting by some $F:\set*{\sfa,\sfb,\sfc}\to \R$ corresponds to replacing
\[
E_\sfa \rightsquigarrow E_\sfa + F_\sfa, \qquad
E_\sfb \rightsquigarrow E_\sfb + F_\sfb, \qquad\text{and}\qquad
E_\sfc \rightsquigarrow E_\sfc + F_\sfc,
\]
where $F_\sfa$ and $F_\sfb$ are the energy perturbations (tilts) at the wells $\sfa$ and $\sfb$, and $F_\sfc$ is the perturbation at the saddle. 
The corresponding exponents are given by $f_\sfx =\upbeta F_\sfx$ for $\sfx = \sfa,\sfb,\sfc$.
This leads to the tilted equation
\begin{align*}
\partial_t \rho_a 
&= D \sqrt{\rho_\sfa\rho_\sfb} \ee^{\frac\upbeta2 (E_\sfa+E_\sfb-2E_\sfc)}
  \ee^{\frac\upbeta2 \bra*{F_\sfa+F_\sfb-2F_\sfc}}
  \cdot 2\sinh\frac12 \bra*{\log \frac{\rho_\sfb}{\rho_\sfa} + \upbeta(E_\sfb + F_\sfb -E_\sfa -F_\sfa)}\\
&= D \sqrt{\frac{\rho_\sfa\rho_\sfb}{\pi_\sfa\pi_\sfb}} \ee^{-\upbeta E_\sfc}
  \ \ee^{\frac\upbeta2 \bra*{F_\sfa+F_\sfb-2F_\sfc}}
  \cdot 2\sinh\frac12 \bra*{\log \frac{\rho_\sfb\pi_\sfa}{\pi_\sfb\rho_\sfa}+ \upbeta( F_\sfb -F_\sfa) }.
\end{align*}
This expression is the evolution equation for the gradient--system--with--tilting given by
\begin{align*}
\calE(\rho) &= \rho_\sfa\log \frac{\rho_\sfa}{\pi_\sfa} + \rho_\sfb\log \frac{\rho_\sfb}{\pi_\sfb}, 
\qquad \sfF := \set[\Big]{\calF(\rho) =  \upbeta(\rho_\sfa F_\sfa + \rho_\sfb F_\sfb)},\\
\calR^*(\rho,\Xi\,;\,\calF) &= \sigma(\rho,\calF)\, \sfC^*(\Xi)
\qquad \text{in which } 
\begin{aligned}[t]
\sigma(\rho,\calF) &:= D \sqrt{\rho_\sfa\rho_\sfb} \ee^{\frac\upbeta2 (E_\sfa+E_\sfb-2E_\sfc)}
  \ee^{\frac\upbeta2 \bra*{F_\sfa+F_\sfb-2F_\sfc}}
 \end{aligned}
\end{align*}
This exact form of dependence on the tilts $F$ is also seen in the Kramers high-activation-energy limit (see Theorem~\ref{t:lower-bound}). 
The change on the level of the activity $a\rightsquigarrow a^F$ corresponds to case~\ref{ex:tilt:chem} in Example~\ref{ex:particular-DB-parametrisations}
with the choice $\omega_{\sfa\sfb} = \ee^{-\upbeta F_c}$, and we also have by comparison with~\eqref{eqdef:kappa:chemical} the identities
\begin{equation*}
	\sigma(\rho,\calF) = \sqrt{\rho_\sfa \rho_\sfb} \, a^{F}_{\sfa\sfb} = \sqrt{\rho_\sfa \rho_\sfb} \,  a_{\sfa\sfb} \ee^{\frac\upbeta2 \bra*{F_\sfa+F_\sfb-2F_\sfc}}.
\end{equation*}

\subsection{Characterization of tilt-independent gradient structures}
\label{ss:char-tilt-indep-GS}

Certain evolution equations are known to have many different gradient structures, such as reaction-diffusion  equations of the form of  Example~\ref{ex:FP}+\ref{ex:chem-reactions} (see e.g.~\cite{Mielke11}) or the simple heat flow on a graph of Example~\ref{ex:heat-flow}~\cite{Mielke11,ErbarMaas2014,MielkeStephan2020}.

However, if we also require the gradient structure to be tilt-independent, then the range of gradient systems for a given evolution equation reduces drastically, and in some cases this leads to a unique characterization of such tilt-independent systems. 
This was first observed by Mielke and Stephan~\cite{MielkeStephan2020}, and we review their result in Section~\ref{sss:MS}.

In Section~\ref{sss:tilt-indepdent-GS-MC} we prove a related result (Theorem~\ref{th:MC:gradient:tilt}), using the  additive definition of tilting of this paper and allowing a more general class of dissipation potentials. 
This leads to a characterization of tilt-independent gradient structures that is significantly broader than just the ones that arise from tilting of Markov processes and large deviations in Sections~\ref{ss:tilting:MCs} and~\ref{ss:tilting-LDPs}.

In Section~\ref{sss:MC:quadratic:untilting} we revisit the `de-tilting' procedure of Remark~\ref{rem:unitlting:tilt-GS}, and in Section~\ref{sss:MC:tilt:SG-upwind} we connect the findings to finite-volume schemes.

\subsubsection{Tilt-independent gradient structures for finite Markov chains}\label{sss:tilt-indepdent-GS-MC}

We focus on Example~\ref{ex:heat-flow}; let $\nodes$ be a finite set, and let given jump rates $\kappa:\edges:=\nodes\times\nodes\to [0,\infty)$ satisfy the detailed-balance condition~\eqref{eq:def:graph:DBC}.  
The aim is to show that in a broad class of  gradient systems that all generate the equation~\eqref{eq:heat-flow-intro}, the requirement of tilt-independence singles out a single one.

\bigskip
\emph{Energies and their tilting.}
A natural choice for the class of driving functionals is the class of entropies of the form 
\begin{equation}\label{eqdef:H:Phi}
	\calH_\GenEnt(\rho | \pi) := \sum_{\sfx\in\nodes} \pi_x \GenEnt\bra*{\frac{\rho_\sfx}{\pi_\sfx}} , 
\end{equation}
where $\GenEnt$ satisfies
\begin{equation}\label{ass:MC:H}
\GenEnt\in C\bra*{[0,\infty);[0,\infty)}\cap C^2((0,\infty)), \text{ strictly convex and bounded from below.}
\end{equation}
Note that $\calH_\GenEnt$ is the `perspective version' of the convex energy $\calE_\GenEnt(\rho) := \sum_{\sfx\in \nodes} \GenEnt(\rho_\sfx)$ following Definition~\ref{defn:perspective-function}.  The standard relative entropy $\calH$ corresponds to $\sfH(s) = \eta(s|1) = s\log s - s + 1$.

We consider tilting by \emph{potential tilts} $\sfF_{\Pot}$ defined in~\eqref{eqdef:MCs:potential-tilts}, that is $\calF^F_{\Pot}\in \sfF_{\Pot}$ is of the form
\begin{equation*}
  	\calF^F_{\mathup{Pot}}(\rho) := \sum_{\sfx\in\nodes} \rho_\sfx F_\sfx \qquad\text{for some } F:\nodes \to \R. 
\end{equation*}
Following the energetic interpretation of tilting of Section~\ref{sss:tilting:energies}, we make the choice  that a tilt~$\calF$ changes the functional $\RelEnt_\GenEnt$ by \emph{addition}, i.e.
\begin{equation}
\label{eq:MS-type-result:def-of-tilting}
\calE := \RelEnt_\GenEnt(\cdot|\pi) \ \xmapsto{\text{tilt by }\calF}\  \calE^\calF := \RelEnt_\GenEnt(\cdot|\pi)+\calF.
\end{equation}
(See Section~\ref{sss:MS} for an alternative choice.) 
It follows that minimizers $\pi^F$ of $\calE^\calF$ are for a given potential-tilt $\calF^F_\Pot\in \sfF_{\Pot}$ of the form $v\pi$, where $v$ is characterized by 
\begin{equation}
\label{eq:MS-type-result:char-v-Phi}
\GenEnt'(v_\sfx) + F_\sfx = \lambda, \qquad\text{and $\lambda$ is a normalization constant.}
\end{equation}
In the Boltzmannian case $\GenEnt(s) = s\log s $ we recover the exponential characterization $v_\sfx = c\, \ee^{-F_\sfx}$, but other choices of $\GenEnt$ lead to different relations between $v$ and $F$. 

\bigskip

\emph{Dynamics.} 
We now turn to tilt-dependence and tilt-independence. Condition~\eqref{eq:MS-type-result:char-v-Phi} characterizes how a potential tilt $\calF(\cdot) = \dual F \cdot $ changes the stationary measure~$\pi$ to a new measure~$\pi^F$; Proposition~\ref{prop:Jump:tilt} shows that there remains a significant freedom in choosing jump rates $\kappa^F$ that fix $\pi^F$ and still satisfy detailed balance. 

By contrast, the following theorem shows that the only gradient systems in this class that both (a) are tilt-independent and (b) yield a flux that is linear in $\rho$ are the Boltzmann-cosh combinations described in Example~\ref{ex:heat-flow} in Section~\ref{ss:flux-GS-intro}.

\begin{theorem}
\label{th:MC:gradient:tilt}
Let the gradient--system--with--tilting $(\nodes,\edges,\ona,\calE,\sfF_{\mathrm{pot}},\calR)$ satisfy
\begin{enumerate}
	\item $\calE(\rho) = \RelEnt_\GenEnt(\rho|\pi)$ defined in~\eqref{eqdef:H:Phi} with fixed $\GenEnt$ satisfying~\eqref{ass:MC:H} and $\pi\in \ProbMeas^+(\nodes)$;
	\item\label{l:th:MC:convex} $\calR(\rho,j;\calF) = \calR_{\Redge}(\rho,j)$ is tilt-independent and has the form 
	\begin{equation}\label{eqdef:R:psi}
			\calR_{\Redge}(\rho,j) := \sum_{\sfx\sfy\in \edges} \Redge_{\sfx\sfy}\bra*{\rho;j_{\sfx\sfy}},
	\end{equation}
	where $\Redge_{\sfx\sfy}: \calM_{\geq0}(\nodes) \times \R \to [0,\infty)$ satisfies
	\begin{enumerate}
	\item 	$ \Redge_{\sfx\sfy}(\rho;\cdot)$ is differentiable and convex;
	\item  $\Redge_{\sfx\sfy}(\rho;0)=0$ for all $\rho\in \calM_{\geq0}(V)$ and $\sfx\sfy\in\edges$;
	\item $\Redge_{\sfx\sfy}^*(\rho;\cdot)$ is differentiable.
	\end{enumerate}
	\item For each $\calF^F_\Pot = \dual F \cdot \in \sfF_{\mathrm{pot}}$, 
	the induced flux $j_{\sfx\sfy}^F$,
	\begin{equation*}
		j_{\sfx\sfy}^F :=\rmD_2 \Redge_{\sfx\sfy}\bra*{\rho; -\ona_{\sfx\sfy} \rm D \bra*{ \calE + \calF^F_\Pot}},
	\end{equation*}
	can be expressed as 
	\begin{equation}\label{eq:MC:gradient:tilt:flux}
		j_{\sfx\sfy}^F = \frac{1}{2}(\rho_\sfx \kappa_{\sfx\sfy}^F - \rho_{\sfy}\kappa_{\sfy\sfx}^F),
	\end{equation}
	where the tilted rates  $\kappa^F$ satisfy detailed balance with respect to some stationary measure $\pi^F\in \ProbMeas^+(\nodes)$.
	
	Assume that $\{\sfx\sfy: \kappa^F_{\sfx\sfy}>0\}$ does not depend on $F$; we set $\kappa:= \kappa^0$ (i.e., $\kappa^F$ for $F=0$).
\end{enumerate}
Then there exist $\gamma>0$, $c_1,c_2\in \R$
and a jointly symmetric function $\theta: \edges \times \R_{\geq0}^2 \to \R_{\geq0}$ (that is~\eqref{eq:Jump:Phi:symmetry}) 
with $\theta_{\sfx\sfy}(\cdot,\cdot)$ being non-decreasing and one-homogeneous,
such that for every $s\geq0$, $\Xi\in \R$, and $\sfx\sfy\in \edges$,
\begin{equation}
\label{eq:th:MS-type-result}
\GenEnt(s) = \gamma s\log s + c_1 s + c_2,
\qquad
\rmD_2 \Redge^*_{\sfx\sfy}(\rho;\Xi) =  \frac{\pi_\sfx\kappa_{\sfx\sfy}}2 \theta_{\sfx\sfy}\bra*{\frac{\rho_\sfx}{\pi_\sfx} \ee^{-\Xi/2\gamma},\frac{\rho_\sfy}{\pi_\sfy} \ee^{\Xi/2\gamma}} \,(\sfC^*)'\bra*{\frac{\Xi}{\gamma}}.
\end{equation}
\end{theorem}
The properties~\eqref{eq:th:MS-type-result} represent a family of gradient systems indexed by parameters $\gamma, c_1$, and $c_2$. 
The affine component $s\mapsto c_1s+c_2$ in $\GenEnt$ is an affine transformation of $\calE$, and does not affect the evolution.
The parameter $\gamma$ characterizes the standard rescaling freedom in gradient systems discussed in Remark~\ref{rem:scaling-energy-in-GS}, and can be also interpreted as a viscosity parameter (see Section~\ref{sss:MC:tilt:SG-upwind}).

\begin{proof}
Proposition~\ref{prop:Jump:tilt} characterizes~$\pi^F$ and~$\kappa^F$ in terms of a function $f:\nodes\to\R$ and a jointly symmetric function $\theta:\edges\times \R\times \R\to\R$ as 
\begin{equation*}
\pi^F_\sfx  := \ee^{-f_\sfx}\pi_\sfx
\qquad\text{and}\qquad
\kappa^F_{\sfx\sfy} = \kappa_{\sfx\sfy} \ee^{f_\sfx} \theta_{\sfx\sfy}\bra*{\ee^{-f_\sfx},\ee^{-f_\sfy}} . 
\end{equation*}
Consequently the flux $j$ in~\eqref{eq:MC:gradient:tilt:flux} with jump rates $\kappa^F$ can be written as
\begin{equation}
\label{eq:MS-type-result:char-j-1}
j_{\sfx\sfy} =  \frac12 \theta_{\sfx\sfy}\bra*{\ee^{-f_\sfx},\ee^{-f_\sfy}}
  \pra*{\ee^{f_\sfx}u_\sfx \pi_\sfx \kappa_{\sfx\sfy}- \ee^{f_\sfy}u_\sfy\pi_\sfy \kappa_{\sfy\sfx}}
  = \frac{k_{\sfx\sfy}}2	\theta_{\sfx\sfy}(v_\sfx,v_\sfy)  \pra*{\frac{u_\sfx}{v_\sfx} - \frac{u_\sfy}{v_\sfy}},
\end{equation}
where we write
\[
v_\sfx := \ee^{-f_\sfx}
\qquad\text{and}\qquad 
k_{\sfx\sfy} = k_{\sfy\sfx} := \pi_\sfx\kappa_{\sfx\sfy}.
\]
We already mentioned that the function $v:\nodes\to\R$ is characterized by~\eqref{eq:MS-type-result:char-v-Phi}; this also fixes the relation between $F$ and $f$.

We introduce the notation $\redge_{\sfx\sfy}^*(\rho;\Xi) := \rmD_2 \Redge^*_{\sfx\sfy}(\rho;\Xi)$; the assumptions on $\Redge_{\sfx\sfy}$ imply that $\Xi\mapsto \redge_{\sfx\sfy}^*(\rho;\Xi)$ has a single zero at zero.
For given $\rho\in \calM_{\geq0}(\nodes)$, the functionals $\calE$, $\calR$, and $\calF$ generate the flux
\begin{align}
  j_{\sfx\sfy} &\ =\ \redge_{\sfx\sfy}^*\bra[\big]{\rho;\GenEnt'(u_\sfx) - \GenEnt'(u_\sfy) + F_\sfx - F_\sfy} \nonumber\\
  &\leftstackrel{\mathclap{\eqref{eq:MS-type-result:char-v-Phi}}}{\ =\ } \redge_{\sfx\sfy}^*\bra[\big]{\rho;\GenEnt'(u_\sfx) - \GenEnt'(u_\sfy) -\GenEnt'(v_\sfx) + \GenEnt'(v_\sfy)}.
\label{eq:MS-type-result:char-j-2}
\end{align}
Since by assumption this expression should coincide with~\eqref{eq:MS-type-result:char-j-1}, we deduce that both vanish simultaneously, i.e.
\[
\GenEnt'(u_\sfx) - \GenEnt'(u_\sfy) -\GenEnt'(v_\sfx) + \GenEnt'(v_\sfy)
= 0 
\qquad\iff\qquad
\frac{u_\sfx}{v_\sfx} = \frac{u_\sfy}{v_\sfy}.
\]
This implies that for all $a,b>0$,
\[
a \GenEnt''(a) - b \GenEnt''(b) = \frac{\rmd}{\dx t}\bra*{\GenEnt'(ta)-\GenEnt'(tb)}\Big|_{t=1} =0.
\]
It follows that there exist constants $\gamma,c_1,c_2\in \R$ such that
\[
\GenEnt(s) = \gamma\, s \log s + c_1 s + c_2,
\]
and since $\GenEnt$ is strictly convex we find $\gamma>0$. We then also have  $f_\sfx = (F_\sfx-\lambda)/\gamma$.

\medskip
With this expression for $\GenEnt$ we rewrite the equality between~\eqref{eq:MS-type-result:char-j-1} and~\eqref{eq:MS-type-result:char-j-2} as 
\begin{align*}
\redge_{\sfx\sfy}^*\bra*{\rho;\gamma \log\frac{u_\sfx v_\sfy}{u_\sfy v_\sfx}}
&= 	\frac{k_{\sfx\sfy}}{2} \theta_{\sfx\sfy}(v_\sfx,v_\sfy)
  \pra*{\frac{u_\sfx}{v_\sfx} - \frac{u_\sfy}{v_\sfy}}
\end{align*}
Since the right-hand side does not depend on other values of $\rho$ than $\rho_\sfx = u_\sfx \pi_\sfx$ and $\rho_\sfy = u_\sfy \pi_\sfy$, $\redge^*_{\sfx\sfy}$ is a function only of $u_\sfx$ and $u_\sfy$, and we write with slight abuse of notation $\redge^*_{\sfx\sfy}(u_\sfx,u_\sfy;\cdot)$ accordingly. 

In addition, note that the right-hand side is one-homogeneous in $u$ and hence also $\redge^*_{\sfx\sfy}$ is one-homogeneous in $\rho$, 
which implies that $(a,b)\mapsto \redge^*_{\sfx\sfy}(a,b;\cdot)$ is jointly one-homogeneous. 
This leads to the identity
\[
	 \redge_{\sfx\sfy}^*\bra*{u_\sfx,u_\sfy;\gamma \log\frac{u_\sfx v_\sfy}{u_\sfy v_\sfx}}
 = 	\frac{k_{\sfx\sfy}}{2} \theta_{\sfx\sfy}(v_\sfx,v_\sfy)
\pra*{\frac{u_\sfx}{v_\sfx} - \frac{u_\sfy}{v_\sfy}}.
\]
Dividing by $\sqrt{u_\sfx u_\sfy}$ and substituting
\[
	\mu := \frac{u_\sfx}{u_\sfy}\qquad\text{and}\qquad \nu := \frac{v_\sfx}{v_\sfy} ,
\]
we arrive at
\[
\redge_{\sfx\sfy}^*\bra*{\sqrt{\mu},\frac 1{\sqrt{\mu}} ; \gamma \log\frac{\mu}{\nu}}
=
\frac{k_{\sfx\sfy}}{2} \frac{\theta_{\sfx\sfy}(v_\sfx,v_\sfy)}{\sqrt{v_\sfx v_\sfy}}
\pra*{\sqrt{\frac{\mu}{\nu}} - \sqrt{\frac{\nu}{\mu}}}.
\]
By recalling from~\eqref{eq:C'} the elementary identity
\[
(\sfC^*)'\bra*{\log \frac{\mu}{\nu}}= 2\sinh\bra*{\frac{1}{2}\log \frac{\mu}{\nu}}= \sqrt{\frac{\mu}{\nu}} - \sqrt{\frac{\nu}{\mu}} ,
\]
we define $\redge^*_{\sfx\sfy}(a,b;s) = \frac{k_{\sfx\sfy}}{2} (\sfC^*)'(s/\gamma) h_{\sfx\sfy}(a,b;s/\gamma)$ for some $h_{\sfx\sfy}:\R_{\geq0}\times \R_{\geq0} \times \R$. 
Note that~$h$ inherits the one-homogeneity  in $(a,b)$ from $\redge^*$. By this Ansatz, we then arrive at
\[
  h_{\sfx\sfy}\bra*{\sqrt{\mu},\frac1{\sqrt{\mu}}; \log \frac\mu\nu } = \frac{\theta_{\sfx\sfy}(v_\sfx,v_\sfy)}{\sqrt{v_\sfx v_\sfy}} .
\]
Since the left-hand side is jointly zero-homogeneous in $(v_\sfx,v_\sfy)$, the same holds for the right-hand side, which therefore only depends on $\nu = v_\sfx/v_\sfy$.
It follows that $\theta_{\sfx\sfy}$ has to be jointly one-homogeneous and the following identity is satisfied
\[
   h_{\sfx\sfy}\bra*{\sqrt{\mu},\frac1{\sqrt{\mu}}; \log \frac\mu\nu } = \frac{\theta_{\sfx\sfy}(v_\sfx,v_\sfy)}{\sqrt{v_\sfx v_\sfy}} = \theta_{\sfx\sfy}\bra*{\sqrt{\nu},\frac1{\sqrt{\nu}}} . 
\]
Applying this identity, with $\mu=\sqrt{a/b}$ and $\nu=\ee^{-t} \sqrt{a/b}$ for any $a,b> 0$ and $s\in \R$, we find that $h$ satisfies
\[
  h_{\sfx\sfy}(a,b;t) = \sqrt{a b} \, h_{\sfx\sfy}\bra*{\sqrt{\mu},\frac1{\sqrt{\mu}}; \log \frac\mu\nu }
  = \sqrt{a b} \, \theta_{\sfx\sfy}\bra*{ \sqrt{\frac{a}{b}} \ee^{-t/2}, \sqrt{\frac{b}{a}} \ee^{t/2}} 
  = \theta_{\sfx\sfy}\bra*{ a \ee^{-t/2}, b \ee^{t/2}} . 
\]
By recalling all the definitions, we obtain~\eqref{eq:MS-type-result:char-j-1}.

It is left to show that $\theta_{\sfx\sfy}(\cdot,\cdot)$ is non-decreasing. Assumption~\ref{l:th:MC:convex} implies that $\theta$ satisfies
\begin{equation}\label{eq:th:theta:mon}
	\R \ni \xi \mapsto \theta_{\sfx\sfy}\bra*{a \ee^{-\xi},b \ee^\xi} \sinh(\xi) \text{ is non-decreasing for all } \sfx\sfy\in \edges \text{ and } a,b\geq 0 .
\end{equation}
We first show that $\theta_{\sfx\sfy}(a,\cdot)$ is non-decreasing for all $a\in [0,\infty)$.
By the joint one-homogeneity of $\theta_{\sfx\sfy}(\cdot,\cdot)$, we can equivalently rewrite~\eqref{eq:th:theta:mon} by expanding $2\sinh(\xi)$ as 
\begin{equation*}
	\theta_{\sfx\sfy}\bra*{a \ee^{-\xi},b \ee^\xi} 2\sinh(\xi) 
	=
	\theta_{\sfx\sfy}\bra*{a , b\ee^{2\xi} } - \theta_{\sfx\sfy}\bra*{a \ee^{-2\xi}, b }
	=
	\theta_{\sfx\sfy}\bra*{a, b x} \bra*{ 1- x^{-1}},
\end{equation*}
where we introduced the monotone change of variable $x:= \ee^{2\xi}\in (0,\infty)$.
Then, the monotonicity condition~\eqref{eq:th:theta:mon} implies for all $a,b\in [0,\infty)$ the monotonicity
\begin{equation}\label{eq:th:t:mon}
	\theta_{\sfx\sfy}(a, b x_0) \bra*{1-\tfrac{1}{x_0}} \leq \theta_{\sfx\sfy}(a, b x_1) \bra*{1-\tfrac{1}{x_1}} \qquad\text{for all } 0<x_0\leq x_1 .
\end{equation}
By contradiction, we assume that there exists $\sfx\sfy\in \edges$, $a\in [0,\infty)$ and $0< y_0 < y_1$ such that $\theta_{\sfx\sfy}(a,y_0)> \theta_{\sfx\sfy}(a,y_1)$. 
Now, we can choose for any $b>0$, $x_{i,b}:=b^{-1} y_i$ for $i=0,1$ in~\eqref{eq:th:t:mon} and obtain for all $b>0$ the inequality
\begin{equation*}
	\theta_{\sfx\sfy}(a,y_0)\bra*{1-\tfrac{b}{y_0}} \leq \theta_{\sfx\sfy}(a,y_1) \bra*{1-\tfrac{b}{y_1}} ,
\end{equation*}
which contradicts for $b$ sufficiently small the assertion $\theta_{\sfx\sfy}(y^0)> \theta_{\sfx\sfy}(y^1)$, hence a contradiction.
So $\theta_{\sfx\sfy}(a,\cdot)$ is non-decreasing for all $a\in [0,\infty)$. 
A similar argument shows that also  $\theta_{\sfx\sfy}(\cdot,b)$ is non-decreasing for all $b\in[0,\infty)$, completing the proof.
%
%
\end{proof}

Let us connect Theorem~\ref{th:MC:gradient:tilt} above with the characterization of detailed balance Markov jump kernels in Proposition~\ref{prop:Jump:tilt}.
For an admissible tilting $(\pi^f,\kappa^f)$ satisfying~\eqref{cond:DB-kernels-tilting},
we obtain from Proposition~\ref{prop:Jump:tilt}
that $\kappa^f$ is of the form~\eqref{e:Jump:Phi}
\begin{equation}\label{e:Jump:Phi:p0}
	\kappa^f_{\sfx\sfy} = \kappa_{\sfx\sfy} \ee^{f_\sfx}\theta_{\sfx\sfy}\bra*{\ee^{-f_\sfx},\ee^{-f_\sfy}} 
\end{equation}
for some $\theta:\edges \times \R^2 \to (0,\infty)$, jointly symmetric in the sense of~\eqref{eq:Jump:Phi:symmetry}.

We obtain in the following corollary sufficient conditions on the detailed balance jump kernels $\kappa^f$ ensuring that the induced evolution equation possess a tilt-independent gradient structure. 
\begin{cor}\label{cor:MC:tilt:kappa:norm}
For a family of admissible tilted pairs $\set[\big]{(\pi^f,\kappa^f)}_{f:\nodes \to \R}$ satisfying~\eqref{cond:DB-kernels-tilting} the following are equivalent:
\begin{enumerate}
	\item
	$\kappa^f_{\sfx\sfy}$ satisfies for all $\sfx\sfy\in\edges$
	\begin{equation}\label{eq:Jump:tilt:rate:normalization}
		\kappa_{\sfx\sfy}^f = \kappa_{\sfx\sfy}^{f+\alpha}, \qquad\text{for all } \alpha\in \R;
	\end{equation}
	and
	\begin{equation}\label{eq:Jump:tilt:rate:monotone}
		\R^2\ni (f_{\sfx},f_{\sfy})\mapsto \ee^{-f_{\sfx}}\kappa^f_{\sfx\sfy}  \ \text{ is jointly non-increasing.}
	\end{equation}
	\item
	The flux $j^f_{\sfx\sfy}=\frac{1}{2}(\rho_\sfx \kappa_{\sfx\sfy}^f - \rho_{\sfy}\kappa_{\sfy\sfx}^f)$ is induced by a gradient system $(\nodes,\edges,\ona,\calH(\cdot|\pi), \sfF_{\mathrm{Pot}},\calR_{\Redge})$
	with potential tilts $\sfF_{\mathrm{Pot}}$ given in~\eqref{eqdef:MCs:potential-tilts},
	tilt-independent $\calR_{\Redge}$ given in~\eqref{eqdef:R:psi},
	and $\set*{\Redge_{\sfx\sfy}}_{\sfx\sfy\in\edges}$ given through its dual in~\eqref{eq:th:MS-type-result} for some  
	jointly symmetric function $\theta: \edges \times \R_{\geq0}^2 \to (0,\infty)$ (that is~\eqref{eq:Jump:Phi:symmetry}) 
	with $\theta_{\sfx\sfy}(\cdot,\cdot)$ being non-decreasing and one-homogeneous.
\end{enumerate}
\end{cor}
\begin{proof}
By Theorem~\ref{th:MC:gradient:tilt}, it is enough to show that \eqref{eq:Jump:tilt:rate:normalization} and~\eqref{eq:Jump:tilt:rate:monotone} is equivalent to $\theta_{\sfx\sfy}(\cdot,\cdot)$ being one-homogeneous and non-decreasing.

First, let $\kappa^f$ satisfy~\eqref{eq:Jump:tilt:rate:normalization}, then we get for $\alpha\in \R$ the condition
\begin{equation*}
	\kappa_{\sfx\sfy} \ee^{f_\sfx} \theta_{\sfx\sfy}\bra*{\ee^{-f_\sfx},\ee^{-f_\sfy}} = \kappa_{\sfx\sfy}^f = \kappa_{\sfx\sfy}^{f+\alpha} = \kappa_{\sfx_\sfy} \ee^{f_\sfx + \alpha} \theta_{\sfx\sfy}\bra*{\ee^{-\alpha} \ee^{-f_\sfx},  \ee^{-\alpha}  \ee^{-f_\sfy}} ,
\end{equation*}
showing that $\theta_{\sfx\sfy}$ is positively one-homogeneous. 

Conversely, if $\theta_{\sfx\sfy}(\cdot,\cdot)$ is one-homogeneous, then~\eqref{eq:Jump:tilt:rate:normalization} is immediate from the form of $\kappa^f$ in~\eqref{e:Jump:Phi:p0}. Hence, \eqref{eq:Jump:tilt:rate:normalization} is equivalent to the one-homogeneity of $\theta_{\sfx\sfy}(\cdot,\cdot)$. 

For the second equivalence, we note that by~\eqref{e:Jump:Phi:p0} we have  $\ee^{-f_\sfx}\kappa^f_{\sfx\sfy}=\kappa_{\sfx\sfy} \theta(\ee^{-f_\sfx},\ee^{-f_\sfy})$, which makes the monotonicity relationship between the condition~\eqref{eq:Jump:tilt:rate:monotone} and~$\theta_{\sfx\sfy}(\cdot,\cdot)$ apparent. 
\end{proof}
\begin{remark}[On the condition~\eqref{eq:Jump:tilt:rate:normalization} in Corollary~\eqref{cor:MC:tilt:kappa:norm}]
On the one hand, the normalization~\eqref{eq:Jump:tilt:rate:normalization} can  similarly be compared to the normalization of $\pi^f$ in~\eqref{eq:Jump:tilt:pi} as a probability measure. 
Indeed, the latter immediately entails that shifts $f \mapsto f + \alpha$ for some $\alpha\in \R$ leave $\pi^f = \pi^{f+\alpha}$ invariant. 
On the other hand, note that the chemical reaction rates as discussed in Section~\ref{ss:case-study-chem-reactions-tilt-dependent} do not satisfy the normalization condition~\eqref{eq:Jump:tilt:rate:normalization}, which already shows that those cannot come from a tilt-independent gradient structure.
In a physical sense, the condition~\eqref{eq:Jump:tilt:rate:normalization} can be understood as stating that the rates only depend on relative energy levels. 
For chemical reaction rates in Section~\ref{ss:case-study-chem-reactions-tilt-dependent}, we observe that energy levels at nodes need to be compared with energies of edges (see for instance~\eqref{eqdef:kappa:chemical}), which forces the use of the same reference energy and hence property~\eqref{eq:Jump:tilt:rate:normalization} can not generically be satisfied.
\end{remark}

\begin{remark}[Symmetries of the kinetic relation defined through~\eqref{eq:th:MS-type-result}]\label{rem:MS-type:symmetries}\leavevmode
\begin{enumerate}
	\item 	
	Since the fluxes in the assumption~\eqref{eq:MC:gradient:tilt:flux} are assumed to be skew-symmetric, 
	the kinetic relation defined by $\set[\big]{\rmD_2 \Redge_{\sfx\sfy}^*(\rho;\Xi)}_{\sfx\sfy\in\edges}$ in~\eqref{eq:th:MS-type-result} by construction also satisfies the skew-symmetry property
	\begin{equation*}
	\rmD_2 \Redge_{\sfx\sfy}^*(\rho;\Xi) = -\rmD_2 \Redge_{\sfy\sfx}^*(\rho;-\Xi) \qquad\text{for all } \rho\in \ProbMeas(\nodes) \text{ and }  \Xi\in \R. 
	\end{equation*}
	Equivalently, $\Redge_{\sfx\sfy}(\rho;j)=\Redge_{\sfy\sfx}(\rho;-j)$, i.e. the frictional dissipation potential of a flux $j$ along the forward edge $\sfx\sfy$ is equal to the  dissipation potential of $-j$ along the backward edge $\sfy\sfx$.
	\item\label{rem:MS-type:symmetries:symmetric}
	If we suppose, in addition, that $\Redge_{\sfx\sfy}$ in~\eqref{eqdef:R:psi} is symmetric on every $\sfx\sfy\in \edges$, that is
\begin{equation*}
	\Redge_{\sfx\sfy}(\rho; -s) = \Redge_{\sfx\sfy}(\rho; s) \qquad\text{for all $s\in \R$},
\end{equation*}
then the induced kinetic relation $\rmD_2 \Redge^*_{\sfx\sfy}$ is skew-symmetric on every edge $\sfx\sfy\in\edges$,
\[
 \rmD_2 \Redge^*_{\sfx\sfy}(\rho;\Xi)=-\rmD_2 \Redge^*_{\sfx\sfy}(\rho;-\Xi) \qquad\text{for all } \Xi\in \R .
\]
It follows from~\eqref{eq:th:MS-type-result} that
\begin{equation*}
	\theta_{\sfx\sfy}\bra*{a,b x}=\theta_{\sfx\sfy}\bra*{a x,b } \qquad\text{for all } a,b\in [0,\infty), x\in (0,\infty). 
\end{equation*}
By choosing $b=x$ and $a=1$, we get $\theta_{\sfx\sfy}\bra*{1,x^2}=x \theta_{\sfx\sfy}\bra*{1,1}$ for all $x>0$, implying that $\theta_{\sfx\sfy}\bra*{1,x}= \omega_{\sfx\sfy} \sqrt{x}$ with $\omega_{\sfx\sfy}:=\sqrt{\theta_{\sfx\sfy}\bra*{1,1}}$. 
Hence, we obtain by the one-homogeneity
\begin{equation*}
\theta_{\sfx\sfy}\bra*{a,b} = a \theta_{\sfx\sfy}\bra*{1,a^{-1}b} = \omega_{\sfx\sfy} \sqrt{a\, b} ,
\end{equation*}
which is exactly the choice leaving the activity invariant; see Proposition~\ref{prop:activity-tilt} and Example~\ref{ex:particular-DB-parametrisations}.\ref{ex:tilt:sym}.
\item 
In Section~\ref{sss:cosh-from-ldp-intro} we obtained a variational description of uni-directional fluxes in terms of the large deviation functional~\eqref{eqdef:J}, and the behaviour under tilting is discussed in Sections~\ref{ss:tilting:MCs} and~\ref{ss:tilting-LDPs}.
The resulting tilt-independent gradient structure is the one for which $\theta_{\sfx\sfy}(a,b)=\omega_{\sfx\sfy}\sqrt{a b}$ for some $\omega_{\sfx\sfy}\in (0,\infty)$.

We are not aware of an immediate stochastic interpretation as large deviation rate functional for all the other possibilities for $\theta$ in~Theorem~\ref{th:MC:gradient:tilt}; also variational characterizations, as a contraction of unidirectional fluxes similar to~\eqref{eq:formula-1}, are left for future research. 
Nevertheless, we show that specific other choices for $\theta$ emerge in numerical schemes and comment on it in Section~\ref{sss:MC:tilt:SG-upwind}.
\end{enumerate}
\end{remark}
We want to point out that the discussion of tilting is not restricted to linear functionals of potential type $\sfF_{\mathrm{Pot}}$.
\begin{remark}[More general tilts]\label{rem:MC:general_tilts}
	In the discussion of Theorem~\ref{th:MC:gradient:tilt}, we have restricted ourselves to tilts characterized by functions $F:\nodes\to\R$, corresponding to potential energies~\eqref{eqdef:MCs:potential-tilts}. For this class the derivative $F = \rmD\calF(\rho)$ is independent of $\rho$.
	
	It is natural to consider more general tilts $\calF$; in~\eqref{eq:Kramers:tilts} in Section~\ref{s:Kramers}, for instance, we consider the class of admissible tilts
	\begin{equation}\label{eq:MCs:tilts}
		\sfF := \set[\Big]{\;\calF \in C^1(\ProbMeas(\nodes)): \sup_{\rho\in\ProbMeas(\nodes)} \norm[\big]{\rmD\calF(\rho)}_{L^\infty(\nodes)}<\infty\;} .
	\end{equation}
	A general class of tilts in the class $\sfF$ in~\eqref{eq:MCs:tilts} are introduced in~\cite{BDFR15a,BDFR15b}. Besides simple `potential' energies of the form~\eqref{eqdef:MCs:potential-tilts}, this definition allows also for `interaction' energies of the form
	\begin{equation*}
		\calF_2(\rho) = \frac{1}{2} \sum_{\sfx\in \nodes}\sum_{\sfy\in \nodes} G_{\sfx\sfy} \rho_\sfx \rho_\sfy ,
	\end{equation*}
	for some symmetric interaction kernel $G:\nodes \times \nodes \to \R$. 
	
	Even more general energies in the class~\eqref{eq:MCs:tilts} are given for some $K\in C^2(\nodes\times \calM_{\geq0}(\nodes))$ in the form
	\begin{equation*}
		\calF_K(\rho) = \sum_{\sfx\in \nodes} K_{\sfx}(\rho) \rho_\sfx .
	\end{equation*}
	These define a family of local equilibrium states for the energy $\calE+ \calF_K$ of the form
	\begin{equation}\label{eq:MC:tilt:local-equilibrium}
		\pi_\sfx^\rho = \frac{\pi_\sfx}{\PartSum^\rho} \ee^{-H_\sfx(\rho)} 
		\quad\text{with}\quad 
		H_\sfx(\rho) := \frac{\partial}{\partial\rho_\sfx} \calF_K(\rho) 
		\quad\text{and}\quad 
		\PartSum^\rho :=  \sum_{\sfx\in \nodes} \pi_\sfx \ee^{-H_\sfx(\rho)} .
	\end{equation} 
	In this sense, the tilt of energy from $\calE$ to $\calE+\calF_K$ gives rise to the change of the $\rho$-independent equilibrium to the family of $\rho$-dependent local equilibrium defined in~\eqref{eq:MC:tilt:local-equilibrium}.
	In \cite{ErbarFathiLaschosSchlichting16,ErbarFathiSchlichting2020} a gradient flow structure and curvature notions for the free energy $\calE + \calF_K$ are investigated based on quadratic dissipation potentials.
\end{remark}

\subsubsection{Alternative effect of tilting: changing the reference measure}\label{sss:MS}
In~\cite{MielkeStephan2020}, Mielke and Stephan consider a slightly more restricted class in comparison to~\eqref{eqdef:R:psi} of gradient systems $(\nodes,\edges,\gnabla,\calH_\GenEnt(\cdot|\pi),\calR_{\alpha,\Redge})$ with dissipation potential given by
\begin{equation}\label{eqdef:R:psi-MS}
	\calR_{\alpha,\Redge}^*(\rho,\Xi) := \sum_{\sfx\sfy\in \edges} \alpha_{\sfx\sfy}(\rho)\Redge^*_{\sfx\sfy}\bra*{\Xi_{\sfx\sfy}},
\end{equation}
where $\Redge^*_{\sfx\sfy}\in C^2(\R,\R_{\geq0})$ satisfies  $\Redge_{\sfx\sfy}^*(0)=0$ and $(\Redge_{\sfx\sfy}^*)''(0)>0$, and $\alpha$ satisfies $\alpha_{\sfx\sfy}\in C^1(\calM_{\geq 0}(\nodes);\R_{\geq0})$ for ${\sfx\sfy\in\edges}$. We assume that $\alpha$ and $\Redge$ also satisfy the symmetry relations 
\[
\alpha_{\sfx\sfy}(\rho) = \alpha_{\sfy\sfx} (\rho)
\qquad\text{and}\qquad
\Redge_{\sfx\sfy}(s) = \Redge_{\sfy\sfx}(s) = \Redge_{\sfx\sfy}(-s)= \Redge_{\sfy\sfx}(-s).
\]
Note that the structural assumption~\eqref{eqdef:R:psi-MS} with $\alpha$ symmetric leads to the setting discussed in Remark~\ref{rem:MS-type:symmetries}.\ref{rem:MS-type:symmetries:symmetric}.
The requirement that the gradient structure $(\nodes,\edges,\gnabla,\calH_\GenEnt(\cdot|\pi),\calR_{\alpha,\Redge})$ induces the evolution equation~\eqref{eq:heat-flow-intro} translates into the requirement
\begin{equation}
	\label{eq:cond:a-Phi-Psi}
	\alpha_{\sfx\sfy}(\rho) = \frac{\sqrt{\kappa_{\sfx\sfy}\pi_\sfx \kappa_{\sfy\sfx}\pi_\sfy}\bigl(\frac{\rho_\sfx}{\pi_\sfx}-\frac{\rho_\sfy}{\pi_\sfy}\bigr)}
	{(\Redge_{\sfx\sfy}^*)'\bigl(\GenEnt'(\frac{\rho_\sfx}{\pi_\sfx})-\GenEnt'(\frac{\rho_\sfy}{\pi_\sfy})\bigr)}
	\qquad \text{for }\frac{\rho_\sfx}{\pi_\sfx}\not=\frac{\rho_\sfy}{\pi_\sfy}, 
\end{equation}
with the value of $\alpha_{\sfx\sfy}(\rho)$ for $\rho_\sfx/\pi_\sfx = \rho_\sfy/\pi_\sfy$ being defined by continuity and $\Redge_{\sfx\sfy}^*$ denoting the Legendre dual of $\Redge_{\sfx\sfy}$. Since there are $2|\edges|+1$  functions to be chosen, and~\eqref{eq:cond:a-Phi-Psi} represents only $|\edges|$ conditions on these, a wide range of gradient systems can be constructed that all induce the same equation~\eqref{eq:heat-flow-intro}. 

Instead of considering~$\calF$ as an addition to $\RelEnt_\GenEnt(\cdot|\pi)$ as in~\eqref{eq:MS-type-result:def-of-tilting}, the authors of~\cite{MielkeStephan2020} assume  that tilting preserves the entropic structure of $\calE$ and only changes the reference measure $\pi$, i.e.
\begin{equation}
\label{eq:MS-type-result:def-of-tilting-MS}
\calE := \RelEnt_\GenEnt(\cdot|\pi) \ \xmapsto{\text{tilt by }\calF}\  \calE^\calF := \RelEnt_\GenEnt(\cdot|\pi^F), 
\end{equation}
where $\pi^F$ depends in an unspecified way on $\calF$. This leads to a characterization that is very similar to Theorem~\ref{th:MC:gradient:tilt} above:
\begin{theorem}[{\cite[Prop.~4.1]{MielkeStephan2020}}]
	Let $\alpha$, $\Redge$, and $\GenEnt$ satisfy~\eqref{eq:cond:a-Phi-Psi}. If $\alpha$ is independent of~$\pi$, then there exist $c_1,c_2\in \R$ and $\omega_{\sfx\sfy},\gamma>0$ for $\sfx\sfy\in\edges$ such that 
\begin{equation*}
\GenEnt(s) = \gamma s\log s + c_1 s + c_2,
\quad
\Redge^*_{\sfx\sfy}(t) = \gamma \omega_{\sfx\sfy} \,\sfC^*(t/\gamma),
\quad\text{and}
\quad
\alpha_{\sfx\sfy}(\rho) = \frac1{2\omega_{\sfx\sfy}}  \sqrt{\rho_\sfx \rho_\sfy \kappa_{\sfx\sfy} \kappa_{\sfy\sfx}}.
\end{equation*}
\end{theorem}
To better understand the difference in assumptions about the form of tilting, first note that when $\GenEnt(s) = s\log s$ and $\pi^F = (\PartSum^F)^{-1} \ee^{-F}\pi$ the two definitions coincide:
\begin{equation*}
	\calH(\rho|\pi^F)  = \calH(\rho | \pi) + \skp*{\log \frac{\pi}{\pi^F}, \rho} = \calH(\rho | \pi) + \skp*{ F + \PartSum^F, \rho} .
\end{equation*}
Hence, the exponential tilting of the equilibrium $\pi\mapsto \pi^F = (\PartSum^F)^{-1} \ee^{-F}\pi$  is equivalent to the additive tilt of the energy $\calH(\rho | \pi) \mapsto \calH(\rho | \pi) + \calF^F_\Pot(\rho)$ with 
\begin{equation*}
	\calF^F_\Pot(\rho) = \sum_{x\in \nodes} \bra*{F_x + \PartSum^F} \rho_{\sfx}.
\end{equation*}
However, this equivalence is specific to the Boltzmannian case $\GenEnt(s) = s\log s$, as the lemma below shows: if $\GenEnt$ is such that `tilting  by addition' is equivalent to `tilting by modifying $\pi$', then $\GenEnt$ is Boltzmannian. It even is sufficient to assume that the two forms of tilting only are equivalent on the set of \emph{normalized} measures.
\begin{lemma}
Let  $\pi\in\calM_{>0}(\nodes)$ be fixed. Assume that there exists a map
\[
F\mapsto \pi^F
\]
that defines how $\pi$ is modified by the tilting $F:\nodes\to\R$, with $0\mapsto \pi$, and assume that in addition we have the identity
\begin{equation}\label{eq:RelEnt:linear:tilt}
	  \RelEnt_\GenEnt(\rho | \pi^F) = 	  \RelEnt_\GenEnt(\rho | \pi) + \skp{F,\rho}
	  + c_F,\qquad\text{for all }\rho\in \ProbMeas(\nodes),
\end{equation}
where for each $F$, $c_F$ is a constant. 
Then $\GenEnt(s) =\gamma \,s \log s + c_1 s+ c_2$ for some constants $\gamma>0$ and $c_1,c_2\in \R$.
\end{lemma}

\begin{proof}
By differentiating~\eqref{eq:RelEnt:linear:tilt} with respect to $\rho$ while preserving the mass of $\rho$ we find
\[
\GenEnt'\bra*{\frac{\rho_\sfx}{\pi^F_\sfx}} 
 - \GenEnt'\bra*{\frac{\rho_\sfy}{\pi^F_\sfy}} = 
\GenEnt'\bra*{\frac{\rho_\sfx}{\pi_\sfx}} - \GenEnt'\bra*{\frac{\rho_\sfy}{\pi_\sfy}} + F_\sfx - F_\sfy.
\]
Setting $v = \pi^F/\pi$ and choosing $\rho = \pi^F = v\pi $ we find that $\GenEnt'(v_\sfx) + F_\sfx = \lambda$, similarly to~\eqref{eq:MS-type-result:char-v-Phi}, and therefore $F_\sfx-F_\sfy = \GenEnt'(v_\sfy) - \GenEnt'(v_\sfx)$.
Setting $u = \rho/\pi$ we find
\[
\GenEnt'\bra*{\frac{u_\sfx}{v_\sfx}} 
 - \GenEnt'\bra*{\frac{u_\sfy}{v_\sfy}} 
= \GenEnt'(u_\sfx) - \GenEnt'(u_\sfy) - \GenEnt'(v_\sfx) + \GenEnt'(v_\sfy).
\]
Differentiating this expression with respect to $u_\sfx$ we find
\[
\GenEnt''\bra*{\frac{u_\sfx}{v_\sfx}}  = v_\sfx \GenEnt''(u_\sfx) 
\qquad\text{for all }u_\sfx, v_\sfx >0,
\]
from which we deduce,  as in the proof of Theorem~\ref{th:MC:gradient:tilt}, that $\GenEnt(s) = \gamma s\log s + c_1 s + c_2$ for some $\gamma>0$ and $c_1,c_2\in\R$.
\end{proof}

To conclude, for the Boltzmann case $\GenEnt(s) = \gamma s\log s + c_1 s = c_2$, tilting by addition (as in~\eqref{eq:MS-type-result:def-of-tilting}) and tilting by modifying the reference measure (as in~\eqref{eq:MS-type-result:def-of-tilting-MS}) are equivalent, but for any other entropy function (such as Tsallis entropies) the two lead to different concepts of tilting.

\subsubsection{`De-tilting' of tilt-dependent gradient structures for finite Markov chains}\label{sss:MC:quadratic:untilting}

In this section, we apply the construction of `de-tilting' described in Remark~\ref{rem:unitlting:tilt-GS} to the quadratic gradient structure of~\cite{ChowHuangLiZhou12,Maas11,Mielke11}.

We first show how the quadratic gradient system is tilt-dependent. 
For doing so, a pair~$(\pi,\kappa)$ satisfying the detailed-balance condition~\eqref{eq:def:graph:DBC} is fixed.
We consider potential tilts~\eqref{eqdef:MCs:potential-tilts}, that is, elements of $\sfF_{\mathup{Pot}} = \set*{ \calF^F_\Pot(\rho) = \sum_{\sfx\in\nodes}F_\sfx\rho_\sfx \,|\, F:\nodes \to \R}$, and we assume that the  tilted jump rates are given by $\kappa_{\sfx\sfy}^F := \kappa_{\sfx\sfy} \ee^{-\tfrac{1}{2}\gnabla F_{\sfx\sfy}}$,
which is the form obtained for the tilting of Markov processes in~\eqref{eq:tilting-MCs-transformed-kappa-F} in Section~\ref{ss:tilting:MCs}.
Based on~\cite{ChowHuangLiZhou12,Maas11,Mielke11}, we postulate that for $\calF^F_\Pot\in\sfF_\Pot$ the 
quadratic dissipation potential is given by
\begin{equation}\label{eqdef:MC:quad-GS}
	\calR_2^*(\rho,\Xi;\calF^F_\Pot):=  \sum_{\sfx\sfy\in\edges} \Redge^*_2\bra*{\sfx\sfy;\rho_\sfx,\rho_\sfy;\Xi_{\sfx\sfy};-\gnabla F_{\sfx\sfy}} ,
\end{equation}
where the dependence of $\Redge_2$ in the parameters $\sfx\sfy$, $(\rho_\sfx,\rho_\sfy)$, and $\Xi_{\sfx\sfy}$ is assumed given by~\cite{ChowHuangLiZhou12,Maas11,Mielke11}.
The dependence on the tilt $-\ona F_{\sfx\sfy}$ is still to be determined, and we will find this dependence by requiring that the induced flux satisfies 
\[
j^F_{\sfx\sfy}=\frac{1}{2}\bra*{\rho_\sfx \kappa_{\sfx\sfy}^F- \rho_\sfy \kappa_{\sfy\sfx}^F},
\]
corresponding to the induced equation~\eqref{eq:heat-flow-intro} and an assumption of skew-symmetry of $j^F$.
Combining these requirements we find that $\Redge_2$ has the form
\begin{align*}
 \Redge^*_2\bra*{\sfx\sfy;\rho_\sfx,\rho_\sfy;\Xi_{\sfx\sfy};-\gnabla F_{\sfx\sfy}}
  &:= \frac{1}{4} \Lambda\bra*{\rho_{\sfx}\kappa_{\sfx\sfy}^F , \rho_{\sfy}\kappa_{\sfy\sfx}^F} \abs*{\Xi_{\sfx\sfy}}^2 \\
  &=
  \frac{\pi_\sfx \kappa_{\sfx\sfy}}{4} \Lambda\bra*{\frac{\rho_\sfx}{\pi_{\sfx}} \ee^{-\tfrac{1}{2}\gnabla F_{\sfx\sfy}},\frac{\rho_\sfy}{\pi_{\sfy}}\ee^{\tfrac{1}{2}\gnabla F_{\sfx\sfy}}} \abs*{\Xi_{\sfx\sfy}}^2
\end{align*}
where $\Lambda:\R_{\geq 0}^2 \to \R_{\geq 0}$ is the logarithmic mean~\eqref{eqdef:log-mean} defined in Remark~\ref{rem:other-gradient-structures}.
Indeed, in this case the induced flux $j^F$ is given by
\begin{align*}
	j^F_{\sfx\sfy} &= {\rmD_2 \calR_2^*\bra*{\rho,-\gnabla \bra[\big]{\log\tfrac{\rho}{\pi}+F};\calF^F_\Pot}}_{\sfx\sfy}\\
	&= \rmD_3 \Redge^*_2\bra*{\sfx\sfy;\rho_\sfx,\rho_\sfy;-\gnabla_{\sfx\sfy} \bra[\big]{\log\tfrac{\rho}{\pi}+F};-\gnabla F_{\sfx\sfy}}\\
	&= \frac{1}{2} \bra*{ \rho_\sfx \kappa^F_{\sfx\sfy} - \rho_{\sfy}\kappa_{\sfy\sfx}^F} 
\end{align*}

We now construct the `de-tilted' gradient structure as described in Remark~\ref{rem:unitlting:tilt-GS}.
For $\pi\in \ProbMeas^+(\nodes)$, we observe that $\dom \calH(\cdot|\pi)=\ProbMeas(\nodes)$ and hence the assumption~\eqref{ass:untilting:domain} of Remark~\ref{rem:unitlting:tilt-GS} is satisfied. 
By the construction~\eqref{eqdef:MC:quad-GS}, $\calR_2^*$ also satisfies the structural assumption~\eqref{ass:untilting:F-dependence}.
Hence, we can define the \emph{evolution-equivalent} (in the sense of~\eqref{eqdef:evolution-equivalent}) gradient structure by following~\eqref{eqdef:untilted:Psis} and setting
\begin{align*}
	\Redge^*\bra*{\sfx\sfy;\rho_\sfx,\rho_\sfy;\Xi_{\sfx\sfy}}
	&=\int_0^{\Xi_{\sfx\sfy}} \rmD_3\Redge^*_2\bra*{\sfx\sfy;\rho_\sfx,\rho_\sfy;\xi;\xi+\gnabla_{\sfx\sfy}\log \tfrac{\rho}{\pi}} \dx{\xi} \\
	&= \frac{\pi_\sfx\kappa_{\sfx\sfy}}{2} \int_0^{\Xi_{\sfx\sfy}} \Lambda\bra*{\frac{\rho_\sfx}{\pi_{\sfx}} \ee^{-\tfrac{1}{2}\xi-\tfrac{1}{2}\gnabla_{\sfx\sfy}\log \tfrac{\rho}{\pi}},\frac{\rho_\sfy}{\pi_{\sfy}}\ee^{\tfrac{1}{2}\xi+\tfrac{1}{2}\gnabla_{\sfx\sfy}\log \tfrac{\rho}{\pi}}}\, \xi \dx{\xi}\\
	&=\frac{\pi_\sfx\kappa_{\sfx\sfy}}{2} \sqrt{\frac{\rho_\sfx\rho_\sfy}{\pi_{\sfx}\pi_{\sfy}}} \int_0^{\Xi_{\sfx\sfy}} \Lambda\bra*{ \ee^{-\tfrac{1}2\xi},\ee^{\tfrac12 \xi}}\xi \dx\xi \\
	&=\frac12\sqrt{\kappa_{\sfx\sfy}\kappa_{\sfy\sfx}\rho_\sfx\rho_\sfy} \int_0^{\Xi_{\sfx\sfy}} (\sfC^*)'(\xi) \dx{\xi}
	= \frac12\sqrt{\kappa_{\sfx\sfy}\kappa_{\sfy\sfx}\rho_\sfx\rho_\sfy}  \,\sfC^*\bra*{\Xi_{\sfx\sfy}}.
\end{align*}
Hence $\calR(\rho,\Xi)=\sum_{\sfx\sfy} \Redge^*\bra*{\sfx\sfy;\rho_\sfx,\rho_\sfy;\Xi_{\sfx\sfy}}$ is exactly the cosh structure for the detailed-balance heat flow defined in~\eqref{eq:ER-heat-flow-intro} for Example~\ref{ex:heat-flow}.

\subsubsection{Finite-volume schemes as generalized gradient structure}\label{sss:MC:tilt:SG-upwind}
There is a recent interest in bringing numerical finite-volume schemes into the framework of gradient flows; see for instance~\cite{ChowHuangLiZhou12,DisserLiero2015,CancesGallouetTodeschi2019,HeidaKantnerStephan2021,EspositoPatacchiniSchlichtingSlepcev2021,SchlichtingSeis2021,HraivoronskaTse2022TR}.
For finite-volume schemes, the individual nodes $\nodes$ correspond to  cells in a tessellation of the physical domain $\Omega\subset \R^d$.
Two cells $\sfx$ and $\sfy$ are called neighbours if they have a common interface, in which case we write $\sfx\sfy\in\edges$. 
Every pair $\sfx\sfy\in\edges$ of neighbours has a symmetric transmission coefficient $\tau_{\sfx\sfy}\in(0,\infty)$, which usually is taken proportional to the area of the interface; we assume that $\tau_{\sfx\sfy}=\tau_{\sfy\sfx}$.
With this interpretation $(\nodes,\edges)$ is a weighted graph.

To fix a pair $(\pi,\kappa)$ on $(\nodes,\edges)$, we choose $\pi_{\sfx}=\frac{\vol_\sfx}{\vol(\Omega)}$, where $\vol(\sfx)$ is the $d$-dimensional volume of the cell $\sfx$. 
We define rates $\kappa_{\sfx\sfy}:={\tau_{\sfx\sfy}}/{\pi_{\sfx}}$ such that $\pi_\sfx \kappa_{\sfx\sfy}=\tau_{\sfx\sfy}=\tau_{\sfy\sfx}=\pi_\sfy \kappa_{\sfy\sfx}$ and hence $(\pi,\kappa)$ defines a detailed-balance Markov chain on the weighted graph $(\nodes,\edges)$. 
For a measure $\rho\in \ProbMeas(\nodes)$, usually obtained by discretizing a probability measure on the physical domain $\Omega$, we define the relative cell densities $u_{\sfx}=\frac{\rho_{\sfx}}{\pi_\sfx}$.
	
A finite-volume scheme specifies the flux $j_{\sfx\sfy}$ resulting from a generalized forcing $\Xi_{\sfx\sfy}$ across the interface $\sfx\sfy$, for given values $u_\sfx$ and $u_\sfy$ on the ends of the edge;
this is exactly a specification of a kinetic relation.
For drift-diffusion equation such as  the Fokker-Planck equation~\eqref{eq:FP-intro}, a popular choice is the Scharfetter-Gummel flux interpolation~\cite{SG1969} (see also~\cite{Farrell_etal2017,Cances_etal2019}).
For consistency with the result in Theorem~\ref{th:MC:gradient:tilt}, we describe this choice in the double directed edge setup (see Remark~\ref{rem:edges-conversion}).

In the Scharfetter-Gummel scheme the flux is defined as follows. Fix a `viscosity' parameter $\gamma>0$. Writing $\xi$ for $\Xi_{\sfx\sfy}$, let the pair $(u,j_{\sfx\sfy})$ with $u:[0,1]\to \R$ solve the two-point boundary-value problem
\begin{align*}
	\frac{2 j_{\sfx\sfy}}{\tau_{\sfx\sfy}} &= - \gamma u(x) \bra*{ \partial_x \log u(x)+\log \frac{u_{\sfx}}{u_{\sfy}}}  + \xi u(x) , \qquad\text{on } (0,1) ,\\
	&\qquad u(0)=u_\sfx \quad\text{and}\quad u(1)=u_\sfy .
\end{align*}
As a boundary-value problem for $u$ this problem is overdetermined, and $j_{\sfx\sfy}=j_{\sfx\sfy}[u_\sfx,u_\sfy;\xi]\in\R$ can be interpreted as a Lagrange multiplier that ensures unique solvability.
Hereby, the term $\gamma u(x)\log \tfrac{u_\sfx}{u_\sfy}$ encodes Fick's law and ensures the normalization property $j_{\sfx\sfy}[u_\sfx,u_\sfy;0]=0$.

By noting that $\ee^{-\xi x/\gamma}$ is an integrating factor, the Lagrange multiplier $j_{\sfx\sfy}$ can be characterized as
\begin{align}\label{eq:SG:cell-problem}
	j_{\sfx\sfy}[u_\sfx,u_\sfy;\xi]
	&= \frac{\tau_{\sfx\sfy}}{2} \gamma \Lambda_{-1}\bra*{ u_{\sfx} \ee^{-\frac{\xi}{2\gamma}}, u_{\sfy} \ee^{\frac{\xi}{2\gamma}}} \,2\sinh\bra*{\frac{\xi}{2\gamma}} \\
	&= \frac{\pi_{\sfx}\kappa_{\sfx\sfy}}{2} \gamma \Lambda_{-1}\bra*{ \frac{\rho_{\sfx}}{\pi_\sfx} \ee^{-\frac{\xi}{2\gamma}},  \frac{\rho_{\sfy}}{\pi_\sfy} \ee^{\frac{\xi}{2\gamma}}} \,(\sfC^*)'\bra*{\frac{\xi}{\gamma}} , \notag 
\end{align}
where $\Lambda_{-1}:\R_{\geq0}^2 \to \R_{\geq0}$ is the harmonic-logarithmic mean given by
\begin{equation}
\label{eqdef:HarmLogMean}
	\Lambda_{-1}(a,b):= 
	\begin{cases}
	\,\dfrac{ab(\log b-\log a)}{b-a} = 
	\Lambda(a^{-1},b^{-1})^{-1} ,& \text{for } ab>0 \\
	\,0 , & \text{for } ab=0 .
	\end{cases}
\end{equation}

Hence, we recover a tilt-independent gradient structure as in~\eqref{eq:th:MS-type-result} from Theorem~\ref{th:MC:gradient:tilt}, 
with the dual dissipation potential given by $\calR_{\SG,\gamma}^*(\rho;\Xi) = \sum_{\sfx\sfy\in\edges} \Redge^*_{\SG,\gamma}(\sfx\sfy;\rho;\Xi_{\sfx\sfy})$ with
\begin{equation*}
	\Redge^*_{\SG,\gamma}(\sfx\sfy;\rho;\Xi_{\sfx\sfy}) := 
	\frac{\pi_{\sfx}\kappa_{\sfx\sfy}}{2} \int_0^{\Xi_{\sfx\sfy}} \gamma\Lambda_{-1}\bra*{ \frac{\rho_{\sfx}}{\pi_\sfx} \ee^{-\frac{\xi}{2\gamma}},  \frac{\rho_{\sfy}}{\pi_\sfy} \ee^{\frac{\xi}{2\gamma}}} \,(\sfC^*)'\bra*{\frac{\xi}{\gamma}} \dx{\xi}. 
\end{equation*}
Hence, for any viscosity $\gamma>0$, the numerical Scharfetter-Gummel scheme has the gradient structure $(\nodes,\edges,\ona,\gamma \calH(\cdot|\pi),\sfF_\Pot,\calR_{\SG,\gamma})$. The cell-problem~\eqref{eq:SG:cell-problem} is generalized in~\cite{EymardFuhrmannGaertner06} to different mobilities and internal energies providing flux interpolations for nonlinear diffusions. 
We expect a close connection of the so obtained kinetic relations to the tilt-independent gradient structures of the form~\eqref{eq:th:MS-type-result}.

In the zero-viscosity limit $\gamma\to 0$ a non-trivial limiting gradient structure $(\nodes,\edges,\ona,0,\sfF_\Pot$, $\calR_{\Upwind})$ is formally obtained by observing that
\begin{equation*}
	 \gamma \Lambda_{-1}\bra*{ u_{\sfx} \ee^{-\frac{\Xi}{2\gamma}}, u_{\sfy} \ee^{\frac{\Xi}{2\gamma}}} \,2\sinh\bra*{\frac{\Xi}{2\gamma}}
	 \to
	 u_{\sfx}\Xi_+ - u_{\sfy} \Xi_-\qquad\text{as } \gamma\to 0,
\end{equation*}
where $\Xi_+ := \max\set*{0, \Xi}$ and $\Xi_-:=\max\set*{0,-\Xi}$ denote the positive and negative part, respectively.
This defines the well-known \emph{upwind} flux interpolation and it is readily checked that the associated dual dissipation potential is given by
\begin{equation*}
	\calR_{\Upwind}^*(\rho;\Xi) := \frac{1}{4}\sum_{\sfx\sfy\in\edges} \pi_{\sfx}\kappa_{\sfx\sfy} \bra*{ u_{\sfx}\bra*{\Xi_+}^2+u_{\sfx}\bra*{\Xi_-}^2} .
\end{equation*}
Gradient structures based on this dissipation potential are studied from the numerical point of view in~\cite{CancesGallouetTodeschi2019} and from an analytic point of view in~\cite{EspositoPatacchiniSchlichtingSlepcev2021}.
	
The choice of $\theta^\gamma(a,b):=\sqrt{ab}$ in $\Redge^*$ in~\eqref{eq:th:MS-type-result} gives in the zero-viscosity limit $\gamma\to 0$ the kinetic relation
\[
	\rmD_2\Redge^*_{\sfx\sfy}(\rho;\Xi) = \frac{\pi_\sfx \kappa_{\sfx\sfy}}{2} \sqrt{u_{\sfx} u_{\sfy}} \; \Xi ,
\]
which defines a quadratic gradient structure. However, for numerical applications it has the undesirable property that the support of the densities is constant in time and cannot expand. Another choice, resembling some features of the upwind scheme in the zero-viscosity limit, is the Metropolis-Hastings interpolation given in Example~\ref{ex:particular-DB-parametrisations}.\ref{ex:tilt:Metropolis} by $\theta(a,b):=\min\set*{a,b}$, leading the kinetic relation
\[
  \rmD_2\Redge^*_{\sfx\sfy}(\rho;\Xi) = \frac{\pi_\sfx \kappa_{\sfx\sfy}}{2}\bra*{ u_\sfx \, \bONE_{(0,\infty)}(\Xi) - u_\sfy \, \bONE_{(0,\infty)}(-\Xi)} .
\]
We leave a more systematic study of these and further choices to future work.

\newpage

\section{Reduction of two-terminal networks}\label{s:two-terminal-reduction}
\label{s:ha}

\subsection{Setting and definition of gradient system}
\label{ss:ha:setup}

The purpose of this section is to investigate networks of the type of Example~\ref{ex:heat-flow}, in a limit in which the network `collapses' because  nearly all of the jump rates converge to $\infty$. Only two sets of jump rates remain bounded, which are those connected to two `terminals'. In the limit the topology of the network reduces to a very simple chain involving only the two terminals. The challenges are to prove the limit and to characterize the properties of the limit chain. 

\medskip

We recall the setup of a detailed-balance chain on the finite state space $\nodes$ from Example~\ref{ex:heat-flow} and Section~\ref{ss:mod-jump-rates}. 
The stationary measure $\pi\in \ProbMeas^+(\nodes)$ is assumed to be strictly positive, and the jump rates $\kappa:\nodes \times \nodes \to \R_{\geq0}$ are assumed to satisfy the detailed-balance condition~\eqref{eq:Jump:DBC}. We consider the edge set $\edges$ to be implicitly defined by $\edges :=\{\sfx\sfy\in \nodes\times\nodes: \kappa_{\sfx\sfy}>0\}$, and we assume that the resulting graph $(\nodes ,\edges)$ is connected. 

We mark two nodes in $\nodes$ as `terminal nodes' and call them~$\sfa$ and~$\sfb$, 
leading to the disjoint splitting $\nodes = \nodes_{\fast}\dot\cup \nodes_{\term}$ 
with $\nodes_{\term}:=\set{\sfa,\sfb}$ and $\sfV_\fast := \sfV \setminus \sfV_\term$.
We rescale the dynamics of the network in such a way 
that the jump rate from a non-terminal node to any other node is fast, by setting
\begin{equation}\label{eqdef:ha:kappa_eps}
	\kappa^\e_{\sfx \sfy} := \begin{cases}
		\e^{-1} \kappa_{\sfx \sfy}, & \sfx \in \nodes_{\fast},\  \sfy \in \nodes ; \\
		\kappa_{\sfx \sfy},  & \sfx \in \nodes_{\term} , \ \sfy\in \nodes  .
	\end{cases} 
\end{equation}
With this rescaling $\kappa^\eps$ satisfies 
the detailed balance condition with respect to the stationary measure $\pi^\eps$ defined by
\begin{equation}\label{eq:ha:pi_eps}
 \pi^\e_\sfx = \frac{\eps \pi_\sfx}{\PartSum^\e}\text{ for }\sfx\in \nodes_{\fast}, \quad \pi^\e_\sfa= \frac{\pi_\sfa}{\PartSum^\eps},  
 \quad \pi^\e_\sfb= \frac{\pi_\sfb}{\PartSum^\eps}, 
 \quad\text{with}\quad 
 	\PartSum^\eps :=  \pi_\sfa + \pi_\sfb + \eps \!\!\sum_{\sfx \in \nodes_{\fast}}\!\! \pi_\sfx  . 
\end{equation}
The rescaling implies the convergence $\pi^\eps\to \pi^0\in \ProbMeas(\nodes)$ with
\begin{equation*}
	 \pi^0_\sfx = 0 \quad \text{for }\sfx\in\nodes_{\fast}; \qquad \pi^0_\sfa = \frac{\pi_\sfa}{\pi_\sfa+ \pi_\sfb} \quad\text{and}\quad \pi^0_\sfb = \frac{\pi_\sfb}{\pi_\sfa+ \pi_\sfb} . 
\end{equation*}
From this convergence of $\pi^\e$ we expect that the evolution
\begin{equation}\label{eq:ha:master:eps}
	\partial_t \rho_\sfx^\e(t) = \sum_{\sfy\in \nodes} \bra*{\kappa_{\sfy\sfx}^\eps \rho_\sfy^\eps(t)  - \kappa_{\sfx\sfy}^\eps\rho_\sfx^\e(t)}, \qquad \sfx \in \nodes ,
\end{equation}
 converges to a limiting dynamics on the two-node set $\nodes_\term=\set{\sfa,\sfb}$ satisfying
\begin{equation*}
	\partial_t \rho_{\sfa}^0(t) = \kappa_{\sfa\sfb}^0 \rho_\sfb^0(t) - \kappa_{\sfb\sfa}^0 \rho_\sfa^0(t) = - \partial_t \rho_{\sfb}^0(t) ,
\end{equation*}
for some effective rates $\kappa_{\sfa\sfb}^0, \kappa_{\sfb\sfa}^0$ that 
automatically satisfy the detailed balance condition with respect to~$\pi^0$.

\medskip
Note that by~\eqref{eqdef:ha:kappa_eps} and~\eqref{eq:ha:pi_eps} we have 
\begin{equation*}
   \pi^\e_\sfx \kappa^\e_{\sfx\sfy} =: k_{\sfx\sfy}^\e = k_{\sfy\sfx}^\e> 0 \qquad\text{for } \sfx\sfy\in \edges, 
\end{equation*}
and
\begin{equation}\label{eqdef:ha:k0}
	 k_{\sfx\sfy}^\e = \frac{\pi_{\sfx} \kappa_{\sfx\sfy}}{\PartSum^\e} \longrightarrow  \frac{\pi_{\sfx} \kappa_{\sfx\sfy}}{\PartSum^0} = \frac{\pi_\sfx \kappa_{\sfx\sfy}}{\pi_{\sfa} + \pi_{\sfb}} =: k^0_{\sfx\sfy}>0 \qquad\text{as } \eps \to 0 \qquad\text{for } \sfx\sfy\in \edges.
\end{equation}
Hence, the effective unidirectional equilibrium fluxes are of order $O(1)$ in the limit $\eps\to 0$, which  explains their occurrence below in the representation formula for the effective conductivity of the remaining edge $\sfa\sfb$.

\begin{example}[A linear chain]\label{ex:N-chain:illustration}
	Fix $N\geq3$. We consider the setup of an \emph{$N$-chain}, a linear chain of $N$ nodes with nearest-neighbour connections, as depicted in Figure~\ref{fig:MarkvN-2}; this is a generalization of the 3-chain example of~\cite[\S 3.3]{LieroMielkePeletierRenger17}. The terminals are the end points $\nodes_{\term}:=\set{1,N}$, and the remaining nodes $\nodes_{\fast} := \{2,\dots,N\mathord-1\}$ are fast. The only nonzero rates are the nearest-neighbour rates $\kappa_{i,i+1} = \kappa_{i+1,i}>0$. 
	
	Before rescaling, the stationary measure is chosen to be $\pi_i=1/{N}$ for $i\in \nodes$. 
	After the rescaling~\eqref{eqdef:ha:kappa_eps} we then have 
	\[
	\pi_1^\eps=\frac{1}{2+\eps(N-2)}=\pi_N^\eps \qquad\text{and}\qquad \pi_i^\eps = \frac{\eps}{2+\eps(N-2)}\quad\text{for } i=2,\dots, N\mathord-1.
	\]
	Hence, in the limit $\eps\to 0$ we obtain $\pi_1^0={1}/{2}=\pi_N^0$ and $\pi_i^0=0$ for $i=2,\dots,N-1$.
	\let\k\kappa
	\begin{figure}[ht]
		\centering
		\begin{tikzpicture}[scale=0.8]
			\tikzstyle{every node}=[draw,shape=circle,fill=black!15];
			\node (v1) at (2,0) {$\quad1\quad$};
			\node (v2) at (4,0) {$2$};
			\node (v3) at (6,0) {$3$};
			\node (vN-1) at (10,0){\footnotesize{$\!\!\!N\!\!-\!\!1\!\!\!$}};
			\node (vN) at (12,0) {$\quad N\quad$};
			\node (vv1) at (16,0) {$\quad1\quad$};
			\node (vvN) at (20,0) {$\quad N\quad$};
			\tikzstyle{every node}=[];
			
			\node (vdot) at (8,0) {$\cdots$};
			
			\draw[->, very thick] (v1) to [out=30, in=135] (v2);
			\draw[->, very thick] (v2) to [out=45, in=135] (v3);
			\draw[->, very thick] (v2) to [out=-135, in=-30] (v1);
			\draw[->, very thick] (v3) to [out=45, in=135] (vdot);
			\draw[->, very thick] (v3) to [out=-135, in=-45] (v2);
			\draw[->, very thick] (vdot) to [out=-135, in=-45] (v3);
			\draw[->, very thick] (vN-1) to [out=-135, in=-45] (vdot);
			\draw[->, very thick] (vdot) to [out=45, in=135] (vN-1);
			\draw[->, very thick] (vN-1) to [out=45, in=150] (vN);
			\draw[->, very thick] (vN) to [out=-150, in=-45] (vN-1);
			
			\draw[->, very thick] (vv1) to [out=30, in=150] (vvN);
			\draw[->, very thick] (vvN) to [out=210, in=-30] (vv1);
			
			\draw[shorten >=0.3cm, shorten <=0.3cm,->, double, very thick] (vN) to [out=0, in=180] (vv1);
			
			\node[above] at (3,.7) {$\k_{12}$};
			\node[below] at (3,-.7) {$\frac{\k_{12}}\eps$};
			\node[above] at (5,.7) {$\frac{\k_{23}}\eps$};
			\node[below] at (5,-.7) {$\frac{\k_{23}}\eps$};
			\node[above] at (7,.7) {$\frac{\k_{34}}\eps$};
			\node[below] at (7,-.7) {$\frac{\k_{34}}\eps$};
			\node[above] at (9,.7) {$\frac{\k_{N-2,N-1}}\eps$};
			\node[below] at (9,-.7) {$\frac{\k_{N-2,N-1}}\eps$};
			\node[above] at (11,.7) {$\frac{\k_{N-1,N}}\eps$};
			\node[below] at (11,-.7) {$\k_{N-1,N}$};
			
			\node[above] at (18,0.7) {$\kappa_{1,N}^0$};
			\node[below] at (18,-0.7) {$\kappa_{N,1}^0$};

			\node[above] at (14,0.25) {$\eps\to 0$};
		\end{tikzpicture}
		\caption{\emph{Left:} $N$-state Markov process with high rate of leaving states
			$2, \dots, N\mathord-1$. \\
			\emph{Right:} The limit for $\eps\to 0$ gives a two-state Markov process.}
		\label{fig:MarkvN-2}
	\end{figure}%
	
\indent
We return to this example after proving the main Theorem~\ref{thm:ha:capacity} below.
\end{example}



We take as gradient system for~\eqref{eq:ha:master:eps} that of Example~\ref{ex:heat-flow} with a specific choice of tilting, which we now describe. We consider the  class of `potential' tilts $\sfF_{\mathrm{Pot}} = \{\calF_\Pot^F(\rho) = \sum_{\sfx\in\nodes} \rho_\sfx F_\sfx\}$ as defined in~\eqref{eqdef:MCs:potential-tilts}, so that the tilted energy can be written as 
\begin{equation*}
	\calE_{\e,F}(\rho) := \calE_\e(\rho) + \calF^F_\Pot(\rho)  = \calH(\rho \,|\, \pi^{\e,F}) + \log \frac{\PartSum^\e}{\PartSum^{\e,F}} ,
\end{equation*}
where the tilted stationary measures satisfy
\begin{equation*}
 \pi^{\e,F}_\sfx = \frac1{\PartSum^{\eps,F}}{\pi^{\e}_\sfx \ee^{-F_\sfx}}\quad\text{for }\sfx\in \nodes ,\qquad\text{with}\quad \PartSum^{\e,F} := \sum_{\sfx\in \nodes} \pi^\e_\sfx \ee^{-F_\sfx} . 
\end{equation*}
As discussed in Section~\ref{ss:mod-jump-rates}, there exist various choices of how to tilt the jump rates while remaining compatible with detailed-balance stationarity of $\pi^{\e,F}$; here we follow the tilt from Example~\ref{ex:particular-DB-parametrisations}.\ref{ex:tilt:sym}, 
which by Proposition~\ref{prop:activity-tilt} is the only one that leaves the activity invariant:
\begin{equation*}
	\kappa^{\e,F}_{\sfx\sfy} := \kappa^{\e}_{\sfx\sfy} \ee^{\frac12 (F_\sfx-F_\sfy)} .
\end{equation*}
We leave the other possibilities of tilting as identified in Proposition~\ref{prop:Jump:tilt} to be discussed elsewhere.
The tilted evolution of~\eqref{eq:ha:master:eps} is then given by
\begin{equation}\label{eq:ha:master:eps:F}
	\partial_t \rho_\sfx^{\e,F}(t) = \sum_{\sfx\in \nodes} \bra*{\kappa_{\sfy\sfx}^{\eps,F} \rho_\sfy^{\eps,F}(t)  - \kappa_{\sfx\sfy}^{\eps,F} \rho_\sfx^{\e,F}(t)}, \qquad \sfx \in \nodes .
\end{equation}
By the characterization of tilt-independence of Theorem~\ref{th:MC:gradient:tilt}, 
the only gradient structure for~\eqref{eq:ha:master:eps:F} that is independent of the tilt $F\in\sfF_{\mathrm{Pot}}$ is given by the dissipation potential
\begin{equation*}
	\calR_\e(\rho,j) =\sum_{\sfx\sfy\in \edges}\sfC\bra*{j_{\sfx\sfy} \,\middle|\, \sigma_{\sfx\sfy}^\eps(\rho)}
	\qquad\text{with}\qquad
	\sigma_{\sfx\sfy}^\eps(\rho) =  \frac{1}{2}\sqrt{\kappa^\e_{\sfx\sfy}\kappa^\e_{\sfy\sfx}\rho_\sfx  \rho_\sfy} . 
\end{equation*}
This tilt-independence also is obvious from the observation that indeed
$\sqrt{\kappa^{\e,F}_{\sfx\sfy} \kappa^{\e,F}_{\sfy\sfx} } = \sqrt{\kappa^\e_{\sfx\sfy}\kappa^\e_{\sfy\sfx}}$.
In summary, at finite $\e>0$ we fix the gradient structure with tilts $(\nodes,\edges,\gnabla,\calE_\e,\sfF_{\mathrm{Pot}},\calR_\e)$ given above.

\subsection{EDP convergence}

The basis of EDP convergence for tilt gradient systems is the pair of functions $\calI_{\e,F}^T$ and $\calD_{\e,F}^T$ (see Section~\ref{ss:convergence-of-GS}).
The EDP functional $\calI_{\e,F}^T$ is given by 
\begin{equation*}
	\calI_{\e,F}^T(\rho,j) := \calE_{\e,F}(\rho(T)) - \calE_{\e,F}(\rho(0)) + \calD_{\e,F}^T(\rho,j),
\end{equation*}
and the dissipation functional $\calD_{\e,F}^T$, which has the formal definition 
\begin{equation*}
	\calD_{\e,F}^T(\rho,j) := \int_0^T \pra*{ \calR_{\e}(\rho,j) + \calR_{\e}^*(\rho, - \gnabla \rmD \calE_{\e,F})} \dx{t}.
\end{equation*}
can be defined rigorously (see~\eqref{eq:Rstar:Hellinger}) as
\begin{multline*}
	\calD_{\e,F}^T(\rho^\e,j^\e) :=  \int_0^T \sum_{\sfx\sfy\in \edges} 
\set*{\sfC\bra*{j^\eps_{\sfx\sfy} \middle| 
        \tfrac12 k^{\e,F}_{\sfx\sfy}\sqrt{u^{\e,F}_\sfx u^{\e,F}_\sfy}} 
      + k^{\e,F}_{\sfx\sfy}\bra*{ \sqrt{u^{\e,F}_\sfx} - \sqrt{u^{\e,F}_\sfy}}^2} \dx{t},\\
      \text{provided }(\rho^\e,j^\e)\in \CE(0,T),
\end{multline*}
with the notation
\begin{equation*}
	u^\e := \pderiv{\rho}{{\pi^\e}}, \qquad u^{\e,F} := u^\e \ee^{F},
	  \qquad\text{and}\qquad k^{\e,F}_{\sfx\sfy} := k^\e_{\sfx\sfy} \ee^{-\frac{F_\sfx+F_\sfy}{2}}.
\end{equation*}
Note that in view of~\eqref{eqdef:N:explicit:rigrous} we also have 
\begin{equation}
\label{eq:ha:D-G}
\calD_{\e,F}^T(\rho^\e,j^\e) :=  \int_0^T \sum_{\sfx\sfy\in \edges} 
\CCs\bra*{j^\e_{\sfx\sfy}, u_{\sfx}^{\e,F}, u_\sfy^{\e,F}; \tfrac12 k^{\e,F}_{\sfx\sfy}} \dx{t},\qquad 
      \text{provided }(\rho^\e,j^\e)\in \CE(0,T).
\end{equation}
Whenever $\calD_{\e,F}^T(\rho,j)<\infty$, the superlinearity of $\sfC$ implies that $j\in L^1(0,T;\edges)$ and therefore $t\mapsto \rho(t)$ is absolutely continuous.

\medskip

\noindent
We show in the sections below that 
\begin{enumerate}
\item Sequences $(\rho^\e,j^\e)$ along which  $\calI_{\e,F}^T$ and $\calE_{\e,F}(\rho^\e(0))$ are  bounded are compact in an appropriate sense;
\item The energies $\calE_\e$ $\Gamma$-converge, and in the limit the pair $(\rho,j)$ satisfies a contracted continuity equation;
\item The dissipation functionals $\calD_{\e,F}^T$ satisfy a lower bound.
\end{enumerate}
Together these establish EDP convergence of gradient systems with tilting, as given by Definitions~\ref{defn:EDP-convergence} and~\ref{defn:EDP-convergence:reduce}.

\subsubsection{Compactness}

\begin{lemma}[Compactness]
\label{l:ha:compactness}
Fix $F\in \sfF_{\mathrm{Pot}}$, let $\e_n\to0$, and let $(\rho^{\eps_n},j^{\eps_n})\in \CE(0,T)$ be a sequence such that 
\begin{equation}
\label{est:ha-compactness}
\sup_{n} \calE_{\e_n,F}(\rho^{\e_n}(0)) + \calD_{\e_n,F}^T(\rho^{\eps_n},j^{\eps_n})\;=: C \;<\;\infty.
\end{equation}
Then there is a subsequence (not relabeled) along which $(\rho^{\e_n},j^{\e_n})$ converges in $\CE(0,T)$ to a limit $(\rho,j)\in \CE(0,T)$ (see Def.~\ref{def:converge-in-CE}). More precisely,
\begin{enumerate}
\item \label{l:ha:compactness:j}
	For each $\sfx\sfy\in \edges$, $j_{\sfx\sfy}^{\e_n}$ converges narrowly to $j_{\sfx\sfy}\in \calM([0,T])$, i.e.
\[
\int_0^T\varphi(t) j^{\e_n}_{\sfx\sfy}(t) \dx t \longrightarrow
\int_0^T \varphi(t) j_{\sfx\sfy}(\dx t)
\qquad\text{ for all $\varphi\in C([0,T])$.}
\]
The limit $j$ satisfies $(\odiv j)(\sfx,A)=0$ for all $\sfx\in\nodes_\fast$ and $A\subset [0,T]$.
\item \label{l:ha:compactness:uab} 
For $\sfx\in\nodes_\term$, $u_\sfx^{\e_n,F}$ converges in $C([0,T])$ to a limit $u_\sfx^F$.
\item \label{l:ha:compactness:u-rest}
For each $\sfx\in \nodes_\fast$, $u_\sfx^{\e_n,F}$ converges narrowly to a finite measure on $[0,T]$, and $\rho^{\e_n}_\sfx$ converges narrowly to zero.
\end{enumerate}
\end{lemma}

\begin{proof}
The proof follows the arguments of~\cite[Cor.~3.10]{PeletierRenger21}.
We suppress the subscript~$n$ from $\e_n$ for simplicity. 

At the expense of doubling the constant $C$ in~\eqref{est:ha-compactness} we can assume that all $j^\e_{\sfx\sfy}$ are non-negative. Since $\pi^{\e,F}_{\sfa,\sfb}$ are bounded away from zero, $u^{\e,F}_{\sfa,\sfb}$ are bounded in $L^\infty(0,T)$ independently in $\e$.
For a neighbour~$\sfx$ of $\sfa$ we then use~\eqref{est:ha-compactness} and the inequality
\[
\int_0^T u^{\e,F}_{\sfx}(t)\dx t \leq 
2\int_0^T u^{\e,F}_{\sfa}(t)\dx t
+ 2\int_0^T \bra*{\sqrt{u^{\e,F}_\sfa} - \sqrt{u^{\e,F}_\sfx}}^2\dx t
\]
to bound $u^{\e,F}_\sfx$ in $L^1(0,T)$. By the connectedness of the graph we can extend this boundedness to all $\nodes_\fast$ and find
\begin{equation}
\label{est:ha-compactness:u-Vfast}
\sup_{\e>0,\sfy\in \nodes_\fast} \int_0^T u^{\e,F}_\sfy\dx t < \infty.
\end{equation}
Along a subsequence we can assume that for all $\sfy\in \nodes_\fast$, $u^{\e,F}_\sfy$  converges narrowly to a finite measure, and since $\pi^\e$ converges to zero on $\nodes_\fast$, $\rho^\e$ converges narrowly to zero on $\nodes_\fast$. This proves part~\ref{l:ha:compactness:u-rest}.

From the inequality $|\xi j|\leq \sfC(j|a) + a\sfC^*(\xi)$ that holds for all $j,\xi\in\R$ and $a\geq 0$ we deduce that 
\[
\int_0^T |j^\e_{\sfx\sfy}| \leq \int_0^T \pra*{
  \sfC\bra*{j^\e_{\sfx\sfy}|\sigma^\e_{\sfx\sfy} } 
+ \sigma^\e_{\sfx\sfy} \sfC^*(1)} \dx t
\stackrel{\eqref{est:ha-compactness}}\leq C , 
\qquad \sigma^\e_{\sfx\sfy} = \tfrac12 k^{\e,F}_{\sfx\sfy}\sqrt{u^{\e,F}_\sfx u^{\e,F}_\sfy},
\]
so that all fluxes $j^\e_{\sfx\sfy}$ are bounded in $L^1(0,T)$, and can be assumed to converge narrowly in $\calM([0,T])$. This proves the convergence of part~\ref{l:ha:compactness:j}; the vanishing of the divergence on $\nodes_\fast$ follows from the continuity equation, because $\rho$ is zero on $\nodes_\fast$.

For the terminal-origin fluxes $j^\e_{\sfa\sfx}$, $j^\e_{\sfx\sfa}$, $j^\e_{\sfb\sfy}$, and $j^\e_{\sfy\sfb}$ we derive a stronger bound. First we note that for any $\sfx\in\nodes_\fast$ and any $\xi\in C_b([0,T])$,
\begin{align*}
\int_0^T \xi(t)j^\e_{\sfa\sfx}(t) \dx t
&\leq \int_0^T \sfC\bra*{j^\e_{\sfa\sfx}\big| \sigma^\e_{\sfa\sfx}}\dx t
+ \int_0^T \sigma^\e_{\sfa\sfx}{\sfC^*(\xi(t))} \dx t\\
&\leftstackrel{\eqref{est:ha-compactness}} \leq C 
+ \frac14 \bra*{k_{\sfa\sfx}^{\e,F}}^2\norm{u^{\e,F}_\sfa}_\infty  \int_0^T u^{\e,F}_\sfx\dx t 
+ \frac14 \int_0^T \sfC^*(\xi(t))^2 \dx t.
\end{align*}
The second term on the right-hand side is bounded by~\eqref{est:ha-compactness:u-Vfast} and~\eqref{eqdef:ha:k0}, and it then follows (see e.g.~\cite[(2.8)]{PeletierRenger21}) that $j^\e_{\sfa\sfx}$ is bounded in the Orlicz space $L^\Psi(0,T)$, where $\Psi:\R\to\R$ is defined through its Legendre dual $\Psi^*(\xi) := \sfC^*(\xi)^2$.
The superlinearity of $\Psi$ implies that therefore $j^\e_{\sfa\sfx}$ is equi-integrable  for all $\sfx\in \nodes$, and similarly for $j^\e_{\sfb\sfy}$. 

To bound $j^\e_{\sfx\sfa}$ and $j^\e_{\sfy\sfb}$ we remark that by the continuity equation, 
\begin{align*}
\sum_{\sfx\in\nodes_\fast}{\partial_t \rho^\e_\sfx}
&= \sum_{\substack{\sfx\in \nodes_\fast\\ \sfy\in \nodes}} 
  \bra*{j^\e_{\sfy\sfx}-j^\e_{\sfx\sfy}}
= \sum_{\substack{\sfx\in \nodes_\fast\\ \sfy\in \nodes_\term}} 
  \bra*{j^\e_{\sfy\sfx}-j^\e_{\sfx\sfy}}
= \sum_{\sfx\in\nodes_\fast}\bra*{j^\e_{\sfa\sfx} + j^\e_{\sfb\sfx}}
- \sum_{\sfx\in\nodes_\fast}\bra*{j^\e_{\sfx\sfa} + j^\e_{\sfx\sfb}}.
\end{align*}
The left-hand side converges to zero weakly as $\e\to0$, and the first term on the right-hand side is equi-integrable as we just observed. Therefore there exists a modulus of continuity $\omega:[0,\infty)\to[0,\infty)$ such that 
\begin{equation}
\label{est:ha:compactness:approx-equicontinuity}
\limsup_{\e\to0} \sup_{\substack{0\leq t_0< t_1\leq T:\\|t_1-t_0|<\delta}}\;
\int_{t_0}^{t_1}\sum_{\sfx\in \nodes_\fast}\bra*{ j^\e_{\sfx\sfa} + j^\e_{\sfx\sfb}}\dx t 
\leq \omega(\delta).
\end{equation}
Combining~\eqref{est:ha:compactness:approx-equicontinuity} with the equi-integrability of $j^\e_{\sfa\sfx}$ and the continuity equation $\pi^\e_\sfa \partial_t u^\e_\sfa = -\odiv j^{\e}(\sfa)$, we conclude by a modified Arzela-Ascoli theorem~\cite[Th.~A.1]{PeletierRenger21} that $u^\e_\sfa$  converges (along a subsequence)  in $C([0,T])$. This proves part~\ref{l:ha:compactness:uab} of the Lemma.

Note that $t\mapsto \rho_\sfx(t)$ is continuous: for $\sfx\in\nodes_\fast$ this follows because $\rho_\sfx(t)=0$, and for $\sfx=\sfa,\sfb$ this follows from part~\ref{l:ha:compactness:uab}. Therefore the convergence is in the sense of $\CE(0,T)$.
\end{proof}

\subsubsection{Convergence of energy and contracted continuity equation}
\label{ss:ha:contracted-CE}

The convergence of $\pi^\e \to \pi^0$ implies the $\Gamma$-convergence of $\calE_{\e}$ in $\ProbMeas(\nodes)$ (see e.g.~\cite[Lemma~6.2]{AmbrosioSavareZambotti09})
\begin{equation*}
	\calE_\e = \calH(\cdot | \pi^\e) \xrightarrow{\Gamma} \calH(\cdot | \pi^0) =: \calE_0.
\end{equation*}
Consequently we also get for any $F\in \sfF_{\mathrm{Pot}}$ the convergence
\[
	\calE_{\e,F} = \calE_\e + \calF^F_\Pot \xrightarrow{\Gamma} \calE_0 + \calF^F_\Pot  =: \calE_{0,F}.
\]

Note that $\supp \pi^0 = \nodes_{\term}$ and hence for any $\rho$ with $\calE_0(\rho)<\infty$ we also have $\supp \rho\subseteq \nodes_{\term}$.
It follows that  if $(\rho^\eps,j^\eps)\in \CE(0,T)$ converges to $(\rho,j)\in\CE(0,T)$ and $\sup_{\eps,t} \calE_\eps(\rho^\eps(t))<\infty$, then in the limit $\rho_\sfx\equiv0$ for any $\sfx\in\nodes_\fast$. We therefore define
\begin{subequations}
\label{eqdef:ha:contracted-CE}
\begin{alignat}3
\edges_\term &:= \{\sfa\sfb\}, &\qquad 
\wt\rho &:= \rho|_{\nodes_\term}, 	&
\wt\calE_\e(\wt\rho\,) := \calE_\e(\rho),
\\
(\ona f)_{\sfa\sfb} &:= f_\sfb-f_\sfa, &
\wt\jmath_{\sfa\sfb} &:= (\odiv j)(\sfa),&
\end{alignat}
\end{subequations} 
Then $(\wt\rho,\wt\jmath\,)$ satisfies the contracted continuity equation associated with $(\nodes_\term,\edges_\term, \gnabla)$, since 
\[
\partial_t \wt\rho_\sfa = -\partial_t \wt\rho_\sfb = \partial_t \rho_\sfa = -(\odiv j)(\sfa)
= 
-\wt\jmath_{\sfa\sfb} = -(\odiv \wt\jmath\,)(\sfa).
\]

\subsubsection{Lower bound on the dissipation function}

The main argument in the proof of the lower bound on $\calD_{\e,F}^T$ is the following. By applying a simple minimization argument to~\eqref{eq:ha:D-G} we find
\begin{align}\label{ineq:ha:D-to-cell-pb}
	\MoveEqLeft\calD_{\e,F}^T(\rho^\eps,j^\eps) 
	=\int_0^T \sum_{\sfx\sfy\in\edges} 
		\CCs\bra*{j_{\sfx\sfy}^\e;u^{\e,F}_\sfx,u^{\e,F}_\sfy; \tfrac{1}{2} k_{\sfx\sfy}^{\e,F} } \dx{t} \\
	&\geq \int_0^T \inf_{\substack{ u\in\calM_{\geq 0}(\nodes)\\ j\in \calM(\edges)}}\set[\bigg]{  \sum_{\sfx\sfy\in\edges} \CCs\bra*{j_{\sfx\sfy}; u_\sfx, u_\sfy; \tfrac{1}{2} k_{\sfx\sfy}^{\e,F}} \dx{t}:  u|_{\nodes_{\term}} = u^{\e,F}|_{\nodes_{\term}},\  \gdiv j=\gdiv j^\eps}  \dx t .
\notag
\end{align}
In the limit $\e\to0$ the divergence $\odiv j^\e$ concentrates onto $\nodes_{\term} = \{\sfa,\sfb\}$, and the infimum above takes the form (see Def.~\ref{defn:ha:cell-formula} below)
\begin{align}
\label{eq:ha:cell-problem}
&\scrN[\nodes,\nodes_\term,k]\bra*{u, s} = \\
&\inf_{\substack{ \ol u\in\calM_{\geq 0}(\nodes)\\ j\in \calM(\edges)}}
\set[\bigg]{ \sum_{\sfx\sfy\in\edges} \CCs\bra*{j_{\sfx\sfy}; \ol u_\sfx, \ol u_\sfy; \tfrac{1}{2} k_{\sfx\sfy}^{F}} \dx{t}:  
    \ol u|_{\nodes_{\term}} \!= u^{F}|_{\nodes_{\term}},  \ 
    \gdiv j|_{\nodes_\term} \!= (s,-s) , \  
    \gdiv j|_{\nodes_\fast} \!= 0 }. \notag
\end{align}

\medskip

A truly remarkable fact is that this `cell problem'  can again be characterized by a $\sfC$-$\sfC^*$ structure, with a modified activity $k$ given by a capacity:

\begin{prop}[Effective dissipation density]\label{prop:ha:effectiveN}
The infimum in~\eqref{eq:ha:cell-problem}
is equal to 
\begin{equation*}
\scrN[\nodes,\nodes_\term,k]\bra*{u, s} 
= \CCs\bra*{s,u_\sfa,u_\sfb;\capacity^F_{\sfa\sfb}} ,
\end{equation*}
	with
	\begin{equation}\label{eqdef:effective:capacity}
		\capacity^F_{\sfa\sfb} := 
		\inf\set[\bigg]{ \frac{1}{2} \sum_{\sfx\sfy\in \edges} k_{\sfx\sfy}^F 	\abs*{\gnabla_{\sfx\sfy} h }^2 \,\bigg|\, 
			h:\nodes \to \R, h_\sfa=1, h_\sfb=0} ,
	\end{equation}
	where $k_{\sfx\sfy}^F = k_{\sfx\sfy}^0 \ee^{-(F_\sfx+F_\sfy)/2}$ and 
	$k_{\sfx\sfy}^0=\frac{\pi_\sfx \kappa_{\sfx\sfy}}{\pi_{\sfa} + \pi_{\sfb}}$ as in~\eqref{eqdef:ha:k0}.
\end{prop}

\noindent The proof of Proposition~\ref{prop:ha:effectiveN} is given in Section~\ref{ss:ha:cell-formula} below. 

\bigskip

The following lemma makes the transition from~\eqref{ineq:ha:D-to-cell-pb} to~\eqref{eq:ha:cell-problem} rigorous. 

\begin{lemma}[Lower bound]
Fix $F\in \sfF_{\mathrm{Pot}}$. Let $\e_n\to0$,  let $(\rho^{\e_n},j^{\e_n})\in \CE(0,T)$ converge to $(\rho,j)\in \CE(0,T)$ in the sense of Lemma~\ref{l:ha:compactness}, and let $(\wt\rho,\wt\jmath\,)$ be associated to $(\rho,j)$ as in~\eqref{eqdef:ha:contracted-CE}. 
Then 
\[
\liminf_{n\to\infty} \calD_{\e_n,F}^T(\rho^{\eps_n},j^{\eps_n})
\geq 
\wt\calD_{0,F}^T(\wt\rho,\wt\jmath\,).
\]	
The limit functional $\wt\calD_{0,F}^T$ is defined by 
\begin{equation*}
\wt\calD_{0,F}^T(\wt\rho,\wt\jmath\,):= 
\begin{cases}
\ds\int_0^T \CCs\bra*{\wt\jmath_{\sfa\sfb}, u_\sfa^F,u_\sfb^F; \capacity_{\sfa\sfb}^F}\dx t,
&\text{provided $(\wt\rho,\wt\jmath\,)\in {\CE_{\term}(0,T)}$},\\
+\infty & \text{otherwise}.
\end{cases}
\end{equation*}
In this definition $u^{\e}_\sfx = (\dx\rho/\dx\pi^0)(\sfx) = (\dx\wt\rho/\dx\pi^0)(\sfx)$ and $u^{\e,F}_\sfx = u^\e_\sfx \ee^{F_\sfx}$ for $\sfx\in \nodes_\term$.
\end{lemma}

\begin{proof}
We can assume without loss of generality that $\sup_n \calD_{\e_n,F}^T(\rho^{\e_n},j^{\e_n})<\infty$. By the compactness of Lemma~\ref{l:ha:compactness} we can take a subsequence along which $j^{\e_n}$, $\rho^{\e_n}$,  and $u^{\e_n,F}$ converge in the sense given by the Lemma. The limit $j$ is an $\R^\edges$-valued measure on $[0,T]$, the limit $\rho$ is a measure concentrated on $[0,T]\times \nodes_\term$; for the tilted density $u^F$ we will only be using the values on $\nodes_\term=\{\sfa,\sfb\}$, given by $u^F_\sfx(t) = \ee^{F_\sfx} \dx\rho/\dx\pi^0(t,\sfx)$ for $\sfx\in\{\sfa,\sfb\}$. Again we suppress the subscript $n$ from $\e_n$.
As discussed in Section~\ref{ss:ha:contracted-CE}, by defining $(\wt\rho,\wt\jmath\,)$ as in~\eqref{eqdef:ha:contracted-CE} we have $(\wt\rho,\wt\jmath\,)\in \CE_\term(0,T)$.

Following the discussion in Section~\ref{ss:lsc-G} we define the Hilbert space $\calZ := H^{-1}(0,1)^\edges \times \R^\nodes $ and the $\calZ$-valued measure $\mu^\e$ on $[0,T]$ given by
\[
\dual \varphi {\mu^\e}
:= \int_0^T \pra[\bigg]{\;
\sum_{\sfx\sfy\in\edges} \dual{\varphi_{\sfx\sfy}^1(t)}\bONE j^\e_{\sfx\sfy}(t)
  + \sum_{\sfx\in\nodes} \varphi^2_{\sfx}(t)u^{\e,F}_{\sfx}(t) 
}\dx t
\]
for all $\varphi \in C_b\bra*{[0,T];H^1_0(0,1)^\edges \times \R^\nodes}$.
The {compactness} given in Lemma~\ref{l:ha:compactness} implies that $\mu^\e$ converges narrowly to a limit $\mu\in \calM([0,T];\calZ)$. Note that while $\mu^\e$ is Lebesgue absolutely continuous in $t$, the limit $\mu$ need not be.

For each $\sfx\sfy\in\edges$ we have $k_{\sfx\sfy}^{\e,F}\to k_{\sfx\sfy}^F$. 
By applying Lemma~\ref{l:G-lsc} with $A=[0,T]$  we find that 
\begin{equation}
\label{ineq:ha:compactness:1}
\liminf_{\e\to0} \calD_{\e,F}^T (\rho^\e,j^\e)
\geq
\int\limits_{[0,T]} \sum_{\sfx\sfy\in \edges} 
\CCs\bra*{\frac{\dx\mu}{\dx |\mu|}(t); \tfrac{1}{2}k^{F}_{\sfx\sfy}} |\mu|(\dx{t}).
\end{equation}
Note that we explicitly include the boundary points $0$ and $T$ in the integral above, because we have narrow convergence on $[0,T]$, not necessarily on $(0,T)$; for instance, initial boundary layers may lead to concentration of $|\mu|$ at $t=0$ (see Remark~\ref{rem:well-preparedness-initial-data}).

We next show that we can contract the sum above using Proposition~\ref{prop:ha:effectiveN}.
Fix $\psi\in H^1_0(0,1)$ with $\int_0^1 \psi = 1$.
Given a $\varphi\in C^1_c((0,T);\R^\nodes)$, define $\eta_{\sfx}\in C_c^1((0,T);H^1_0(0,1))$ by $\eta_{\sfx}(t) := \psi \varphi_{\sfx}(t)$ for each $\sfx\in \nodes$. We then calculate
\begin{align*}
\int_0^T \sum_{\sfx\sfy\in\edges} 
\dual{\ona_{\sfx\sfy} \eta(t)}{\mu^1_{\sfx\sfy}(\dx t)}
&= \lim_{\e\to0} 
\int_0^T \sum_{\sfx\sfy\in\edges} 
\dual{\ona_{\sfx\sfy} \eta(t)}{\bONE}j^\e_{\sfx\sfy}(t)\dx t
= \lim_{\e\to0} 
\int_0^T \sum_{\sfx\sfy\in\edges} 
 \ona_{\sfx\sfy} \varphi(t)j^\e_{\sfx\sfy}(t)\dx t\\
&= -\lim_{\e\to0} 
\int_0^T \sum_{\sfx\in\nodes} 
 \partial_t\varphi_\sfx(t)\rho^\e_{\sfx}(t)\dx t
= \int_0^T \sum_{\sfx\in \nodes} 
 \partial_t\varphi_\sfx(t)\rho_{\sfx}(t)\dx t\\
 &= \int_0^T \sum_{\sfx\in \nodes_\term} 
 \partial_t\varphi_\sfx(t)\rho_{\sfx}(t)\dx t.
\end{align*}
Therefore the measure $\odiv \mu^1$ vanishes on $(0,T)\times \nodes_\fast$ and equals $-\partial_t \rho$ in the sense of distributions on $(0,T)\times \nodes$. Define the $\calY$-valued measure $\nu$ on $[0,T]$ by
\[
\nu(\dx t) :=  \bra*{  \odiv\bra*{\frac{\dx\mu^1}{\dx|\mu|}}(t,\sfa), 
\frac{\dx\mu^2_{\sfa}}{\dx|\mu|}(t), \frac{\dx\mu_{\sfb}^2}{\dx|\mu|}(t)}|\mu|(\dx t).
\]
We then calculate
\begin{align*}
\liminf_{\e\to0} \calD_{\e,F}^T (\rho^\e,j^\e)\quad
&\leftstackrel{\eqref{ineq:ha:compactness:1}} \geq	 \int\limits_{[0,T]} \sum_{\sfx\sfy\in \edges} 
\CCs\bra*{\frac{\dx\mu}{\dx |\mu|}(t); \tfrac12 k^{F}_{\sfx\sfy}} |\mu|(\dx{t})\\
&\geq 
 \int\limits_{(0,T)} \sum_{\sfx\sfy\in \edges} 
\CCs\bra*{\frac{\dx\mu}{\dx |\mu|}(t);  \tfrac12 k^{F}_{\sfx\sfy}} |\mu|(\dx{t})\\
&\leftstackrel{\text{Prop.~\ref{prop:ha:effectiveN}}}= 
\int\limits_{(0,T)} \CCs\bra*{\frac{\dx\nu^1}{\dx|\nu|}(t), 
\frac{\dx\nu^2}{\dx|\nu|}(t), \frac{\dx\nu^3}{\dx|\nu|}(t); \capacity_{\sfa\sfb}^F}|\nu|(\dx t).
\end{align*}
Finally, the convergence of $u^{\e,F}_{\sfa,\sfb}$ in $C([0,T])$ implies that $\mu^2_{\sfa}(\dx t) = \nu^2(\dx t) = u^{F}_{\sfa}(t)\dx t$, and similarly for $\nu^3$. Therefore $\nu^{2,3}$ are Lebesgue absolutely continuous on $[0,T]$, and by part~\ref{l:G-lsc:part2} of Lemma~\ref{l:G-lsc} the same is true for $\nu^1$. Writing $\nu^1 (\dx t)=: \wt\jmath_{\sfa\sfb}\dx t$ 
as in~\eqref{eqdef:ha:contracted-CE} we then have 
\[
\liminf_{\e\to0} \calD_{\e,F}^T (\rho^\e,j^\e)
\geq
\int\limits_{(0,T)}\CCs\bra*{\wt\jmath_{\sfa\sfb}(t), 
u^F_\sfa(t), u^F_\sfb(t); \capacity_{\sfa\sfb}^F}\dx t
= \calD_{0,F}^T(\wt\rho,\wt\jmath\,).
\qedhere
\]
\end{proof}

\subsection{Main result and discussion}\label{ss:ha:EDP:discussion}

With this, we can conclude the EDP convergence statement. 

\begin{theorem}[EDP convergence for two-terminal networks]\label{thm:ha:capacity}
The gradient system with tilting $(\nodes,\edges,\gnabla,\calE_\e,\sfF_{\mathrm{Pot}},\calR_\e)$ 
as described in Section~\ref{ss:ha:setup} EDP-converges to the contracted gradient system with tilting
$(\nodes_{\term},\edges_{\term},\gnabla,\wt\calE,\sfF_{\mathrm{Pot}},\wt\calR)$, where for $F\in \sfF_{\mathrm{Pot}}$ the dissipation potential is defined as 
\begin{equation}\label{eq:ha:effective-R}
	\wt \calR(\wt\rho,\wt\jmath;F) =  
	\begin{cases}
	\sfC\bra*{\wt\jmath_{\sfa\sfb} \middle| \,\capacity^F_{\sfa\sfb} \sqrt{u^F_{\sfa} u^F_{\sfb}}} & \text{if }(\wt\rho,\wt\jmath)\in \CE_\term(0,T),\\
	+\infty &\text{otherwise}.
	\end{cases}
\end{equation}
\end{theorem}

This theorem shows that gradient structures with dissipation potentials based on the functions $\sfC$-$\sfC^*$ 
are stable under minimization over the values at intermediate nodes; effectively this theorem gives a far-reaching generalization of  the simple series and parallel laws of Corollary~\ref{cor:N:series} and~\ref{cor:N:parallel}.

The use of potential theory (capacities) and the use of electrical networks 
to analyze and represent Markov chains on graphs is a classical topic, see for instance 
\cite{DoyleSnell1984}, \cite[\S 9.4]{LevinPeres2017}, \cite[\S 2.3]{LyonsPeres2016} and \cite[\S 1.3]{Grimmett2018}. 
This observation gives fundamental interconnections among Markov chains, linear algebra, graph theory and physics, and has a long history~\cite{Ohm1827,Kennelly1899,Kron1939,Truemper1992,Ziegler1995,CurtisIngermanMorrow1998};
see also~\cite{DorflerBullo2013} for more recent applications in electrical engineering.

The specific role of the capacity as defined in~\eqref{eqdef:effective:capacity} emerges 
also as the effective conductivity
in the investigation of metastability in discrete stochastic systems~\cite{BovierEckhoffGayrardKlein2002,LandimMisturiniTsunoda2015,BovierDenHollander16,SchlichtingSlowik2019}.

The possible connection of the capacity as defined in~\eqref{eqdef:effective:capacity} and the infimum~\eqref{eq:ha:cell-problem} becomes more apparent
by considering the Euler-Lagrange equation for $h$, 
which in this case is the weighted graph-Laplacian associated to the conductivities $\set{k_{\sfx\sfy}^F}_{\sfx\sfy\in \edges}$ defined by 
\begin{align}\label{eq:ha:harmonic}
	- \gdiv (k^F\gnabla h)(\sfx) &= 0 \quad \text{ for } \sfx \in \nodes\setminus \nodes_\term ,
	\quad\text{and}\quad h_\sfa=1, h_{\sfb}=0 .
\end{align}
In the electrical-network interpretation, the edges in the network are resistors, and the solution~$h$ is the voltage at points of  $\nodes$ 
driven through an applied voltage difference of one volt at the terminal nodes $\nodes_{\term}=\set{\sfa,\sfb}$. 
The resulting flux, or  electric current, is $j^F_{\sfx\sfy} = k^F_{\sfx\sfy}\gnabla h_{\sfx\sfy}$ 
and by construction it is divergence-free in $\nodes\setminus \nodes_{\term}$. 
Hence $\capacity^F(\sfa,\sfb)$ can be also interpreted as ($1/2$ times)  the rate of dissipation of  energy of the driven system.

The infimum $\scrN$ in~\eqref{eq:ha:cell-problem} has a similar interpretation.
The values $u|_{\nodes_\term} = \set{u_\sfa,u_\sfb}$ act as a Dirichlet boundary condition of the system.
The connection is immediate for $s=0$, where the minimum in $j$ is obtained for $j\equiv 0$ and we obtain by recalling~\eqref{eqdef:N:explicit:rigrous} and \eqref{eqdef:effective:capacity} that
\begin{align*}
	\scrN[\nodes,\nodes_{\term},k^F]\bra*{(u_\sfa,u_\sfb),(0,0)} 
	&= \inf_{u\in \calM(\nodes\setminus\nodes_{\term})}
		 {\sum_{\sfx\sfy\in \edges} k_{\sfx\sfy}^F  \bra*{\sqrt{u_\sfx}-\sqrt{u_\sfy}}^2 } \\
	&= \inf\set[\bigg]{\sum_{\sfx\sfy\in \edges} k_{\sfx\sfy}^F \abs{\gnabla h_{\sfx\sfy}}^2
		\bigg|  h: \nodes \to \R, h_\sfa=\sqrt{u_{\sfa}} , h_\sfb=\sqrt{u_{\sfb}}}\\
	&= 2\bra*{\sqrt{u_\sfa}-\sqrt{u_\sfb}}^2 \capacity^F\bra*{\sfa,\sfb} ,
\end{align*}
However, as a crucial difference, the cell formula~\eqref{eq:ha:cell-problem} additionally forces  an in- and outflow on the system by the divergence constraint
$\gdiv j|_{\nodes_\term}=s$. 
Hence the functional $\scrN$ measures a dynamic dissipation of the system 
simultaneously forced by Dirichlet and flux boundary conditions.
In particular, the joint minimizers $(u^*,j^*)$ in~\eqref{eqdef:ha:cell-formula} are in general 
not coupled through an identity such as~$j^*_{\sfx\sfy} = k^F_{\sfx\sfy} \bra{\gnabla\sqrt{u^*}}_{\sfx\sfy}$, 
as one could conjecture from the solution to~\eqref{eq:ha:harmonic}. 
Instead, by inspection of the associated Euler-Lagrange equation for the minimization problem~\eqref{eqdef:ha:cell-formula},
the pair $(u^*,j^*)$ satisfy a coupled system involving nonlinear graph Laplacians. 
We leave the study of the resulting nonlinear potential theory for the future, and limit ourselves to the proof of the `replacement lemma' Proposition~\ref{prop:ha:effectiveN}.

\bigskip

As a final conclusion, we consider as an application of Theorem~\ref{thm:ha:capacity} to Example~\ref{ex:N-chain:illustration}. 

\begin{example}[Effective conductance of an $N$-chain]
From the identity~\eqref{eqdef:ha:k0}, we obtain that $k_{i,i+1}^0=\frac{\kappa_{i,i+1}}{2}$. 
Hence, we obtain for $F\in\sfF_{\mathrm{Pot}}$ the effective conductance
\begin{equation}\label{eqdef:kChain}
	k_{\Chain}^N := 2\bra*{\sum_{i=1}^{N-1} \frac{1}{\kappa_{i,i+1} \exp\bra*{-(F_i+F_{i+1})/2}}}^{-1} .
\end{equation}
This formula resembles a discretization of the effective conductivity~\eqref{eq:def:thin-membrane:sigma} 
obtained for the thin membrane limit in Section~\ref{s:thin-membrane}. 
Indeed, let $a\in C([0,1],(0,\infty))$ and $\ol F\in C([0,1])$ be given.
For any $N\in \N$, let $\kappa_{i,i+1}^N:= a(i/N)/(N-1)$ for $i=1,\dots, N-1$ and $F_i^N := \ol F(i/N)$ for $i=1,\dots,N$.
Then, we can pass to the limit in~\eqref{eqdef:kChain} and obtain 
\begin{equation*}
	k_{\Chain}^N \to 2\bra[\bigg]{\int_0^1 \frac{\ee^{\ol F(x)}}{a(x)} \dx{x} }^{-1} . 
\end{equation*}
Hence, the fast terminal $N$-chain and thin-membrane limit in Section~\ref{s:thin-membrane} show a very similar tilt-dependence.
\end{example}

\subsection{Proof of the capacitary cell formula}
\label{ss:ha:cell-formula}

In this section we prove Proposition~\ref{prop:ha:effectiveN}.
We first introduce a more general version of the cell formula in~\eqref{eq:ha:cell-problem}.
\begin{definition}\label{defn:ha:cell-formula}
	For an irreducible graph $(\nodes,\edges)$ with symmetric edges set $\edges$, given symmetric $k: \edges \to (0,\infty)$ and any $\nodes_{0}\subseteq \nodes$, define for $u\in \calM_{\geq 0}(\nodes_0)$ and $s\in \calM(\nodes_0)$ the 
	functional $\scrN[\nodes,\nodes_0,k]: \calM_{\geq 0}(\nodes_0)\times \calM(\nodes_0)\to \R$ by
	\begin{equation}\label{eqdef:ha:cell-formula}
		\scrN[\nodes,\nodes_0,k]\bra*{u, s} := 
		\inf_{\substack{ u\in\calM_{\geq 0}(\nodes \setminus\nodes_0)\\ j\in \calM(\edges)}}
		\set[\bigg]{\sum_{\sfx\sfy\in\edges} \CCs\bra*{j_{\sfx\sfy}; u_\sfx, u_\sfy;  \tfrac12 k_{\sfx\sfy}} \dx{t} : \gdiv j|_{\nodes_0} = s ,  \gdiv j|_{\nodes\setminus\nodes_0} = 0 }.
	\end{equation}
	Note the dual role of $u$ in the minimization above: the values of $u$ on $\nodes_0$ are given through the argument of $\scrN$, and the values of $u$ on the complement $\nodes\setminus \nodes_0$ are minimized in the infimum.
\end{definition}
At consequence of Definition~\ref{defn:ha:cell-formula} is the following localization property of the cell formula.
\begin{lemma}\label{l:ha:cell-formula:localization}
	In the setting of Definition~\ref{defn:ha:cell-formula}, if $\nodes_0 \subseteq \nodes_1 \subseteq \nodes$, then we have
	\begin{equation}\label{eq:ha:cell-formula:localization}
		\scrN[\nodes,\nodes_0,k]\bra*{u, s} = \inf_{ u\in\calM_{\geq 0}(\nodes_1\setminus \nodes_0)} \set*{ \scrN[\nodes,\nodes_1,k]\bra*{ u, s} },
	\end{equation}
	where $s$ is implicity extended by $0$ to $\nodes_1$.
\end{lemma}The localization formula~\eqref{eq:ha:cell-formula:localization} makes it feasible to calculate $\scrN$ by iterated minimization, such that in each iteration we calculate the effective conductivity for a cell problem with just one degree of freedom.
Before stating the general case, let us illustrate this situation for the contraction of a series law, which reduces to the formula~\eqref{eq:N:series}.

\begin{figure}[ht]
	\centering
	{\small
		\begin{tikzpicture}[scale=0.8, shorten >=3pt, shorten <=3pt]
			\tikzstyle{every node}=[draw,shape=circle,fill=black!15];
			\node (vv1) at (6,0) {$\quad\sfa\quad$};
			\node (vv2) at (10,0) {$\ \sfw\ $};
			\node (vv3) at (14,0) {$\quad\sfb\quad$};
			\tikzstyle{every node}=[];

			\draw[->, very thick] (vv1) to (vv2);
			\draw[->, very thick] (vv2) to (vv3);

			\node[above] at (8,0.1) {$k_{\sfa\sfw}$};
			\node[above] at (12,0.1) {$k_{\sfw\sfb}$};
			
			\tikzstyle{every node}=[draw,shape=circle,fill=black!15];
			\node (zv1) at (20,0) {$\quad\sfa\quad$};
			\node (zv3) at (24,0) {$\quad\sfb\quad$};
			\tikzstyle{every node}=[];
			
			\draw[->, very thick] (zv1) to (zv3);
			
			\node[above] at (22,0.1) {$\hat k_{\sfa\sfb}$};
			
			\draw[shorten >=0.7cm, shorten <=0.7cm,->, double, very thick] 
				(vv3) -- (zv1) node[midway,shift={(-0.1,0.5)}] {\text{\footnotesize contraction}};

		\end{tikzpicture}
	}
	\caption{The setup of the contraction of a series law}
	\label{fig:series-law-contraction}
\end{figure}
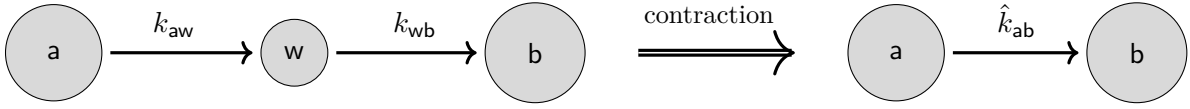
For this, we consider the graph given by $\nodes_{\Series} = \nodes_{\term} \cup \set{\sfw}$ with single directed edges $\edges_{\Series}=\set*{\sfa\sfw,\sfb\sfw}$ as in  Figure~\ref{fig:series-law-contraction}.
In this setup, the divergence constraints $\gdiv j(\sfw)=0$ and $\gdiv j(\sfa)=s=-\gdiv j(\sfb)$ enforce $j_{\sfa\sfw}=j_{\sfw\sfb}=s$. Hence, the contraction is a direct consequence of Corollary~\ref{cor:N:series} and takes the form
\begin{equation}\label{eq:ha:series}
	\inf_{u_{\sfw} \geq 0} 
\bra[\big]{ \CCs\bra*{s; u_\sfa,u_\sfw;  k_{\sfa\sfw}} 
	+ \CCs\bra*{s; u_\sfw,u_\sfb;   k_{\sfw\sfb}}} \\
= \CCs\bra*{s; u_{\sfa},u_{\sfb};  \hat k_{\sfa\sfb}} ,
\end{equation}
with
\begin{equation*}
	\hat k_{\sfa\sfb} := \bra*{\frac{1}{k_{\sfa\sfc}}+\frac{1}{k_{\sfb\sfc}}}^{-1} . 
\end{equation*}
We note that $\hat k_{\sfa\sfb}$ is exactly the effective conductance of a electrical network 
consisting of two conductors with conductances $k_{\sfa\sfw}$ and $k_{\sfb\sfw}$ in series, and in addition one can directly check that 
\begin{equation*}
	\hat k_{\sfa\sfb} = \capacity_{\Series}(\sfa,\sfb) 
	= \inf\set*{ k_{\sfa\sfw} \abs*{ \nabla h_{\sfa\sfw}}^2 +k_{\sfw\sfb} \abs*{ \nabla h_{\sfw\sfb}}^2 
		\,\middle|\, h:\nodes \to \R , h_{\sfa}=1, h_{\sfb}=0 } .
\end{equation*}
See also \cite{DoyleSnell1984}, \cite[\S 9.4]{LevinPeres2017}, \cite[\S 2.3]{LyonsPeres2016} and \cite[\S 1.3]{Grimmett2018}.
Hence, the statement~\eqref{eq:ha:series} can be read as follows:
The effective dynamic dissipation of the network $(\nodes_{\Series},\edges_{\Series},k)$ upon removing the node $\sfw$
is the same as the dynamic dissipation of the two-point network $(\nodes_{\term},\edges_{\term},\hat k)$. 

To generalize this to more complex networks than a linear chain, 
we need to obtain similar statements for more general transformations. 
The series reduction, going back to Ohm~\cite[pp.~19--20]{Ohm1827}, 
was later generalized to the $\rmY$-$\Delta$ transformation by Kennelly~\cite{Kennelly1899}: 
Three edges with a common node $\set*{\sfx\sfw, \sfy\sfw, \sfz\sfw}$ 
are replaced by a triangle consisting of the edges $\set*{\sfx\sfy, \sfy\sfz , \sfz\sfx}$. 
This process removes the node $\sfw$ and introduces new edges, possibly in parallel to existing ones.
Those emerging multi-edges can be contracted to a single edge 
by the parallel law reduction in~\eqref{eq:N:parallel} from Corollary~\ref{cor:N:parallel}.
Hence, by graph-theoretic arguments~\cite{Epifanov1966,Truemper1989}, 
it follows that the effective conductivity of any two-terminal planar two-connected graph 
can be obtained by a sequence of series, $\rmY$-$\Delta$ and parallel reduction steps. 

Instead of providing the result for the $\rmY$-$\Delta$ transformation, 
we will directly consider the general \emph{star-mesh}~\cite{Kron1939,Shew1947} or \emph{star-clique}~\cite[\S 2.11 Ex.~(2.69)]{LyonsPeres2016}, transformation:
A node $\sfw \in \nodes\setminus \nodes_{\term}$ with neighbours $\sfN_{\sfw}\subseteq \nodes$ 
is removed and between each pair of edges in $\sfN_{\sfw}$ an edge with a specific conductance is introduced. 
 
We show this statement on the level of the dynamic dissipation functional $\scrN$ 
from Definition~\ref{defn:ha:cell-formula} in the following Lemma~\ref{l:ha:star-mesh:local}, 
which results in a general reduction step in Proposition~\ref{prop:ha:EffCond:step}.

To formalize the setting, 
let $\sfN_\sfw = \set{\sfx \in \nodes: k_{\sfx\sfw} >0}$ be the neighbours of $\sfw$ 
and also write $\overline \sfN_\sfw := \sfN_\sfw \cup  \sfw$. 
We are going to restrict the optimization problem to the subgraph induced by~$\sfN_{\sfw}$. 
For the general case, we work with double directed edges (cf.~Remark~\ref{rem:edges-single-double}) and 
use the decomposition of the edge set induced by $\kappa$, 
\begin{equation}\label{eqdef:ha:cell-formula:edge-decomp}
	\edges_\sfw := \set*{ \sfx\sfy\in \overline\sfN_\sfw \times \overline\sfN_{\sfw}: \sfx\ne\sfy} 
	= \sfM_\sfw \mathop{\dot\cup} \sfS_\sfw 
\end{equation}
with the \emph{mesh} and \emph{star} edges given by
\begin{equation*}
\sfM_\sfw  := \set*{ \sfx\sfy\in \sfN_\sfw \times \sfN_{\sfw}: \sfx\ne\sfy} 
\qquad\text{and}\qquad
\sfS_\sfw  := \sfN_\sfw \times \sfw  \cup \sfw \times \sfN_\sfw . 
\end{equation*}
Moreover, we denote with 
\begin{equation*}
	k|_{\sfS_\sfw}(\sfx\sfy) :=
	\begin{cases}
		k_{\sfx\sfy} , & \sfx\sfy \in \sfS_\sfw ;\\
		0 , & \sfx\sfy \not\in \sfS_\sfw
	\end{cases}
\end{equation*}
the restriction of $k$ to the edges of the star $\sfS_\sfw$. 
With this preliminary notation, we obtain the star-mesh transformation as follows.
\begin{lemma}[Star-mesh identity]\label{l:ha:star-mesh:local}
  For any $u\in \calM_{\geq0}(\sfN_\sfw)$ and $s\in \calM(\sfN_\sfw)$ we have
  \begin{equation}\label{eq:ha:star-mesh:local}
  	\begin{split}
  		\scrN[\overline\sfN_\sfw,\sfN_\sfw,k|_{\sfS_\sfw}]\bra*{u,s}
  		= \scrN[\sfN_\sfw,\sfN_\sfw, \tilde k](u, s) ,
  	\end{split}
  \end{equation}
  where for $\sfx\sfy \in \sfM_\sfw$ the conductivity $\tilde k$ is given by
  \begin{equation*}
  	\tilde k_{\sfx\sfy} :=  \frac{k_{\sfx\sfw} k_{\sfy\sfw}}{\sum_{\sfz\in \nodes\setminus \sfw} k_{\sfz\sfw}} .
  \end{equation*}
\end{lemma}	
By making use of the localization property~\eqref{eq:ha:cell-formula:localization} from Lemma~\ref{l:ha:cell-formula:localization} 
and  possibly the contraction of emerging parallel edges during the star-mesh reduction, 
we arrive at the general result for removing the vertex $\sfw\in \nodes \setminus \nodes_{\term}$. 
\begin{prop}[Reduction step towards effective conductance]\label{prop:ha:EffCond:step}
	We have for all $(u,s)\in\calM_{\geq0}(\nodes_\term)\times \calM(\nodes_\term)$ 
	the identity
	\begin{equation}\label{eq:ha:cell-formula:1d:reduce}
		\inf_{u_\sfw\geq 0} \scrN[\nodes,\nodes_\term\cup \sfw,k]\bra*{u, s}
		=  \scrN[\nodes\setminus\sfw,\nodes_\term,\hat k]\bra*{u, s} ,
	\end{equation}
	where
	\begin{equation*}
		\hat k_{\sfx\sfy} := k_{\sfx\sfy}+ \frac{k_{\sfx\sfw} k_{\sfy\sfw}}{\sum_{\sfz\in \nodes\setminus \sfw} k_{\sfz\sfw}} 
		\qquad\text{for } \sfx, \sfy\in \nodes \setminus \sfw : \sfx\ne\sfy , 
	\end{equation*}
	and $s_\sfw$ is defined to be zero.
\end{prop}
Based on Proposition~\ref{prop:ha:EffCond:step}, 
we can reduce the cell formula $\scrN[\nodes,\nodes_{\term},k]$ step by step and conclude the proof of Proposition~\ref{prop:ha:effectiveN}. 
\begin{proof}[Proof of Proposition~\ref{prop:ha:effectiveN}]
	We fix a labeling of nodes $\nodes_\fast = \set{\sfw^1,\dots,\sfw^n}$. 
	Then, we iteratively apply~\eqref{eq:ha:cell-formula:1d:reduce} from Proposition~\ref{prop:ha:EffCond:step} to obtain
	\begin{align*}
		\scrN\pra*{\nodes,\nodes_\term,k^F}
		&= \inf_{u_{\sfw^n}\geq 0} \cdots \inf_{u_{\sfw^2}\geq 0}\inf_{u_{\sfw^1}\geq 0}
		\scrN\pra*{\nodes,\nodes_\term\cup \set{\sfw^1,\dots,\sfw^n},k^F} \\
		&= \inf_{u_{\sfw^n}\geq 0} \cdots \inf_{u_{\sfw^2}\geq 0}
		\scrN\pra*{\nodes\setminus\set{\sfw^1},\nodes_\term\cup \set{\sfw^2,\dots,\sfw^n},\hat k^1} \\
		&\vdotswithin{=} \\
		&= \inf_{u_{\sfw^n}\geq 0}  \scrN\pra*{\nodes\setminus\set{\sfw^1,\dots,\sfw^{n-1}},\nodes_\term\cup \set{\sfw^n},\hat k^{n-1}} \\
		&= \scrN\pra*{\nodes\setminus\set{\sfw^1,\dots,\sfw^n},\nodes_\term,\hat k^n},
	\end{align*}
	where the conductivities $\hat k^i : \bra*{\nodes \setminus\set{\sfw^1,\dots,\sfw^i}}^2 \to (0,\infty)$ for $i=1,\dots, n$ are defined iteratively by 
	\begin{equation}\label{eq:ha:Gauss_elimination}
		\hat k^{i}_{\sfx\sfy} := \hat k^{i-1}_{\sfx\sfy} 
		+ {\hat k^{i-1}_{\sfx\sfw^{i}}\hat k^{i-1}_{\sfy\sfw^{i}}}
		\bra*{\sum_{\sfz \in \nodes\setminus\set*{\sfw^1,\dots,\sfw^{i}}} \hat k^{i-1}_{\sfz\sfw^{i}}} ^{-1}
		\qquad\text{and}\qquad \hat k^0 := k^F .
	\end{equation}
	Hence, we obtain by inspecting Definition~\ref{defn:ha:cell-formula} for any $u_\sfa,u_\sfb\geq 0$ and $s\in \R$ the identity
	\begin{equation*}
		\scrN[\nodes,\nodes_{\term},k]\bra*{(u_\sfa,u_\sfb),(s,-s)}
		= \CCs\bra*{\tfrac{1}{2}s;u_\sfa,u_\sfb;\tfrac{1}{2}\hat k^n} + \CCs\bra*{-\tfrac{1}{2}s;u_\sfb,u_\sfa;\tfrac{1}{2}\hat k^n} = \CCs\bra*{s;u_\sfa,u_\sfb; \hat k^n} ,
	\end{equation*}
	where we used the symmetry of $\CCs$ in its second and third argument and the one-homogeneity in the first and fourth argument from Lemma~\ref{l:props-N}.\ref{l:props-N:one-homogeneous}.
	It is left to show that $\hat k^n \overset{!}= k^F_{\sfa\sfb}= \capacity^F(\sfa,\sfb)$. 
	First, we note that by using the unique solution of~\eqref{eq:ha:harmonic} in~\eqref{eqdef:effective:capacity} 
	and summation by parts, we obtain the identity
	\begin{equation}\label{eq:ha:weightedLap:capacity}
		- \gdiv (k^F\gnabla h)(\sfa) = \capacity^F(\sfa,\sfb). 
	\end{equation}
	It is a classical observation by Kron~\cite[\S X]{Kron1939} in electrical engineering 
	that each of the transformations~\eqref{eq:ha:Gauss_elimination} corresponds to a Gauss elimination step
	in the partial inversion of the equation~\eqref{eq:ha:harmonic} for the row 
	corresponding to the terminal nodes $\sfa\in \nodes_{\term}$. 
	In this way, the Schur complement of the system~\eqref{eq:ha:harmonic} 
	with respect to the terminal nodes $\nodes_{\term}$ is calculated, 
	which provides the identity $\hat k^n = \capacity^F(\sfa,\sfb)$ thanks to~\eqref{eq:ha:weightedLap:capacity},
	independently of the labeling of the nodes in $\nodes\setminus\nodes_{\term}$.
\end{proof}
We still need to prove the single reduction step, Proposition~\ref{prop:ha:EffCond:step}
\begin{proof}[Proof of Proposition~\ref{prop:ha:EffCond:step}]
	We will decompose several divergences and introduce 
	the restricted divergence for $j\in \calM(\edges)$ and any $\hat\sfE \subseteq \edges$ by 
	\begin{equation*}
		\gdiv_{\hat\edges} j (x) :=  \sum_{\sfx\sfy\in \hat\sfE} j_{\sfx\sfy} - \sum_{\sfy\sfx\in \hat\sfE} j_{\sfy\sfx} .
	\end{equation*}
	From the constraint in the optimization problem in the left-hand side of~\eqref{eq:ha:cell-formula:1d:reduce}, 
	we get for $\sfx\in \sfN_{\sfw}$ the decomposition
	\begin{equation*}
		\hat s_\sfx := \gdiv_{\edges_\sfw} j(\sfx) =  s_\sfx - \gdiv_{\edges\setminus \edges_\sfw} j(x) .
	\end{equation*}
	and by definition of $\sfS_\sfw$ and since $j \ll k$, we have 
	$\gdiv_{\sfS_\sfw} j(\sfw) =  \gdiv j(\sfw)  = 0$ 
	from the constraint in the optimization problem~\eqref{eq:ha:cell-formula:1d:reduce}. Hence, we set $\hat s_\sfw := 0$.

	By comparison of the definitions of $\scrN$ in the left and right-hand side of~\eqref{eq:ha:cell-formula:1d:reduce}, 
	it is left to show that
	\begin{equation}\label{eq:ha:cell-formula:sub-localization}
		\begin{split}
			\MoveEqLeft\scrN[\overline \sfN_\sfw,\sfN_\sfw,k](u,\hat s) = 
			\inf_{\substack{ u_\sfw \geq 0 \\ j \in \calM(\edges_\sfw)}} 
			\set[\bigg]{  \sum_{\sfx\sfy\in\edges_\sfw} \CCs\bra*{j_{\sfx\sfy};u_\sfx,u_\sfy; \tfrac12 k_{\sfx\sfy}} :
				\gdiv_{\edges_\sfw} j = \hat s} \\
			&\stackrel{!}{=} 
			\inf_{\hat\jmath \in \calM(\sfM_\sfw)}
			\set[\bigg]{  \sum_{\sfx\sfy\in\sfM_\sfw} \CCs\bra*{\hat \jmath_{\sfx\sfy}; u_\sfx, u_\sfy; \tfrac12 \hat k_{\sfx\sfy}} 
				: \gdiv_{\sfM_\sfw} \hat\jmath = \hat s|_{\sfN_\sfw}} 
			= \scrN[\sfN_\sfw,\sfN_\sfw, \hat k](u,\hat s) . 
		\end{split}
	\end{equation}
	We use the disjoint splitting~\eqref{eqdef:ha:cell-formula:edge-decomp}
	to decompose for fixed $u_\sfw$ and $j$ the sum on the left-hand side of~\eqref{eq:ha:cell-formula:sub-localization},
	\begin{align*}
		\MoveEqLeft \sum_{\sfx\sfy\in\edges_\sfw} \CCs\bra*{j_{\sfx\sfy}; u_\sfx, u_\sfy; \tfrac12 k_{\sfx\sfy}} \\
		&=
		\sum_{\sfx\sfy\in\sfM_\sfw}
		\CCs\bra*{j_{\sfx\sfy}; u_\sfx, u_\sfy; \tfrac12 k_{\sfx\sfy}}
		+ \sum_{\sfx \in \sfN_\sfw} 
		\CCs\bra*{j_{\sfx\sfw};u_\sfx,u_\sfw; \tfrac12 k_{\sfx\sfw}}
		+
		\sum_{\sfy \in \sfN_\sfw} 
		\CCs\bra*{j_{\sfw\sfy};u_\sfw,u_\sfy; \tfrac12 k_{\sfw\sfy}} .
	\end{align*}
	By also decomposing the divergence constraint further into
	\begin{equation*}
		\hat s = \gdiv_{\edges_\sfw} j = \gdiv_{\sfM_\sfw} j + \gdiv_{\sfS_\sfw} j ,
	\end{equation*}
	we arrive at the following nested optimization problem for the left-hand side of~\eqref{eq:ha:cell-formula:sub-localization},
	\begin{equation}\label{eq:ha:cell-formula:star-mesh-decomposition}\begin{split}
			\scrN[\overline \sfN_\sfw,\sfN_\sfw,k](u,\hat s) 
			=
			\inf_{j\in \calM(\sfM_\sfw)}
			\biggl\{
				\sum_{\sfx\sfy\in\sfM_\sfw}
				\CCs\bra*{j_{\sfx\sfy};u_\sfx,	u_\sfy; \tfrac12 k_{\sfx\sfy}}
				+ 
				\scrN[\overline\sfN_\sfw,\sfN_\sfw,k|_{\sfS_\sfw}]\bra*{u,\hat s-\gdiv_{\sfM_\sfw} j} 
			\biggr\}.
	\end{split}\end{equation}
	The inner optimization problem is supported only on the star part $\sfS_\sfw$
	and we can apply~\eqref{eq:ha:star-mesh:local} in the inner optimization problem of~\eqref{eq:ha:cell-formula:star-mesh-decomposition}.
	Using Lemma~\ref{l:ha:star-mesh:local} to replace $\scrN[\overline\sfN_\sfw,\sfN_\sfw,k|_{\sfS_\sfw}]\bra*{u,\hat s-\gdiv_{\sfM_\sfw} j}$ by $\scrN[\sfN_\sfw,\sfN_\sfw, \tilde k](u, \hat s-\gdiv_{\sfM_\sfw} j)$, and expanding the definition of the latter object, we arrive at an identity to which the parallel reduction formula~\eqref{eq:N:parallel} from Corollary~\ref{cor:N:parallel} is applicable:
	\begin{equation*}\begin{split}
			\scrN[\overline \sfN_\sfw,\sfN_\sfw,k](u,\hat s) 
			&= \begin{multlined}[t]
				\inf_{j,\hat\jmath\in \calM(\sfM_\sfw)}
				\biggl\{
				\sum_{\sfx\sfy\in\sfM_\sfw}
				\bra[\Big]{
					\CCs\bra*{j_{\sfx\sfy};u_\sfx,	u_\sfy; \tfrac12 k_{\sfx\sfy}}
					+ \CCs\bra*{\hat \jmath_{\sfx\sfy}; u_\sfx, u_\sfy; \tfrac12 \tilde k_{\sfx\sfy}}} : \\
				\gdiv_{\sfM_\sfw} j + \gdiv_{\sfM_\sfw} \hat\jmath = \hat s |_{\sfN_\sfw}
				\biggr\} 
			\end{multlined}\\
			&= \inf_{\hat \jmath \in \calM(\sfM_\sfw)} 
			\set[\bigg]{ \sum_{\sfx\sfy\in\sfM_\sfw}
				\CCs\bra*{\hat \jmath_{\sfx\sfy}; u_\sfx, u_\sfy; \tfrac12 k_{\sfx\sfy} + \tfrac12 \tilde k_{\sfx\sfy}}
				: \gdiv_{\sfM_\sfw} \hat\jmath = \hat s |_{\sfN_\sfw}} . \qedhere
	\end{split}\end{equation*}
\end{proof}
It is left to prove the localized star-mesh identity~\eqref{eq:ha:star-mesh:local}.
\begin{proof}[Proof of Lemma~\ref{l:ha:star-mesh:local}]	
	We first note that if $\sum_{\sfx\in \sfN_\sfw} s_\sfx \ne 0$ 
	then the divergence constraint on both sides of~\eqref{eq:ha:star-mesh:local} cannot be satisfied by any flux 
	and hence both sides are an infimum over an empty set, which is $+\infty$ by convention.
	
	Hence, from now on, we fix $s \in \calM(\calN_\sfw)$ such that $\sum_{\sfx\in \sfN_\sfw} s_\sfx = 0$ 
	and first expand the right-hand side of~\eqref{eq:ha:star-mesh:local} as
	\begin{equation}\label{eq:ha:cell-formula:star-optimization}
		\scrN[\overline\sfN_\sfw,\sfN_\sfw,k|_{\sfS_\sfw}]\bra*{u, s}=
		\begin{multlined}[t]
			\inf_{\substack{ u_\sfw \geq 0 \\ j \in \calM(\sfS_\sfw)}}
			\biggl\{
			\sum_{\sfx \in \sfN_\sfw} 
			\CCs\bra*{j_{\sfx\sfw};u_\sfx,u_\sfw; \tfrac12 k_{\sfx\sfw}}
			+
			\sum_{\sfy \in \sfN_\sfw} 
			\CCs\bra*{j_{\sfw\sfy};u_\sfw,u_\sfy; \tfrac12 k_{\sfw\sfy}} : \\
			\gdiv_{\sfS_\sfw} j (\sfx) = j_{\sfx\sfw} - j_{\sfw\sfx}  = s_\sfx ,\  \sfx \in \sfN_\sfw; \ \gdiv_{\sfS_\sfw} j(\sfw) = 0
			\biggr\} .
		\end{multlined}
	\end{equation}
	The divergence constraint in~\eqref{eq:ha:cell-formula:star-optimization} is satisfied by setting
	$j_{\sfx\sfw} = \tfrac12{s_\sfx} + r_{\sfx}$ and $j_{\sfw\sfx} = -\tfrac12{s_\sfx}  + r_{\sfx}$, 
	where $r_{\sfx}\in \R$ is arbitrary. 
	Hence, we arrive at the individual optimization problems for each $\sfx\in\sfN_\sfw$,
	\begin{align*}
		\MoveEqLeft\inf_{\hat\jmath_{\sfx\sfw}-\hat\jmath_{\sfw\sfx}= s_\sfx}
		\set[\big]{ \CCs\bra*{j_{\sfx\sfw};u_\sfx,u_\sfw;  \tfrac12 k_{\sfx\sfw}}
			+
			\CCs\bra*{j_{\sfw\sfx};u_\sfw,u_\sfx; \tfrac12 k_{\sfw\sfx}}} \\
		&= \inf_{r_{\sfx}\in \R}  \set[\big]{ \CCs\bra*{\tfrac12{s_\sfx}  + r_{\sfx};u_\sfx,u_\sfw; \tfrac12 k_{\sfx\sfw}}
			+\CCs\bra*{-\tfrac12{s_\sfx}  + r_{\sfx};u_\sfw,u_\sfx; \tfrac12 k_{\sfw\sfx}} } \\
		&=  \CCs\bra*{\tfrac12{s_\sfx};u_\sfx,u_\sfw; \tfrac12 k_{\sfx\sfw}}
			+\CCs\bra*{\tfrac12{s_\sfx};u_\sfx,u_\sfw; \tfrac12 k_{\sfx\sfw}} 
		= \CCs\bra*{s_\sfx;u_\sfx,u_\sfw; k_{\sfx\sfw}} ,
	\end{align*}
	where we used the symmetry and convexity properties of $\CCs$.
	This shows that the left-hand side of~\eqref{eq:ha:star-mesh:local} reduces to
	\begin{equation*}
		\scrN[\overline \sfN_\sfw,\sfN_\sfw,k](u,s) = 
		\inf_{ u_\sfw \geq 0 } 
		\sum_{\sfx\in \sfN_\sfw } \CCs\bra*{s_\sfx; u_\sfx,u_\sfw; k_{\sfx\sfw}} 
	\end{equation*}
	We similarly rewrite the divergence constraint on the right-hand side of~\eqref{eq:ha:star-mesh:local} by setting 
	\begin{equation*}
		\hat \jmath_{\sfx\sfy} =  \frac{s_\sfx - s_\sfy}{2d} + r_{\sfx\sfy} ,
	\end{equation*}
	where $d=\abs{\sfN_\sfw}$ and $r \in \calM(\sfM_\sfw)$ satisfies $\gdiv_{\sfM_\sfw} r=0$.
	Indeed, for this choice we have 
	\[
		\gdiv_{\sfM_\sfw}\hat \jmath(\sfx) = 2\sum_{\sfy\in\sfN_\sfw\setminus \sfx}\frac{s_\sfx - s_\sfy}{2d}
		 = \frac{1}{d}\sum_{\sfy\in \sfN_\sfw} \bra*{s_\sfx - s_\sfy} = s_{\sfx},
	\]
	where we used the fact that $\sum_{\sfy\in \sfN_\sfw} s_\sfy=0$. 
	In this way, we arrive at the following reformulation of the right-hand side of~\eqref{eq:ha:star-mesh:local}:
	\begin{equation}\label{eq:ha:dual-char-2}
		\scrN[\sfN_\sfw,\sfN_\sfw, \tilde k](u, s) 
		=
		\inf_{r \in \calM(\sfM_\sfw)}
		\set[\bigg]{ \sum_{\sfx\sfy\in\sfM_\sfw} 
				\CCs\bra*{\frac{s_\sfx - s_\sfy}{2d} + r_{\sfx\sfy}; u_\sfx, u_\sfy; \tfrac12 \tilde k_{\sfx\sfy}} 
			: \gdiv_{\sfM_\sfw} r =0} .
	\end{equation}
	By using~\eqref{eq:props-N:dual-char-1} from Lemma~\ref{l:props-N} and setting $f_\sfx := -{s_\sfx}/{2d}$ for $\sfx\in\sfN_\sfw$,
	we can further rewrite this as
	\begin{equation}
		\label{eq:l:ha:dual-char-2:intermediate}	
		\inf_{r\in \calM(\sfM_\sfw)} \max_{\zeta :\sfM_\sfw\to \R} \set[\bigg]{\sum_{\sfx\sfy\in\sfM_\sfw} 
			\pra*{\zeta_{\sfx\sfy} \bra[\big]{\ona f_{\sfx\sfy} +r_{\sfx\sfy}}
				+ \tilde k_{\sfx\sfy} \bra*{ u_\sfx + u_\sfy - 2 \sqrt{u_\sfx u_\sfy} \cosh\tfrac12{\zeta_{\sfx\sfy}}}}
			\ :\ \odiv_{\sfM_\sfw} r = 0
		} .
	\end{equation}
	The saddle point is achieved; for $\zeta$ this is given by~\eqref{eq:props-N:dual-char-1}, 
	and for $r$ this follows in the left-hand side of~\eqref{eq:ha:dual-char-2} from the coercivity of $\CCs$ in its first variable. 
	Since the minimization in $r$ is only constrained by the divergence condition, 
	we find by differentiating that 
	\[
		0 = \sum_{\sfx\sfy\in\sfM_\sfw} \zeta_{\sfx\sfy} \tilde r_{\sfx\sfy} 
		\qquad\text{for all $\tilde r$ such that $\odiv_{\sfM_\sfw} \tilde r=0$.}
	\]
	Using the Helmholtz decomposition of edge functions (e.g.~\cite[(3.6)]{Lim20}) 
	it follows that $\zeta$ can be written as $\zeta_{\sfx\sfy} = \ona \xi_{\sfx\sfy} + c$ 
	for a function $\xi:\sfN_\sfw\to\R$ and a constant $c\in \R$. 
	Without loss of generality we can set the constant to zero, 
	and we find that the expression in braces in~\eqref{eq:l:ha:dual-char-2:intermediate} is independent of $r$. It follows that~\eqref{eq:l:ha:dual-char-2:intermediate} equals the right-hand side of~\eqref{eq:ha:dual-char-2}. 

	In this way the statement to be proven  becomes 
	\begin{equation}
		\label{eq:l:ha:alt-char}
		\inf_{u_\sfw \geq 0} \sum_{\sfx\in \sfN_\sfw } \CCs\bra*{s_\sfx ; u_\sfx,u_\sfw; k_{\sfx\sfw}}  
		\overset{!}{=}
		\sup_{\xi: \sfN_\sfw\to \R} \sum_{\sfx\sfy\in\sfM_\sfw} \pra*{
				\frac{s_\sfx - s_\sfy}{2d} \ona \xi_{\sfx\sfy}
			+ \tilde k_{\sfx\sfy} \bra*{ u_\sfx + u_\sfy - 2 \sqrt{u_\sfx u_\sfy} \cosh\tfrac12{\ona\xi_{\sfx\sfy}}}}.
	\end{equation}
	For the proof of~\eqref{eq:l:ha:alt-char}, we first consider the case that the infimum over $u_\sfw$ is achieved at $u_\sfw=0$. 
	By part~\ref{l:props-N:est-B} of Lemma~\ref{l:props-N} we then find that $s_\sfx=0$ for all $\sfx\in \sfN_\sfw$. 
	By inspecting the resulting formula one finds that $u_\sfw=0$ only is minimal if all $u_\sfx$ are zero as well, 
	in which case both sides of~\eqref{eq:l:ha:alt-char} evaluate to zero. 
	
	We continue under the assumption that $s$ is not identically zero, 
	so that the infimum is not achieved at $u_\sfw=0$; 
	by~\eqref{eqdef:N:explicit:rigrous} we find that the infimum also is not achieved as $u_\sfw\to\infty$, 
	and therefore the infimum is achieved at a $u_\sfw\in(0,\infty)$.
	
	Applying once more part~\ref{l:props-N:est-B} of Lemma~\ref{l:props-N} and using that now $s$ is not identically equal to zero, we obtain that both sides in~\eqref{eq:l:ha:alt-char} are finite if and only if $u_{\sfx}>0$ for all $\sfx\in \sfN_\sfw$.
	
	By~\eqref{eq:props-N:dual-char-1} from Lemma~\ref{l:props-N} the left-hand side of~\eqref{eq:l:ha:alt-char} can be written as
	\begin{align}
		\notag
		&\inf_{u_\sfw\geq0} \sup_{\xi\colon \ol\sfN_\sfw \to \R}
		\sum_{\sfx\in\sfN_\sfw} \pra*{
				\ona\xi_{\sfx\sfw} s_\sfx
				+ 2 k_{\sfx\sfw} \bra*{ u_\sfx + u_\sfw - 2 \sqrt{u_\sfx u_\sfw} \cosh\tfrac12{\ona \xi_{\sfx\sfy}}}}\\
		&= \inf_{u_\sfw\geq0} \sup_{\xi\colon \ol\sfN_\sfw \to \R}
		{\sum_{\sfx\in\sfN_\sfw} \pra*{
				\bra*{\ona\xi_{\sfx\sfw}+\ona_{\sfx\sfw}\log u} s_\sfx
				+ 2 u_\sfx k_{\sfx\sfw} (1-\ee^{\ona\xi_{\sfx\sfw}/2})
				+ 2 u_\sfw k_{\sfx\sfw} (1-\ee^{-\ona\xi_{\sfx\sfw}/2}) } } ,\notag \\
		&= \inf_{u_\sfw\geq0} \sup_{\xi\colon \ol\sfN_\sfw \to \R}
			{\sum_{\sfx\in\sfN_\sfw} \pra*{
					\bra*{\xi_{\sfx}+\log u_{\sfx}} s_\sfx
					+ 2 u_\sfx k_{\sfx\sfw} (1-\ee^{\ona\xi_{\sfx\sfw}/2})
					+ 2 u_\sfw k_{\sfx\sfw} (1-\ee^{-\ona\xi_{\sfx\sfw}/2}) } } ,
		\label{eq:l:ha:alt-char:intermediate}
	\end{align}
	where we replaced $\xi_{\sfx} \mapsto \xi_{\sfx} + \log u_{\sfx}$ for $\sfx\in\ol \sfN_\sfw$ in the first step using the fact that $u>0$ on~$\ol\sfN_\sfw$, the identity $2\sqrt{ab}\cosh\bra*{\tfrac12 x+\tfrac{1}{2}\log\tfrac{a}{b}}=a \ee^{x/2} + b \ee^{-x/2}$, and also $\sum_{\sfx\in \sfN_\sfw} s_\sfx=0$ in the second step.
	
	Since the infimum is achieved at $u_\sfw\in (0,\infty)$, we find by differentiating the right-hand side of~\eqref{eq:l:ha:alt-char:intermediate} in $u_\sfw$ that 
	\[
		0 = \sum_{\sfy\in\sfN_\sfw} k_{\sfy\sfw} (1-\ee^{-\ona \xi_{\sfy\sfw}/2})
		\qquad\text{or equivalently}\qquad
		\ee^{\xi_\sfw/2}  = \sum_{\sfy\in\sfN_\sfw} \frac{k_{\sfy\sfw}}{\ol k}\ee^{\xi_\sfy/2}
		\qquad \text{with }\ol k := \sum_{\sfy\in\sfN_\sfw} k_{\sfy\sfw}.
	\]
	We then write for any $\sfx\in \sfN_\sfw$
	\begin{align*}
		1 - \ee^{\ona \xi_{\sfx\sfw}/2} 
		= 	1 - \ee^{\xi_\sfw/2}\ee^{-\xi_{\sfx}/2} 
		= 1 - \sum_{\sfy\in\sfN_\sfw} \frac{k_{\sfy\sfw}}{\ol k}\ee^{(\xi_\sfy-\xi_\sfx)/2}
		= \sum_{\sfy\in\sfN_\sfw} \frac{k_{\sfy\sfw}}{\ol k}(1-\ee^{\ona \xi_{\sfx\sfy}/2}).
	\end{align*}
	With this identity we rewrite~\eqref{eq:l:ha:alt-char:intermediate}, using the shorthand $\xi^u_\sfx:=\xi_{\sfx}+\tfrac{1}{2}\log u_{\sfx}$, as
	\begin{align*}
		\MoveEqLeft\sup_\xi 
		\set[\bigg]{\sum_{\sfx\in\sfN_\sfw} \pra[\Big]{
				\xi_{\sfx}^u s_\sfx
				+   2 u_\sfx  k_{\sfx\sfw}\sum_{\sfy\in\sfN_\sfw} \frac{k_{\sfy\sfw}}{\ol k}(1-\ee^{\ona \xi_{\sfx\sfy}/2})}
			\ : \ 
			\xi\colon \ol\sfN_\sfw \to \R}\\
		&= \sup_\xi 
		\set[\bigg]{\sum_{\sfx\sfy\in\sfM_\sfw} \pra[\Big]{
				\xi_{\sfx}^u \frac{s_\sfx}{d}
				+   2 u_\sfx  \tilde k_{\sfx\sfy}(1-\ee^{\ona \xi_{\sfx\sfy}/2})}
			\ : \ 
			\xi\colon \ol\sfN_\sfw \to \R}\\
		&= \sup_\xi 
		\set[\bigg]{\sum_{\sfx\sfy\in\sfM_\sfw} \pra[\Big]{
				\xi_{\sfx}^u \frac{s_\sfx}{d}
				+   u_\sfx  \tilde k_{\sfx\sfy}(1-\ee^{\ona \xi_{\sfx\sfy}/2})
				+   u_\sfy  \tilde k_{\sfx\sfy}(1-\ee^{-\ona \xi_{\sfx\sfy}/2})}
			\ : \ 
			\xi\colon \ol\sfN_\sfw \to \R} \\
		&= \sup_\xi 
		\set[\bigg]{\sum_{\sfx\sfy\in\sfM_\sfw} \pra[\Big]{
				\ona\xi_{\sfx\sfy}^u \frac{s_\sfx}{d}
				+   u_\sfx  \tilde k_{\sfx\sfy}(1-\ee^{\ona \xi_{\sfx\sfy}/2})
				+   u_\sfy  \tilde k_{\sfx\sfy}(1-\ee^{-\ona \xi_{\sfx\sfy}/2})}
			\ : \ 
			\xi\colon \ol\sfN_\sfw \to \R}.
	\end{align*}
	Finally, we symmetrize the first term in the sum above to
	\[
		\sum_{\sfx\sfy\in\sfM_\sfw} 
		\ona\xi^u_{\sfx\sfy} \frac{s_\sfx}{d}
		= - \sum_{\sfx,\sfy\in\sfN_\sfw} 
		(\xi^u_{\sfx}-\xi^u_\sfy) \frac{s_\sfx}{d}
		= \sum_{\sfx,\sfy\in\sfN_\sfw} 
		\gnabla \xi^u_{\sfx\sfy} \frac{s_\sfx - s_\sfy}{2d}.
	\]
	By recalling that $\xi^u_\sfx:=\xi_{\sfx}+\log u_{\sfx}$ and undoing this shift by setting $\xi_{\sfx}\mapsto \xi_{\sfx} - \log u_{\sfx}$ for $\sfx\in \sfN_\sfw$, 
	we find that ~\eqref{eq:l:ha:alt-char} indeed holds.
\end{proof}

\appendix

\addtocontents{toc}{\protect\setcounter{tocdepth}{0}}%

\newpage 

\section{Proofs for Kramers' high activation limit}

\subsection{Constant \texorpdfstring{$y$}{y}-flux for limiting measure}

\begin{lemma}[Limiting measures $\rho$ have constant $y$-flux]\label{lem:Kramers:limit:jy}
	\label{l:char:j0}
	Let $(\rho,j)= (\rho,(j^x,j^y))\in \CE(0,T)$, and assume that $\rho$ and $j^x$ are supported on $[0,T]\times \Omega\times \{a,b\}$. 
	Then $j^y$ is of the form~\eqref{char:j0}, i.e.
	\[
		j^y(t,\dxx xy) = \overline \jmath (t,\dx x) \bONE_{[a,b]}(y)\dx y
	\]
	for some $(\overline \jmath(t,\cdot))_t\in \calM(\Omega)$.
\end{lemma}
\begin{proof}
	Write $\Upsilon = [c,d]\subset \R$. 
	Fix $\varphi\in C_{\mathrm c}^1((0,T)\times \nodes)$ and set $\Phi(t,x,y) := \int_c^y\varphi(t,x,y')\dx y'$; we extend $\Phi$ smoothly outside of $\nodes_T$ to a function in $C_{\mathrm c}^1((0,T)\times \R^{d+1})$. Then by~\eqref{eq:Kramers:weak-form-CE},
\begin{multline*}	
	\int_{\nodes_T}  \varphi(t,x,y) j^y(t,\dxx xy)\dx t
	= 
	\int_{\Omega_T} \int_\R \partial_y \Phi(t,x,y) j^y(t,\dxx xy) \dx t\\
	= - \int_{\Omega_T}\int_\R  \Bigl[ \nabla_x \Phi(t,x,y) j^x(t,\dxx xy)\dx t
	+  \partial_t \Phi(t,x,y)\rho(t,\dxx xy) \dx t\Bigr].
\end{multline*}
	
	Since $j^x$ and $\rho$ in the right-hand side above are concentrated on $\Omega_T\times \{a,b\}$, we find using $\varphi$ with $\varphi=0$ for $y\geq a$ that $j^y$ is zero for $y<a$. Similarly $j^y=0$ for $y>b$, and therefore $j^y$ is supported on $\Omega_T\times [a,b]$. From the second equality above  we also find that $\partial _y j^y$ is zero for $a<y<b$ in the sense of distributions. This establishes the structure~\eqref{char:j0}.
\end{proof}

\subsection{Rescaling}

In this section we describe the rescaling~\eqref{eqdef:hat-transformed-objects} of the continuity equation and the dissipation function $\calD_0^T$ in~\eqref{eq:FP-Kramerslimit:D:tilted}.

\subsubsection{Rescaling of the continuity equation}\label{ss:Kramers:append:rescaling:CE}

We claimed in Section~\ref{ss:Kramers-setting} that for $(\rho,j)\in \CE(0,T)$, the rescaled pair $(\hrho,\hj)$ also is an element of $\CE(0,T)$, both on $\hnodes_{\e T}$ and on $\hnodes_T$. Here we provide the details. 

The definition of $\hrho$ and $\hj$ in~\eqref{eqdef:hat-transformed-objects} implies that the two first requirements of Definition~\ref{defn:Kramers:CE} follow from those of $\rho$ and $j$. To show the third condition, first note that we have for any $\varphi\in C^1_{\mathrm c}((0,T)\times\R^d\times \R)$
\[
\int_{\hnodesT}\partial_t \varphi(t,x,z) \, \hrho(t,\dxx xz)\dx t
= 
\int_{\hnodes_{\e T}} \partial_t \varphi(t,x,z) \,\hrho(t,\dxx xz)\dx t
= 
\int_{\nodes_T} \partial_t \varphi(t,x,z_\e(y))\, \rho(t,\dx xy)\dx t,
\]
where the first identity follows since $\hrho$ is supported on $\hnodes_{\e T}$. Similarly we have
\[
\begin{split}
\int_{\hnodesT}\nabla_x \varphi(t,x,z)  \,\hj^x(t,\dxx xz)\dx t
&=\int_{\hnodes_{\e T}} \nabla_x \varphi(t,x,z) \, \hj^x(t,\dxx xz) \dx t \\
&=\int_{\nodes_T}  \nabla_x \varphi(t,x,z_\e(y)) \,j^x(t,\dxx xy) \dx t.
\end{split}
\]
For the transformation of $j^y$ we write
\begin{multline*}
	\int_{\hnodesT} \partial_z \varphi(t,x,z) \, \hj^y(t,\dxx xz)\dx t
	= \int_{\hnodes_{\e T}}\partial_z \varphi(t,x,z) \, \hj^y(t,\dxx xz)\dx t\\
	= \int_{\nodes_T} z_\e'(y)\partial_z \varphi(t,x,z_\e(y))\,j^y(t,\dxx xy) \dx t
	= \int_{\nodes_T} \partial_y \varphi(t,x,z_\e(y))\, j^y(t,\dxx xy) \dx t.
\end{multline*}

\subsubsection{Rescaling of the dissipation functional}\label{ss:Kramers:append:rescaling:D}

Since $(\dx \hj^x/{\dx \hrho}) \circ \Psi_\e = \dx j^x/{\dx \rho}$ we find 
\begin{align*}
	\int_\nodes \frac1{2m_\Omega} \abs*{\frac{\dx j^x}{\dx\rho}}^2 \dx \rho
	&= \int_{\hnodes_\e} \frac1{2m_\Omega} \abs*{\frac{\dx \hj^x}{\dx\hrho}}^2 \dx \hrho,
\end{align*}
while for the $y$-term we find, writing in terms of Lebesgue densities,
\begin{align*}
	\int_\nodes \frac{|j^y(x,y)|^2}{2\tau_\e\rho(x,y)}  \dxx xy
	&= \int_\nodes \frac{|j^y(x,y)|^2}{2\tau_\e\pi_\e(x,y)u_\e(x,y)}  \dxx xy\\
	&\leftstackrel{\eqref{eq:z'tau}} = \frac{|\Omega|}{2m_\Upsilon} \int_\nodes \frac{z_\e'(y)}{u_\e(x,y)} |j^y(x,y)|^2 \dxx xy
	= \frac{|\Omega|}{2m_\Upsilon} \int_{\hnodes_\e} \frac1{\hu(x,z)} |\hj^y(x,z)|^2 \dxx xz.
\end{align*}

\medskip
Turning to the $\calR^*$-term in $\calD^T_\e$, we have
\begin{align*}
	\int_\nodes 2m_\Omega \ee^{-F_\e^\rho}\abs*{\nabla_x \sqrt {u^\calF}}^2\pi_\e(\dxx xy)
	&= \int_{\hnodes_\e} 2m_\Omega \ee^{-\hF_\e^\rho}\abs*{\nabla_x \sqrt {\hu^\calF}}^2\hpi_\e(\dxx xz).
\end{align*}
For the other term we calculate
\begin{align*}
	\int_\nodes 2\tau_\e \ee^{-F_\e^\rho} \abs*{\partial_y \sqrt {u^\calF}}^2\pi_\e(\dxx xy )\ \ 
	&\leftstackrel{\eqref{eq:z'tau}} = \frac{2m_\Upsilon}{|\Omega|} \int_\nodes \exp\bigl(-F_\e^\rho(t,x,y)\bigr) \bigl|\partial_z \sqrt{ \hu^\calF }\bigr|^2(t,x,z_\e(y)) \,z'_\e(y)\dxx xy\\
	&= \frac{2m_\Upsilon}{|\Omega|} \int_{\hnodes_\e} {\exp\bigl(-\hF_\e^\rho(t,x,z)\bigr)} \bigl|\partial_z\sqrt{ \hu^\calF}\bigr|^2(t,x,z)\dxx xz.
\end{align*}

\subsection{Compactness: Proof of Theorem~\ref{t:compactness}}

The theorem below is the main compactness result of Section~\ref{s:Kramers}. The technique is taken from~\cite[Th.~3.1-2]{ArnrichMielkePeletierSavareVeneroni12}, with some modifications to deal with the additional variable $x$ and the tilting function $\calF$. In addition, we use a slightly different method to prove the lower bound of Theorem~\ref{t:lower-bound} in Appendix~\ref{ss:lower-bound} below, which allows us to establish weaker forms of compactness for the traces.

\begin{theorem}[Compactness]
	\label{t:compactness}
	Let $\e_n$ be a sequence that converges to zero, and let $(\calF_{\e_n})_n$ satisfy Assumption~\ref{ass:F}. Let the sequence $(\rho_{\e_n},j_{\e_n})_n\subset \CE(0,T)$ satisfy the uniform dissipation bound~\eqref{est:initial-energy-D}.
	Then there exists a subsequence (for which we use the same notation) along which 
	\begin{enumerate}
	\item For the untransformed variables $(\rho_{\e_n}, j_{\e_n})$:
	\begin{enumerate}[ref={(\alph*)}]
		\item\label{t:compactness:rho_e}
		For each $t\in[0,T]$, $\rho_{\e_n}(t) \weaksto \rho_0(t)$ in $\calM_{\geq0}(\nodes)$, and 
		\begin{equation}
			\label{char:rho_e-limit}
			\rho_0(t,\dxx xy) = u^-(t,x) \,\gamma^a \delta_a(\dx y)\dx x  +u^+(t,x)\,\gamma^b \delta_b(\dx y)\dx x,
		\end{equation}
		for some functions $u^\pm\in L^1(\Omega_T)$.
		\item \label{t:compactness:j_e^x}
		$j_{{\e_n}}^x \weaksto j_{0}^x$ in $\calM(\nodes_T)$ and
		\begin{equation}
		\label{char:jx0}
		j_{0}^x(\dxxx txy) = \mathrm j_{0}^{x,-}(t,x)\dxx t x\,\delta_a(\dx y) 
		+ \mathrm j_{0}^{x,+}( t,x)\dxx t x\, \delta_b(\dx y) ,
		\end{equation}
		for some\/ $\mathrm j^{x,\pm}_{0}\in L^1(\Omega_T)$.
		\item \label{t:compactness:j_e^y}$j_{{\e_n}}^y\weaksto j_{0}^y$ in duality with $C_{\mathrm c}^1((0,T)\times \nodes)$, and 
		\begin{equation}
			\label{char:jy0}
			j_{0}^y(\dxxx txy) = 
			\ol {\textup\j}(t,x) \bONE_{[a,b]}(y)\dxxx tx y,
		\end{equation}
		for some $\ol{\textup\j} \in L^1(\Omega_T)$.

		\item \label{t:compactness:traces}
		The traces $u_{{\e_n}}^-(t,x) :=  u_{\e_n}(t,x,a)$ and $u_{{\e_n}}^+(t,x) := u_{\e_n}(t,x,b)$ of $u_\e := \dx{\rho_\e}/{\dx{\pi_\e}}$ are well defined and bounded in $L^1(\nodes_T)$, and converge in duality with $C_{\mathrm b}(\Omega_T)$ to the functions $u^\pm$ in~\eqref{char:rho_e-limit}.
	\end{enumerate}
	\item For the transformed variables $(\hrho_{\e_n}, \hj_{\e_n})$:
	\begin{enumerate}[resume,ref={(\alph*)}]
		\item \label{t:compactness:hrhoe} For each $t\in[0,T]$, $\hrho_{\e_n}(t) \weaksto \hrho_0(t)$ in $\calM_{\geq0}(\hnodesT)$ and 
			\[
			\hrho_0(t,\dxx xz) = u^-(t,x) \,\gamma^- \delta_0(\dx z)\dx x  +u^+(t,x)\,\gamma^+ \delta_1(\dx z)\dx x.
			\]
			
		\item \label{t:compactness:hj_e^x}
		$\hj_{{\e_n}}^x \weaksto \hj_{0}^x$ in $\calM(\hnodesT)$ and
			\[
			\hj_{0}^x(\dxxx txz) 
			= \mathrm j_{0}^{x,-}(t,x)\dxx t x\,\delta_0(\dx z) 
         		+ \mathrm j_{0}^{x,+}( t,x)\dxx t x\, \delta_1(\dx z) ,
			\]
			
		\item \label{t:compactness:hjey}
		$\hj_{{\e_n}}^y \weaksto \hj_{0}^y$ in $\calM(\hnodesT)$ and
		\begin{equation*}
			\hj_{0}^y(\dxxx txz) = 
			\ol {\textup\j}(t,x) \bONE_{[0,1]}(z)\dxxx tx z.
		\end{equation*}
	\end{enumerate}
	Note that the functions $u^\pm$, $\mathrm j_0^{x,\pm}$, and $\ol{\textup\j}$ above are the same as those in the characterizations~\eqref{char:rho_e-limit}, \eqref{char:jx0}, and~\eqref{char:jy0}.
\item Finally, \label{t:compactness:F}
		for each time $t$, the derivatives $F_{\e_n}^\rho(t) = \rmD\calF_{\e_n}(\rho_{\e_n}(t))$ converge uniformly on~$\nodes$ to the limit $F_0^\rho(t) = \rmD\calF_0(\rho_0(t))$. The transformed derivatives $\hF_{\e_n}^\rho(t,x,z)$ converge pointwise to $F_0^\rho(t,x,c)$ for all $0<z<1$ and all $(t,x)\in \Omega_T$.
\end{enumerate}
\end{theorem}

\begin{remark}[Convergence in $\CE$]
This compactness result implies that the transformed pair $(\hrho_{\e_n},\hj_{\e_n})$ converges in $\CE(0,T)$ (on the transformed domain $[0,T]\times \hnodes$). For the original pair $(\rho_{\e_n},j_{\e_n})$ the convergence of $j^y_{\e_n}$ is a little weaker than that required in Definition~\ref{def:converge-in-CE}, but strong enough to pass to the limit in the continuity equation.
\end{remark}

\begin{proof}
	In this proof we will write $\e$ instead of $\e_n$ or a subsequence of $\e_n$ to simplify the notation. This also means that we extract subsequences without this being reflected in the notation. We will use constants $C$ that vary from line to line, but which only depend on the supremum in~\eqref{est:initial-energy-D} and the bound
	\begin{equation}
		\label{est:F}
		A := \sup_{\e} \|F_\e^\rho\|_\infty <\infty.
	\end{equation}
	As remarked in Section~\ref{ss:Kramers-setting} we can assume without loss of generality that $\rho_\e(t)$ has mass one for each $t$.
	
	\medskip
	
	\refstepcounter{step}
	\textbf{Step \thestep. Basic estimates.}
	From the bounds~\eqref{est:initial-energy-D} and~\eqref{est:F} we deduce the following basic estimates:
	\begin{alignat}2
		\label{est:basic1}
		&\int_{\nodes_T} \abs*{\frac{\dx j_\e^x}{\dx\rho_\e}}^2\dx\rho_\e\dx t\leq C, &\qquad 
		&\int_{\nodes_T}\bra*{\,\abs*{\nabla_x \sqrt {u^\calF_\e}}^2 + \abs*{\partial_y \sqrt {u^\calF_\e}}^2\,}\,\pi_\e(\dxx xy )\dx t\leq C,
		\\
		&\int_{\hnodes_{\e T}}  \frac{\abs*{\hj^x_\e}^2}{\hu_\e}\dxxx xz t \leq C,
		    &\qquad 
		&\int_{\hnodes_{\e T}} \abs*{\partial_z \sqrt{\hu^\calF_\e}}^2 \dxxx xzt \leq C,
		\label{est:basic2}
	\end{alignat}

	\refstepcounter{step}
	\textbf{Step \thestep. First convergence results.}
	\label{step:basic-estimates}
	By the chain rule, Lemma~\ref{l:Kramers-CR}, we have for each~$t$ the estimate
	\begin{equation}
		\label{est:basic}
		\sup_{\e,t} \RelEnt(\rho_\e(t)|\pi_\e) 
		\leq  \calE_\e^\calF(\rho_\e(t=0)) + |\calF_\e(\rho_\e(t))| + \calD_\e^T(\rho_\e,j_\e;\calF_\e) \leq C,
	\end{equation}
	which implies that we can extract a subsequence along which $\rho_\e$ converges narrowly on $\nodes_T$ to a limit $\rho_0$ of the form
	\begin{align*}
		\rho_0(\dxxx txy) = \rho_{0}^-(\dxx tx) \delta_a(\dx y) + \rho_{0}^+(\dxx tx) \delta_b(\dx y).
	\end{align*}
	Since  $\rho_\e(\dx t, \nodes) = \dx t$, the narrow convergence implies that also $\rho_0(\dx t, \nodes)= \dx t$, and we can therefore disintegrate $\rho_0$ and $\rho_0^\pm$ as  $\rho_0(\dxxx txy) = \rho_0(t,\dxx xy) \dx t$ and $\rho_0^\pm(\dxx tx) = \rho_0^\pm(t,\dx x)\dx t$.

	In addition we estimate
	\begin{align*}
		& \int_{\nodes_T} |j_\e^x | \leq \bra*{\int_{\nodes_T} \abs*{\frac{\dx j_\e^x}{\dx\rho_\e}}^2\dx\rho_\e\dx t}^{1/2} 
		\bra*{ \int_{\nodes_T} \dx \rho_\e\dx t}^{1/2} \stackrel{\eqref{est:basic1}}\leq C ,
	\end{align*}
	from which it follows that $j_\e^x \stackrel *\longrightharpoonup j_0^x$ in $\calM(\nodes_T)$ for some $j_0^x\in \calM(\nodes_T)$. By lower semicontinuity~\cite[Th.~2.38]{AmbrosioFuscoPallara00} we have
	\begin{equation}
		\label{est:lsc-jex}
		\int_{\nodes_T} \abs*{\frac{\dx j_0^x}{\dx\rho_0}}^2\dx\rho_0\dx t
		\leq \liminf_{n\to\infty} \int_{\nodes_T}\abs*{\frac{\dx j_\e^x}{\dx\rho_\e}}^2\dx\rho_\e\dx t ,
	\end{equation}
	which implies that $j_0^x\ll \rho_0$, so that we can write 
	\begin{equation}
		\label{eq:char-j0}
		j_{0}^x(t,\dxx xy) = j_{0}^{x,-}(t,\dx x)\,\delta_a(\dx y) 
		+ j_{0}^{x,+}(t,\dx x)\,\delta_b(\dx y) .
	\end{equation}
	From the narrow convergence of $j^x_\e$ and $\rho_\e$ and the equation $\partial_t \rho_\e + \div_x j^x_\e + \div_y j^y_\e=0$ we deduce that
	\[
	j_\e^y \stackrel*\longrightharpoonup j^y_0\qquad\text{in duality with }C^1_{\mathrm c}((0,T)\times \nodes),
	\]
	and by Lemma~\ref{l:char:j0} we find that $j_0^y$  has the structure~\eqref{char:jy0}
	\begin{equation}
	\label{char:j0-app}
	j^y(t,\dxx xy) = \overline \jmath (t,\dx x) \bONE_{[a,b]}(y)\dx y
	\qquad \text{for some $(\overline \jmath(t))_t\subset \calM(\Omega)$.}
\end{equation}
It follows that $(\rho_0,(j_0^x,j_0^y))$ satisfies~\eqref{eq:Kramers:weak-form-CE} and $(\rho_0,(j_0^x,\ol\jmath))$ satisfies~\eqref{eq:weak-form-CE0}.
	
	\medskip
	\refstepcounter{step}
	\textbf{Step \thestep. Estimate on $y$-traces.} \label{step:y-traces}
	We next prove the following estimate. 
	\begin{equation*}
		\label{est:y-trace}
		\int_{\Omega_T}\sup_{y\in \Upsilon} u_\e(t,x,y)\dxx tx 
		= \int_{\Omega_T} \sup_{z\in \hUpsilone}\hu_\e(t,x,z)\dxx tx \leq C.
	\end{equation*}
	
	Set $\hp_\e := \sqrt{\hu_\e^\calF}$. Since $\partial_z \hp_\e$ has finite  $L^2(\Omega_T\times \hUpsilone)$-norm by~\eqref{est:basic2}, the trace of $\sqrt{\hu_\e^\calF}$ at any $z\in \hUpsilone$ is well defined as element of $L^2(\Omega_T)$; consequently, the trace of $\hu_\e^\calF$ is well defined as element of $L^1(\Omega_T)$, and by the continuity of $\hF_\e^\rho$ in $z$ also the trace of $\hu_\e$ is well defined in $L^1(\Omega_T)$. By transformation the same holds for $u_\e$. 
	
	We also have 
	\[
	\htheta_\e(t,x) := \sup_{z,z'\in \hUpsilone} |\hp_\e(t,x,z)-\hp_e(t,x,z')|
	\leq |\hUpsilone|^{1/2} \bra*{\int_{\hUpsilone} {\partial_z\hp_\e(t,x,z)}^2 \dx z}^{1/2},
	\]
	so that using \eqref{est:basic2} and the {boundedness} of $\hUpsilone$ we find
	\[
	\int_{\Omega_T} \htheta_\e(t,x)^2 \dxx tx \leq C.
	\]
	It follows that for any $(t,x)\in\Omega_T$ and $z\in \hUpsilone$, 
	\begin{align*}
		\hp_\e(t,x,z) &\leq 
		|\Omega| \int_{\hUpsilone} \hp_\e(t,x,z')\,\hpi_\e^y(\dx z')
		+ \htheta_\e(t,x)\\
		&\leq C \bra*{\int_{\hUpsilone} u_\e^\calF(t,x,z')\,\hpi_\e^y(\dx z')}^{1/2}
		+ \htheta_\e(t,x)\\
		&\leq C \ee^{A/2} \sqrt{\hrho_\e(t,x,\hUpsilone)} + \htheta_\e(t,x).
	\end{align*}
	Then
	\begin{align*}
		\int_{\Omega_T} \sup_{y\in \Upsilon}u_\e(t,x,y)\dxx tx
		&= \int_{\Omega_T}  \sup_{z\in \hUpsilone}\hu_\e(t,x,z)\dxx tx
		\leq \ee^A   \int_{\Omega_T} \sup_{z\in \hUpsilone} \hp_\e(t,x,z)^2\dxx tx\\
		&\leq C\pra*{ \int_{\Omega_T} \hrho_\e(t,x,\hUpsilone)\dxx tx + \int_{\Omega_T} \htheta(t,x)^2\dxx tx}
		\leq C.
	\end{align*}
	\medskip
	
	\refstepcounter{step}
	\textbf{Step \thestep. Bounded oscillation of the density $u_\e$ away from the saddle.}
	We now establish the following property: for any open interval $\Upsilon_1\subset \Upsilon$ such that (a) $\sup_{\Upsilon_1} H < H(c)$ and (b) $a\in \Upsilon_1$ or $b\in \Upsilon_1$, 
	\begin{equation}
		\label{est:u-constant-away-from-saddle}
		\limsup_{\e\to0} \int_{\Omega_T} \sup_{y,y'\in \Upsilon_1} \abs*{u_\e(t,x,y)-u_\e(t,x,y')}\dxx tx\leq C\moc(|\Upsilon_1|),
	\end{equation}
	where the modulus of continuity $\moc$ is introduced in~\eqref{eqdef:moc}.

	Set $p_\e := \sqrt{u^\calF_\e}$ and define
	\[
	\theta^{\Upsilon_1}_\e(t,x) := 
	\sup_{y,y'\in \Upsilon_1} |p_\e(t,x,y)-p_\e(t,x,y')|
	\leq \bra*{\int_{\Upsilon_1} (\partial_y p_\e)^2(t,x,y)\pi_\e^y(y)\dx y}^{1/2}
	\bra*{\int_{\Upsilon_1} \frac{\dx y}{\pi_\e^y(y)}}^{1/2},
	\]
	for which we estimate
	\begin{align*}
		\int_{\Omega_T} {\theta^{\Upsilon_1}_\e}(t,x)^2 \dxx tx
		&\stackrel{\eqref{est:basic1}} \leq \frac{C}{\e\tau_\e} \, {{\PartSum_\e}} \int_{\Upsilon_1} \ee^{H(y)/\e}\dx y  
		\longrightarrow 0\qquad\text{as }\e\to0
	\end{align*}
	where the convergence to zero follows  from observing that $\PartSum_\e =O(\sqrt\e)$, $\tau_\e \sim \e \ee^{H(c)/\e}$, and $\sup_{\Upsilon_1} H < H(c)$.
	
	Similarly to Step~\ref{step:y-traces} we then derive that 
	\[
	\sup_{y\in \Upsilon_1} p_\e(t,x,y)
	\leq (\pi_\e^y(\Upsilon_1))^{-1}\int_{\Upsilon_1} p_\e(t,x,y')\pi_\e^y(\dx y') + \theta^{\Upsilon_1}_\e(t,x).
	\]
	Note that since $a$ or $b$ is an element of $\Upsilon_1$, we have $\liminf_{\e\to0} \pi_\e^y(\Upsilon_1)>0$, and again similarly to Step~\ref{step:y-traces} we find that 
	\begin{align}
		\notag	
		\MoveEqLeft \int_{\Omega_T} \sup_{y,y'\in \Upsilon_1}
		\abs*{u_\e^\calF(t,x,y) - u_\e^\calF(t,x,y')}\dxx tx\\
		\notag
		&\leq \int_{\Omega_T} \theta^{\Upsilon_1}_\e(t,x)
		\pra*{(\pi_\e^y(\Upsilon_1))^{-1/2}\bra*{\int_{\Upsilon_1} u_\e^\calF(t,x,y')\pi_\e^y(\dx y')} + \theta^{\Upsilon_1}_\e(t,x)}\dxx tx\\
		&\longrightarrow 0 \qquad\text{as } \e\to0.
		\label{est:diff-u-F-to-zero}
	\end{align}
	
	Finally, to derive~\eqref{est:u-constant-away-from-saddle} we note that 
	\begin{align*}
		\MoveEqLeft\abs*{u_\e(t,x,y) - u_\e(t,x,y')}\\
		&\leq \ee^A \abs*{u_\e^\calF(t,x,y) - u_\e^\calF(t,x,y')} + u_\e(t,x,y) \abs*{1-\exp\bigl(F_\e^\rho(t,x,y)-F_\e^\rho(t,x,y')\bigr)}\\
		&\leq \ee^A \abs*{u_\e^\calF(t,x,y) - u_\e^\calF(t,x,y')}
		+ \ee^A u_\e(t,x,y) \abs*{F_\e^\rho(t,x,y)-F_\e^\rho(t,x,y')}
	\end{align*}
	Then
	\begin{align*}
		\MoveEqLeft\int_{\Omega_T} \sup_{y,y'\in \Upsilon_1} \abs*{u_\e(t,x,y) - u_\e(t,x,y')}
		\leq C \underbrace{\int_{\Omega_T} \sup_{y,y'\in \Upsilon_1} \abs*{u^\calF_\e(t,x,y) - u^\calF_\e(t,x,y')}}_{\to 0\text{ by }\eqref{est:diff-u-F-to-zero}}\\
		&+ C
		\underbrace{\sup_{(t,x)\in \Omega_T} \sup_{y',y''\in \Upsilon_1} \abs*{F_\e^\rho(t,x,y')-F_\e^\rho(t,x,y'')}}_{\leq \moc(|\Upsilon_1|)}\;
		\underbrace{
			\sup_{y\in \Upsilon_1} \int_{\Omega_T} u_\e(t,x,y) \dxx tx
	}_{(*)}.
\end{align*}
The term marked $(*)$ is bounded as $\e\to0$ by~\eqref{est:y-trace}, and 
the estimate~\eqref{est:u-constant-away-from-saddle}  follows.

\medskip

\refstepcounter{step}
\label{step:conv-of-traces}
\textbf{Step \thestep. Narrow convergence of the traces.}
By~\eqref{est:y-trace},  the traces
\[
u_{\e}^-(t,x) := u_\e(t,x,a) \qquad\text{and}\qquad
u_{\e}^+(t,x) := u_\e(t,x,b)
\]
are bounded in $L^1(\Omega_T)$. Along a subsequence they therefore converge narrowly, in duality with $C_{\mathrm b}(\Omega_T)$, to measure limits $\mu^\pm\in \calM_{\geq0}(\Omega_T)$. In this step we connect these limits to  $\rho_0^\pm$:
\[
\rho_0^-(t,\dx x) = \gamma^a\mu^-(t,\dx x) 
\qquad\text{and}\qquad
\rho_0^+(t,\dx x) = \gamma^b\mu^+(t,\dx x).
\]
We give the details for $\mu^+$ and $\rho_0^+$, and the other case follows from analogous arguments. 

Define the following `projection' of $\rho_\e$ onto $y=b$, 
\[
\rho_{\e}^+(t, \dx x) := \int_{y\in\Upsilon} z_\e(y)\rho_\e(t, \dxx xy).
\]
This sequence converges narrowly on $\Omega_T$ to $\rho_0^+$, since $\rho_\e z_\e(y)$ concentrates onto $y=b$. 

We then calculate for $\varphi\in C_{\mathrm b}(\Omega_T)$
\begin{align*}
\int_{\Omega_T} \varphi \rho_0^+ 
&= \lim_{\e\to0} \int_{\Omega_T} \varphi \rho_\e^+ 
= \lim_{\e\to0} \int_{\nodes_T} \varphi(t,x) z_\e(y) \rho_\e(t,\dxx xy)\dx t\\
&= \lim_{\e\to0} \int_{\nodes_T} \varphi(t,x) z_\e(y) u_\e(t,x,y)\, \pi_\e(\dxx xy)\dx t\\
&= \lim_{\e\to0} \int_{\nodes_T} \varphi(t,x) z_\e(y) u_\e^+(t,x)\, \pi_\e(\dxx xy)\dx t
+ \lim_{\e\to0} R_\e,
\end{align*}
with 
\[
R_\e = \int_{\nodes_T} \varphi(t,x) z_\e(y) \bra[\big]{u_\e(t,x,y)-u_\e(t,x,b)}\, \pi_\e(\dxx xy)\dx t.
\]
We now show that $R_\e$ converges to zero, using the previous two steps. We split the integral into two parts following $\Upsilon = \Upsilon_1\dot \cup \Upsilon_2$, where $\Upsilon_1$ is an open neighbourhood of $b$ such that $\sup_{\Upsilon_1}H < H(c)$. To estimate the integral over $\Upsilon_1$, we write
\begin{align*}
	\MoveEqLeft\abs*{\int_{\Omega_T\times \Upsilon_1}\varphi(t,x)\bra[\big]{u_\e(t,x,y)-u_\e(t,x,b)}z_\e(y) \pi_\e(\dxx xy)\dx t\,}\\
	&\leq\  \norm\varphi_\infty \int_{\Omega_T} 
	{\sup_{y\in\Upsilon_1} \abs*{u_\e(t,x,y)-u_\e(t,x,b)}}\dxx tx 
	\;\int_{\Upsilon_1} \pi_\e^y(\dx y),
\end{align*}
and by~\eqref{est:u-constant-away-from-saddle} this expression can be made as small as required. We estimate the remaining integral by 
\begin{align*}
	\MoveEqLeft\abs*{ \int_{\Omega_T\times \Upsilon_2}\varphi(t,x)\bra[\big]{u_\e(t,x,y)-u_\e(t,x,b)}
	\,z_\e(y) \,\pi_\e(\dxx xy)\dx t\,}\\
	&\leq \|\varphi\|_\infty 
	\int_{\Omega_T} {\sup_{y\in \Upsilon_2} u_\e(t,x,y) }\dxx xt 
	\;\int_{\Upsilon_2} \abs{z_\e(y)}\, \pi_\e^y(\dx y)
	\longrightarrow 0,
\end{align*}
where the final integral vanishes because $|z_\e|\,\pi_\e^y$ concentrates onto $y=b\not\in \Upsilon_2$. 

We therefore have 
\begin{align*}
\int_{\Omega_T} \varphi(t,x) \rho_0^+(t,\dx x)\dx t 
&= \lim_{\e\to0} \int_{\nodes_T} \varphi(t,x) z_\e(y) u_\e^+(t,x)\, \pi_\e(\dxx xy)\dx t\\
&= \lim_{\e\to0} \int_{\Omega_T} \varphi(t,x)  u_\e^+(t,x)\dxx xt 
\; \underbrace{\int_\Upsilon z_\e(y) \pi_\e^y(\dx y)}_{\to \;\gamma^b}\\[-2\jot]
&= \gamma^b \int_{\Omega_T} \varphi(t,x) \mu^+(t,x)\dxx xt, 
\end{align*}
which establishes the identification $\rho_0^+(t,\dx x) = \gamma^b \mu^+(t,\dx x) $. 

\medskip
\refstepcounter{step}
\textbf{Step~\thestep. Convergence of $\hrho_\e$, $\hj_\e^x$, and $\hj_\e^y$, and characterization of the limits.}
\label{step:hat-convergence}
We next show that, in the sense of narrow convergence on $\nodes_T$,
\begin{subequations}
\begin{align}
	\label{eq:char-hrho0}
	\hrho_\e &\stackrel *\longrightharpoonup \hrho_0, \quad \text{with}
	&& \hrho_{0}(\dxx txz) = \hrho_{0}^-(t,\dx x)\delta_0(\dx z) \dx t+
	\hrho_{0}^+(t,\dx x)\delta_1(\dx z)\dx t,
	\\
	\hj_\e^x &\stackrel *\longrightharpoonup \hj_0^x,  \quad\text{with}
	&& \hj_0^x(\dxxx txz) = j^{x,-}_{0}(t,\dx x)\delta_0(\dx z)\dx t +
	j^{x,+}_{0}(t,\dx x)\delta_1(\dx z)\dx t ,
	\label{eq:char-hj0x}
	\\
	\hj_\e^y &\stackrel *\longrightharpoonup \hj_0^y,  \quad \text{with}
	&&\hj_{0}^y(\dxxx txz) = \ol\jmath(t,\dx x) \bONE_{[0,1]}(z)\dxx zt.
	\label{eq:char-hj0y}
\end{align}
\end{subequations}
Note that the restricted measures $j^{x,\pm}_0$ and $\ol\jmath$ in~\eqref{eq:char-hj0x} and~\eqref{eq:char-hj0y} above are the same as in the limits~\eqref{eq:char-j0} and~\eqref{char:j0-app} of the untransformed fluxes $j_\e$. We show in Step~\ref{step:convergence-rho-each-time} below that also $\hrho_0^{\pm} = \rho_0^\pm$. 

The convergence of $\hrho_\e$ and $\hj_\e^x$ follows along the same lines as in Step~\ref{step:basic-estimates}; in particular we find a similar estimate
\begin{equation}
	\label{est:hjex}
	\int_{\hnodesT} |\hj^x_\e| \leq C,
\end{equation}
and the convergence
\[	
\hj_\e^x \stackrel *\longrightharpoonup \hj_0^x,  \quad\text{with}\quad 
	 \hj_0^x(\dxxx txz) = \hj^{x,-}_{0}(t,\dx x)\delta_0(\dx z)\dx t +
	\hj^{x,+}_{0}(t,\dx x)\delta_1(\dx z)\dx t ,
\]
To show that $\hj_0^{x,\pm} = j_0^{x,\pm}$ we pick $\varphi\in C_{\mathrm c}^1((0,T)\times \Omega;\R^d)$ and calculate
\begin{align*}
\MoveEqLeft\int_{\Omega_T}\varphi(t,x) j_0^{x,+}(t,\dx x)\dx t
= \int_{\nodes_T}\varphi(t,x) \,\bONE\{y>c\} \, j_0^{x}(t,\dx x)\dxx y t \\
&\leftstackrel{(1)}= \lim_{\e\to0}  \int_{\nodes_T}\varphi(t,x)\, z_\e(y)\,  j_\e^{x}(t,\dx x)\dxx y t 
\stackrel{(2)}= \lim_{\e\to0}  \int_{\hnodes_T}\varphi(t,x) \,z \, \hj_\e^{x}(t,\dx x)\dxx z t \\
&=  \int_{\hnodes_T}\varphi(t,x) \,z \, \hj_0^{x}(t,\dx x)\dxx z t 
= \int_{\Omega_T} \varphi(t,x) \, \hj_0^{x,+}(t,\dx x) \dx t.
\end{align*}
The identity $(1)$ above follows because $z_\e(y)$ converges to $\bONE\{y>c\}$, uniformly away from $y=c$, and $j_0^x$ does not charge $\{y=c\}$. Identity $(2)$ is the definition~\eqref{eqdef:hat-transformed-objects} of $\hj_\e$. This chain of identities implies that $\hj_0^{x,+} = j_0^{x,+}$, and similarly we have  $\hj_0^{x,-} = j_0^{x,-}$.

For $\hj_\e^y$  we estimate 
\begin{align}
	\notag
	\int_{\hnodesT} |\hj_\e^y | 
	=\int_{\hnodes_{\e T}} |\hj_\e^y | 
	&\leq \bra*{\int_{\hnodes_{\e T}}  \frac{\abs*{\hj^x_\e}^2}{\hu_\e}\dxxx xz t}^{1/2} 
	\bra*{\int_{\hnodes_{\e T}} \hu_\e\dxxx txz}^{1/2}\\
	&\leftstackrel{\eqref{est:basic2}} \leq  C
	+ C \int_{\Omega_T} \sup_{z\in\hUpsilone} \hu_\e(t,x,z)\dxx tx
	\stackrel{\eqref{est:y-trace}}\leq C.
	\label{est:hjey}
\end{align}
We find that we can extract a subsequence such that 
\[
\hj_\e^y \stackrel*\longrightharpoonup \hj^y_0
\qquad\text{on } \hnodesT
\]
for some finite measure $\hj^y_0$ supported on $\hnodesT$.

The continuity equation $\partial_t \hrho_\e +\div_x \hj_\e^x + \div_y \hj_\e^y = 0$ passes to the limit under narrow convergence. By Lemma~\ref{l:char:j0} we find that $\hj_0^y$  has the structure
\begin{equation}
\label{eq:compactness-structure-hjy0}
\hj_{0}^y(\dxxx txz) = \hat{\ol\jmath}(t,\dx x) \bONE_{[0,1]}(z)\dxx zt.
\end{equation}
To show that $\hat{\ol\jmath} = \ol\jmath$, we pick $\varphi\in C_{\mathrm c}^1((0,T)\times \Omega)$ and set 
\begin{gather*}
\psi(t,x,y) := \frac{y-a}{b-a} \varphi(t,x),\\
\varphi_\e(t,x,y) := z_\e(y) \varphi(t,x), 
\qquad
\varphi_0(t,x,y) := \varphi(t,x)\bONE\{y>c\}	.
\end{gather*}
We then similarly calculate
\begin{align*}
-\int_{\Omega_T} \ol\jmath \,\varphi 
&= -\int_{\nodes_T} \ol\jmath \,\partial_y \psi \, \bONE[a,b](y) 
\stackrel{\eqref{char:j0-app}} = -\int_{\nodes_T} j_0^y\,\partial_y \psi
\stackrel{\eqref{eq:Kramers:weak-form-CE}} = \int_{\nodes_T} \bigl[ \rho_0 \partial_t\psi + j_0^x \nabla_x \psi\bigr]\\
&
=\int_{\Omega_T} \bigl[ \rho_0^+ \partial_t\varphi + j_0^{x,+} \nabla_x \varphi\bigr] 
= \int_{\nodes_T} \bigl[ \rho_0 \partial_t\varphi_0 + j_0^x \nabla_x \varphi_0\bigr]
= \lim_{\e\to0} \int_{\nodes_T} \bigl[ \rho_\e \partial_t\varphi_\e + j_\e^x \nabla_x \varphi_\e\bigr]\\
&
= \lim_{\e\to0} \int_{\hnodes_T} \bigl[ \hrho_\e \, z\, \partial_t\varphi + \hj_\e^x \, z \, \nabla_x \varphi\bigr]
=  -\lim_{\e\to0} \int_{\hnodes_T} \hj^y_\e \, \varphi 
\stackrel{\eqref{eq:compactness-structure-hjy0}} = - \int_{\Omega_T} \hat{\ol\jmath} \, \varphi .
\end{align*}
It follows that $\hat{\ol\jmath} = {\ol\jmath}$.

\medskip
\refstepcounter{step}\label{step:convergence-rho-each-time}
\textbf{Step~\thestep. Convergence of $\hrho_\e$ and $\rho_\e$ at each time $t$, and equality of $\rho_0^\pm$ with $\hrho_0^\pm$.}
The continuity equation implies that for all $\varphi\in C_{\mathrm b}^1(\hnodes)$ and all $0\leq t_0<t_1\leq T$,
\[
\int_{\hnodes} \varphi\, \bra[\big]{\dx \hrho_\e(t_1) - {\dx\hrho_\e(t_0)}}
= - \int_{t_0}^{t_1}\!\! \int_{\hnodes} \,\bra[\big]{ \hj^x_\e \nabla_x \varphi + \hj^y_\e \nabla_y \varphi}\dxxx txz.
\]
We then estimate the $L^1$-Wasserstein distance $W_1$~\cite[\S 7.1]{AmbrosioGigliSavare08}  by
\begin{align*}
	W_1(\hrho_\e(t_1),\hrho_\e(t_0)) 
	&:= \sup \set*{\int_{\hnodes} \varphi\, \bigl({\dx \hrho_\e(t_1)} - {\dx\hrho_\e(t_0)}\bigr)
		: \varphi\in C_{\mathrm b}^1(\hnodes), \ \|\nabla\varphi\|_\infty \leq 1}\\
	&\leq \int_{t_0}^{t_1} \!\!\int_{\hnodes}\bigl(|\hj^x_\e| + |\hj^y_\e|\bigr) \stackrel{\eqref{est:hjex},\eqref{est:hjey}} \leq C.
\end{align*}
This implies that the sequence of functions $t\mapsto \hrho_\e(t)$ has bounded variation in $W_1$, and by Helly's compactness theorem we can extract a subsequence that converges at every~$t$. Since the supports of $\hrho_\e$ are uniformly bounded, convergence in $W_1$ coincides with narrow convergence. Since the narrow limit of $\hrho_\e$ on $\nodes_T$ is characterized by~\eqref{eq:char-hrho0}, we identify the limit of  $\hrho_\e(t)$ as $\hrho_0^-(t,\dx x)\delta_0(\dx z)+ \hrho_0^+(t,\dx x) \delta_1(\dx z)$.

\smallskip

We next show that also $\rho_\e(t)$ converges for each $t$. Fix $t\in [0,T]$; momentarily using explicit sequences $\e_n$ again, use the bound~\eqref{est:basic} to extract a subsequence $\e'_n$ such that $\rho_{\e'_n}(t)$ converges narrowly on $\nodes$ to a limit $\wt\rho^-(\dx x)\delta_a(\dx y) + \wt\rho^+(\dx x)\delta_b(\dx y)$ for some $\wt\rho^\pm\in \calM_{\geq0} (\Omega)$. If we show that $\wt\rho^\pm$ equal $\hrho_0^\pm$ then the convergence holds for the whole sequence $\e_n$, and consequently for all $t$. 

Fix a function $\varphi\in C_{\mathrm b}(\Omega\times\R)$ with $\varphi(x,1) = 0$ for all $x$,  and define $\psi_\e(x,y) := \varphi(x,z_\e(y))$. Note that $\psi_\e$ converges pointwise  to the function $\psi_0$, with $\psi_0(x,y)$ equal to $\varphi(x,0)$ for $y<c$ and equal to  $0$ for $y>c$, and the convergence is uniform away from $y=c$. We then calculate
\begin{align*}
	\int_\Omega \wt\rho^-(t,\dx x)\varphi(x,0)
	&=\lim_{n\to\infty} \int_\nodes \rho_{\e'_n} (t,\dxx xy) \psi_{\e'_n}(x,y)
	= \lim_{n\to\infty} \int_\nodes \hrho_{\e'_n} (t,\dxx xz) \varphi(x,z)\\
	&= \int_\Omega \hrho_0^-(t,\dx x)\varphi(x,0) .
\end{align*}
This proves that $\wt\rho^- = \hrho_{0}^-(t)$, and by a similar argument $\wt\rho^+ = \hrho_{0}^+(t)$. Therefore $\rho_\e(t)$ converges narrowly at each $t$ to $\hrho_0^-(t,\dx x)\delta_a(\dx y)+ \hrho_0^+(t,\dx x) \delta_b(\dx y)$. From now on we no longer distinguish between $\rho_0^\pm$ and $\hrho_0^\pm$.

\medskip
\refstepcounter{step}
\textbf{Step \thestep. Uniform convergence of $F_\e^\rho(t)$ on $\nodes$ to $F_0^\rho(t)$ for  each $t$, and pointwise convergence of $\hF_\e^\rho(t,x,z)$ to $F_0^\rho(t,x,c)$ for each $(t,x)\in \Omega_T$ and $0<z<1$.}
The narrow convergence of $\rho_\e(t)$ to $\rho_0(t)$ at each time $t$ implies with Assumption~\ref{ass:F} that $F_\e^\rho(t)$ converges uniformly on $\nodes$ to $F_0^\rho(t) := \rmD\calF_0(\rho_0(t))$ at each time $t$. Since $z_\e^{-1}(z) \to c$ for all $0<z<1$, the transformed versions $\hF_\e^\rho(t,x,z) = F_\e^\rho(t,x,z_\e^{-1}(z))$ converge to $F_0^\rho(t,x,c)$ for all $(t,x)\in \Omega_T$ and $z\in (0,1)$. This concludes part~\ref{t:compactness:F} of Theorem~\ref{t:compactness}.

\medskip
\refstepcounter{step}
\textbf{Step~\thestep. Lower semicontinuity of the $x$-Fisher-information and absolute continuity of the measures $\rho_{0}^\pm$.}
Using the duality characterization for $\mu\in\ProbMeas(\Omega)$ and $f\in C^1_b(\Omega)$
\begin{equation*}
	\sup_{\varphi\in C^1_{\mathrm c}(\Omega;\R^d)} \int_{\Omega} \mu \Bigl[\ee^{f}\div_x \varphi - \frac12 \ee^{2f}\abs*{\varphi}^2 \Bigr]
	= 
	\begin{cases}
		\ds
		2\int_{\Omega}\ee^{-f} \abs*{\nabla_x \sqrt {u\ee^f}}^2\dx x 
		& \text{if }\ds\mu(\dx x) = u(x)\dx x, \\
		+\infty & \text{otherwise,}  
	\end{cases}
\end{equation*}
we find that the limit $\rho_0$ satisfies
\[
\sup_{\varphi\in C^1_{\mathrm c}(\nodes_T;\R^d)} \int_{\nodes_T} \rho_0 \Bigl[\ee^{F_0^\rho}\div_x \varphi - \frac12 \ee^{2F_0^\rho}\abs*{\varphi}^2 \Bigr] 
\leq M := \liminf_{\e\to0} 2\int_{\nodes_T}\ee^{-F_\e^\rho} \abs*{\nabla_x \sqrt {u^\calF_\e}}^2\pi_\e(\dxx xy)\dx t.
\]
Rewriting this expression as 
\[
\sup_{\varphi\in C_{\mathrm c}^1(\nodes_T;\R^d)}\sum_{y=a,b}\int_0^T\!\!\int_{\Omega} \rho_0(t,\dx x,y) \Bigl[\ee^{F_0^\rho}\div_x \varphi-\frac12 \ee^{2F_0^\rho}|\varphi|^2\Bigr](t,x,y)\dx t\leq M,
\]
we conclude that $\rho_{0}(t,\cdot,a),\rho_{0}(t,\cdot,b)\in \calM_{\geq0}(\Omega)$ are Lebesgue absolutely continuous for almost all $t$. Writing $\rho_0(t,\dx x,\cdot) = u^\pm(t,x) \gamma^{a,b} \dx x$ this provides the characterization~\eqref{char:rho0} of part~\ref{t:compactness:rho_e} of Theorem~\ref{t:compactness}, and since $\hrho_0^\pm = \rho_0^\pm$ also part~\ref{t:compactness:hrhoe}. In addition, in combination with Step~\ref{step:conv-of-traces} this also implies part~\ref{t:compactness:traces}. 

With an additional rewrite we also deduce the following estimate, which will be useful in the proof of Theorem~\ref{t:lower-bound} below,
\begin{equation}
	\label{est:lsc-FIx}
	2\int_{\nodes_T} \ee^{-F_0^\rho} \abs*{\nabla _x\sqrt{u_0^\calF}}^2  \pi_0(\dxx xy)\dx t
	\leq M,
\end{equation}
where we set 
$u_0(t,x,a) := u^-(t,x)$, $u_0(t,x,b) := u^+(t,x)$, and
and $u_0^\calF := u_0 \ee^{F_0^\rho}$. 

\medskip
\refstepcounter{step}
\textbf{Step~\thestep. Absolute continuity of $j_0^x$, $\hj_0^x$, and $\ol\jmath$.}
In Steps~\ref{step:basic-estimates} and~\ref{step:hat-convergence} we found that $j_0^x\ll \rho_0$ and $\hj_0^x\ll \hrho_0$, which by the previous step implies that 
\[
j_0^{x,\pm}(t,\dx x) = \hj_0^{x,\pm}(t,\dx x) = {\mathrm{j}}_0^{x,\pm}(t,x)\dx x,
\]
for some $\mathrm{j}_0^{x,\pm}\in L^1(\Omega_T)$. This concludes the proof of parts~\ref{t:compactness:j_e^x} and~\ref{t:compactness:hj_e^x} of Theorem~\ref{t:compactness}.

In the course of the proof of Theorem~\ref{t:lower-bound}, in the next section, we prove that the uniform dissipation bound~\eqref{est:initial-energy-D} guarantees that $\ol\jmath$ is Lebesgue absolutely continuous (see~\eqref{eq:j-perp-is-zero}). This concludes the proof of parts~\ref{t:compactness:j_e^y} and~\ref{t:compactness:hjey}, and thereby also the proof of Theorem~\ref{t:compactness}. 
\end{proof}

\subsection{Lower bound: Proof of Theorem~\ref{t:lower-bound}}
\label{ss:lower-bound}

The inequalities for the first integral $\calD_\e^{T,x}(\rho_\e,j_\e;\calF_\e)$ in~\eqref{eqdef:D-part1} have already been established in the proof of Theorem~\ref{t:compactness}, in~\eqref{est:lsc-jex} and~\eqref{est:lsc-FIx}.

\medskip

The starting point of the proof for the lower semicontinuity of the second integral $\calD_\e^{T,y}$ in~\eqref{eqdef:D-part2} is the rewriting~\eqref{eq:Kramers:lb:Dy}, which we reproduce for the convenience of the reader
\[
\calD_\e^{T,y}(\rho_\e,j_\e;\calF_\e) \geq\int_{\Omega_T} \CCs\bra*{\hj^y_\e(t,x,\cdot)\big|_{[0,1]};\hu_\e^\calF(t,x,a), \hu_\e^\calF(t,x,b); k_\e(t,x,\cdot)}\dxx tx ,
\]
where $k_\e(t,x,z) := m_\Upsilon |\Omega|^{-1} \ee^{-\hF^\rho_\e(t,x,z)}$.
Here the functional $\CCs:\calY \times L^1(0,1)\to [0,\infty]$ with the space $\calY := H^{-1}(0,1)\times \R^2$ defined in~\eqref{eqdef:calN}.


\medskip
Define the sequence of $\calY$-valued measures $\mu_\e$ on $\Omega_T$ to represent the argument of $\CCs$ in the expression above:
\begin{multline}
	\label{eqdef:mu_e}
	\dual {\mu_\e}\varphi := \int_{\Omega_T}
	\pra*{\int_0^1\hj^y_\e(\dxxx txz) \varphi_1(t,x)(z)
		+ \pra*{u^\calF_\e(t,x,a)\varphi_2(t,x) + u^\calF_\e(t,x,b)\varphi_3(t,x) }\dxx tx},
\end{multline}
for any $\varphi\in C_{\mathrm b}(\Omega_T; H^1_0(0,1)\times \R\times\R)$.

Since the sequence  $\mu_\e$ has bounded \emph{semivariation}~\cite[p.~72]{MarzShortt94}
\begin{align*}
	\|\mu_\e\|(\Omega_T) &\leq \int_{\Omega_T} \Bigl[
	\|\hj_\e^y(\dxx tx, \cdot)\|_{H^{-1}(0,1)} 
	+ \abs*{u^\calF_\e(t,x,a)} + \abs*{u^\calF_\e(t,x,b)}
	\Bigr]\\
	&\leftstackrel{\substack{\text{Th.~\ref{t:compactness},}\\\text{part~\ref{t:compactness:traces}}}}
	\leq C + C\int_{\Omega_T} \|\hj_\e^y(\dxx tx, \cdot)\|_{L^1(0,1)} \stackrel{\text{part~\ref{t:compactness:hjey}}} \leq C,
\end{align*}
we can extract a subsequence that converges weakly-\textasteriskcentered\ to a limit $\mu_0$~\cite[Cor.~1.4]{MarzShortt94}.
We can characterize this limit in terms of the limit objects given by Theorem~\ref{t:compactness}:
	for any $\varphi\in C_{\mathrm b}(\Omega_T; H^1_0(0,1)\times \R\times\R)$, we have 
	\begin{multline}
		\label{eqdef:m0}
		\dual {\mu_0}\varphi = 
		\int_{\Omega_T} \pra*{\, \ol \jmath(\dxx tx)\int_0^1\varphi_1(t,x)(z)\dx z
			+  \pra*{u^\calF_0(t,x,a)\varphi_2(t,x) + u^\calF_0(t,x,b)\varphi_3(t,x) }
			\dxx tx}
	\end{multline}
Indeed, since $\hj^y_\e|_{z\in[0,1]}$ converges narrowly in $\calM(\Omega_T\times [0,1])$ to a limit of the form $\ol\jmath(\dxx tx)\dx z$, the first term in~\eqref{eqdef:mu_e} converges to the first term in~\eqref{eqdef:m0}. By the narrow convergence of the traces $u_\e^\calF$ at $y=a,b$ (part~\ref{t:compactness:traces} of Theorem~\ref{t:compactness}) the corresponding integrals in~\eqref{eqdef:mu_e} converge as well. 

When $(t_\e,x_\e)\to(t,x)$, by part~\ref{t:compactness:F} of Theorem~\ref{t:compactness},
\[
k_\e(t_\e,x_\e)\to k(t,x):=\frac{m_\Upsilon}{|\Omega|}\ee^{-F_0^\rho(t,x,c)}  \quad\text{in }L^1(0,1)\text{ and with lower bound }k_\e\geq k_0>0.
\]
It follows from part~\ref{l:G-lsc:part1} of Lemma~\ref{l:G-lsc}
\[
\liminf_{\e\to0}\int_{\Omega_T}  \CCs\bra*{\frac{\dx{\mu_\e}}{\dx{|\mu_\e|}}(t,x);k_\e(t,x)} |\mu_\e|(\dxx tx)
\geq 
\int_{\Omega_T}  \CCs\bra*{\frac{\dx{\mu_0}}{\dx{|\mu_0|}}(t,x); k(t,x)} |\mu_0|(\dxx tx).
\]	
We  decompose  $|\mu_0| = |\mu_0|^{(r)}\dxx tx + |\mu_0|^\perp$ into a Lebesgue regular and a singular part, and we apply the same decomposition to $\mu_0$ by 
\[
\mu_0^{(r)}(t,x)\dxx tx := \frac{\dx{\mu_0}}{\dx{|\mu_0|}}(t,x)|\mu_0|^{(r)}(\dxx tx)
\qquad\text{and}\qquad
\mu_0^\perp :=  \frac{\dx{\mu_0}}{\dx{|\mu_0|}} |\mu_0^\perp|.
\]
Using the characterization~\eqref{eqdef:m0} we can write this decomposition as 
\begin{align*}
	\mu_0^{(r)}(t,x) &=: \bigl(\ol\jmath(t,x)\bONE_{[0,1]} ,  u_0^{\calF,a}(t,x), u_0^{\calF,b}(t,x) \bigr),\\
	\mu_0^\perp(\dx tx) &=: \bigl(\ol\jmath^\perp(\dxx tx), 0,0\bigr),
\end{align*}
in terms of an analogous decomposition $\ol\jmath(\dxx tx) = \ol\jmath(t,x)\dxx tx + \ol\jmath^\perp(\dxx tx)$.
Part~\ref{l:G-lsc:part2} of Lemma~\ref{l:G-lsc} then implies that 
\begin{equation}
\label{eq:j-perp-is-zero}
	\ol\jmath^\perp =0
\end{equation}
and therefore $\mu_0^\perp=0$.

By part~\ref{l:props-N:char-const-j} of  Lemma~\ref{l:props-N}  we then find
\begin{multline*}
	\liminf_{\e\to0}\int_{\Omega_T}  \CCs\bra*{\frac{\dx{\mu_\e}}{\dx{|\mu_\e|}}(t,x);k_\e(t,x)} |\mu_\e|(\dxx tx)\\
	\geq \int_{\Omega_T}\biggl[ \, \sfC\Bigl({\ol\jmath(t,x)}\Big|{\sigma(t,x;\rho,\calF_0)}\Bigr)
	+ \frac{2m_\Upsilon}{|\Omega|}\ee^{-F_0^\rho(t,x,c)}\bra*{\sqrt{u_0^{\calF,a}(t,x)} - 
		\sqrt{u_0^{\calF,b}(t,x)} }^2\,\biggr]\dxx tx
\end{multline*}
with
\begin{align*}
	\sigma(t,x;\rho,\calF_0) &:= \frac{m_\Upsilon}{|\Omega|}\ee^{-F_0^\rho(t,x,c)}\sqrt{u_0^{\calF,a}u_0^{\calF,b}}(t,x)\\
	&= \frac{m_\Upsilon}{|\Omega|}\exp\Bigl(-F_0^\rho(t,x,c) + \tfrac12 (F_0^\rho(t,x,a)+F_0^\rho(t,x,b)\Bigr)
	\sqrt{u_0^{a}u_0^{b}}(t,x).
\end{align*}
This concludes the proof of Theorem~\ref{t:lower-bound}.
		
\subsection{The chain rule: Proof of Lemma~\ref{l:CRLB-Kramers-limit}}
\label{app:CRLB}
\newcommand{\ld}{\lambda,\delta}

\begin{lemma}[Chain rule bound]
\label{l:Kramers-CR}
Let $\calF\in \sfF$ (see~\eqref{eq:Kramers:tilts}) and let $(\rho,j)\in \wt\CE(0,T)$ (see Definition~\ref{defn:Kramers:CE0}).  If $\calE^\calF_0(\rho(0))<\infty$ then 
\begin{equation}
\label{ineq:Kramers-CR}
\abs*{\calE^\calF_0(\rho(T)) - \calE^\calF_0(\rho(0))} \leq \calD_0^T(\rho,j; \calF).
\end{equation}
\end{lemma}

\begin{proof}
Fix a pair $(\rho,j)\in \wt\CE(0,T)$ as in the Lemma; by assumption we have $\calE^\calF_0(\rho(0))<\infty$, and without loss of generality we can assume that $\calD_0^T(\rho,j;\calF) < \infty$. Note that this implies that $j^x, \, u, \, |j^x|^2/u,\, \nabla_x\sqrt{u}\in L^1(0,T;L^1(\nodes))$ and $\ol\jmath\in L^1(0,T;L^1(\Omega))$.
	
\medskip
	
First note that it is sufficient to prove~\eqref{ineq:Kramers-CR} under the assumption that $u = \dx\rho/\dx\pi_0$ is bounded away from zero. This follows from noting that a regularization of the form 
\[
\rho^\lambda := \lambda \rho + (1-\lambda) \pi_0, \qquad j^\lambda = \lambda j,
\qquad \lambda\in (0,1)
\]
achieves this lower bound and preserves the continuity equation; in the required inequality each of the terms passes to the limit in $\lambda$ by either the Monotone Convergence Theorem or the Dominated Convergence Theorem.

\medskip

Proceeding under the assumption that $u$ is bounded away from zero, we next claim that it is sufficient to show that 
\begin{equation}
\label{eq:Kramers-CR-no-F}
-\calE_0(\rho)\Big|_0^T = \int_0^T \set*{ 2\int_\nodes \frac{j^x}{\sqrt u} \nabla_x \sqrt u
+ \int_\Omega B(u_a,u_b,\ol \jmath)}.
\end{equation}
Here $B(u_a,u_b,\ol\jmath)$ is the extension of $\ol\jmath \log (u_b/u_a)$ given in~\eqref{eqdef:B}.
This can be seen as follows:
\begin{enumerate}
\item If one proves~\eqref{eq:Kramers-CR-no-F}, then from the smoothness of $\calF$ one directly obtains the corresponding identity with $\calF$, 
\begin{equation}
\label{eq:Kramers-CR-with-F}
-\calE^\calF_0(\rho)\Big|_0^T = \int_0^T \set*{ 2\int_\nodes \frac{j^x}{\sqrt u^\calF} \nabla_x \sqrt {u^\calF}
+ \int_\Omega B\bra*{u_a^\calF,u_b^\calF,\ol \jmath}},
\end{equation}
where $u^\calF := u \ee^{F^\rho}$.
\item The required inequality then follows by applying the estimates described in Section~\ref{ss:Kramers-setting}.
\end{enumerate}

To prove~\eqref{eq:Kramers-CR-no-F} we use the fact that $\Omega$ is a $d$-dimensional rectangle.  We  interpret $\Omega$ as a single rectangular patch in a rectangular grid, and extend $\rho$ and $j$ outside of $\Omega$ by reflection:
\begin{itemize}
\item $\rho$ and $\ol\jmath$ are reflected symmetrically across $\partial\Omega$;
\item $j^x$ is reflected symmetrically across $\partial\Omega$,  with an additional sign flip applied to the component normal to $\partial\Omega$.
\end{itemize}
We then convolve $\rho$ and $j$ with a strictly positive regularizing sequence of the form $\gamma_\delta(x) := \delta^{-d} \gamma(x/\delta)$:
\[
\rho^{\delta} := \gamma_\delta {*} \rho,
\quad
j^{\delta}  := (j^{x,\delta},\ol\jmath^\delta) := (\gamma_\delta {*} j^x, \gamma_\delta{*}\ol\jmath), 
\qquad \text{for } \delta>0.
\]
Note that $\pi_0$ (after extension) is invariant under convolution with $\gamma_\delta$. 

By the structure of the reflection, the pair $(\rho^\delta,j^\delta)$ solves the continuity equation on the periodic cell $\Omega$, i.e. for any  $\varphi\in C^1([0,T)\times\Omega\times \{a,b\})$ we have
\begin{multline*}
\int_0^T\!\!\int\limits_{\Omega} \sum_{y=a,b}\bra[\big]{\partial_t \varphi(t,x,y)\, \rho^\delta(\dxx tx,y)
  + \nabla_x \varphi(t,x,y)\,j^{x,\delta}(\dxx tx,y) } \dd t \\
+ \int_0^T\!\!\int\limits _{\Omega} \bigl(\varphi(t,x,b)-\varphi(t,x,a)\bigr)\overline\jmath^\delta(\dxx tx)
  + \int\limits_{\Omega} \sum_{y=a,b} \varphi(0,x,y) \rho^\circ(\dx x,y) = 0.
\end{multline*}
The pair $(\rho^\delta,j^\delta)$ has sufficient  regularity for the identity~\eqref{eq:Kramers-CR-no-F} to hold. We next pass to the limit in each of the terms.

Using $j,\,u,\,|j^x|^2/u\in L^1$ and Jensen's inequality we note that
\[
\frac{\abs{j^{x,\delta}}^2}{u^\delta} \leq \gamma_\delta * \frac{\abs{j^{x}}^2}{u},
\]
so that by Lemma~\ref{l:extended-DCT} below the sequence $\abs{j^{x,\delta}}^2/u^\delta$ converges in $L^1(\Omega)$, and therefore ${j^{x,\delta}}/\sqrt{u^\delta}$ converges in $L^2$. By the same argument $\nabla_x \sqrt{u^\delta} = \nabla_x u^\delta /2\sqrt{u^\delta}$ also converges in~$L^2$. It follows that the first integral on the right-hand side in~\eqref{eq:Kramers-CR-with-F} converges. 

For the second integral, note that 
\[
|B(u_a,u_b,\ol\jmath) |\stackrel{\eqref{ineq:props-N:BG}} \leq 
\CCs\bra*{\ol\jmath;u_a,u_b;\frac{m_\Upsilon}{|\Omega|}}        
\stackrel{\eqref{eqdef:N:explicit:rigrous}} = \sfC(\ol\jmath|\sigma) + \frac{2m_\Upsilon}{|\Omega|}\bra*{\sqrt{u_b}-\sqrt{u_a}}^2, 
\qquad \sigma = \frac{m_\Upsilon}{|\Omega|}\sqrt{u_b u_a}.
\]
Similarly, 
\begin{align*}
|B(u_a^\delta,u_b^\delta,\ol\jmath^\delta) |
&\leq \sfC(\ol\jmath^\delta|\sigma^\delta) + \frac{2m_\Upsilon}{|\Omega|}\bra*{\sqrt{u_b^\delta}-\sqrt{u_a^\delta}}^2, 
\qquad \sigma^\delta = \frac{m_\Upsilon}{|\Omega|}\sqrt{u_b^\delta u_a^\delta}\\
&\leq \gamma_\delta * \bra*{\sfC(\ol\jmath|\sigma) + \frac{2m_\Upsilon}{|\Omega|}\bra*{\sqrt{u_b}-\sqrt{u_a}}^2}.
\end{align*}
The function $B$ is continuous except at  $(u_a,u_b,\ol\jmath)$ with $u_au_b=0$. Since $\sfC(\ol\jmath|\sigma)$ is finite a.e., it follows that the set of points where $\ol\jmath\not=0$ and  $u_a =0$, $u_b=0$, or both, is a null set. From Lemma~\ref{l:extended-DCT} it again follows that  $B(u_a^\delta,u_b^\delta,\ol\jmath^\delta)$ converges in $L^1$.
Therefore the right-hand side in~\eqref{eq:Kramers-CR-no-F} passes to the limit, and the limit is finite. 

\medskip
Turning to the left-hand side, the finiteness of $\calE_0(\rho(0))$ implies (again by Lemma~\ref{l:extended-DCT}) that we have $\calE_0(\rho^\delta(0))\to \calE_0(\rho(0))$. This implies that $\calE_0(\rho^\delta(T))$ converges as $\delta\to0$, and in particular is bounded by some constant $C$. 

It then follows that $\calE_0(\rho(T))\leq C$, by the following argument: for any $M>1$ and $a\geq0$ we have $\eta(a\wedge M) \leq \eta(a)$, so that, writing $\rho = u\pi_0$ and $\rho^\delta =u^\delta\pi_0$, 
\[
\int_{\Omega}\eta\bra*{u^\delta(T)\wedge M}\dx \pi_0
\leq 
\int_{\Omega}\eta\bra*{u^\delta(T)}\dx \pi_0
=\calE_0(\rho^\delta(T))
 \leq C.
\]
Passing to the limit in $\delta$ using the Dominated Convergence theorem  we find for all $M>1$
\[
\int_{\Omega}\eta\bra*{u(T)\wedge M}\dx \pi_0 \leq C.
\]
From the Monotone Convergence theorem we then obtain
\[
\calE_0(\rho(T)) = \int_{\Omega}\eta\bra*{u(T)}\dx \pi_0 \leq C,
\]
which establishes the finiteness of $\calE_0(\rho(T))$. It then follows by the same argument that $\calE_0(\rho^\delta(T))\to \calE_0(\rho(T))$, which concludes the proof. 
\end{proof}

The following lemma is probably well known, but we could not find a good reference.
\begin{lemma}
\label{l:extended-DCT}
Let $\Omega$ be a $d$-dimensional rectangle and let $v\in L^1(\Omega;\R^m)$. Extend $v$ to $\R^d\setminus \Omega$ by reflection (either symmetric or antisymmetric), as above. 

Assume that $f:\R^m\to\R$ is continuous except on a set $N\subset \R^m$ such that $\{x\in\Omega: v(x)\in N\}$ is a Lebesgue null set. 

Assume that there exists $g\in L^1(\Omega)$ such that 
\[
|f(\gamma_\delta{*}v)| \leq \gamma_\delta{*} g \qquad\text{a.e. in }\Omega.
\]
Then $f(\gamma_\delta{*}v)\to f(v)$ in $L^1(\Omega)$ as $\delta\to0$.
\end{lemma}

\begin{proof}
Note that $\gamma_\delta{*}v \to v$ in $L^1(\Omega;\R^m)$ and $\gamma_\delta{*} g\to g$ in $L^1(\Omega)$. By the Inverse Dominated Convergence Theorem (e.g.~\cite[Th.~4.9]{Brezis11}) there exists a subsequence $\delta_n\to 0$ such that $\gamma_{\delta_n}{*}v\to v$ almost everywhere and $0\leq \gamma_{\delta_n}{*}g\leq h$ for some $h\in L^1$. 

From the a.e.\ convergence $\gamma_{\delta_n}{*}v \to v$ and the continuity properties of $f$ it follows that $f(\gamma_\delta{*}v)\to f(v)$ a.e.\ on $\Omega$.
Since $|f(\gamma_\delta{*}v)|$ is dominated by $h$, it follows that $f(\gamma_{\delta_n}{*}v)\to f(v)$ in $L^1$, and since the limit is unique the convergence holds as $\delta\to0$.
\end{proof}

\subsection{Lower semicontinuity of integrals of one-homogeneous convex functions}
\begin{lemma}[Reshetnyak lower semicontinuity]
	\label{l:Reshetnyak}
	Let $\calY$ be a separable Hilbert space, and let $\Omega$ be a compact subset of $\R^d$. Let $f_n, f:\Omega\times \calY\to [0,\infty]$ be  convex and $1$-homogeneous in the second variable. Assume that they satisfy 
	\begin{equation*}
		\liminf_{n\to\infty} f_n(x_n,\xi_n)\geq f(x,\xi) \qquad\text{for any $x_n\to x$ and any weakly converging sequence }\xi_n\weakto \xi.
	\end{equation*}
	Let $\mu_n$ be a sequence of\/ $\calY$-valued Borel measures on $\Omega$ that converges weakly to a measure~$\mu$, i.e.
	\begin{equation}
	\label{l:Reshetnyak:ass-lsc}
	\text{for all }\varphi\in C(\Omega), \qquad
	\int_\Omega \varphi(x) \mu_n(\dx x) \longrightharpoonup \int_\Omega \varphi(x) \mu(\dx x) 
	\qquad\text{in }\calY.
	\end{equation}
	Then
	\[
	\liminf_{n\to\infty} \int_\Omega f_n\bra*{x, \frac{\dx \mu_n}{\dx {|\mu_n|}}(x)}  |\mu_n|(\dx x)
	\geq 
	\int_\Omega f\bra*{x, \frac{\dx \mu}{\dx {|\mu|}}(x)}  |\mu|(\dx x).
	\]
\end{lemma}
Since Hilbert spaces possess the Radon-Nikodym property~\cite[Th.~1.3.21]{HytonenVanNeervenVeraarWeis16-I}, the derivatives $\dx \mu_n/\dx \abs*{\mu_n}$ and $\dx \mu/\abs*{\mu}$ are well-defined.
\begin{proof}
	We  extend the proof of~\cite[Th.~2.38]{AmbrosioFuscoPallara00} to Hilbert-valued functions.	
	Set $B := \{\xi\in \calY: \|\xi\|\leq 1\}$. 
	Since $\calY$ satisfies the 	Radon-Nikodym property, there exist Borel measurable $B$-valued functions $g_n$ and $g$ such that $\mu_n =g_n |\mu_n|$ and $\mu = g|\mu|$. Set 
	\[
	\nu_n := |\mu_n|\otimes \delta_{g_n(x)}, \qquad\text{i.e.}\qquad
	\int_{\Omega\times \calY}\phi(x,y)\nu_n(\dxx xy) = \int_\Omega \phi(x,g_n(x))|\mu_n|(\dx x).
	\]
	Since $\Omega\times B$ is a weakly compact subset of $\R^d\times \calY$ and $|\nu_n|(\Omega\times B) = |\mu_n|(\Omega)$  is bounded, there exists a  subsequence (for which we do not change notation) $\nu_n\weakto \nu$ in the narrow sense (against~$C_{\mathrm b}$). 
	
	If $\pi:\Omega\times \calY\mapsto \Omega$ is the projection onto the first variable, then it follows that $|\mu_n| = \pi_\# \nu_n$ converges narrowly to a limit $\lambda = \pi_\#\nu$. 
	By the lower semicontinuity of the variation~\cite[Prop.~1.62(b)]{AmbrosioFuscoPallara00} we have $|\mu|\leq \lambda$.
	
	We now disintegrate the measure $\nu$ \cite[Th.~5.3.1]{AmbrosioGigliSavare08}: there exists a Borel map $x\mapsto \nu_x\in \ProbMeas(B)$ such that $\nu = \lambda \otimes \nu_x$. By the same argument as in~\cite[Th.~2.38]{AmbrosioFuscoPallara00} we find that 
	\[
	\int_{B} y \nu_x (\dx y) = g(x) \frac{\dx {|\mu|}}{\dx \lambda}(x)
	\qquad \text{for $\lambda$-a.e. $x$.}
	\]
	We then calculate
	\begin{align*}
		\liminf_{n\to\infty} &\int_{\Omega} f_n\bra*{x,g_n(x)}\, |\mu_n|(\dx x)
		= \liminf_{n\to\infty} \int_{\Omega\times B} f_n\bra*{x,y} \,\nu_n(\dxx xy)\\
		&\leftstackrel{(*)}\geq \int_{\Omega\times B} f\bra*{x,y} \,\nu(\dxx xy)
		= \int_\Omega \bra*{\int_{B}f(x,y)\, \nu_x(\dx y)}\lambda(\dx x)\\
		&\leftstackrel{\text{Jensen}}\geq \int_{\Omega}f\bra*{x,\int_{B}y\, \nu_x(\dx y)}\lambda (\dx x)
		= \int_\Omega f(x,g(x))\, |\mu(\dx x)|.
	\end{align*}
	The inequality marked $(*)$ can be found e.g.\ in~\cite[Lemma~3.2]{FathiSimon16}. Since the limiting inequality is independent of the subsequence, it holds for the whole sequence. 
\end{proof}

\footnotesize
\bibliography{coshref}

\newcommand{\etalchar}[1]{$^{#1}$}
\begin{thebibliography}{HVNVW16}
\expandafter\ifx\csname url\endcsname\relax
  \def\url#1{\texttt{#1}}\fi
\expandafter\ifx\csname doi\endcsname\relax
  \def\doi#1{\burlalt{doi:#1}{http://dx.doi.org/#1}}\fi
\expandafter\ifx\csname urlprefix\endcsname\relax\def\urlprefix{URL }\fi
\expandafter\ifx\csname href\endcsname\relax
  \def\href#1#2{#2}\fi
\expandafter\ifx\csname burlalt\endcsname\relax
  \def\burlalt#1#2{\href{#2}{#1}}\fi

\bibitem[ADPZ11]{AdamsDirrPeletierZimmer2011}
S.~Adams, N.~Dirr, M.~A. Peletier, and J.~Zimmer.
\newblock From a large-deviations principle to the {W}asserstein gradient flow:
  A new micro-macro passage.
\newblock {\em Comm. Math. Phys.}, 307(3):791--815, 2011.
\newblock \doi{10.1007/s00220-011-1328-4}.

\bibitem[ADPZ13]{AdamsDirrPeletierZimmer2013}
S.~Adams, N.~Dirr, M.~Peletier, and J.~Zimmer.
\newblock Large deviations and gradient flows.
\newblock {\em Philos. Trans. R. Soc. Lond. Ser. A Math. Phys. Eng. Sci.},
  371(2005):20120341, 17, 2013.
\newblock \doi{10.1098/rsta.2012.0341}.

\bibitem[AFP00]{AmbrosioFuscoPallara00}
L.~Ambrosio, N.~Fusco, and D.~Pallara.
\newblock {\em {Functions of Bounded Variation and Free Discontinuity
  Problems}}.
\newblock Oxford Mathematical Monographs. Oxford University Press, first
  edition, 2000.

\bibitem[AGS08]{AmbrosioGigliSavare08}
L.~Ambrosio, N.~Gigli, and G.~Savar\'e.
\newblock {\em {Gradient Flows in Metric Spaces and in the Space of Probability
  Measures}}.
\newblock Lectures in Mathematics ETH Z{\"u}rich. Birkh{\"a}user, 2008.

\bibitem[AL83]{AltLuckhaus1983}
H.~W. Alt and S.~Luckhaus.
\newblock Quasilinear elliptic-parabolic differential equations.
\newblock {\em Math. Z.}, 183(3):311--341, 1983.
\newblock \doi{10.1007/BF01176474}.

\bibitem[AMP{\etalchar{+}}12]{ArnrichMielkePeletierSavareVeneroni12}
S.~Arnrich, A.~Mielke, M.~A. Peletier, G.~Savar\'e, and M.~Veneroni.
\newblock {Passing to the limit in a Wasserstein gradient flow: From diffusion
  to reaction}.
\newblock {\em Calculus of Variations and Partial Differential Equations},
  44:419--454, 2012.

\bibitem[ASZ09]{AmbrosioSavareZambotti09}
L.~Ambrosio, G.~Savar{\'e}, and L.~Zambotti.
\newblock {Existence and stability for Fokker--Planck equations with
  log-concave reference measure}.
\newblock {\em Probability theory and related fields}, 145(3):517--564, 2009.

\bibitem[AWTSK18]{ArroyoWalaniTorres-SanchezKaurin18}
M.~Arroyo, N.~Walani, A.~Torres-S{\'a}nchez, and D.~Kaurin.
\newblock Onsager's variational principle in soft matter: Introduction and
  application to the dynamics of adsorption of proteins onto fluid membranes.
\newblock In {\em The Role of Mechanics in the Study of Lipid Bilayers}, pages
  287--332. Springer, 2018.

\bibitem[BBB20]{BasileBenedettoBertini2020}
G.~Basile, D.~Benedetto, and L.~Bertini.
\newblock A gradient flow approach to linear {B}oltzmann equations.
\newblock {\em Ann. Sc. Norm. Super. Pisa Cl. Sci. (5)}, 21:943--975, 2020.
\newblock \doi{10.2422/2036-2145.201811\_012}.

\bibitem[BBBO21]{BasileBenedettoBertiniOrrieri2021}
G.~Basile, D.~Benedetto, L.~Bertini, and C.~Orrieri.
\newblock Large deviations for {K}ac-like walks.
\newblock {\em Preprint arXiv:2101.05481}, 2021.

\bibitem[BBRW18]{BrunaBurgerRanetbauerWolfram2018TR}
M.~Bruna, M.~Burger, H.~Ranetbauer, and M.-T. Wolfram.
\newblock Asymptotic gradient flow structures of a nonlinear {F}okker-{P}lanck
  equation.
\newblock {\em Preprint arXiv:1708.07304}, 2018.

\bibitem[BCP86]{BCP86}
J.~M. Ball, J.~Carr, and O.~Penrose.
\newblock {The Becker-D{\"{o}}ring cluster equations: Basic properties and
  asymptotic behaviour of solutions}.
\newblock {\em Commun. Math. Phys.}, 104(4):657--692, 1986.
\newblock \doi{10.1007/BF01211070}.

\bibitem[BD35]{BD1935}
R.~Becker and W.~D\"{o}ring.
\newblock {Kinetische Behandlung der Keimbildung in \"{u}bers\"{a}ttigten
  D\"{a}mpfen.}
\newblock {\em Ann. der Physik}, 24:719--752, 1935.

\bibitem[BDFR15a]{BDFR15a}
A.~Budhiraja, P.~Dupuis, M.~Fischer, and K.~Ramanan.
\newblock Limits of relative entropies associated with weakly interacting
  particle systems.
\newblock {\em Electron. J. Probab.}, 20:no. 80, 22, 2015.
\newblock \doi{10.1214/EJP.v20-4003}.

\bibitem[BDFR15b]{BDFR15b}
A.~Budhiraja, P.~Dupuis, M.~Fischer, and K.~Ramanan.
\newblock Local stability of {K}olmogorov forward equations for finite state
  nonlinear {M}arkov processes.
\newblock {\em Electron. J. Probab.}, 20:no. 81, 30, 2015.
\newblock \doi{10.1214/EJP.v20-4004}.

\bibitem[BDH16]{BovierDenHollander16}
A.~Bovier and F.~Den~Hollander.
\newblock {\em Metastability: A Potential-Theoretic Approach}, volume 351.
\newblock Springer, 2016.

\bibitem[BEGK02]{BovierEckhoffGayrardKlein2002}
A.~Bovier, M.~Eckhoff, V.~Gayrard, and M.~Klein.
\newblock Metastability and low lying spectra in reversible {M}arkov chains.
\newblock {\em Comm. Math. Phys.}, 228(2):219--255, 2002.
\newblock \doi{10.1007/s002200200609}.

\bibitem[BEGK04]{BovierEckhoffGayrardKlein04}
A.~Bovier, M.~Eckhoff, V.~Gayrard, and M.~Klein.
\newblock {Metastability in reversible diffusion processes~I: Sharp asymptotics
  for capacities and exit times}.
\newblock {\em Journal of the European Mathematical Society}, 6(4):399--424,
  2004.

\bibitem[Ber89]{Berge1989}
C.~Berge.
\newblock {\em Hypergraphs: Combinatorics of finite sets}, volume~45 of {\em
  North-Holland Mathematical Library}.
\newblock North-Holland Publishing Co., Amsterdam, 1989.

\bibitem[Ber13]{Berglund13}
N.~Berglund.
\newblock Kramers' law: Validity, derivations and generalisations.
\newblock {\em Markov Processes and Related Fields}, 19:459--490, 2013.

\bibitem[BGL14]{BakryGentilLedoux14}
D.~Bakry, I.~Gentil, and M.~Ledoux.
\newblock {\em Analysis and Geometry of {M}arkov Diffusion Operators}.
\newblock Springer, 2014.

\bibitem[BGSRS20]{BodineauGallagherSaint-RaymondSimonella20TR}
T.~Bodineau, I.~Gallagher, L.~Saint-Raymond, and S.~Simonella.
\newblock {Statistical dynamics of a hard sphere gas: Fluctuating {B}oltzmann
  equation and large deviations}.
\newblock {\em Preprint arXiv:2008.10403}, 2020.

\bibitem[BHP21]{BurgerHumpertPietschmann2021}
M.~Burger, I.~Humpert, and J.-F. Pietschmann.
\newblock Dynamic optimal transport on networks.
\newblock {\em Preprint arXiv:2101.03415}, 2021.

\bibitem[BNK03]{NK2003}
E.~Ben-Naim and P.~L. Krapivsky.
\newblock {Exchange-driven growth}.
\newblock {\em Phys. Rev. E}, 68(3):031104, 2003.
\newblock \doi{10.1103/PhysRevE.68.031104}.

\bibitem[{Bol}72]{Boltzmann1872}
L.~{Boltzmann}.
\newblock {Weitere Studien \"uber das W\"armegleichgewicht unter
  Gasmolek\"ulen.}
\newblock {\em {Wien. Ber.}}, 66:275--370, 1872.

\bibitem[Bol64]{Boltzmann1964}
L.~Boltzmann.
\newblock {\em Lectures on Gas Theory}.
\newblock University of California Press, 12 1964.
\newblock \doi{10.1525/9780520327474}.

\bibitem[Bou20]{Bouchet2020}
F.~Bouchet.
\newblock Is the {B}oltzmann equation reversible? {A} large deviation
  perspective on the irreversibility paradox.
\newblock {\em Journal of Statistical Physics}, 181(2):515--550, 10 2020.
\newblock \doi{10.1007/s10955-020-02588-y}.

\bibitem[BP16]{BonaschiPeletier16}
G.~A. Bonaschi and M.~A. Peletier.
\newblock Quadratic and rate-independent limits for a large-deviations
  functional.
\newblock {\em Continuum Mechanics and Thermodynamics}, 28:1191--1219, 2016.

\bibitem[BPW13]{BurgerPietschmannWolfram2013}
M.~Burger, J.-F. Pietschmann, and M.-T. Wolfram.
\newblock Identification of nonlinearities in transport-diffusion models of
  crowded motion.
\newblock {\em Inverse Probl. Imaging}, 7(4):1157--1182, 2013.
\newblock \doi{10.3934/ipi.2013.7.1157}.

\bibitem[Bra14]{Braides14}
A.~Braides.
\newblock {\em {Local Minimization, Variational Evolution and
  {$\Gamma$}-Convergence}}, volume 2094 of {\em Lecture Notes in Mathematics}.
\newblock Springer, 2014.

\bibitem[Bre11]{Brezis11}
H.~Brezis.
\newblock {\em Functional Analysis, Sobolev Spaces and Partial Differential
  Equations}.
\newblock Springer, New York, 2011.

\bibitem[Cas45]{Casimir45}
H.~B.~G. Casimir.
\newblock On {O}nsager's principle of microscopic reversibility.
\newblock {\em Reviews of Modern Physics}, 17(2-3):343, 1945.

\bibitem[CCHFG21]{Cances_etal2019}
C.~Canc\`es, C.~Chainais-Hillairet, J.~Fuhrmann, and B.~Gaudeul.
\newblock A numerical-analysis-focused comparison of several finite volume
  schemes for a unipolar degenerate drift-diffusion model.
\newblock {\em IMA J. Numer. Anal.}, 41(1):271--314, 2021.
\newblock \doi{10.1093/imanum/draa002}.

\bibitem[CD99]{CioranescuDonato99}
D.~Cioranescu and P.~Donato.
\newblock {\em {An Introduction to Homogenization}}, volume~17 of {\em Oxford
  lecture series in mathematics and its applications}.
\newblock Oxford Science Publications, 1999.

\bibitem[Cer88]{Cercignani88}
C.~Cercignani.
\newblock {\em The {B}oltzmann Equation and its Applications}.
\newblock Springer, 1988.

\bibitem[CGT20]{CancesGallouetTodeschi2019}
C.~Canc\`es, T.~O. Gallou\"{e}t, and G.~Todeschi.
\newblock A variational finite volume scheme for {W}asserstein gradient flows.
\newblock {\em Numer. Math.}, 146(3):437--480, 2020.
\newblock \doi{10.1007/s00211-020-01153-9}.

\bibitem[CHLZ12]{ChowHuangLiZhou12}
S.-N. Chow, W.~Huang, Y.~Li, and H.~Zhou.
\newblock Fokker--{P}lanck equations for a free energy functional or {M}arkov
  process on a graph.
\newblock {\em Archive for Rational Mechanics and Analysis}, 203(3):969--1008,
  2012.

\bibitem[CIM98]{CurtisIngermanMorrow1998}
E.~B. Curtis, D.~Ingerman, and J.~A. Morrow.
\newblock Circular planar graphs and resistor networks.
\newblock {\em Linear Algebra Appl.}, 283(1-3):115--150, 1998.
\newblock \doi{10.1016/S0024-3795(98)10087-3}.

\bibitem[CIP94]{CercignaniIllnerPulvirenti1994}
C.~Cercignani, R.~Illner, and M.~Pulvirenti.
\newblock {\em The Mathematical Theory of Dilute Gases}.
\newblock Springer New York, 1994.
\newblock \doi{10.1007/978-1-4419-8524-8}.

\bibitem[CLSS10]{CarrilloLisiniSavareSlepcev10}
J.~Carrillo, S.~Lisini, G.~Savar{\'e}, and D.~Slepcev.
\newblock {Nonlinear mobility continuity equations and generalized displacement
  convexity}.
\newblock {\em Journal of Functional Analysis}, 258(4):1273--1309, 2010.

\bibitem[Cok01]{Coker2001}
A.~K. Coker.
\newblock {\em Modeling of {Chemical} {Kinetics} and {Reactor} {Design}},
  volume~13.
\newblock Gulf {P}ublishing {C}ompany, 2001.

\bibitem[Com18]{Combettes18}
P.~Combettes.
\newblock Perspective functions: Properties, constructions, and examples.
\newblock {\em Set-Valued Var. Anal}, 26:247--264, 2018.

\bibitem[Con90]{Connors1990}
K.~A. Connors.
\newblock {\em Chemical kinetics: {The} study of reaction rates in solution}.
\newblock VCH, 1990.
\newblock \urlprefix\url{https://books.google.de/books?id=nHux3YED1HsC}.

\bibitem[CT15]{ChetriteTouchette15}
R.~Chetrite and H.~Touchette.
\newblock Nonequilibrium {M}arkov processes conditioned on large deviations.
\newblock In {\em Annales Henri Poincar{\'e}}, volume~16, pages 2005--2057.
  Springer, 2015.
\newblock \doi{10.1007/s00023-014-0375-8}.

\bibitem[CV90]{ColliVisintin90}
P.~Colli and A.~Visintin.
\newblock {On a class of doubly nonlinear evolution equations}.
\newblock {\em Communications in Partial Differential Equations},
  15(5):737--756, 1990.

\bibitem[DB13]{DorflerBullo2013}
F.~Dorfler and F.~Bullo.
\newblock Kron reduction of graphs with applications to electrical networks.
\newblock {\em {IEEE} Transactions on Circuits and Systems I: Regular Papers},
  60(1):150--163, January 2013.
\newblock \doi{10.1109/tcsi.2012.2215780}.

\bibitem[DFM19]{DondlFrenzelMielke19}
P.~Dondl, T.~Frenzel, and A.~Mielke.
\newblock {A gradient system with a wiggly energy and relaxed EDP-convergence}.
\newblock {\em ESAIM: Control, Optimisation and Calculus of Variations}, 25:68,
  2019.

\bibitem[DGMT80]{DeGiorgiMarinoTosques80}
E.~De~Giorgi, A.~Marino, and M.~Tosques.
\newblock Problems of evolution in metric spaces and maximal decreasing curve.
\newblock {\em Atti Accad. Naz. Lincei Rend. Cl. Sci. Fis. Mat. Natur. (8)},
  68(3):180--187, 1980.

\bibitem[dH00]{DenHollander00}
F.~den Hollander.
\newblock {\em Large Deviations}, volume~14 of {\em Fields Institute
  Monographs}.
\newblock American Mathematical Society, Providence, RI, 2000.
\newblock \doi{10.1007/s00440-009-0235-5}.

\bibitem[DL15]{DisserLiero2015}
K.~Disser and M.~Liero.
\newblock On gradient structures for {M}arkov chains and the passage to
  {W}asserstein gradient flows.
\newblock {\em Netw. Heterog. Media}, 10(2):233--253, 2015.
\newblock \doi{10.3934/nhm.2015.10.233}.

\bibitem[DM93]{DalMaso93}
G.~Dal~Maso.
\newblock {\em An Introduction to {$\Gamma$}-Convergence}, volume~8 of {\em
  Progress in Nonlinear Differential Equations and Their Applications}.
\newblock Birkh{\"a}user, Boston, 1993.

\bibitem[DMDM06]{Dal-MasoDeSimoneMora06}
G.~Dal~Maso, A.~DeSimone, and M.~G. Mora.
\newblock Quasistatic evolution problems for linearly elastic--perfectly
  plastic materials.
\newblock {\em Archive for rational mechanics and analysis}, 180(2):237--291,
  2006.

\bibitem[DNS09]{DolbeaultNazaretSavare2008}
J.~Dolbeault, B.~Nazaret, and G.~Savar\'{e}.
\newblock A new class of transport distances between measures.
\newblock {\em Calc. Var. Partial Differential Equations}, 34(2):193--231,
  2009.
\newblock \doi{10.1007/s00526-008-0182-5}.

\bibitem[Doi11]{Doi11}
M.~Doi.
\newblock Onsager's variational principle in soft matter.
\newblock {\em Journal of Physics: Condensed Matter}, 23(28):284118, 2011.

\bibitem[DS84]{DoyleSnell1984}
P.~G. Doyle and J.~L. Snell.
\newblock {\em Random Walks and Electric Networks}, volume~22 of {\em Carus
  Mathematical Monographs}.
\newblock Mathematical Association of America, Washington, DC, 1984.

\bibitem[DS10]{DaneriSavare10TR}
S.~Daneri and G.~Savar{\'e}.
\newblock Lecture notes on gradient flows and optimal transport.
\newblock {\em Preprint arXiv:1009.3737}, 2010.

\bibitem[DZ98]{DemboZeitouni98}
A.~Dembo and O.~Zeitouni.
\newblock {\em {Large Deviations Techniques and Applications}}.
\newblock Springer Verlag, 1998.

\bibitem[EFG06]{EymardFuhrmannGaertner06}
R.~Eymard, J.~Fuhrmann, and K.~G\"{a}rtner.
\newblock A finite volume scheme for nonlinear parabolic equations derived from
  one-dimensional local {D}irichlet problems.
\newblock {\em Numer. Math.}, 102(3):463--495, 2006.
\newblock \doi{10.1007/s00211-005-0659-5}.

\bibitem[EFLS16]{ErbarFathiLaschosSchlichting16}
M.~Erbar, M.~Fathi, V.~Laschos, and A.~Schlichting.
\newblock Gradient flow structure for {M}c{K}ean-{V}lasov equations on discrete
  spaces.
\newblock {\em Discrete Contin. Dyn. Syst.}, 36(12):6799--6833, 2016.
\newblock \doi{10.3934/dcds.2016096}.

\bibitem[EFMM21]{ErbarForkertMaasMugnolo2021}
M.~Erbar, D.~Forkert, J.~Maas, and D.~Mugnolo.
\newblock Gradient flow formulation of diffusion equations in the {W}asserstein
  space over a metric graph.
\newblock {\em Preprint arXiv:2105.05677}, 2021.

\bibitem[EFS20]{ErbarFathiSchlichting2020}
M.~Erbar, M.~Fathi, and A.~Schlichting.
\newblock Entropic curvature and convergence to equilibrium for mean-field
  dynamics on discrete spaces.
\newblock {\em ALEA Lat. Am. J. Probab. Math. Stat.}, 17(1):445--471, 2020.
\newblock \doi{10.30757/alea.v17-18}.

\bibitem[EGSS21]{EspositoGvalaniSchlichtingSchmidtchen2021TR}
A.~Esposito, R.~S. Gvalani, A.~Schlichting, and M.~Schmidtchen.
\newblock On a novel gradient flow structure for the aggregation equation.
\newblock {\em Preprint arXiv:2112.08317}, 2021.

\bibitem[EK09]{EthierKurtz09}
S.~N. Ethier and T.~G. Kurtz.
\newblock {\em Markov processes: Characterization and convergence}, volume 282.
\newblock John Wiley \& Sons, 2009.

\bibitem[EM14]{ErbarMaas2014}
M.~Erbar and J.~Maas.
\newblock Gradient flow structures for discrete porous medium equations.
\newblock {\em Discrete Contin. Dyn. Syst.}, 34(4):1355--1374, 2014.
\newblock \doi{10.3934/dcds.2014.34.1355}.

\bibitem[EMR15]{ErbarMaasRenger2015}
M.~Erbar, J.~Maas, and D.~R.~M. Renger.
\newblock From large deviations to {W}asserstein gradient flows in multiple
  dimensions.
\newblock {\em Electron. Commun. Probab.}, 20:no. 89, 12, 2015.
\newblock \doi{10.1214/ECP.v20-4315}.

\bibitem[Epi66]{Epifanov1966}
G.~V. Epifanov.
\newblock Reduction of a plane graph to an edge by star-triangle
  transformations.
\newblock {\em Dokl. Akad. Nauk SSSR}, 166:19--22, 1966.

\bibitem[EPSS21]{EspositoPatacchiniSchlichtingSlepcev2021}
A.~Esposito, F.~S. Patacchini, A.~Schlichting, and D.~Slep\v{c}ev.
\newblock Nonlocal-interaction equation on graphs: Gradient flow structure and
  continuum limit.
\newblock {\em Arch. Ration. Mech. Anal.}, 240(2):699--760, 2021.
\newblock \doi{10.1007/s00205-021-01631-w}.

\bibitem[Erb14]{Erbar2014}
M.~Erbar.
\newblock Gradient flows of the entropy for jump processes.
\newblock {\em Ann. Inst. Henri Poincar\'{e} Probab. Stat.}, 50(3):920--945,
  2014.
\newblock \doi{10.1214/12-AIHP537}.

\bibitem[Erb16]{Erbar2016TR}
M.~Erbar.
\newblock A gradient flow approach to the {B}oltzmann equation.
\newblock {\em Preprint arXiv:1603.00540}, 2016.

\bibitem[ET16]{EvansTabrizian16}
L.~C. Evans and P.~R. Tabrizian.
\newblock Asymptotics for scaled {K}ramers--{S}moluchowski equations.
\newblock {\em SIAM Journal on Mathematical Analysis}, 48(4):2944--2961, 2016.

\bibitem[Fei72]{Feinberg1972}
M.~Feinberg.
\newblock On chemical kinetics of a certain class.
\newblock {\em Arch. Rational Mech. Anal.}, 46:1--41, 1972.
\newblock \doi{10.1007/BF00251866}.

\bibitem[Fen99]{Feng99}
J.~Feng.
\newblock Martingale problems for large deviations of {M}arkov processes.
\newblock {\em Stochastic processes and their applications}, 81(2):165--216,
  1999.

\bibitem[FK06]{FengKurtz06}
J.~Feng and T.~G. Kurtz.
\newblock {\em {Large Deviations for Stochastic Processes}}, volume 131 of {\em
  Mathematical Surveys and Monographs}.
\newblock American Mathematical Society, 2006.

\bibitem[FL21]{FrenzelLiero21}
T.~Frenzel and M.~Liero.
\newblock Effective diffusion in thin structures via generalized gradient
  systems and {EDP}-convergence.
\newblock {\em Discrete \& Continuous Dynamical Systems-S}, 14(1):395, 2021.

\bibitem[Fle82]{Fleming82}
W.~H. Fleming.
\newblock Logarithmic transformations and stochastic control.
\newblock In {\em Advances in Filtering and Optimal Stochastic Control}, pages
  131--141. Springer, 1982.

\bibitem[FM21]{FrenzelMielke21TA}
T.~Frenzel and A.~Mielke.
\newblock Deriving the kinetic flux relation for nonlinear diffusion through a
  membrane using {EDP}-convergence.
\newblock {\em In preparation}, 2021.

\bibitem[FRD{\etalchar{+}}17]{Farrell_etal2017}
P.~Farrell, N.~Rotundo, D.~H. Doan, M.~Kantner, J.~Fuhrmann, and T.~Koprucki.
\newblock {\em {Handbook of Optoelectronic Device Modeling and Simulation}},
  volume~2, chapter Chapter 50 ``Drift-Diffusion Models'', pages 733--771.
\newblock CRC Press Taylor \& Francis Group, 2017.
\newblock \doi{10.4324/9781315152318-25}.

\bibitem[FS16]{FathiSimon16}
M.~Fathi and M.~Simon.
\newblock The gradient flow approach to hydrodynamic limits for the simple
  exclusion process.
\newblock In {\em From particle systems to partial differential equations.
  {III}}, volume 162 of {\em Springer Proc. Math. Stat.}, pages 167--184.
  Springer, 2016.

\bibitem[FSS15]{FlammStadlerStadler2015}
C.~Flamm, B.~M.~R. Stadler, and P.~F. Stadler.
\newblock Chapter 13 - {Generalized} {Topologies}: {Hypergraphs}, {Chemical}
  {Reactions}, and {Biological} {Evolution}.
\newblock In S.~C. Basak, G.~Restrepo, and J.~L. Villaveces, editors, {\em
  Advances in {Mathematical} {Chemistry} and {Applications}}, pages 300--328.
  Bentham Science Publishers, 2015.
\newblock \doi{10.1016/B978-1-68108-053-6.50013-2}.

\bibitem[FW98]{FreidlinWentzell98}
M.~I. Freidlin and A.~D. Wentzell.
\newblock {\em {Random Perturbations of Dynamical Systems}}, volume 260 of {\em
  Grundlehren der Mathematischen Wissenschaften}.
\newblock Springer, 1998.

\bibitem[Gal96]{Gallavotti96}
G.~Gallavotti.
\newblock Extension of {O}nsager's reciprocity to large fields and the chaotic
  hypothesis.
\newblock {\em Physical Review Letters}, 77(21):4334, 1996.

\bibitem[GC95]{GallavottiCohen95}
G.~Gallavotti and E.~G.~D. Cohen.
\newblock Dynamical ensembles in nonequilibrium statistical mechanics.
\newblock {\em Physical Review Letters}, 74(14):2694, 1995.

\bibitem[GCdRC84]{Garcia-ColinRio-Correa84}
L.~S. Garc{\'\i}a-Col{\'\i}n and J.~L. del Rio-Correa.
\newblock Further generalization of the {O}nsager reciprocity theorem.
\newblock {\em Physical Review A}, 30(6):3314, 1984.

\bibitem[GKM20]{GladbackKopferMaas2020}
P.~Gladbach, E.~Kopfer, and J.~Maas.
\newblock Scaling limits of discrete optimal transport.
\newblock {\em SIAM J. Math. Anal.}, 52(3):2759--2802, 2020.
\newblock \doi{10.1137/19M1243440}.

\bibitem[GKMP20]{GladbackKopferMaasPortinale2020}
P.~Gladbach, E.~Kopfer, J.~Maas, and L.~Portinale.
\newblock Homogenisation of one-dimensional discrete optimal transport.
\newblock {\em Journal de Math\'{e}matiques Pures et Appliqu\'{e}es.},
  139:204--234, 2020.
\newblock \doi{10.1016/j.matpur.2020.02.008}.

\bibitem[GKZD00]{GorbanKarlinZmievskiiDymova2000}
A.~Gorban, I.~Karlin, V.~Zmievskii, and S.~Dymova.
\newblock Reduced description in the reaction kinetics.
\newblock {\em Physica A: Statistical Mechanics and its Applications},
  275(3):361--379, 2000.
\newblock \doi{10.1016/S0378-4371(99)00402-1}.

\bibitem[GLPN93]{GalloLongoPallottino1993}
G.~Gallo, G.~Longo, S.~Pallottino, and S.~Nguyen.
\newblock Directed hypergraphs and applications: {C}ombinatorial strcutures and
  algorithms.
\newblock {\em Discrete Appl. Math.}, 42:177--201, 1993.
\newblock \doi{10.1016/0166-218X(93)90045-P}.

\bibitem[GNRK16]{GahnNeuss-RaduKnabner16}
M.~Gahn, M.~Neuss-Radu, and P.~Knabner.
\newblock Homogenization of reaction--diffusion processes in a two-component
  porous medium with nonlinear flux conditions at the interface.
\newblock {\em SIAM Journal on Applied Mathematics}, 76(5):1819--1843, 2016.

\bibitem[GNRK17]{GahnNeuss-RaduKnabner17}
M.~Gahn, M.~Neuss-Radu, and P.~Knabner.
\newblock Derivation of effective transmission conditions for domains separated
  by a membrane for different scaling of membrane diffusivity.
\newblock {\em Discrete \& Continuous Dynamical Systems-S}, 10(4):773, 2017.

\bibitem[GNRK18]{GahnNeuss-RaduKnabner18}
M.~Gahn, M.~Neuss-Radu, and P.~Knabner.
\newblock Effective interface conditions for processes through thin
  heterogeneous layers with nonlinear transmission at the microscopic
  bulk-layer interface.
\newblock {\em Networks \& Heterogeneous Media}, 13(4):609, 2018.

\bibitem[Gri18]{Grimmett2018}
G.~Grimmett.
\newblock {\em Probability on Graphs: {R}andom processes on graphs and
  lattices}, volume~8 of {\em Institute of Mathematical Statistics Textbooks}.
\newblock Cambridge University Press, Cambridge, second edition, 2018.
\newblock \doi{10.1017/9781108528986}.

\bibitem[Grm93]{Grmela93}
M.~Grmela.
\newblock Weakly nonlocal hydrodynamics.
\newblock {\em Physical Review E}, 47(1), 1993.

\bibitem[Grm02]{Grmela2002}
M.~Grmela.
\newblock Reciprocity relations in thermodynamics.
\newblock {\em Physica A: Statistical Mechanics and its Applications},
  309(3):304--328, 2002.
\newblock \doi{10.1016/S0378-4371(02)00564-2}.

\bibitem[Grm10]{Grmela2010}
M.~Grmela.
\newblock Multiscale equilibrium and nonequilibrium thermodynamics in chemical
  engineering.
\newblock In D.~H. West and G.~Yablonsky, editors, {\em Advances in Chemical
  Engineering}, volume~39 of {\em Advances in Chemical Engineering}, chapter~2,
  pages 75--129. Academic Press, 2010.
\newblock \doi{10.1016/S0065-2377(10)39002-8}.

\bibitem[Gya70]{Gyarmati70}
I.~Gyarmati.
\newblock {\em Non-Equilibrium Thermodynamics}.
\newblock Springer, 1970.

\bibitem[Has70]{Hastings1970}
W.~K. Hastings.
\newblock Monte {C}arlo sampling methods using {M}arkov chains and their
  applications.
\newblock {\em Biometrika}, 57(1):97--109, 1970.
\newblock \doi{10.1093/biomet/57.1.97}.

\bibitem[HEL11]{HyonEisenbergLiu11}
Y.~Hyon, R.~Eisenberg, and C.~Liu.
\newblock A mathematical model for the hard sphere repulsion in ionic
  solutions.
\newblock {\em Communications in Mathematical Sciences}, 9(2):459--475, 2011.

\bibitem[Hey21]{Heydecker2021}
D.~Heydecker.
\newblock Large deviations of {K}ac's conservative particle system and energy
  non-conserving solutions to the {B}oltzmann equation: A counterexample to the
  predicted rate function.
\newblock {\em Preprint arXiv:2103.14550}, 2021.

\bibitem[HFEL12]{HyonFonsecaEisenbergLiu12}
Y.~Hyon, J.~E. Fonseca, B.~Eisenberg, and C.~Liu.
\newblock Energy variational approach to study charge inversion (layering) near
  charged walls.
\newblock {\em Discrete and Continuous Dynamical Systems -- Series B},
  17(8):2725--2743, 2012.

\bibitem[HG82]{HurleyGarrod82}
J.~Hurley and C.~Garrod.
\newblock Generalization of the {O}nsager reciprocity theorem.
\newblock {\em Physical Review Letters}, 48(23):1575, 1982.

\bibitem[HKN04]{HelfferKleinNier04}
B.~Helffer, M.~Klein, and F.~Nier.
\newblock Quantitative analysis of metastability in reversible diffusion
  processes via a {W}itten complex approach.
\newblock {\em Matem{\'a}tica contempor{\^a}nea}, 26:41--85, 2004.

\bibitem[HKS21]{HeidaKantnerStephan2021}
M.~Heida, M.~Kantner, and A.~Stephan.
\newblock Consistency and convergence for a family of finite volume
  discretizations of the {F}okker-{P}lanck operator.
\newblock {\em ESAIM Math. Model. Numer. Anal.}, 55(6):3017--3042, 2021.
\newblock \doi{10.1051/m2an/2021078}.

\bibitem[HLLE12]{HorngLinLiuEisenberg12}
T.-L. Horng, T.-C. Lin, C.~Liu, and B.~Eisenberg.
\newblock {PNP} equations with steric effects: A model of ion flow through
  channels.
\newblock {\em The Journal of Physical Chemistry B}, 116(37):11422--11441,
  2012.

\bibitem[HN06]{HelfferNier06}
B.~Helffer and F.~Nier.
\newblock Quantitative analysis of metastability in reversible diffusion
  processes via a {W}itten complex approach: the case with boundary.
\newblock {\em M\'{e}m. Soc. Math. Fr. (N.S.)}, 105:vi+89, 2006.
\newblock \doi{10.24033/msmf.417}.

\bibitem[HN11]{HerrmannNiethammer11}
M.~Herrmann and B.~Niethammer.
\newblock {Kramers' formula for chemical reactions in the context of
  Wasserstein gradient flows}.
\newblock {\em Communications in Mathematical Sciences}, 9(2):623--635, 2011.

\bibitem[Hor97]{Hornung97}
U.~Hornung.
\newblock {\em {Homogenization and Porous Media}}.
\newblock Springer Verlag, 1997.

\bibitem[HT19]{HuesmannTrevisan2019}
M.~Huesmann and D.~Trevisan.
\newblock A {B}enamou-{B}renier formulation of martingale optimal transport.
\newblock {\em Bernoulli}, 25(4A):2729--2757, 2019.
\newblock \doi{10.3150/18-BEJ1069}.

\bibitem[HT22]{HraivoronskaTse2022TR}
A.~Hraivoronska and O.~Tse.
\newblock Diffusive limit of random walks on tessellations via generalized
  gradient flows.
\newblock {\em Preprint arXiv:2202.06024}, 2022.

\bibitem[HUL93]{HiriartUrrutyLemarechal1993}
J.-B. Hiriart-Urruty and C.~Lemar{\'{e}}chal.
\newblock {\em {Convex Analysis and Minimization Algorithms I}}.
\newblock Springer Berlin Heidelberg, 1993.
\newblock \doi{10.1007/978-3-662-02796-7}.

\bibitem[HVNVW16]{HytonenVanNeervenVeraarWeis16-I}
T.~Hyt{\"o}nen, J.~Van~Neerven, M.~Veraar, and L.~Weis.
\newblock {\em {Analysis in Banach spaces, Volume I: Martingales and
  Littlewood-Paley Theory}}, volume~63 of {\em A Series of Modern Surveys in
  Mathematics}.
\newblock Springer, 2016.

\bibitem[JKO97]{JordanKinderlehrerOtto97}
R.~Jordan, D.~Kinderlehrer, and F.~Otto.
\newblock {Free energy and the Fokker-Planck equation}.
\newblock {\em Physica D: Nonlinear Phenomena}, 107(2-4):265--271, 1997.

\bibitem[JKO98]{JordanKinderlehrerOtto98}
R.~Jordan, D.~Kinderlehrer, and F.~Otto.
\newblock The variational formulation of the {F}okker-{P}lanck equation.
\newblock {\em SIAM J. Math. Anal.}, 29(1):1--17, 1998.
\newblock \doi{10.1137/S0036141096303359}.

\bibitem[JM19]{JostMulas2019}
J.~Jost and R.~Mulas.
\newblock Hypergraph {L}aplace operators for chemical reaction networks.
\newblock {\em Adv. Math.}, 351:870--896, 2019.
\newblock \doi{10.1016/j.aim.2019.05.025}.

\bibitem[Ken99]{Kennelly1899}
A.~E. Kennelly.
\newblock The equivalence of triangles and three-pointed stars in conducting
  networks.
\newblock {\em Electrical World and Engineer}, 34(12):413--414, 1899.

\bibitem[KHT09]{KlamtHausTheis2009}
S.~Klamt, U.-U. Haus, and F.~Theis.
\newblock Hypergraphs and cellular networks.
\newblock {\em PLoS Comput. Biol.}, 5(5):e1000385, 6, 2009.
\newblock \doi{10.1371/journal.pcbi.1000385}.

\bibitem[KJZ18]{KaiserJackZimmer18}
M.~Kaiser, R.~L. Jack, and J.~Zimmer.
\newblock Canonical structure and orthogonality of forces and currents in
  irreversible {M}arkov chains.
\newblock {\em Journal of Statistical Physics}, 170(6):1019--1050, 2018.

\bibitem[Kra40]{Kramers1940}
H.~A. Kramers.
\newblock Brownian motion in a field of force and the diffusion model of
  chemical reactions.
\newblock {\em Physica}, 7(4):284--304, 4 1940.
\newblock \doi{10.1016/s0031-8914(40)90098-2}.

\bibitem[Kro39]{Kron1939}
G.~Kron.
\newblock {\em {Tensor Analysis of Networks}}.
\newblock John Wiley \& Sons, New York, 1939.
\newblock \urlprefix\url{https://hdl.handle.net/2027/mdp.39015017568331}.

\bibitem[Lai87]{Laidler1987}
K.~Laidler.
\newblock {\em Chemical Kinetics}.
\newblock Harper \& Row, New York, 1987.

\bibitem[LB15]{BinderLandau2015}
D.~P. Landau and K.~Binder.
\newblock {\em A guide to {M}onte {C}arlo simulations in statistical physics}.
\newblock Cambridge University Press, Cambridge, fourth edition, 2015.

\bibitem[L{\'e}o95]{Leonard95a}
C.~L{\'e}onard.
\newblock On large deviations for particle systems associated with spatially
  homogeneous {B}oltzmann type equations.
\newblock {\em Probability Theory and Related Fields}, 101(1):1--44, 1995.

\bibitem[Lim20]{Lim20}
L.-H. Lim.
\newblock Hodge {L}aplacians on graphs.
\newblock {\em Siam Review}, 62(3):685--715, 2020.

\bibitem[LMPR17]{LieroMielkePeletierRenger17}
M.~Liero, A.~Mielke, M.~A. Peletier, and D.~R.~M. Renger.
\newblock On microscopic origins of generalized gradient structures.
\newblock {\em Discrete and Continuous Dynamical Systems-Series S}, 10(1):1,
  2017.

\bibitem[LMT15]{LandimMisturiniTsunoda2015}
C.~Landim, R.~Misturini, and K.~Tsunoda.
\newblock Metastability of reversible random walks in potential fields.
\newblock {\em J. Stat. Phys.}, 160(6):1449--1482, 2015.
\newblock \doi{10.1007/s10955-015-1298-6}.

\bibitem[LP16]{LyonsPeres2016}
R.~Lyons and Y.~Peres.
\newblock {\em Probability on trees and networks}, volume~42 of {\em Cambridge
  Series in Statistical and Probabilistic Mathematics}.
\newblock Cambridge University Press, New York, 2016.
\newblock \doi{10.1017/9781316672815}.

\bibitem[LP17]{LevinPeres2017}
D.~A. Levin and Y.~Peres.
\newblock {\em Markov Chains and Mixing Times}, volume 107.
\newblock American Mathematical Soc., 2017.

\bibitem[LSU68]{LSU1968}
O.~A. Lady\v{z}enskaja, V.~A. Solonnikov, and N.~N. Ural'ceva.
\newblock {\em Linear and Quasilinear Equations of Parabolic Type}.
\newblock Translations of Mathematical Monographs, Vol. 23. American
  Mathematical Society, Providence, R.I., 1968.
\newblock Translated from the Russian by S. Smith.

\bibitem[Maa11]{Maas11}
J.~Maas.
\newblock Gradient flows of the entropy for finite {M}arkov chains.
\newblock {\em Journal of Functional Analysis}, 261(8):2250--2292, 2011.

\bibitem[Mae20]{Maes2020}
C.~Maes.
\newblock Frenesy: Time-symmetric dynamical activity in nonequilibria.
\newblock {\em Phys. Rep.}, 850:1--33, mar 2020.
\newblock \doi{10.1016/j.physrep.2020.01.002}.

\bibitem[Mar15]{Marcelin15}
M.~R. Marcelin.
\newblock Contribution {\`a} l'{\'e}tude de la cin{\'e}tique physico-chimique.
\newblock In {\em Annales de physique}, volume~9, pages 120--231. EDP Sciences,
  1915.

\bibitem[Mie11]{Mielke11}
A.~Mielke.
\newblock A gradient structure for reaction-diffusion systems and for
  energy-drift-diffusion systems.
\newblock {\em Nonlinearity}, 24:1329--1346, 2011.

\bibitem[Mie16]{Mielke16a}
A.~Mielke.
\newblock On evolutionary {$\Gamma$}-convergence for gradient systems.
\newblock In {\em Macroscopic and Large Scale Phenomena: Coarse Graining, Mean
  Field Limits and Ergodicity}, pages 187--249. Springer, 2016.

\bibitem[MM20]{MaasMielke20}
J.~Maas and A.~Mielke.
\newblock Modeling of chemical reaction systems with detailed balance using
  gradient structures.
\newblock {\em Journal of Statistical Physics}, 181(6):2257--2303, 2020.

\bibitem[MMP21]{MielkeMontefuscoPeletier21}
A.~Mielke, A.~Montefusco, and M.~A. Peletier.
\newblock Exploring families of energy-dissipation landscapes via tilting:
  Three types of {EDP} convergence.
\newblock {\em Continuum Mechanics and Thermodynamics}, 33:611--637, 2021.

\bibitem[MN07]{MaesNetocny07}
C.~Maes and K.~Neto{\v{c}}n{\`y}.
\newblock Minimum entropy production principle from a dynamical fluctuation
  law.
\newblock {\em Journal of mathematical physics}, 48(5):053306, 2007.

\bibitem[Mor70]{Moreau70}
J.~J. Moreau.
\newblock Sur les lois de frottement, de plasticit{\'e} et de viscosit{\'e}.
\newblock {\em Comptes rendus hebdomadaires des s{\'e}ances de l'Acad{\'e}mie
  des sciences}, 271:608--611, 1970.

\bibitem[MPR14]{MielkePeletierRenger14}
A.~Mielke, M.~A. Peletier, and D.~R.~M. Renger.
\newblock On the relation between gradient flows and the large-deviation
  principle, with applications to {M}arkov chains and diffusion.
\newblock {\em Potential Analysis}, 41(4):1293--1327, 2014.

\bibitem[MPR16]{MielkePeletierRenger16}
A.~Mielke, M.~A. Peletier, and D.~R.~M. Renger.
\newblock A generalization of {O}nsager's reciprocity relations to gradient
  flows with nonlinear mobility.
\newblock {\em Journal of Non-Equilibrium Thermodynamics}, 41(2):141--149,
  2016.

\bibitem[MPS21]{MielkePeletierStephan21}
A.~Mielke, M.~A. Peletier, and A.~Stephan.
\newblock {EDP}-convergence for nonlinear fast--slow reaction systems with
  detailed balance.
\newblock {\em Nonlinearity}, 34(8):5762, 2021.

\bibitem[MR15]{MielkeRoubicek15}
A.~Mielke and T.~Roub{\'\i}cek.
\newblock {\em {Rate-Independent Systems}}.
\newblock Springer, 2015.

\bibitem[MRR{\etalchar{+}}53]{Metropolis1953}
N.~Metropolis, A.~W. Rosenbluth, M.~N. Rosenbluth, A.~H. Teller, and E.~Teller.
\newblock Equation of state calculations by fast computing machines.
\newblock {\em The Journal of Chemical Physics}, 21(6):1087--1092, June 1953.
\newblock \doi{10.1063/1.1699114}.

\bibitem[MRS09]{MielkeRossiSavare09}
A.~Mielke, R.~Rossi, and G.~Savar{\'e}.
\newblock Modeling solutions with jumps for rate-independent systems on metric
  spaces.
\newblock {\em Discrete and Continuous Dynamical Systems A}, 25(2), 2009.

\bibitem[MRS12a]{MielkeRossiSavare12a}
A.~Mielke, R.~Rossi, and G.~Savar{\'e}.
\newblock {BV} solutions and viscosity approximations of rate-independent
  systems.
\newblock {\em ESAIM: Control, Optimisation and Calculus of Variations},
  18(01):36--80, 2012.

\bibitem[MRS12b]{MielkeRossiSavare12}
A.~Mielke, R.~Rossi, and G.~Savar{\'e}.
\newblock Variational convergence of gradient flows and rate-independent
  evolutions in metric spaces.
\newblock {\em Milan Journal of Mathematics}, 80(2):381--410, 2012.

\bibitem[MRS13]{MielkeRossiSavare13}
A.~Mielke, R.~Rossi, and G.~Savar{\'e}.
\newblock Nonsmooth analysis of doubly nonlinear evolution equations.
\newblock {\em Calculus of Variations and Partial Differential Equations},
  46:253--310, 2013.

\bibitem[MS94]{MarzShortt94}
M.~M{\"a}rz and R.~M. Shortt.
\newblock Weak convergence of vector measures.
\newblock {\em Publicationes Mathematicae Debrecen}, 45:71--92, 1994.

\bibitem[MS06]{MartyushevSeleznev06}
L.~Martyushev and V.~Seleznev.
\newblock {Maximum entropy production principle in physics, chemistry and
  biology}.
\newblock {\em Physics reports}, 426(1):1--45, 2006.

\bibitem[MS20]{MielkeStephan2020}
A.~Mielke and A.~Stephan.
\newblock Coarse-graining via {EDP}-convergence for linear fast-slow reaction
  systems.
\newblock {\em Math. Models Methods Appl. Sci.}, 30(9):1765--1807, 2020.
\newblock \doi{10.1142/S0218202520500360}.
\newblock Erratum for Lemma 3.4 under arXiv:1911.06234.

\bibitem[MST89]{MarinoSacconTosques89}
A.~Marino, C.~Saccon, and M.~Tosques.
\newblock Curves of maximal slope and parabolic variational inequalities on
  nonconvex constraints.
\newblock {\em Ann. Scuola Norm. Sup. Pisa Cl. Sci. (4)}, 16(2):281--330, 1989.

\bibitem[MTL02]{MielkeTheilLevitas02}
A.~Mielke, F.~Theil, and V.~I. Levitas.
\newblock A variational formulation of rate-independent phase transformations
  using an extremum principle.
\newblock {\em Archive for Rational Mechanics and Analysis}, 162(2):137--177,
  2002.

\bibitem[Nit06]{Nitzan06}
A.~Nitzan.
\newblock {\em Chemical Dynamics in Condensed Phases: Relaxation, Transfer and
  Reactions in Condensed Molecular Systems}.
\newblock Oxford University Press, 2006.

\bibitem[NRJ07]{Neuss-RaduJager07}
M.~Neuss-Radu and W.~J{\"a}ger.
\newblock Effective transmission conditions for reaction-diffusion processes in
  domains separated by an interface.
\newblock {\em SIAM Journal on Mathematical Analysis}, 39(3):687--720, 2007.

\bibitem[Ohm27]{Ohm1827}
G.~S. Ohm.
\newblock {\em Die galvanische Kette: mathematisch bearbeitet}.
\newblock T.H. Riemann Berlin, 1827.

\bibitem[OM53]{OnsagerMachlup1953}
L.~Onsager and S.~Machlup.
\newblock Fluctuations and irreversible processes.
\newblock {\em Physical Review}, 91(6):1505--1512, 1953.

\bibitem[Ons31]{Onsager31}
L.~Onsager.
\newblock Reciprocal relations in irreversible processes {I} \& {II}.
\newblock {\em Physical Review}, 37:405--426 and 38:2265--2279, 1931.

\bibitem[{\"O}tt97]{Ottinger97}
H.~C. {\"O}ttinger.
\newblock {GENERIC} formulation of {B}oltzmann's kinetic equation.
\newblock {\em Journal of Non-Equilibrium Thermodynamics}, 22(4):386--391,
  1997.

\bibitem[Ott01]{Otto01}
F.~Otto.
\newblock The geometry of dissipative evolution equations: The porous medium
  equation.
\newblock {\em Communications in Partial Differential Equations}, 26:101--174,
  2001.

\bibitem[{\"O}tt19]{Ottinger19}
H.~C. {\"O}ttinger.
\newblock {On the Combined Use of Friction Matrices and Dissipation Potentials
  in Thermodynamic Modeling}.
\newblock {\em Journal of Non-Equilibrium Thermodynamics}, 44(3):295--302,
  2019.

\bibitem[OV00]{OttoVillani2000}
F.~Otto and C.~Villani.
\newblock Generalization of an inequality by {T}alagrand and links with the
  logarithmic {S}obolev inequality.
\newblock {\em Journal of Functional Analysis}, 173(2):361--400, 2000.
\newblock \doi{10.1006/jfan.1999.3557}.

\bibitem[Pay61]{Paynter61}
H.~M. Paynter.
\newblock {\em {Analysis and Design of Engineering Systems}}.
\newblock MIT press, 1961.

\bibitem[Pek05]{Pekar05}
M.~Peka{\v{r}}.
\newblock Thermodynamics and foundations of mass-action kinetics.
\newblock {\em Progress in Reaction Kinetics and Mechanism}, 30(1-2):3--113,
  2005.

\bibitem[Pel14]{PeletierVarMod14TR}
M.~A. Peletier.
\newblock {Variational Modelling: Energies, Gradient Flows, and Large
  Deviations}.
\newblock {\em Preprint arXiv:1402:1990}, 2014.

\bibitem[Pet17]{Peters17}
B.~Peters.
\newblock {\em Reaction rate theory and rare events}.
\newblock Elsevier, 2017.

\bibitem[PR21]{PeletierRenger21}
M.~A. Peletier and D.~R.~M. Renger.
\newblock Fast reaction limits via {$\Gamma$}-convergence of the flux rate
  functional.
\newblock {\em Journal of Dynamics and Differential Equations}, pages 1--42,
  2021.

\bibitem[PRS21]{PattersonRengerSharma21TR}
R.~I.~A. Patterson, D.~R. Renger, and U.~Sharma.
\newblock Variational structures beyond gradient flows: a macroscopic
  fluctuation-theory perspective.
\newblock {\em Preprint arXiv:2103.14384}, 2021.

\bibitem[PRST22]{PeletierRossiSavareTse22}
M.~A. Peletier, R.~Rossi, G.~Savar{\'e}, and O.~Tse.
\newblock Jump processes as generalized gradient flows.
\newblock {\em Calculus of Variations and Partial Differential Equations},
  61(1):1--85, 2022.

\bibitem[PRV14]{PeletierRedigVafayi14}
M.~A. Peletier, F.~Redig, and K.~Vafayi.
\newblock Large deviations in stochastic heat-conduction processes provide a
  gradient-flow structure for heat conduction.
\newblock {\em Journal of Mathematical Physics}, 55(9):093301, 2014.

\bibitem[PS21]{PeletierSchlottke21TR}
M.~A. Peletier and M.~C. Schlottke.
\newblock Gamma-convergence of a gradient-flow structure to a non-gradient-flow
  structure.
\newblock {\em Preprint arXiv:2105.03401}, 2021.

\bibitem[PSV10]{PeletierSavareVeneroni10}
M.~A. Peletier, G.~Savar\'e, and M.~Veneroni.
\newblock From diffusion to reaction via {G}amma-convergence.
\newblock {\em SIAM Journal on Mathematical Analysis}, 42(4):1805--1825, 2010.

\bibitem[Ray13]{Rayleigh13}
L.~Rayleigh.
\newblock On the motion of a viscous fluid.
\newblock {\em The London, Edinburgh, and Dublin Philosophical Magazine and
  Journal of Science}, 26(154):776--786, 1913.

\bibitem[Ren18]{Renger18}
D.~R.~M. Renger.
\newblock Flux large deviations of independent and reacting particle systems,
  with implications for macroscopic fluctuation theory.
\newblock {\em Journal of Statistical Physics}, 172(5):1291--1326, 2018.

\bibitem[Ris84]{Risken84}
H.~Risken.
\newblock {\em The {F}okker-{P}lanck equation: {M}ethods of solutions and
  applications}, volume~18 of {\em Springer Series in Synergetics}.
\newblock Springer-Verlag, Berlin, 1984.
\newblock \doi{10.1007/978-3-642-96807-5}.

\bibitem[Roc66]{Rockafellar66}
R.~T. Rockafellar.
\newblock Characterization of the subdifferentials of convex functions.
\newblock {\em Pacific Journal of Mathematics}, 17(3):497--510, 1966.

\bibitem[Roc70]{Rockafellar1970}
R.~T. Rockafellar.
\newblock {\em {Convex Analysis}}.
\newblock Princeton Mathematical Series, No. 28. Princeton University Press,
  Princeton, N.J., 1970.
\newblock \doi{10.1515/9781400873173}.

\bibitem[Rou13]{Roubicek13}
T.~Roub{\'\i}{\v{c}}ek.
\newblock {\em Nonlinear Partial Differential Equations with Applications},
  volume 153.
\newblock Springer Science \& Business Media, 2013.

\bibitem[RW98]{Rockafellar-Wets98}
R.~T. Rockafellar and R.~J.-B. Wets.
\newblock {\em Variational {A}nalysis}.
\newblock Springer-Verlag, Berlin, 1998.

\bibitem[RZ15]{ReinaZimmer15}
C.~Reina and J.~Zimmer.
\newblock Entropy production and the geometry of dissipative evolution
  equations.
\newblock {\em Physical Review E}, 92(5):052117, 2015.

\bibitem[Sav07]{Savare07}
G.~Savar{\'e}.
\newblock Gradient flows and diffusion semigroups in metric spaces under lower
  curvature bounds.
\newblock {\em Comptes Rendus Mathematique}, 345(3):151--154, 2007.

\bibitem[Sch19]{Schlichting2019}
A.~Schlichting.
\newblock Macroscopic limit of the {B}ecker-{D}\"{o}ring equation \textit{via}
  gradient flows.
\newblock {\em ESAIM Control Optim. Calc. Var.}, 25:Paper No. 22, 36, 2019.
\newblock \doi{10.1051/cocv/2018011}.

\bibitem[Sch20]{S-EDG2018}
A.~Schlichting.
\newblock The exchange-driven growth model: basic properties and longtime
  behavior.
\newblock {\em J. Nonlinear Sci.}, 30(3):793--830, 2020.
\newblock \doi{10.1007/s00332-019-09592-x}.

\bibitem[Sei12]{Seifert12}
U.~Seifert.
\newblock Stochastic thermodynamics, fluctuation theorems and molecular
  machines.
\newblock {\em Reports on Progress in Physics}, 75(12):126001, 2012.

\bibitem[Ser11]{Serfaty11}
S.~Serfaty.
\newblock Gamma-convergence of gradient flows on {H}ilbert and metric spaces
  and applications.
\newblock {\em Discrete and Continuous Dynamical Systems A}, 31(4):1427--1451,
  2011.

\bibitem[SG69]{SG1969}
D.~L. Scharfetter and H.~K. Gummel.
\newblock {Large-signal analysis of a silicon Read diode oscillator}.
\newblock {\em IEEE Trans. Electron Devices}, 16(1):64--77, jan jan 1969.
\newblock \doi{10.1109/T-ED.1969.16566}.

\bibitem[She47]{Shew1947}
D.~W.~C. Shen.
\newblock Generalized star and mesh transformations.
\newblock {\em The London, Edinburgh, and Dublin Philosophical Magazine and
  Journal of Science}, 38(279):267--275, April 1947.
\newblock \doi{10.1080/14786444708521594}.

\bibitem[She85]{Sheu85}
S.-J. Sheu.
\newblock Stochastic control and exit probabilities of jump processes.
\newblock {\em SIAM journal on control and optimization}, 23(2):306--328, 1985.

\bibitem[SJ14]{Van-Der-SchaftJeltsema14}
A.~v.~d. Schaft and D.~Jeltsema.
\newblock Port-{H}amiltonian systems theory: An introductory overview.
\newblock {\em Foundations and Trends in Systems and Control}, 1(2-3):173--378,
  2014.

\bibitem[Smo16]{Smoluchowski1916}
M.~V. Smoluchowski.
\newblock {Drei Vortrage uber Diffusion, Brownsche Bewegung und Koagulation von
  Kolloidteilchen}.
\newblock {\em Phys. Zeitschrift}, 17:585--599, 1916.

\bibitem[SR09]{Saint-Raymond09}
L.~Saint-Raymond.
\newblock {\em Hydrodynamic Limits of the Boltzmann Equation}.
\newblock Number 1971 in Lecture Notes in Mathematics. Springer Science \&
  Business Media, 2009.

\bibitem[SS04]{SandierSerfaty04}
E.~Sandier and S.~Serfaty.
\newblock {Gamma-convergence of gradient flows with applications to
  Ginzburg-Landau}.
\newblock {\em Communications on Pure and Applied Mathematics},
  57(12):1627--1672, 12 2004.
\newblock \doi{10.1002/cpa.20046}.

\bibitem[SS19]{SchlichtingSlowik2019}
A.~Schlichting and M.~Slowik.
\newblock Poincar\'{e} and logarithmic {S}obolev constants for metastable
  {M}arkov chains via capacitary inequalities.
\newblock {\em Ann. Appl. Probab.}, 29(6):3438--3488, 2019.
\newblock \doi{10.1214/19-AAP1484}.

\bibitem[SS21]{SchlichtingSeis2021}
A.~Schlichting and C.~Seis.
\newblock {The Scharfetter--Gummel scheme for aggregation--diffusion
  equations}.
\newblock {\em IMA Journal of Numerical Analysis}, 05 2021.
\newblock \doi{10.1093/imanum/drab039}.

\bibitem[ST20]{SeoTabrizian20}
I.~Seo and P.~Tabrizian.
\newblock {Asymptotics for scaled Kramers--Smoluchowski equations in several
  dimensions with general potentials}.
\newblock {\em Calculus of Variations and Partial Differential Equations},
  59(1):1--21, 2020.

\bibitem[Ste21]{Stephan21}
A.~Stephan.
\newblock {EDP}-convergence for a linear reaction-diffusion system with fast
  reversible reaction.
\newblock {\em Calculus of Variations and Partial Differential Equations},
  60(6):1--35, 2021.

\bibitem[Tru89]{Truemper1989}
K.~Truemper.
\newblock On the delta-wye reduction for planar graphs.
\newblock {\em J. Graph Theory}, 13(2):141--148, 1989.
\newblock \doi{10.1002/jgt.3190130202}.

\bibitem[Tru92]{Truemper1992}
K.~Truemper.
\newblock {\em Matroid decomposition}.
\newblock Academic Press, Inc., Boston, MA, 1992.

\bibitem[TSMA19]{Torres-SanchezMillanArroyo19}
A.~Torres-S{\'a}nchez, D.~Mill{\'a}n, and M.~Arroyo.
\newblock Modelling fluid deformable surfaces with an emphasis on biological
  interfaces.
\newblock {\em Journal of fluid mechanics}, 872:218--271, 2019.

\bibitem[Vil08]{Villani2008}
C.~Villani.
\newblock {\em Entropy Production and Convergence to Equilibrium}, volume 1916
  of {\em Lecture Notes in Mathematics}.
\newblock Springer Berlin Heidelberg, Berlin, Heidelberg, 2008.
\newblock \doi{10.1007/978-3-540-73705-6\_1}.

\bibitem[Vis96]{Visintin96}
A.~Visintin.
\newblock {\em {Models of Phase Transitions}}.
\newblock Birkh{\"a}user, 1996.

\bibitem[Weg01]{Wegscheider1901}
R.~Wegscheider.
\newblock {{\"U}ber simultane Gleichgewichte und die Beziehungen zwischen
  Thermodynamik und Reactionskinetik homogener Systeme}.
\newblock {\em Monatshefte f\"ur Chemie}, 22(8):849--906, 8 1901.
\newblock \doi{10.1007/bf01517498}.

\bibitem[XDD16]{XuDiDoi16}
X.~Xu, Y.~Di, and M.~Doi.
\newblock Variational method for liquids moving on a substrate.
\newblock {\em Physics of Fluids}, 28(8):087101, 2016.

\bibitem[ZD18]{ZhouDoi18}
J.~Zhou and M.~Doi.
\newblock Dynamics of viscoelastic filaments based on {O}nsager principle.
\newblock {\em Physical Review Fluids}, 3(8):084004, 2018.

\bibitem[Zie58]{Ziegler58}
H.~Ziegler.
\newblock An attempt to generalize {O}nsager's principle, and its significance
  for rheological problems.
\newblock {\em Zeitschrift f{\"u}r angewandte Mathematik und Physik ZAMP},
  9(5-6):748--763, 1958.

\bibitem[Zie95]{Ziegler1995}
G.~M. Ziegler.
\newblock {\em Lectures on Polytopes}, volume 152 of {\em Graduate Texts in
  Mathematics}.
\newblock Springer-Verlag, New York, 1995.
\newblock \doi{10.1007/978-1-4613-8431-1}.

\end{thebibliography}
\bibliographystyle{alphainitialsdoi}

\end{document}